\documentclass[10pt]{amsart}

\usepackage{latexsym,exscale,enumerate}
\usepackage{amsmath,amsthm,amsfonts,amssymb,amscd, stmaryrd}
\usepackage{hyperref}
\usepackage{cite}

\usepackage[all]{xy}
\SelectTips{cm}{}

\usepackage{graphicx}
\usepackage{tikz}
\usetikzlibrary{decorations.markings}
\usetikzlibrary{decorations.pathreplacing}
\usepackage[outline]{contour}
\contourlength{1.2pt}

\tikzset{->-/.style={decoration={
  markings,
  mark=at position #1 with {\arrow{>}}},postaction={decorate}}}

\tikzset{middlearrow/.style={
        decoration={markings,
            mark= at position 0.5 with {\arrow{#1}} ,
        },
        postaction={decorate}
    }
}
\newcommand{\bbullet}{
\begin{tikzpicture}
  \draw[fill=black] circle (0.55ex);
\end{tikzpicture}
}

\newcommand{\hackcenter}[1]{
 \xy (0,0)*{#1}; \endxy}

\normalfont\upshape

\def\one{1}
\def\onel{1_\lambda}
\def\oneltwo{1_{\lambda+2}}
\def\oneb{\mathbf{1}}
\def\onebl{\mathbf{1}_{\lambda}}
\def\onebltwo{\mathbf{1}_{\lambda+2}}
\def\onebb{\mathbbm{1}}
\def\onebbl{\mathbbm{1}_{\lambda}}
\def\onebbltwo{\mathbbm{1}_{\lambda+2}}

\def\U{U_q}
\def\Upi{U_{q,\pi}}
\def\AU{_{\Ac}U_q}
\def\AUpi{_{\Ac}U_{q,\pi}}
\def\Udot{\dot{U}_q}
\def\Udotpi{\dot{U}_{q,\pi}}
\def\AUdot{_{\Ac}\dot{U}_q}
\def\AUdotpi{_{\Ac}\dot{U}_{q,\pi}}
\def\Uc{\mathcal{U}_{\pi}}
\def\Ucev{\mathcal{U}}
\def\Udotc{\dot{\mathcal{U}}_{\pi}}
\def\Udotcev{\dot{\mathcal{U}}}

\newcommand{\Zt}{\Z_{2}}
\newcommand{\Ett}{\mathtt{E}}
\newcommand{\Ftt}{\mathtt{F}}

\newcommand{\ds}[1]{\ensuremath{\la{#1}\ra}}
\newcommand{\ads}[1]{\ensuremath{\langle{#1}\rangle}}

\newcommand{\ONHc}{\mathcal{ONH}}
\newcommand{\SCat}{\mathcal{SC}at}
\newcommand{\SBim}{\mathcal{SB}im}
\def\Cc{\mathcal{C}}
\def\Dc{\mathcal{D}}

\def\Ec{\mathcal{E}}
\def\Fc{\mathcal{F}}
\def\Gc{\mathcal{G}}

\newcommand{\Mat}{\mathrm{Mat}}
\newcommand{\Gr}{\mathrm{Gr}}

\def\nn{\notag}
\def\la{\langle}
\def\ra{\rangle}

\usepackage[all]{xy}
\SelectTips{cm}{}

\theoremstyle{definition}
\newtheorem{thm}{Theorem}[section]
\newtheorem{cor}[thm]{Corollary}
\newtheorem{conj}[thm]{Conjecture}
\newtheorem{lem}[thm]{Lemma}
\newtheorem{rem}[thm]{Remark}
\newtheorem{prop}[thm]{Proposition}
\newtheorem{defn}[thm]{Definition}
\newtheorem{example}[thm]{Example}

\numberwithin{equation}{section}

\usepackage{bbm}
\def\C{{\mathbbm C}}
\def\N{{\mathbbm N}}
\def\R{{\mathbbm R}}
\def\Z{{\mathbbm Z}}
\def\Q{{\mathbbm Q}}

\def\k{\Bbbk}

\newcommand{\Ac}{\mathcal{A}}

\newcommand{\Hom}{{\rm Hom}}
\newcommand{\HOM}{{\rm HOM}}
\renewcommand{\to}{\rightarrow}
\newcommand{\maps}{\colon}

\newcommand{\del}{\partial}
\newcommand{\Res}{{\rm Res}}
\newcommand{\End}{{\rm End}}
\newcommand{\END}{{\rm END}}

\def\sltwo{\mathfrak{sl}_2}

\def\Res{{\mathrm{Res}}}
\def\Ind{{\mathrm{Ind}}}
\def\lra{{\longrightarrow}}
\def\Id{\mathrm{Id}}
\def\mf{\mathfrak}
\def\shuffle{\,\raise 1pt\hbox{$\scriptscriptstyle\cup{\mskip
               -4mu}\cup$}\,}
\newcommand{\refequal}[1]{\xy {\ar@{=}^{#1}
(-1,0)*{};(1,0)*{}};
\endxy}

\newcommand{\osym}{\mathrm{O\Lambda}}

\newcommand{\spol}{\mathrm{SPol}}
\newcommand{\onh}{\mathrm{ONH}}

\newcommand{\Udown}{
 \xy (0,0)*{\begin{tikzpicture}[scale=0.75]
  \draw[semithick, <-] (0,0) -- (0,1);
\end{tikzpicture}};
(1.5,0)*{};(-1.5,0)*{}; \endxy
    }

\newcommand{\Uupdots}{
 \xy (0,0)*{\begin{tikzpicture}[scale=0.75]
  \draw[semithick, ->] (0,0) -- (0,1);
  \draw[color=blue, thick, dashed, double distance=1pt] (0,.4) to [out=180, in=-90] (-0.5,1);
  \node at (0,.4) {$\scs \bullet$};
\end{tikzpicture}};
(1.5,0)*{};(-2,0)*{}; \endxy
    }

\newcommand{\Ucapl}{\;\;
\begin{tikzpicture}[scale=0.5]
  \draw[semithick, ->] (0.5,0) .. controls (0.5,0.8) and (-0.5,0.8) .. (-0.5,0);
\end{tikzpicture}\;\; }

\newcommand{\Ucapr}{\;\;
\begin{tikzpicture}[scale=0.5]
  \draw[semithick, ->] (-0.5,0) .. controls (-0.5,0.65) and (0.5,0.65) .. (0.5,0)
      node[pos=0.5, shape=coordinate](X){};
  \draw[color=blue, thick, double distance=1pt, dashed] (X) -- (0,1);
\end{tikzpicture}\;\; }

\newcommand{\Ucupl}{\;\;
\xy
(0,0)*{
\begin{tikzpicture}[scale=0.5]
  \draw[semithick, ->] (0.5,1) .. controls (0.5,0.2) and (-0.5,0.2) .. (-0.5,1);
\end{tikzpicture}}; \endxy
\;\; }

\newcommand{\Ucupr}{\;\;
\xy
(0,0)*{
\begin{tikzpicture}[scale=0.5]
  \draw[semithick, ->] (-0.5,1) .. controls (-0.5,0.4) and (0.5,0.4) .. (0.5,1)
  node[pos=0.5, shape=coordinate](X){};;
  \draw[color=blue, thick, double distance=1pt, dashed] (X) -- (0,0);
\end{tikzpicture}}; \endxy
\;\; }

\newcommand{\Ucuprm}{\;\;
\xy
(0,0)*{
\begin{tikzpicture}[scale=0.5]
  \draw[semithick, ->] (-0.5,1) .. controls (-0.5,0.2) and (0.5,0.2) .. (0.5,1)
  node[pos=0.5, shape=coordinate](X){};;
  \draw[color=blue, thick, double distance=1pt, dashed] (X) -- (0,1.1);
\end{tikzpicture}}; \endxy
\;\; }

\newcommand{\Ucrossr}{\;\;
    \vcenter{\xy (0,0)*{\begin{tikzpicture}[scale=0.5]
  \draw[semithick, ->] (-0.5,0) .. controls (-0.5,0.5) and (0.5,0.5) .. (0.5,1)
      node[pos=0.5, shape=coordinate](X){};
    \draw[semithick, <-] (0.5,0) .. controls (0.5,0.5) and (-0.5,0.5) .. (-0.5,1);
  \draw[color=blue,  thick, dashed] (X) to (0,1);
\end{tikzpicture}}; \endxy} \;\; }

\newcommand{\Ucrossl}{\;\;
    \vcenter{\xy (0,0)*{\begin{tikzpicture}[scale=0.5]
  \draw[semithick, <-] (-0.5,0) .. controls (-0.5,0.5) and (0.5,0.5) .. (0.5,1)
      node[pos=0.5, shape=coordinate](X){};
  \draw[semithick, ->] (0.5,0) .. controls (0.5,0.5) and (-0.5,0.5) .. (-0.5,1);
  \draw[color=blue,  thick, dashed] (X) to (0,0);
\end{tikzpicture}}; \endxy} \;\; }

\newcommand{\sUup}{
 \xy (0,0)*{\begin{tikzpicture}[scale=0.45]
  \draw[semithick, ->] (0,0) -- (0,1);
\end{tikzpicture}};
(1.5,0)*{};(-1.5,0)*{}; \endxy
    }

\newcommand{\sUupb}{
 \xy (0,0)*{\begin{tikzpicture}[scale=0.45]
  \draw[color=blue, thick, double distance=1pt, dashed] (0,0) -- (0,1);
\end{tikzpicture}};
(1.5,0)*{};(-1.5,0)*{}; \endxy
    }

\newcommand{\sUupbb}{
 \xy (0,0)*{\begin{tikzpicture}[scale=0.45]
  \draw[color=blue, thick,  dashed] (0,0) -- (0,1);
\end{tikzpicture}};
(1.5,0)*{};(-1.5,0)*{}; \endxy
    }

\newcommand{\sUdown}{
 \xy (0,0)*{\begin{tikzpicture}[scale=0.45]
  \draw[semithick, <-] (0,0) -- (0,1);
\end{tikzpicture}};
(1.5,0)*{};(-1.5,0)*{}; \endxy
    }

\newcommand{\sUupdot}{
 \xy (0,0)*{\begin{tikzpicture}[scale=0.45]
  \draw[semithick, ->] (0,0) -- (0,1);
  \node at (0,.4) {$\scs \bullet$};
  \draw[color=blue, thick, dashed] (0,.4) to [out=180, in=-90] (-0.5,1.1);
\end{tikzpicture}};
(1.5,0)*{};(-2,0)*{}; \endxy
    }

\def\l{\lambda}

\newcommand{\scs}{\scriptstyle}
\numberwithin{equation}{section}

\def\emph#1{{\sl #1\/}}

\let\tilde=\widetilde

\let\phi=\varphi
\let\theta=\vartheta
\let\epsilon=\varepsilon

\usepackage{bbm}
\def\C{{\mathbbm C}}
\def\N{{\mathbbm N}}
\def\R{{\mathbbm R}}
\def\Z{{\mathbbm Z}}
\def\Q{{\mathbbm Q}}

\def\nn{\notag}

\def\la{\langle}
\def\ra{\rangle}

\newcommand{\ep}{\epsilon}

\newcommand{\qbin}[2]{
\left[
 \begin{array}{c}
 #1 \\
 #2 \\
 \end{array}
 \right]
}
\newcommand{\qbins}[2]{
\left[
 \begin{array}{c}
 \scs #1 \\
 \scs #2 \\
 \end{array}
 \right]
}

\allowdisplaybreaks

\title{An odd categorification of $U_q(\sltwo)$}

\author{Alexander P.\ Ellis}
\address{Department of Mathematics, Columbia University, New
York, NY 10027, USA}
\email{ellis@math.columbia.edu}

      \author{Aaron D.\ Lauda}
  \address{Department of Mathematics, University of Southern California, Los Angeles, CA 90089, USA}
  \email{lauda@usc.edu}

\date{February 27, 2014}

\begin{document}
%
% ==============================================================================

\maketitle
\begin{abstract}
We define a 2-category that categorifies the covering Kac-Moody algebra for $\mf{sl}_2$ introduced by Clark and Wang.  This categorification forms the structure of a super-2-category as formulated by Kang, Kashiwara, and Oh.   The super-2-category structure introduces a $\Z\times\Zt$-grading
giving its Grothendieck group the structure of a free module over the group algebra of $\Z\times\Zt$.  By specializing the $\Zt$-action to $+1$ or to $-1$, the construction specializes to an ``odd'' categorification of $\mf{sl}_2$ and to a supercategorification of $\mf{osp}_{1|2}$, respectively.
\end{abstract}

\setcounter{tocdepth}{1}
\tableofcontents

%\begin{classification}
%	20G42, 17B37, 17A70, 57M27.
%\end{classification}
%
%\begin{keywords}
%Covering algebras, categorified quantum groups, cycltomic quotients, odd nilHecke algebra, odd Khovanov homology.
%\end{keywords}

%#####################################################################
%
\section{Introduction}
%
%####################################################################

Categorical representation theory studies actions of Lie algebras and their associated quantum groups on categories, with generators acting by functors, and equations between elements lifting to isomorphisms of functors.  Natural transformations between these functors exemplify the higher structure in 2-representation theory that cannot be accessed in traditional representation theory.
 Solving a conjecture of Igor Frenkel, the second author introduced a 2-category $\Ucev=\Ucev(\mf{sl}_2)$ that categorifies the integral idempotented version $\AUdot$ of the quantum enveloping algebra of $\mf{sl}_2$~\cite{Lau1,Lau4}.  This 2-category governs the higher structure in categorical actions of $\Udot=\Udot(\mf{sl}_2)$.

The categorification $\Ucev$ of $\Udot$ is ubiquitous in 2-representation theory.
In many of the known categorical representations of $\Udot$ one can identify the higher structure needed to lift these actions to full 2-representations of the 2-category $\Ucev$.

Examples of such 2-representations include modules over cyclotomic KLR-algebras~\cite{KK,Kash,CL,Web}, blocks of parabolic category $\mathcal{O}$~\cite{CR,FKS,Sussan,HS}, foam categories~\cite{Mac,MPT,LQR}, and  derived categories of coherent sheaves on cotangent bundles to Grassmannians~\cite{CKL2,CKL3,CKL4,CL}.  This 2-category also leads to a categorification of quantum $\mf{sl}_2$ at prime roots of unity ~\cite{KQ,EQ} using the theory of Hopfological algebra introduced by Khovanov~\cite{KhHopf} and further developed by Qi in \cite{QYHopf}.

The 2-category $\Ucev$ categorifying quantum $\mf{sl}_2$  also plays a fundamental role in link homology theories.  One of the original motivations for categorifying quantum groups was to provide a unified representation theoretic explanation of the link homology theories that categorify various quantum link invariants.  Various steps in this direction have already been achieved~\cite{Web2,Cautis,LQR}.

Khovanov homology is the simplest of these link homology theories, categorifying a certain normalization of the Jones polynomial~\cite{Kh1,Kh2}. Surrounding Khovanov homology is an intricate system of related combinatorial and geometric ideas. Everything from extended 2-dimensional TQFTs~\cite{Kh1,Kh2,LP3,Cap4,CMW}, planar algebras~\cite{BN1,BN2},  category $\mathcal{O}$~\cite{Strop1,Strop2,BrSt3,BFK}, coherent sheaves on  quiver varieties~\cite{CK01}, matrix factorizations~\cite{KhR,KhR2}, homological mirror symmetry~\cite{SeSm,CK01}, arc algebras~\cite{Kh2,ChK,BrSt1,BrSt2,BrSt3,BrSt4}, Springer varieties~\cite{KhSp,Strop1,SW}, and 5-dimensional gauge theories~\cite{Witten,Witten2} appear in descriptions of Khovanov homology. The categorification $\Ucev(\mf{sl}_2)$ of the idempotent version $\Udot$ of the quantum enveloping algebra of $\mf{sl}_2$ demonstrates commonalities between these approaches revealed in terms of higher representation theory.

Given the many connections between Khovanov homology and the sophisticated structures described above, it is  surprising to discover that there exists a distinct categorification of the Jones polynomial. Ozsv\'{a}th, Rasmussen, Szab\'{o} found an {\em odd} analogue of Khovanov homology~\cite{ORS}. This odd homology theory for links agrees with the original Khovanov homology when coefficients are taken modulo 2.  Both of these theories categorify the Jones polynomial, and results of Shumakovitch~\cite{Shum} show that these categorified link invariants are not equivalent. Both can distinguish knots that are indistinguishable in the other theory.

Motivated by the problem of defining odd categorified quantum groups to provide a higher representation theoretic explanation for odd Khovanov homology, the authors in collaboration with Mikhail Khovanov introduced an odd analogue of the nilHecke algebra. The nilHecke algebra plays a central role in the theory of categorified quantum groups, giving rise to an integral categorification of the negative half of $\U(\mf{sl}_2)$~\cite{Lau1,KL3,Rou2}.  In the categorification $\Ucev$ of the entire quantum group $\Udot$, the nilHecke algebra describes 2-endomorphisms of 1-morphisms $\Fc^n\onebb_{\l}$, respectively $\Ec^n\onebb_{\l}$.  This algebra is also closely connected to the geometry of flag varieties and the combinatorics of symmetric functions.   The purpose of this article is to extend the categorification of the negative half of $\U(\mf{sl}_2)$ via the odd nilHecke algebra to an odd categorification of all of $\Udot(\mf{sl}_2)$.

\subsubsection*{Covering Kac-Moody algebras}

The odd nilHecke algebra appears to have numerous connections to other areas of representation theory.
It was independently introduced by Kang, Kashiwara and Tsuchioka~\cite{KKT}
starting from the different perspective of trying to develop super analogues of KLR algebras.  Their quiver Hecke superalgebras become isomorphic to affine Hecke-Clifford superalgebras or affine Sergeev superalgebras after a suitable completion, and the $\sltwo$ case of their construction is isomorphic to the odd nilHecke algebra.   Cyclotomic quotients of quiver Hecke superalgebras supercategorify certain irreducible representations of Kac-Moody algebras~\cite{KKO,KKO2}. A closely related spin Hecke algebra associated to the affine Hecke-Clifford superalgebra appeared in earlier work of Wang~\cite{Wang} and many of the essential features of the odd nilHecke algebra including skew-polynomials appears in this and related works~\cite{Wang2,KW1,KW2,KW4}.

The odd nilHecke algebra has led to a number of surprising new structures including an odd analogue of the ring of symmetric functions~\cite{EK,EKL,EOddLR} and odd analogues of the cohomology groups of Grassmannians~\cite{EKL} and of Springer varieties~\cite{LR}.  These structures possess combinatorics quite similar to those of their even counterparts.  When coefficients are reduced modulo two the theories become identical, but the odd analogues possess an inherent non-commutativity making them distinct from the classical theory.

Clark, Hill, and Wang have supplied convincing evidence that the odd nilHecke algebra and its generalizations are closely related to super Lie theory~\cite{HillWang,ClarkWang, CHW,CHW2,CFLW}.  A key insight of their construction was the introduction of a parameter $\pi$ with $\pi^2=1$.  They introduce the notion of a covering Kac-Moody algebra defined over $\Q(q)[\pi]/(\pi^2-1)$ for certain Kac-Moody Lie algebras.  A list of those finite and affine type Kac-Moody algebras admitting a covering algebra is shown in \cite[Table1]{HillWang}. The specialization to $\pi=1$ gives the quantum enveloping algebra of a Kac-Moody algebra and the specialization to $\pi=-1$ gives a quantum enveloping algebra of a Kac-Moody superalgebra. This idea led Hill and Wang to a novel bar involution $\overline{q} =\pi q^{-1}$ allowing the first construction of canonical bases for positive parts of Lie superalgebras. The canonical basis for the entire quantum supergroup $\U(\mathfrak{osp}_{1|2})$ was constructed by Clark and Wang in \cite{ClarkWang} where the rank one case was fully developed. The covering algebra $\Upi$ can be seen as a simultaneous generalization of the quantum group $\U(\mf{sl}_2)$ and the Lie superalgebra $\U(\mathfrak{osp}_{1|2})$.  This relationship is illustrated below.
\[\xy
  (0,10)*+{\Upi}="t";%
  (-15,-5)*+{\U(\mathfrak{sl}_2)}="bl";
  (15,-5)*+{\U(\mathfrak{osp}_{1|2})}="br";
  {\ar_{\pi \to 1} "t";"bl"};
  {\ar^{\pi \to -1} "t";"br"};
 \endxy\]
The existence of a canonical basis for the covering algebra $\Upi$ led Clark and Wang to conjecture the existence of a categorification of this algebra.  The main theorem of this paper (Theorem~\ref{thm_Groth}) confirms their conjecture.

\subsubsection*{Covering link homologies}
Given the interaction between categorified quantum groups and link homology, the existence of a categorification of covering Kac-Moody algebras hints at the possible existence of ``covering" link homology theories carrying an additional $\Z_2$-grading.  If the conjectured connection between categorified covering Kac-Moody algebras and odd link homologies exists, this would guide the search for new odd link homology theories.  In particular, the covering Kac-Moody perspective suggests that there will not be an odd analog of $\mf{sl}_n$-link homologies for $n>2$.  Indeed, the only covering Kac-Moody algebras that exist in finite type correspond to the Lie algebras $\mf{so}_{2n+1}$.  The covering Kac-Moody explanation of odd Khovanov homology would then be the result of the Lie algebra coincidence $\mf{sl}_2 = \mf{so}_3$.

The existence of odd link homologies associated to $\mf{so}_{2n+1}$ appears to be further corroborated by the topological field theory construction of link homologies due to Witten.  In a forthcoming article,  Mikhaylov and Witten describe an extension of this work in which candidates for odd link homologies associated $\mf{so}_{2n+1}$ naturally arise~\cite{MW}.  This work utilizes the orthosymplectic supergroup suggesting a close interplay with covering Kac-Moody algebras.

A covering homology was recently defined by Putyra \cite{Putyra} using a 2-category of chronological cobordisms.

\subsubsection*{Diagrammatics for super-2-categories}

The previously introduced graphical calculus for the odd nilHecke algebra required certain generators to skew commute.  This requires one to keep careful track of the relative heights of generators in a diagram making graphical computations rather cumbersome.
In Section \ref{sec-prelim} we translate Kang, Kashiwara, and Oh's theory of super-2-categories into a new diagrammatic framework.  This formulation has the advantage that it allows us to recast the odd nilHecke algebra in a graphical calculus with full isotopy invariance at the cost of adding a new type of strand to the graphical calculus.   From the perspective of these diagrammatics for super-2-categories, the skew commutativity of certain odd 2-morphisms is quite natural, allowing for a categorical manifestation of odd nilHecke algebras.

\subsubsection*{Super-2-representation theory of quantum Kac-Moody superalgebras}

As mentioned above, one of the main theorems of this paper is the construction of a super-2-category $\Udotc$ which categorifies Clark and Wang's quantum covering $\sltwo$.  We also prove a 2-representation theoretic structure theorem analogous to a theorem of Cautis and the second author \cite{CL}: that a strong supercategorical action of $\sltwo$ gives rise to a 2-representation of our super-2-category $\Udotc$, the latter notion being \emph{a priori} stronger.  This allows us to upgrade the Kang-Kashiwara-Oh action on cyclotomic quotients of odd nilHecke algebras to a 2-representation of $\Udotc$.

In \cite{Lau1}, the categorification $\Udotcev$ of quantum $\sltwo$ has generating 1-morphisms $\Ec\onebbl$ and $\onebbl\Fc$.  These 1-morphisms are biadjoint up to grading shifts.  It should come as no surprise, then, that the corresponding generating 1-morphisms of $\Udotc$ are biadjoint up to grading \emph{and parity} shifts.  Unlike shifts of the $\Z$-grading, parity shifts are visible in our diagrammatics.  Hence our biadjointness and cyclicity relations \eqref{eq_biadjoint1},\eqref{eq_biadjoint2},\eqref{eqn-dot-cyclicity},\eqref{eqn-crossing-cyclicity} are more involved than those in \cite{Lau1}.

We expect that all the constructions of this paper---the definition of strong supercategorical actions, the definition of the super-2-category, and the theorem on upgrading the former to the latter---can be extended to all Kac-Moody types for which Kang, Kashiwara, and Tsuchioka \cite{KKT} have defined quiver Hecke superalgebras.

\bigskip
%%%%%%%%%%%%%%%%%%%%%%%%%%%%%%%%%%%%%%%%%%%%%%%%%%%%%%%%%%%%%%%%%%
\noindent {\bf Acknowledgments:}
A.P.E was supported by the NSF Graduate Research Fellowship Program and thanks Masaki Kashiwara for a helpful discussion on super-2-categories. A.D.L  was partially supported by NSF grant DMS-1255334,  the Alfred P. Sloan foundation, and by the John Templeton Foundation. A.D.L is grateful to Sabin Cautis for teaching him the methods used in Section~\ref{sec-formal} and to Edward Witten for discussions about gauge theoretic formulations of odd Khovanov homology. Both the authors would like to acknowledge partial support from Columbia University's RTG grant DMS-0739392.  This paper builds directly on joint work with Mikhail Khovanov \cite{EKL}.

%#####################################################################
%
\section{Preliminaries} \label{sec-prelim}
%
%#####################################################################

%---------------------------------------------------------------------
\subsection{Conventions}\label{subsec-conventions}
%---------------------------------------------------------------------

We will work with $(\Z\times\Zt)$-graded algebras.  The $\Z$-degree of an element $x$ we will simply call \emph{degree} $|x|$ and the $\Zt$-degree we will call \emph{parity} $p(x)$.  The braiding on the base category (graded $\Bbbk$-modules) is governed by parity; for instance, if $A$ is an algebra in this category, then the canonical multiplication on $A\otimes A$ is
\begin{equation*}
(a\otimes b)(c\otimes d)=(-1)^{p(b)p(c)}(ac)\otimes(bd)
\end{equation*}
for parity-homogeneous elements $a,b,c,d$.  If $M$ is an $A$-module, we write $\Pi(M)$ for the its parity shift with the left action twisted by the parity involution on $A$,
\begin{equation}\label{eqn-parity-involution}
\iota_A(a)=(-1)^{p(a)}a.
\end{equation}
We write $M\ds{k}$ for the module obtained from $M$ by shifting all degrees down by $k$.  Hom-spaces always refer to maps which preserve both degree and parity; if we write $\Hom^k$ for the space of maps of degree $k$ (so that $\Hom^0$ is the actual hom-space), then
\begin{equation*}
\Hom^k(M,N)=\Hom(M\ds{-k},N)=\Hom(M,N\ds{k}).
\end{equation*}

%---------------------------------------------------------------------
\subsection{Covering Kac-Moody algebras}\label{subsec-covering-sl2}
%---------------------------------------------------------------------

In the subsection, we present a modification of the constructions of \cite{CHW}.  This ``covering'' form of $\U=U_q(\sltwo)$, which we denote by $\Upi$, is the object we will categorify as the main result of this paper.  Define the $(q,\pi)$-analogues of integers, factorials, and binomial coefficients by
\[
[n]=\frac{(\pi q)^n-q^{-n}}{\pi q-q^{-1}},\qquad
[a]!= \prod_{i=1}^{a}[i], \qquad
\left[\!\!
 \begin{array}{c}
   n \\
   a
 \end{array}
 \!\!\right]
 =
 \frac{\prod_{i=1}^a[n+i-a]}{[a]!}.
\]
Note as in \cite{CHW} that
$\left[\!\!
 \begin{array}{c}
   n \\
   a
 \end{array}
 \!\!\right] = \frac{[n]!}{[a]![n-a]!}$ for $n \geq a \geq 0$ and $[-n]=-\pi^n[n]$.  Let $\Ac=\Z[q,q^{-1}]$, $\Ac_{\pi}=\Z[q,q^{-1},\pi]/(\pi^2-1)$, and $\Q(q)^\pi=\Q(q)[\pi]/(\pi^2-1)$.

\begin{defn}\label{defn-covering-sl2} The \emph{idempotented form} of \emph{quantum covering $\sltwo$} is the (non-unital) $\Q(q)^\pi$-algebra $\Udotpi$ generated by orthogonal idempotents $\lbrace\onel:\lambda\in\Z\rbrace$ and elements
\begin{equation}
\oneltwo E\onel=E\onel=\oneltwo E,\qquad\onel F\oneltwo=F\oneltwo=\onel F\qquad\lambda\in\Z
\end{equation}
subject to the \emph{covering $\sltwo$ relation},
\begin{equation}\label{eqn-covering-sl2-relation}
EF\onel-\pi FE\onel=[\l]\onel.
\end{equation}
The \emph{integral idempotented form} is the $\Ac_\pi$-subalgebra $\AUdotpi\subset\Udotpi$ generated by the divided powers
\begin{equation}
E^{(a)}\onel=\frac{E^a\onel}{[a]!},\quad\onel F^{(a)}=\frac{\onel F^a}{[a]!}
\end{equation}
and the (idempotented) $q$-binomial coefficients
\begin{equation}
\left[\!\!
 \begin{array}{c}
   n \\
   a
 \end{array}
 \!\!\right]\onel
\end{equation}
(we are making the shorthand identifications $E^2\onel=E\oneltwo E\onel$, etc.).
\end{defn}

\begin{rem} There is also a non-idempotented form of covering quantum $\sltwo$, but we will not need it; see \cite{CHW}.\end{rem}

The \emph{Clark-Hill-Wang bar involution} on $\Udotpi$ is defined by
\begin{equation*}
\overline{q}=\pi q^{-1},\quad\overline{\pi}=\pi^{-1}=\pi
\end{equation*}
(it does nothing to $\onel$, $E$, and $F$).  ``Anti-linear'' will be meant with respect to the bar involution.  Note that $(q,\pi)$-integers $[\l]$ are bar-invariant.

Define a linear anti-automorphism $\rho$ of $\Udotpi$ by
\begin{equation}\label{eqn-defn-rho}
\rho(q)=\pi q,\quad\rho(\pi)=\pi,\quad\rho(\onel)=\onel,\quad\rho(E\onel)=q^{\l+1}\onel F,\quad\rho(\onel F)=\pi^{\l+1}q^{-\l-1}E\onel
\end{equation}
and an anti-linear anti-automorphism $\tau=\overline{\rho}$ by
\begin{equation}\label{eqn-defn-tau}
\tau(q)=q^{-1},\quad\tau(\pi)=\pi,\quad\tau(\onel)=\onel,\quad\tau(E\onel)=\pi^{\l+1}q^{-\l-1}\onel F,\quad\tau(\onel F)=q^{\l+1}E\onel.
\end{equation}

Define the bilinear form $(\cdot,\cdot)$ on $\Udotpi$ by the properties
\begin{eqnarray}
&\text{different idempotented parts are orthogonal}\\
&(fx,y)=f(x,y)=(x,fy)\\
&(ux,y)=(x,\rho(u)y)\\
&(x,y)=(y,x)\\
&(F^{(a)}\onel,F^{(a)}\onel)=\prod_{s=1}^a(1-\pi^sq^{-2s})^{-1}
\end{eqnarray}
and the sesquilinear form $\la\cdot,\cdot\ra$ by
\begin{equation*}
\la x,y\ra=\overline{(x,\overline{y})}.
\end{equation*}
The following properties of $\la\cdot,\cdot\ra$ follow:
\begin{eqnarray}
&\text{different idempotented parts are orthogonal}\\
&\la\overline{f}x,y\ra=f\la x,y\ra=\la x,fy\ra\\
\label{eqn-tau-adjoint}&\la ux,y\ra=\la x,\tau(u)y\ra\\
&\la x,y\ra=\la\overline{y},\overline{x}\ra\\
\label{eqn-form-on-F}&\la F^{(a)}\onel,F^{(a)}\onel\ra=\prod_{s=1}^a(1-\pi^sq^{2s})^{-1}
\end{eqnarray}
This agrees with the bilinear form introduced by Lusztig~\cite[26.1.1]{Lus4} when $\pi=1$.

Note that $\la\onel F,\onel F\ra=(1-\pi q^2)^{-1}$.
\begin{lem} $\la E\onel,E\onel\ra=(1-\pi q^2)^{-1}$.\end{lem}
\begin{proof}
\begin{equation*}\begin{split}
\la E\onel,E\onel\ra&\refequal{\eqref{eqn-tau-adjoint}}\la\onel,\tau(E\onel)E\onel\ra\\
&\refequal{\eqref{eqn-defn-tau}}\la\onel,\pi^{\l+1}q^{-\l-1}FE\onel\ra\\
&\refequal{\eqref{eqn-covering-sl2-relation}}\la\onel,\pi^\l q^{-\l-1}EF\onel-\pi^\l q^{-\l-1}[\l]\onel\ra\\
&=\la\tau^{-1}(E\one_{\l-2})\onel,\pi^\l q^{-\l-1}F\onel\ra-\pi^\l q^{-\l-1}[\l]\la\onel,\onel\ra\\
&\refequal{\eqref{eqn-defn-tau}}\la q^{-\l-1}\one_{\l-2}F,\pi^\l q^{-\l-1}\one_{\l-2}F\ra-\pi^\l q^{-\l-1}[\l]\\
&\refequal{\eqref{eqn-form-on-F}}\frac{\pi^\l q^{-2\l}}{1-\pi q^2}-\pi^\l q^{-\l-1}\frac{\pi^\l q^\l-q^{-\l}}{\pi q-q^{-1}}\\
&=\frac{1}{1-\pi q^2}.
\end{split}\end{equation*}
\end{proof}

\begin{rem}
There are some rescaling degrees of freedom for the automorphisms and bilinear form defined above.
Our automorphisms and bilinear form differ from those of \cite{CHW}. The specific choice of scaling given would have been difficult to fix without knowing the categorification of $\Udotpi$.  In particular, our
automorphism $\tau$ is the decategorification of ``take the right adjoint''.\end{rem}

%---------------------------------------------------------------------
\subsection{Odd nilHecke algebras and cyclotomic quotients}
%---------------------------------------------------------------------

\subsubsection{Odd nilHecke algebras categorify $\AU^+(\sltwo)$}

The main theorem of this paper is the construction of a categorification of the algebra $\Udotpi$ of Definition \ref{defn-covering-sl2}.  One of the main ingredients in this construction is the categorification of $\U^+$ by the \emph{odd nilHecke algebras} \cite{EKL,KKT}.  This subsection will review these algebras as well as their \emph{cyclotomic quotients}, which categorify simple modules for $\Udot$ \cite{KKO,KKO2}.

Let $\Bbbk$ be a commutative ring (usually we take $\Bbbk=\Z$ or a field) and let
\begin{equation*}
\spol_n=\Bbbk\langle x_1,\ldots,x_n\rangle/(x_ix_j+x_jx_i\text{ if }i\neq j)
\end{equation*}
be the $(\Z\times\Zt)$-graded superalgebra of \emph{skew polynomials} in $n$ variables.  Its generators are given degrees $|x_i|=2$, $p(x_i)=1$.  The symmetric group $S_n$ acts on $\spol_n$ by
\begin{equation*}
w(x_i)=x_{w(i)},\quad w(fg)=w(f)w(g).
\end{equation*}
For $i=1,\dots,n-1$, let $s_i=(i\quad i+1)$ be the $i$-th simple transposition in $S_n$ and define the $i$-th \emph{odd divided difference operator} $\del_i$ to be the map $\spol_n\to\spol_n$ defined by
\begin{eqnarray}
&\del_i(x_j)=\begin{cases}1&j=i,i+1\\0&\text{otherwise,}\end{cases}\\
&\del_i(fg)=\del_i(f)g+(-1)^{|f|}s_i(f)\del_i(g).
\end{eqnarray}
(It is straightforward to check this is well defined.)  For example,
\begin{equation*}
\del_1(x_1^2x_2)=\del_1(x_1)x_1x_2-x_2\del_1(x_1)x_2+x_2^2\del_1(x_2)=x_1x_2.
\end{equation*}
More generally, for any $f\in \spol_n$ the action of the odd divided difference operator is given by the formula
\[
 \partial_i f = \frac{(x_{i+1} - x_i)f - (-1)^{|f|}s_i(f) (x_{i+1} - x_i) }{x_{i+1}^2-x_i^2},
\]
see \cite[equation 4.19]{KKO}.
These operators play a role analogous to that of the divided difference operators of Kostant-Kumar.

\begin{defn}\label{defn-osym} The sub-superalgebra
\begin{equation}
\osym_n=\bigcap_{i=1}^{n-1}\ker(\del_i)
\end{equation}
of $\spol_n$ is called the superalgebra of \emph{odd symmetric polynomials} in $n$ variables.
\end{defn}

\begin{defn}\label{defn-onh} The \emph{odd nilHecke algebra} $\onh_n$ in $n$ variables (or ``on $n$ strands'') is the sub-superalgebra of $\End_\Bbbk(\spol_n)$ generated by the operators $x_1,\ldots,x_n$ (left multiplication by $x_i$) and $\del_1,\ldots,\del_{n-1}$.
\end{defn}
The generators above have degrees $|x_i|=2$, $|\del_i|=-2$, $p(x_i)=p(\del_i)=1$.  Fix a reduced expression $w=i_1\cdots i_r$ for each $w\in S_n$ and define
\begin{equation*}
\del_w=\del_{i_1}\cdots\del_{i_r}.
\end{equation*}
This is independent of choice of reduced expression up to sign only (see the following Proposition).  If $\alpha=(\alpha_1,\ldots,\alpha_n)$ is an $n$-tuple, write $x^\alpha$ for $x_1^{\alpha_1}\cdots x_n^{\alpha_n}$.

The basic properties of $\onh_n$ are as follows.
\begin{prop}[\cite{EKL}] \hfill
\begin{enumerate}
\item Changing the choice of reduced expression only changes $\del_w$ by a possible factor of $-1$, and
\begin{equation*}
\del_w\del_{w'}=\begin{cases}\pm\del_{ww'}&\ell(w)+\ell(w')=\ell(ww'),\\0&\text{otherwise.}\end{cases}
\end{equation*}
\item The algebra $\onh_n$ is a free $\Bbbk$-module.  Either of the sets $\lbrace\del_wx^\alpha:w\in S_n,\alpha_i\in\Z_{\geq0}\rbrace$, $\lbrace x^\alpha\del_w:w\in S_n,\alpha_i\in\Z_{\geq0}\rbrace$ is a basis.
\item The relations
\begin{equation}\label{eqn-onh-relations}\begin{split}
&x_ix_j+x_jx_i=0\text{ if }i\neq j,\qquad
\del_i\del_j+\del_j\del_i=0\text{ if }|i-j|>1,\\
&\del_i^2=0,\qquad
\del_i\del_{i+1}\del_i=\del_{i+1}\del_i\del_{i+1},\\
&\del_ix_j+x_j\del_i=0,\text{ if }j\neq i,i+1,\\
&x_i\del_i+\del_ix_{i+1}=\del_ix_i+x_{i+1}\del_i=1
\end{split}\end{equation}
hold and give a presentation of $\onh_n$.
\item There is an isomorphism of superalgebras
\begin{equation}
\onh_n\cong\End_{\osym_n}(\spol_n).
\end{equation}
\end{enumerate}\end{prop}

In the super-2-category we will construct in Subsection \ref{subsec-defn-odd-udot}, the odd nilHecke algebra $\onh_n$ will play an important role in
governing the endomorphisms of $\Ec^n\onebbl$ and $\onebbl\Fc^n$.
It is often more convenient to work in a diagrammatic notation (which first arose, in fact, from the 2-categorical setting).  Elements of $\onh_n$ are $\Bbbk$-linear combinations of diagrams involving crossings and dots on $n$ strands with fixed endpoints, considered up to isotopies rel their boundaries that do not change the \emph{relative} heights of dots and crossings.  For example,
\begin{equation*}
\hackcenter{\begin{tikzpicture}[scale=0.5]
	\draw[thick] (0,0) [out=90, in=-90] to (2,2);
	\draw[thick] (1,0) [out=90, in=-90] to (0,2);
	\draw[thick] (2,0) .. controls (2,1) and (1,1) .. (1,2)
		node[pos=.2] () {\bbullet};
\end{tikzpicture}}
\;\; -\;\;2\;
\hackcenter{\begin{tikzpicture}[scale=0.5]
	\draw[thick] (0,0) -- (0,2)
		node[pos=.25] () {\bbullet}
		node[pos=.75] () {\bbullet};
	\draw[thick] (1,0) [out=90, in=-90] to (1,2);
	\draw[thick] (2,0) [out=90, in=-90] to (2,2);
\end{tikzpicture}}
\;\;
\in \;\;
\onh_3.
\end{equation*}
For a more discussion of diagrammatic algebra, see any of \cite{KL1,KhDiagrammatics,Lau4}.

In graphical notation, the relations \eqref{eqn-onh-relations} read
\begin{equation}\label{eqn-onh-relations-graphical}\begin{split}
&\hackcenter{\begin{tikzpicture}[scale=0.5]
	\draw[thick] (0,0) -- (0,2)
		node[pos=.75] () {\bbullet};
	\node () at (1,1) {$\cdots$};
	\draw[thick] (2,0) -- (2,2)
		node[pos=.25] () {\bbullet};
\end{tikzpicture}}
\quad+\quad
\hackcenter{\begin{tikzpicture}[scale=0.5]
	\draw[thick] (0,0) -- (0,2)
		node[pos=.25] () {\bbullet};
	\node () at (1,1) {$\cdots$};
	\draw[thick] (2,0) -- (2,2)
		node[pos=.75] () {\bbullet};
\end{tikzpicture}}
\quad=\;\;0,\qquad
\hackcenter{\begin{tikzpicture}[scale=0.5]
	\draw[thick] (0,0) -- (0,1) [out=90, in=-90] to (1,2);
	\draw[thick] (1,0) -- (1,1) [out=90, in=-90] to (0,2);
	\node () at (2,1) {$\cdots$};
	\draw[thick] (3,0) [out=90, in=-90] to (4,1) -- (4,2);
	\draw[thick] (4,0) [out=90, in=-90] to (3,1) -- (3,2);
\end{tikzpicture}}
\quad+\quad
\hackcenter{\begin{tikzpicture}[scale=0.5]
	\draw[thick] (3,0) -- (3,1) [out=90, in=-90] to (4,2);
	\draw[thick] (4,0) -- (4,1) [out=90, in=-90] to (3,2);
	\node () at (2,1) {$\cdots$};
	\draw[thick] (0,0) [out=90, in=-90] to (1,1) -- (1,2);
	\draw[thick] (1,0) [out=90, in=-90] to (0,1) -- (0,2);
\end{tikzpicture}}
\quad= \;\; 0,\\
&\hackcenter{\begin{tikzpicture}[scale=0.5]
	\draw[thick] (0,0) [out=90, in=-90] to (1,1) [out=90, in=-90] to (0,2);
	\draw[thick] (1,0) [out=90, in=-90] to (0,1) [out=90, in=-90] to (1,2);
\end{tikzpicture}}
\quad=\;\; 0,\qquad
\hackcenter{\begin{tikzpicture}[scale=0.5]
	\draw[thick] (0,0) [out=90, in=-90] to (2,2);
	\draw[thick] (1,0) [out=90, in=-90] to (0,1) [out=90, in=-90] to (1,2);
	\draw[thick] (2,0) [out=90, in=-90] to (0,2);
\end{tikzpicture}}
\quad=\quad
\hackcenter{\begin{tikzpicture}[scale=0.5]
	\draw[thick] (0,0) [out=90, in=-90] to (2,2);
	\draw[thick] (1,0) [out=90, in=-90] to (2,1) [out=90, in=-90] to (1,2);
	\draw[thick] (2,0) [out=90, in=-90] to (0,2);
\end{tikzpicture}}
\quad,\\
&\hackcenter{\begin{tikzpicture}[scale=0.5]
	\draw[thick] (0,0) -- (0,1) [out=90, in=-90] to (1,2);
	\draw[thick] (1,0) -- (1,1) [out=90, in=-90] to (0,2);
	\node () at (2,1) {$\cdots$};
	\draw[thick] (3,0) -- (3,2)
		node[pos=.25] () {\bbullet};
\end{tikzpicture}}
\quad+\quad
\hackcenter{\begin{tikzpicture}[scale=0.5]
	\draw[thick] (0,0) [out=90, in=-90] to (1,1) -- (1,2);
	\draw[thick] (1,0) [out=90, in=-90] to (0,1) -- (0,2);
	\node () at (2,1) {$\cdots$};
	\draw[thick] (3,0) -- (3,2)
		node[pos=.75] () {\bbullet};
\end{tikzpicture}}
= \;\; 0,\qquad
\hackcenter{\begin{tikzpicture}[scale=0.5]
	\draw[thick] (1,0) -- (1,2)
		node[pos=.75] () {\bbullet};
	\node () at (2,1) {$\cdots$};
	\draw[thick] (3,0) [out=90, in=-90] to (4,1) -- (4,2);
	\draw[thick] (4,0) [out=90, in=-90] to (3,1) -- (3,2);
\end{tikzpicture}}
\quad+\quad
\hackcenter{\begin{tikzpicture}[scale=0.5]
	\draw[thick] (1,0) -- (1,2)
		node[pos=.25] () {\bbullet};
	\node () at (2,1) {$\cdots$};
	\draw[thick] (3,0) -- (3,1) [out=90, in=-90] to (4,2);
	\draw[thick] (4,0) -- (4,1) [out=90, in=-90] to (3,2);
\end{tikzpicture}}
=\;\;0,\\
&\hackcenter{\begin{tikzpicture}[scale=0.75]
	\draw[thick] (0,0) .. controls (0,.5) and (1,.5) .. (1,1);
	\draw[thick] (1,0) .. controls (1,.5) and (0,.5) .. (0,1)
		node[pos=.75] () {\bbullet};
\end{tikzpicture}}
\quad+\quad
\hackcenter{\begin{tikzpicture}[scale=0.75]
	\draw[thick] (0,0) .. controls (0,.5) and (1,.5) .. (1,1);
	\draw[thick] (1,0) .. controls (1,.5) and (0,.5) .. (0,1)
		node[pos=.25] () {\bbullet};
\end{tikzpicture}}
\quad=\quad
\hackcenter{\begin{tikzpicture}[scale=0.75]
	\draw[thick] (0,0) -- (0,1);
	\draw[thick] (1,0) -- (1,1);
\end{tikzpicture}}
\quad=\quad
\hackcenter{\begin{tikzpicture}[scale=0.75]
	\draw[thick] (0,0) .. controls (0,.5) and (1,.5) .. (1,1)
		node[pos=.25] () {\bbullet};
	\draw[thick] (1,0) .. controls (1,.5) and (0,.5) .. (0,1);
\end{tikzpicture}}
\quad+\quad
\hackcenter{\begin{tikzpicture}[scale=0.75]
	\draw[thick] (0,0) .. controls (0,.5) and (1,.5) .. (1,1)
		node[pos=.75] () {\bbullet};
	\draw[thick] (1,0) .. controls (1,.5) and (0,.5) .. (0,1);
\end{tikzpicture}}
\quad.
\end{split}\end{equation}
One pleasant outcome of the super-2-categorical setting we will introduce in Subsection \ref{subsec-super-2-categories} is that we will be able to use diagrams which are equivalent modulo \emph{all} isotopies rel boundary, not just those which preserve relative heights of generators.

There is a map of superalgebras $\iota_{a,b}:\onh_a\otimes\onh_b\hookrightarrow\onh_{a+b}$ given by juxtaposing diagrams horizontally, placing the diagram on the left above the one on the right.  On generators,
\begin{equation}\label{eqn-iota-onh-ab}
\iota_{a,b}(x_i\otimes1)=x_i,\quad\iota_{a,b}(\del_i\otimes1)=\del_i,\quad\iota_{a,b}(1\otimes x_i)=x_{a+i},\quad\iota_{a,b}(1\otimes\del_i)=\del_{a+i}.
\end{equation}
\begin{prop}[\cite{EKL}]\label{prop-ekl-2} \hfill
\begin{enumerate}
\item Let $\delta=(n-1,n-2,\ldots,1,0)$ and $e_n=\del_{w_0}x^\delta$.  Then $e_n$ is an idempotent in $\onh_n$ and the left module $P_n=\onh_ne_n$ is, up to grading shifts, the unique indecomposable projective $\onh_n$-module.
\item The regular representation of $\onh_n$ decomposes as
\begin{equation}
\onh_n\cong\bigoplus_{[n]!}(\onh_ne_n)\la\qbin{n}{2}\ra.
\end{equation}
Here, for a Laurent polynomial $f=\sum_jf_jq^j$, the notation $\bigoplus_fM$ means the direct sum $\bigoplus_jM\la-j\ra^{\oplus f_j}$.
\item Induction and restriction along the map $\iota_{a,b}$ send finitely generated graded projectives to finitely generated graded projectives.  In particular,
\begin{equation}
\Ind_{a,b}^{a+b}(P_a\otimes P_b)\cong\bigoplus_{\qbins{a}{b}}P_{a+b},
\end{equation}
where we are using the shorthand notations $a$ for $\onh_a$ and $a,b$ for $\onh_a\otimes\onh_b$.
\end{enumerate}\end{prop}

For any graded algebra $A$, consider the free $\Ac$-module with basis consisting of isomorphism classes of finitely generated indecomposable projective modules.  The version of Grothenieck group with which we will concern ourselves is the quotient of this free module by the relations
\begin{itemize}
\item $[B]=[A]+[C]$ if $B\cong A\oplus C$ as graded modules,
\item $[A\la-1\ra]=q[A]$.
\end{itemize}
If $A$ is a graded superalgebra, we use $\Ac_\pi=\Z[q,q^{-1},\pi]/(\pi^2-1)$ instead of $\Ac$.  The parity shift functor $\Pi$ acts on the Grothendieck group by the relation $[\Pi M]=\pi[M]$.  Then $K_0(A)$ is naturally an $\Ac_\pi$-module spanned by all indecomposable projectives (considered up to isomorphism and grading shifts).

Writing $K_0(\onh_\bullet)=\bigoplus_{n\geq0}K_0(\onh_n)$, statement (3) of Proposition \ref{prop-ekl-2} implies that $[V]\otimes[W]\mapsto[\Ind_{a,b}^{a+b}(V\otimes W)]$ determines a homomorphism $\Ac$-modules $K_0(\onh_a)\otimes_{\Ac}K_0(\onh_b)\to K_0(\onh_{a+b})$.  By Proposition \ref{prop-ekl-2} (3), this map takes
\begin{equation*}
[P_a]\otimes[P_b]\mapsto \qbin{a}{b}[P_{a+b}].
\end{equation*}
Restriction determines a map in the other direction.  This nearly immediately implies the following.
\begin{thm}[\cite{EKL}] The map taking $E^{(a)}$ to $[P_a]$ is an isomorphism of $q$-Hopf algebras
\begin{equation}
\xymatrix{\U^+(\sltwo)\ar[r]^-\cong&K_0(\onh_\bullet)\otimes_{\Ac}\Q(q).}
\end{equation}
\end{thm}
With more work (the ``thick calculus'' of \cite{EKL}), the above theorem can be strengthened.
\begin{thm}[\cite{EKL}] The map taking $E^{(a)}$ to $[P_a]$ is an isomorphism of $q$-Hopf algebras
\begin{equation}
\xymatrix{\AU^+(\sltwo)\ar[r]^-\cong&K_0(\onh_\bullet).}
\end{equation}
\end{thm}

\subsubsection{Cyclotomic quotients categorify $_{\Ac_\pi}V^\Lambda$}

\begin{defn} Let $\Lambda\geq0$ be a dominant integral weight for $\sltwo$ (thought of as a non-negative integer).  The corresponding \emph{cyclotomic quotient} $\onh_n^\Lambda$ of $\onh_n$ is defined to be the superalgebra
\begin{equation}
\onh_n^\Lambda=\onh_n/(x_1^\Lambda).
\end{equation}
\end{defn}

While explicit calculations in cyclotomic quotients are generally difficult, there is an isomorphism with a matrix algebra over a less complicated ring.
\begin{equation*}
\onh_n^\Lambda\cong\Mat_{[n]!}(OH_{n,\Lambda}).
\end{equation*}
The notation $\Mat_{[n]!}$ means the algebra of endomorphisms of $OH_{n,\Lambda}^{\oplus[n]!}$ as a $OH_{n,\Lambda}$-module.  The algebra $OH_{n,\Lambda}$ is the \emph{odd Grassmannian algebra} of \cite{EKL}; it is graded local and serves as an odd analogue of the singular cohomology ring of the complex Grassmannian $\Gr(n,\C^\Lambda)$.

The maps $\iota_{a,b}$ are still well-defined on cyclotomic quotients, yielding analogous induction and restriction functors.  These functors have been studied in depth by Kang-Kashiwara-Oh \cite{KKO,KKO2}.  When $a=n$ and $b=1$, the resulting ``right-strand'' induction and restriction functors between $\onh_n^\Lambda$ and $\onh_{n+1}^\Lambda$ are related by natural isomorphisms that give a ``strong supercategorical action'' (see Subsection \ref{subsec-strong-supercat-action} or the even analogue in \cite{CL}) of $\U(\sltwo)$ on the supermodule categories for $\onh_n^\Lambda$.  Kang-Kashiwara-Oh prove this in both a super and a non-super setting.
\begin{thm}[\cite{KKO,KKO2}] Let $_{\Ac}V^\Lambda$, $_{\Ac_\pi}V^\Lambda$ be the simple modules of highest weight $q^\Lambda$ for $\AU$, $\AUpi$ respectively.
\begin{enumerate}
\item There is an isomorphism of $\Ac$-modules
\begin{equation}
\xymatrix{_{\Ac}V^\Lambda\ar[r]^-\cong&K_0(\onh^\Lambda_\bullet)}
\end{equation}
Under this isomorphism, right-strand induction and restriction between the algebras $\onh_n^\Lambda$ decategorify to the action of $F$ and $E$ on, respectively.
\item Let $K_0^{\text{super}}$ be $K_0$ considered as an $\Ac_\pi$-module, where $\pi$ is the decategorification of parity shift.  Then there is an isomorphism of $\Ac_\pi$-modules
\begin{equation}
\xymatrix{_{\Ac_\pi}V^\Lambda\ar[r]^-\cong&K_0^{\text{super}}(\onh^\Lambda_\bullet)}
\end{equation}
Under this isomorphism, right-strand induction and restriction between the superalgebras $\onh_n^\Lambda$ decategorify to the action of $F$ and $E$ on, respectively.
\end{enumerate}\end{thm}

A natural question raised by this theorem is whether the Kang-Kashiwara-Oh strong supercategorical action can be replaced by a genuine super-2-representation---that is, whether there exists a super-2-category $\Udotc$ such that:
\begin{enumerate}
\item $\Udotc$ categorifies the algebra $\AUdotpi(\sltwo)$;
\item there is compatibility with the results of \cite{EKL}: the odd nilHecke algebra $\onh_n$ naturally appears in the 2-hom-spaces between 1-morphisms lifting the elements $E^n\onel$ or $F^n\onel$ of $\AUdotpi(\sltwo)$;
\item there is compatibility with the results of \cite{KKO,KKO2}: supermodules for $\onh_n^\Lambda$ form a 2-representation of $\Udotc$, and the 1-morphisms lifting $E$ and $F$ act as right-strand restriction and induction, respectively;
\item the indecomposable 1-morphisms in $\Udotc$ categorify the Clark-Wang canonical basis for $\AUdot(\sltwo)$.
\end{enumerate}
The rest of this paper is devoted to answering this question in the affirmative.  But first we must develop the setting of super-2-representation theory.

%---------------------------------------------------------------------
\subsection{Super-2-categories}\label{subsec-super-2-categories}
%---------------------------------------------------------------------

In \cite{Lau1}, the second author constructed a 2-category which categorifies $\Udot(\sltwo)$.  The proper framework for categorifying $\Udotpi(\sltwo)$ is that of super-2-categories.  Our definition of super-2-category is slightly different than but equivalent to that of \cite{KKO2}.  Although the example of super-bimodules is actually a bicategory and not a (strict) 2-category, we will treat all bicategories as strict by the appropriate analogue of Mac Lane's coherence theorem.

\subsubsection{Supercategories, superfunctors, and supernatural transformations}

In order to motivate the general definition of a super-2-category, we first work out an example in detail: that of supercategories, superfunctors, and supernatural transformations.

\begin{defn} A \emph{supercategory} is a category $\Cc$ equipped with a strong categorical action of $\Zt$.  A \textit{superfunctor} is a morphism of strong categorical $\Zt$-actions.\end{defn}

Unpacking the above, the data of a supercategory consists of (see \cite{KKO,KKO2}):
\begin{itemize}
\item a category $\Cc$,
\item a functor $\Psi_\Cc:\Cc\to\Cc$,
\item and a natural isomorphism $\xi_\Cc:\Psi_\Cc^2\to\onebb_\Cc$.
\end{itemize}
The only condition is that
\begin{itemize}
\item $\xi_\Cc\otimes\onebb_{\Psi_\Cc}=\onebb_{\Psi_\Cc}\otimes\xi_\Cc$ as natural isomorphisms $\Psi_\Cc^3\to\Psi_\Cc$.
\end{itemize}
Here we are using $\otimes$ for horizontal composition (sometimes we omit the $\otimes$).  The data of a superfunctor $(F,\alpha_F):(\Cc,\Psi_\Cc)\to(\Dc,\Psi_\Dc)$ is:
\begin{itemize}
\item a functor $F:\Cc\to\Dc$ and
\item a natural isomorphism $\alpha_F:F\otimes\Psi_\Cc\to\Psi_\Dc\otimes F$,
\end{itemize}
and the only condition is that
\begin{itemize}
\item $\onebb_F\otimes\xi_\Cc=(\xi_\Dc\otimes\onebb_F)(\onebb_{\Psi_\Dc}\otimes\alpha_F)(\alpha_F\otimes\onebb_{\Psi_\Cc})$.
\end{itemize}
Diagrammatically, we can draw $\Cc,\Dc$ as regions, $\Psi_\Cc,\Psi_\Dc$ as blue dashed lines, $F$ as a solid line, and $\alpha_F$ as a dashed-solid crossing:
\begin{equation*}
\alpha_F=\quad
\hackcenter{\begin{tikzpicture}
    \draw[thick] (0,0) [out=90, in=-90] to (1,1);
    \draw[thick, color=blue, dashed] (1,0) [out=90, in=-90] to (0,1);
\end{tikzpicture}},\qquad
\alpha_F^{-1}=\quad
\hackcenter{\begin{tikzpicture}
    \draw[thick] (1,0) [out=90, in=-90] to (0,1);
    \draw[thick, color=blue, dashed] (0,0) [out=90, in=-90] to (1,1);
\end{tikzpicture}}.
\end{equation*}
Then the relations in the definition of a supercategory can be expressed diagrammatically as follows: We draw a dashed cap for $\xi_\Cc$ and a dashed cup for $\xi_\Cc^{-1}$.  Then the equation $\xi_\Cc\otimes\onebb_{\Psi_\Cc}=\onebb_{\Psi_\Cc}\otimes\xi_\Cc$ reads
\begin{equation}
\label{eqn-psi-3-psi}
\hackcenter{\begin{tikzpicture}
    \draw[thick, color=blue, dashed] (0,0) -- (0,1.5);
    \draw[thick, color=blue, dashed] (.5,0) .. controls (.5,.8) and (1.5,.8) .. (1.5,0);
\end{tikzpicture}}
\quad=\quad
\hackcenter{\begin{tikzpicture}
    \draw[thick, color=blue, dashed] (0,0) .. controls (0,.8) and (1,.8) .. (1,0);
    \draw[thick, color=blue, dashed] (1.5,0) -- (1.5,1.5);
\end{tikzpicture}}
\quad,
\end{equation}
the fact that $\xi_\Cc$ and $\xi_\Cc^{-1}$ are inverses is encoded by
\begin{equation}
\label{eqn-dashed-bubble}
\hackcenter{\begin{tikzpicture}
    \draw[thick, color=blue, dashed] (0,1) .. controls (0,1.8) and (1,1.8) .. (1,1);
    \draw[thick, color=blue, dashed] (0,1) .. controls (0,.2) and (1,.2) .. (1,1);
\end{tikzpicture}}
\quad=\varnothing
\end{equation}
and
\begin{equation}
\label{eqn-dashed-cup-cap}
\hackcenter{\begin{tikzpicture}
    \draw[thick, color=blue, dashed] (0,0) .. controls (0,.8) and (1,.8) .. (1,0);
    \draw[thick, color=blue, dashed] (0,2) .. controls (0,1.2) and (1,1.2) .. (1,2);
\end{tikzpicture}}
\quad=\quad
\hackcenter{\begin{tikzpicture}
    \draw[thick, color=blue, dashed] (0,0) -- (0, 2);
    \draw[thick, color=blue, dashed] (1,0) -- (1,2);
\end{tikzpicture}}
\quad,
\end{equation}
and the fact that $\alpha_F$ and $\alpha_F^{-1}$ are inverses is encoded by
\begin{equation}
\label{eqn-alpha-alpha-inverse}
\hackcenter{\begin{tikzpicture}
    \draw[thick] (0,0) [out=90, in=-90] to (1,1) [out=90, in=-90] to (0,2);
    \draw[thick, color=blue, dashed] (1,0) [out=90, in=-90] to (0,1) [out=90, in=-90] to (1,2);
\end{tikzpicture}}
\quad=\quad
\hackcenter{\begin{tikzpicture}
    \draw[thick] (0,0) -- (0,2);
    \draw[thick, color=blue, dashed] (1,0) -- (1,2);
\end{tikzpicture}}
\qquad\text{and}\qquad
\hackcenter{\begin{tikzpicture}
    \draw[thick, color=blue, dashed] (0,0) [out=90, in=-90] to (1,1) [out=90, in=-90] to (0,2);
    \draw[thick] (1,0) [out=90, in=-90] to (0,1) [out=90, in=-90] to (1,2);
\end{tikzpicture}}
\quad=\quad
\hackcenter{\begin{tikzpicture}
    \draw[thick, color=blue, dashed] (0,0) -- (0,2);
    \draw[thick] (1,0) -- (1,2);
\end{tikzpicture}}\quad.
\end{equation}
These relations already imply as well
\begin{equation}
\hackcenter{\begin{tikzpicture}
    \draw[thick, color=blue, dashed] (0,0) -- (0,1.5);
    \draw[thick, color=blue, dashed] (.5,1.5) .. controls (.5,.7) and (1.5,.7) .. (1.5,1.5);
\end{tikzpicture}}
\quad=\quad
\hackcenter{\begin{tikzpicture}
    \draw[thick, color=blue, dashed] (0,1.5) .. controls (0,.7) and (1,.7) .. (1,1.5);
    \draw[thick, color=blue, dashed] (1.5,0) -- (1.5,1.5);
\end{tikzpicture}}
\quad,\qquad
\hackcenter{\begin{tikzpicture}
    \draw[thick, color=blue, dashed] (0,0) -- (0,.75) .. controls (0,1.15) and (.5,1.15) .. (.5,.75) .. controls (.5,.35) and (1,.35) .. (1,.75) -- (1,1.5);
\end{tikzpicture}}
\quad=\quad
\hackcenter{\begin{tikzpicture}
    \draw[thick, color=blue, dashed] (0,0) -- (0,1.5);
\end{tikzpicture}}
\quad=\quad
\hackcenter{\begin{tikzpicture}
    \draw[thick, color=blue, dashed] (1,0) -- (1,.75) .. controls (1,1.15) and (.5,1.15) .. (.5,.75) .. controls (.5,.35) and (0,.35) .. (0,.75) -- (0,1.5);
\end{tikzpicture}}
\quad.
\end{equation}
In particular, $\xi,\xi^{-1}$ are a biadjoint pair. Then the relation $\onebb_F\otimes\xi_\Cc=(\xi_\Dc\otimes\onebb_F)(\onebb_{\Psi_\Dc}\otimes\alpha_F)(\alpha_F\otimes\onebb_{\Psi_\Cc})$ in the definition of a superfunctor is equivalent to:
\begin{equation}\label{eqn-defn-superfunctor}
\hackcenter{\begin{tikzpicture}
    \draw[thick] (0,0) -- (0,1.5);
    \draw[thick, color=blue, dashed] (.5,0) .. controls (.5,.8) and (1.5,.8) .. (1.5,0);
\end{tikzpicture}}
\quad=\quad
\hackcenter{\begin{tikzpicture}
    \draw[thick] (0,0) [out=90, in=-90] to (1.5,1.5);
    \draw[thick, color=blue, dashed] (.75,0) [out=90, in=-90] to (0,1) .. controls (0,1.8) and (1,1.8) .. (1.5,.5) -- (1.5,0);
\end{tikzpicture}}\quad.
\end{equation}

\begin{lem}\label{lem-dashed-cyclicity} Any superfunctor $(F,\alpha_F):(\Cc,\Psi_\Cc)\to(\Dc,\Psi_\Dc)$ satisfies pitchfork-cyclicity with respect to the biadjunction $\xi\dashv\xi^{-1}\dashv\xi$.\end{lem}
\begin{proof} In equations, the statement of the lemma is that the following both hold:
\begin{equation*}\begin{split}
&(\onebb_F\otimes\xi_\Cc)(\alpha_F^{-1}\otimes\onebb_{\Psi_\Cc})=(\xi_\Dc\otimes\onebb_F)(\onebb_{\Psi_\Dc}\otimes\alpha_F),\\
&(\alpha_F\otimes\onebb_{\Psi_\Cc})(\onebb_F\otimes\xi^{-1})=(\onebb_{\Psi_\Dc}\otimes\alpha_F^{-1})(\xi^{-1}\otimes\onebb_F).
\end{split}\end{equation*}
Or, diagrammatically,
\begin{equation*}
\hackcenter{\begin{tikzpicture}
    \draw[thick, color=blue, dashed] (0,0) .. controls (0,.8) and (1,.8) .. (1,0);
    \draw[thick] (.5,0) [out=90, in=-90] to (0,1) -- (0,1.5);
\end{tikzpicture}}
\quad=\quad
\hackcenter{\begin{tikzpicture}
    \draw[thick, color=blue, dashed] (0,0) .. controls (0,.8) and (1,.8) .. (1,0);
    \draw[thick] (.5,0) [out=90, in=-90] to (1,1) -- (1,1.5);
\end{tikzpicture}}
\quad,\qquad
\hackcenter{\begin{tikzpicture}
    \draw[thick, color=blue, dashed] (0,1.5) .. controls (0,.7) and (1,.7) .. (1,1.5);
    \draw[thick] (.5,1.5) [out=-90, in=90] to (0, .5) -- (0,0);
\end{tikzpicture}}
\quad=\quad
\hackcenter{\begin{tikzpicture}
    \draw[thick, color=blue, dashed] (0,1.5) .. controls (0,.7) and (1,.7) .. (1,1.5);
    \draw[thick] (.5,1.5) [out=-90, in=90] to (1,.5) -- (1,0);
\end{tikzpicture}}
\quad.
\end{equation*}
The first one follows from a diagrammatic calculation:
\begin{equation*}
\hackcenter{\begin{tikzpicture}
    \draw[thick, color=blue, dashed] (0,0) .. controls (0,.75) and (1,.75) .. (1,0);
    \draw[thick] (.5,0) [out=90, in=-90] to (0,1) -- (0,3);
\end{tikzpicture}}
\quad\refequal{\eqref{eqn-dashed-bubble}}\quad
\hackcenter{\begin{tikzpicture}
    \draw[thick, color=blue, dashed] (0,0) .. controls (0,.8) and (1,.8) .. (1,0);
    \draw[thick] (.5,0) [out=90, in=-90] to (0,1) -- (0,3);
    \draw[thick, color=blue, dashed] (.5,1.5) .. controls (.5,2.25) and (1.5,2.25) .. (1.5,1.5) .. controls (1.5,.75) and (.5,.75) .. (.5,1.5);
\end{tikzpicture}}
\quad\refequal{\eqref{eqn-defn-superfunctor}}\quad
\hackcenter{\begin{tikzpicture}
    \draw[thick, color=blue, dashed] (0,0) .. controls (0,.8) and (1,.8) .. (1,0);
    \draw[thick] (.5,0) [out=90, in=-90] to (0,1) -- (0,1) [out=90, in=-90] to (1.5,3);
    \draw[thick, color=blue, dashed] (0,2.5) .. controls (0,3.3) and (1,3.3) .. (1,2.8);
    \draw[thick, color=blue, dashed] (1.5,1.5) .. controls (1.5,.75) and (.5,.75) .. (.5,1.2);
    \draw[thick, color=blue, dashed] (.5,1.2) -- (0,2.5);
    \draw[thick, color=blue, dashed] (1.5,1.5) -- (1,2.8);
\end{tikzpicture}}
\quad\refequal{\eqref{eqn-dashed-cup-cap}}\quad
\hackcenter{\begin{tikzpicture}
    \draw[thick, color=blue, dashed] (0,0) [out=90, in=-90] to (.5,1.2);
    \draw[thick, color=blue, dashed] (1,0) [out=90, in=-90] to (1.5,1.5);
    \draw[thick] (.5,0) [out=90, in=-90] to (0,1) -- (0,1) [out=90, in=-90] to (1.5,3);
    \draw[thick, color=blue, dashed] (0,2.5) .. controls (0,3.3) and (1,3.3) .. (1,2.8);
    \draw[thick, color=blue, dashed] (.5,1.2) -- (0,2.5);
    \draw[thick, color=blue, dashed] (1.5,1.5) -- (1,2.8);
\end{tikzpicture}}
\quad\refequal{\eqref{eqn-alpha-alpha-inverse}}\quad
\hackcenter{\begin{tikzpicture}
    \draw[thick, color=blue, dashed] (0,0) .. controls (0,.8) and (1,.8) .. (1,0);
    \draw[thick] (.5,0) [out=90, in=-90] to (1,1) -- (1,3);
\end{tikzpicture}}
\quad.
\end{equation*}
The second is proved similarly.
\end{proof}

The added utility in explicitly drawing the superfunctor $\Psi$ is that, in later examples, $\Psi$ will act nontrivially in a way we will want to keep visual track of.  See the discussion of the super-bimodule super-2-category below.

\begin{example} Setting $\alpha_\Psi=-\onebb_{\Psi_\Cc^2}$, the pair $(\Psi_\Cc,\alpha_\Psi)$ is a superfunctor for any supercategory $(\Cc,\Psi_\Cc)$.  We draw this as
\begin{equation}
\hackcenter{\begin{tikzpicture}
    \draw[thick, color=blue, dashed] (0,0) [out=90, in=-90] to (1,1);
    \draw[thick, color=blue, dashed] (1,0) [out=90, in=-90] to (0,1);
\end{tikzpicture}}
\quad=\quad-\quad
\hackcenter{\begin{tikzpicture}
    \draw[thick, color=blue, dashed] (0,0) -- (0,1);
    \draw[thick, color=blue, dashed] (.5,0) -- (.5,1);
\end{tikzpicture}}
\quad.
\end{equation}\end{example}

\begin{defn}\label{defn-supernatural-trans} A \emph{supernatural transformation} between the superfunctors $(F,\alpha_F),(G,\alpha_G):(\Cc,\Psi_\Cc)\to(\Dc,\Psi_\Dc)$ is a natural transformation $\varphi:F\to G$ which commutes with the natural isomorphisms $\alpha_F,\alpha_G$ in the sense that the diagram
\begin{equation}
\xymatrix{
F\otimes\Psi_\Cc\ar[rr]^-{\varphi\otimes\onebb_\Psi}\ar[d]^-{\alpha_F}&&G\otimes\Psi_\Cc\ar[d]^-{\alpha_G}\\
\Psi_\Dc\otimes F\ar[rr]^-{\onebb_\Psi\otimes\varphi}&&\Psi_\Dc\otimes G
}
\end{equation}
commutes.\end{defn}

Diagrammatically, if we draw the functors $F,G$ as solid lines, this means:
\begin{equation}
\hackcenter{\begin{tikzpicture}
    \draw[thick] (0,0) -- (0,.75) [out=90, in=-90] to (.5,1.5);
    \draw[thick, color=blue, dashed] (.5,0) -- (.5,.75) [out=90, in=-90] to (0,1.5);
    \node[draw, thick, fill=blue!20,rounded corners=4pt,inner sep=3pt] () at (0,.5) {\small$\varphi$};
\end{tikzpicture}}
\quad=\quad
\hackcenter{\begin{tikzpicture}
    \draw[thick] (0,0) [out=90, in=-90] to (.5,.75) -- (.5,1.5);
    \draw[thick, color=blue, dashed] (.5,0) [out=90, in=-90] to (0,.75) -- (0,1.5);
    \node[draw, thick, fill=blue!20,rounded corners=4pt,inner sep=3pt] () at (.5,1) {\small$\varphi$};
\end{tikzpicture}}
\quad.
\end{equation}
This sort of diagram occurs when $\varphi$ is an \emph{even} supernatural transformation.  If $\vartheta:F\to\Psi G$ is an \emph{odd} supernatural transformation, the diagrammatic presentation of Definition \ref{defn-supernatural-trans} is:
\begin{equation}
\hackcenter{\begin{tikzpicture}
    \draw[thick] (0,0) -- (0,.75) [out=90, in=-90] to (.5,1.5) -- (.5,2);
    \draw[thick, color=blue, dashed] (.5,0) -- (.5,.75) [out=90, in=-90] to (0,1.5) -- (0,2);
    \draw[thick, color=blue, dashed] (0,.8) [out=135, in=-90] to (-.5,1.5) -- (-.5,2);
    \node[draw, thick, fill=blue!20,rounded corners=4pt,inner sep=3pt] () at (0,.5) {\small$\varphi$};
\end{tikzpicture}}
\quad=\quad
\hackcenter{\begin{tikzpicture}
    \draw[thick] (0,0) [out=90, in=-90] to (.5,.75) -- (.5,2);
    \draw[thick, color=blue, dashed] (.5,0) [out=90, in=-90] to (0,.75) -- (0,2);
    \draw[thick, color=blue, dashed] (.5,1.3) [out=135, in=-90] to (-.5,2);
    \node[draw, thick, fill=blue!20,rounded corners=4pt,inner sep=3pt] () at (.5,1) {\small$\varphi$};
\end{tikzpicture}}
\quad=\quad-\quad
\hackcenter{\begin{tikzpicture}
    \draw[thick] (0,0) [out=90, in=-90] to (.5,.75) -- (.5,2);
    \draw[thick, color=blue, dashed] (.5,0) [out=90, in=-90] to (-.5,1) -- (-.5,2);
    \draw[thick, color=blue, dashed] (.5,1.3) [out=135, in=-90] to (0,2);
    \node[draw, thick, fill=blue!20,rounded corners=4pt,inner sep=3pt] () at (.5,1) {\small$\varphi$};
\end{tikzpicture}}
\quad.
\end{equation}
In particular, if $\vartheta:F\to\Psi F'$, $\varphi:G\to\Psi G'$ are two odd supernatural transformations, then
\begin{equation}\label{eqn-odd-commute-1}
(\vartheta\otimes\onebb_{G'})(\onebb_F\otimes\varphi)=(\onebb_{F'}\otimes\varphi)(\vartheta\otimes\onebb_G)
\end{equation}
as maps $FG\to\Psi F'\Psi G'$, but
\begin{equation}\label{eqn-odd-commute-2}\begin{split}
&(\xi\otimes\onebb_{F'G'})(\onebb_\Psi\otimes\vartheta\otimes\onebb_{G'})(\alpha_F\otimes\onebb_{G'})(\onebb_F\otimes\varphi)\\
&\qquad=(\xi\otimes\onebb_{F'G'})(\alpha_\Psi\otimes\onebb_{F'G'})(\onebb_\Psi\otimes\alpha_{F'}\otimes\onebb_{G'})(\onebb_{\Psi F'}\otimes\varphi)(\vartheta\otimes\onebb_G)\\
&\qquad=-(\xi\otimes\onebb_{F'G'})(\onebb_\Psi\otimes\alpha_{F'}\otimes\onebb_{G'})(\onebb_{\Psi F'}\otimes\varphi)(\vartheta\otimes\onebb_G)
\end{split}\end{equation}
as maps $FG\to F'G'$.  Diagrammatically, equations \eqref{eqn-odd-commute-1} and \eqref{eqn-odd-commute-2} are expressed as:
\begin{equation}\label{eqn-s2c-commute-1}
\hackcenter{\begin{tikzpicture}
    \draw[thick] (0,0) -- (0,2);
    \draw[thick, color=blue, dashed] (0,1) [out=135, in=-90] to (-.5,1.5) -- (-.5,2);
    \node[draw, thick, fill=blue!20,rounded corners=4pt,inner sep=3pt] () at (0,1) {\small$\vartheta$};
    \draw[thick] (1,0) -- (1,2);
    \draw[thick, color=blue, dashed] (1,.5) [out=135, in=-90] to (.5, 1) -- (.5,2);
    \node[draw, thick, fill=blue!20,rounded corners=4pt,inner sep=3pt] () at (1,.5) {\small$\varphi$};
\end{tikzpicture}}
\quad=\quad
\hackcenter{\begin{tikzpicture}
    \draw[thick] (0,0) -- (0,2);
    \draw[thick, color=blue, dashed] (0,.5) [out=135, in=-90] to (-.5,1) -- (-.5,2);
    \node[draw, thick, fill=blue!20,rounded corners=4pt,inner sep=3pt] () at (0,.5) {\small$\vartheta$};
    \draw[thick] (1,0) -- (1,2);
    \draw[thick, color=blue, dashed] (1,1) [out=135, in=-90] to (.5, 1.5) -- (.5,2);
    \node[draw, thick, fill=blue!20,rounded corners=4pt,inner sep=3pt] () at (1,1) {\small$\varphi$};
\end{tikzpicture}}\quad,
\end{equation}
\begin{equation}\label{eqn-s2c-commute-2}
\hackcenter{\begin{tikzpicture}
    \draw[thick] (0,0) -- (0,2);
    \draw[thick, color=blue, dashed] (0,1.2) [out=135, in=-90] to (-.5,1.7);
    \node[draw, thick, fill=blue!20,rounded corners=4pt,inner sep=3pt] () at (0,1.2) {\small$\vartheta$};
    \draw[thick] (1,0) -- (1,2);
    \draw[thick, color=blue, dashed] (1,.5) [out=135, in=-90] to (-1,1.5) -- (-1,1.7);
    \node[draw, thick, fill=blue!20,rounded corners=4pt,inner sep=3pt] () at (1,.5) {\small$\varphi$};
    \draw[thick, color=blue, dashed] (-.5,1.7) [out=90, in=90] to (-1,1.7);
\end{tikzpicture}}
\quad=\quad
\hackcenter{\begin{tikzpicture}
    \draw[thick] (0,0) -- (0,2);
    \draw[thick, color=blue, dashed] (0,.5) [out=135, in=-90] to (-.5,1) -- (-.5,1.75);
    \node[draw, thick, fill=blue!20,rounded corners=4pt,inner sep=3pt] () at (0,.5) {\small$\vartheta$};
    \draw[thick] (1,0) -- (1,2);
    \draw[thick, color=blue, dashed] (1,1) [out=135, in=-90] to (-1,1.75);
    \node[draw, thick, fill=blue!20,rounded corners=4pt,inner sep=3pt] () at (1,1) {\small$\varphi$};
    \draw[thick, color=blue, dashed] (-.5,1.75) [out=90, in=90] to (-1,1.75);
\end{tikzpicture}}
\quad=\quad-\quad
\hackcenter{\begin{tikzpicture}
    \draw[thick] (0,0) -- (0,2);
    \draw[thick, color=blue, dashed] (0,.5) [out=135, in=-90] to (-1,1.5) -- (-1,1.75);
    \node[draw, thick, fill=blue!20,rounded corners=4pt,inner sep=3pt] () at (0,.5) {\small$\vartheta$};
    \draw[thick] (1,0) -- (1,2);
    \draw[thick, color=blue, dashed] (1,1) [out=135, in=-90] to (-.5,1.75);
    \node[draw, thick, fill=blue!20,rounded corners=4pt,inner sep=3pt] () at (1,1) {\small$\varphi$};
    \draw[thick, color=blue, dashed] (-.5,1.75) [out=90, in=90] to (-1,1.75);
\end{tikzpicture}}
\quad.
\end{equation}

\subsubsection{Super-2-categories in general}\label{subsubsec-super-2-cat-general}

Having treated the example of supercategories, superfunctors, and supernatural transformations in detail both algebraically and diagrammatically, we now recall the definition of a super-2-category from~\cite[Definition 7.9]{KKO2}.

\begin{defn}\label{defn-super-2-cat-second} A \emph{super-2-category} $\Cc$ consists of:
\begin{itemize}
\item a collection of \emph{objects} (or \emph{0-morphisms}),
\item for each pair of objects $X,Y$ a collection of \emph{1-morphisms} denoted $\Hom(X,Y)$, and
\item for each pair of 1-morphisms with the same domain and codomain $F,G:X\to Y$ a collection of \emph{2-morphisms} denoted $\Hom(F,G)$;
\item for objects $X,Y,Z$, compositions maps for 1-morphisms $\Hom(Y,Z)\times\Hom(X,Y)\to\Hom(X,Z)$ which we write as $(G,F)\mapsto G\otimes F$ or simply $(G,F)\mapsto GF$,
\item for 1-morphisms $F,G,H\in\Hom(X,Y)$, a vertical composition map $\Hom(G,H)\times\Hom(F,G)\to\Hom(F,H)$ which we write as $(\vartheta,\varphi)\mapsto\vartheta\varphi$,
\item for 1-morphisms $F,F'\in\Hom(X,Y)$ and $G,G'\in\Hom(Y,Z)$, a horizontal composition map $\Hom(G\otimes F',G'\otimes F')\times\Hom(G\otimes F,G\otimes F')$ denoted $(\vartheta\otimes\onebb_{F'},\onebb_G\otimes\varphi)\mapsto\vartheta\otimes\varphi$,
\item for each object $X$ an \emph{identity 1-morphism} $\onebb_X\in\Hom(X,X)$,
\item for each 1-morphism $F$ an \emph{identity 2-morphism} $\onebb_F\in\Hom(F,F)$,
\item for each object $X$ a \emph{parity shift 1-morphism} $\Psi_X\in\Hom(X,X)$ and an invertible 2-morphism $\xi\in\Hom(\Psi_X^2,\onebb_X)$, and
\item for each 1-morphism $F:X\to Y$ a 2-morphism $\alpha_F\in\Hom(F\Psi_X,\Psi_YF)$.
\end{itemize}
These data are subject to the following conditions:
\begin{itemize}
\item horizontal composition is associative, and $\onebb_X$ is a unit for each object $X$;
\item vertical composition is associative, and $\onebb_F$ is a unit for each 1-morphism $F$;
\item interchange law: if $F,F',F'',G,G',G''$ are 1-morphisms and $\alpha,\alpha',\beta,\beta'$ are 2-morphisms such that the compositions
\begin{equation*}
(\alpha'\otimes\beta')(\alpha\otimes\beta),\quad(\alpha'\alpha)\otimes(\beta'\beta)
\end{equation*}
both make sense, then these two expressions are equal;
\item for all objects $X$, the 2-morphism $\xi_X:\Psi_X^2\to\onebb_X$ is invertible;
\item if $GF$ is the composition of the 1-morphisms $F$ and $G$, then $\alpha_{G\otimes F}=(\alpha_G\otimes\onebb_F)(\onebb_G\otimes\alpha_F$);
\item for all 1-morphisms $F:X\to Y$, the 2-morphism $\alpha_F:F\otimes\Psi_X\to\Psi_Y\otimes F$ is invertible;
\item for all objects $X$, there is an equality $\xi_X\otimes\onebb_{\Psi_X}=\onebb_{\Psi_X}\otimes\xi_X$ of 2-morphisms $\Psi_X^3\to\Psi_X$;
\item for all 1-morphisms $F:X\to Y$,
\begin{equation}
\onebb_F\otimes\xi_X=(\xi_Y\otimes\onebb_F)(\onebb_{\Psi_Y}\otimes\alpha_F)(\alpha_F\otimes\onebb_{\Psi_X});
\end{equation}
\item for all objects $X$, $\alpha_{\xi_X}=-\onebb_{\xi_X^2}$;
\item for all 2-morphisms $\vartheta:F\to G$, $\alpha_G\vartheta=\vartheta\alpha_F$.
\end{itemize}
\end{defn}

\begin{example} There is a super-2-category $\SCat$ whose objects are supercategories, 1-morphisms are superfunctors, and 2-morphisms are supernatural transformations.\end{example}

The general diagrammatics for super-2-categories are exactly as described above for the example $\SCat$.

\begin{example} There is a super-2-category $\SBim$ whose objects are superalgebras, 1-morphisms are super-bimodules, and 2-morphisms are (even) homomorphisms of super-bimodules.  The parity shift 1-morphism is denoted $\Pi$; if $M$ is an $(A,B)$-super-bimodule, then $\Pi M$ is the $(A,B)$-super-bimodule whose underlying space is $M$ with the $\Zt$ grading reversed and the left $A$-action twisted by the parity involution on $A$.  That is,
\begin{equation*}
a\in A\text{ acts on }\Pi M\text{ by sending }m\text{ to }\iota_A(a)m=(-1)^{p(a)}am.
\end{equation*}
Horizontal compositions of 1-morphisms is given by the tensor product of super-bimodules.  There is also the obvious $\Z$-graded variant of this super-2-category.\end{example}

In the formalism of super-2-categories, the generator of ``the $\Z_2$-action'' is the family of 1-morphisms $\lbrace\Psi_X\rbrace$.  As equations \eqref{eqn-s2c-commute-1}, \eqref{eqn-s2c-commute-2} show, this leads to nontrivial signs between 2-morphisms.  These signs are a more general behavior than the ordinary commutativity resulting from the interchange law for 2-categories.  A basic example: for any object $X$ of a super-2-category, the commutative algebra $\End(\onebb_X)$ is now the even part of a supercommutative superalgebra $\Pi\End(\onebb_X):=\End(\onebb_X\oplus\Psi_X)$.

%---------------------------------------------------------------------
\subsection{Super-2-functors}
%---------------------------------------------------------------------

\begin{defn} A super-2-functor $\Gc:\Cc\to\Dc$ consists of the following data:
\begin{itemize}
\item a function $\Gc:\text{Ob}(\Cc)\to\text{Ob}(\Dc)$,
\item for each pair $X,Y\in\text{Ob}(\Cc)$, a functor $\Hom(X,Y)\to\Hom(\Gc(X),\Gc(Y))$
\end{itemize}
such that:
\begin{itemize}
\item for $X\in\text{Ob}(\Cc)$, the functor $\Hom(X,X)\to\Hom(\Gc(X),\Gc(X))$ takes $\Psi_X$ to $\Psi_{\Gc(X)}$ and $\xi_X$ to $\xi_{\Gc(X)}$,
\item for $X,Y\in\text{Ob}(\Cc)$, the functor $\Hom(X,Y)\to\Hom(\Gc(X),\Gc(Y))$ takes $\alpha_F$ to $\alpha_{\Gc(F)}$.
\end{itemize}
\end{defn}

Or, in other words, a super-2-functor maps objects to objects, 1-morphisms to 1-morphisms, and 2-morphisms, compatibly with horizontal composition, vertical composition, parity shift 1-morphisms, and ``dashed-solid'' crossings.

%---------------------------------------------------------------------
\subsection{Super diagram conventions} \label{subsec-superconventions}
%---------------------------------------------------------------------

A consequence of the axioms
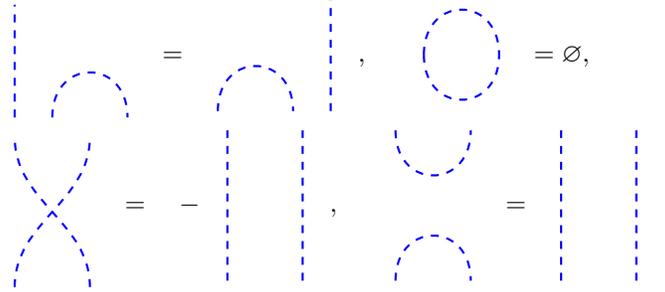
\begin{equation*}\begin{split}
&\hackcenter{\begin{tikzpicture}
    \draw[thick, color=blue, dashed] (0,0) -- (0,1.5);
    \draw[thick, color=blue, dashed] (.5,0) .. controls (.5,.8) and (1.5,.8) .. (1.5,0);
\end{tikzpicture}}
\quad=\quad
\hackcenter{\begin{tikzpicture}
    \draw[thick, color=blue, dashed] (0,0) .. controls (0,.8) and (1,.8) .. (1,0);
    \draw[thick, color=blue, dashed] (1.5,0) -- (1.5,1.5);
\end{tikzpicture}}
\quad,\qquad
\hackcenter{\begin{tikzpicture}
    \draw[thick, color=blue, dashed] (0,1) .. controls (0,1.8) and (1,1.8) .. (1,1);
    \draw[thick, color=blue, dashed] (0,1) .. controls (0,.2) and (1,.2) .. (1,1);
\end{tikzpicture}}
\quad=\varnothing,\\
&\hackcenter{\begin{tikzpicture}
    \draw[thick, color=blue, dashed] (0,0) [out=90, in=-90] to (1,2);
    \draw[thick, color=blue, dashed] (1,0) [out=90, in=-90] to (0,2);
\end{tikzpicture}}
\quad=\quad-\quad
\hackcenter{\begin{tikzpicture}
    \draw[thick, color=blue, dashed] (0,0) -- (0,2);
    \draw[thick, color=blue, dashed] (1,0) -- (1,2);
\end{tikzpicture}}
\quad,\qquad
\hackcenter{\begin{tikzpicture}
    \draw[thick, color=blue, dashed] (0,0) .. controls (0,.8) and (1,.8) .. (1,0);
    \draw[thick, color=blue, dashed] (0,2) .. controls (0,1.2) and (1,1.2) .. (1,2);
\end{tikzpicture}}
\quad=\quad
\hackcenter{\begin{tikzpicture}
    \draw[thick, color=blue, dashed] (0,0) -- (0, 2);
    \draw[thick, color=blue, dashed] (1,0) -- (1,2);
\end{tikzpicture}}
\quad,
\end{split}\end{equation*}
is that for each object $X$ in a super-2-category $\mathcal{C}$, the space of 2-morphisms between 1-morphisms $\Psi^a_X$ and $\Psi^b_X$ with $a$ congruent to $b$ modulo two consisting only of dashed lines is one dimensional.  The relations imply that any diagram pairing the $a+b$ endpoints with no intersecting strands represents the same 2-morphism in $\mathcal{C}$.  Choosing any such diagram we obtain a canonical isomorphism $\Psi^a_X\to\Psi^b_X$ whenever $a$ is congruent to $b$ modulo two.  In this case we introduce a shorthand to simplify our graphical calculus.  We express this isomorphism using a thickened strand
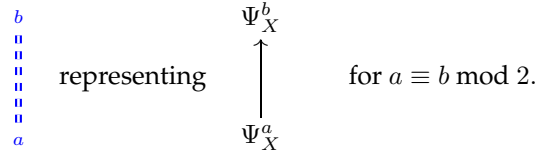
\begin{equation*}
\hackcenter{\begin{tikzpicture}
    \draw[thick, color=blue, double distance=1pt, dashed] (0,0) -- (0,1.2)
        node[above, blue](){$\scs b$};
    \node[blue] at (0,-.25){$\scs a$};
\end{tikzpicture}}\quad \text{representing} \quad
\xy
 (0,-8)*+{\Psi_X^a}="1" ; (0,8)*+{\Psi_X^b}="2";
 {\ar@{->} "1"; "2"};
\endxy
\qquad\text{ for }a\equiv b\text{ mod }2.
\end{equation*}

In most cases the labelling of the source and target will be clear from the context and we will often omit these labels.  It will also be convenient to allow negative labels for these thick strands with the interpretation that a thick strand with source labeled $a$ and target labelled $b$ represents the canonical isomorphism between $\Psi_X^{|a|}$ and $\Psi_X^{|b|}$.

%#####################################################################
%
\section{Odd categorified structures and their actions}
%
%#####################################################################

%---------------------------------------------------------------------
\subsection{Strong supercategorical actions}\label{subsec-strong-supercat-action}
%---------------------------------------------------------------------

We will concern ourselves with two ``strong'' notion of categorification of $\Udotpi$, the covering version of Lusztig's idempotented form of quantum $\sltwo$ at a generic parameter $q$.  The first is that of a strong supercategorical action, which we define in this section.  The second, which is \emph{a priori} stronger, is a 2-functor from a certain super-2-category $\Udotc$.  The first main result of this paper is to prove that a strong supercategorical action always extends to such a 2-functor.

\begin{defn} \label{def_strong}
Let $\Cc$ be a graded idempotent complete $\Bbbk$-linear 2-category such that
\begin{itemize}
\item The objects of $\Cc$ are indexed by the integral weights of $\sltwo$ (which we identify with $\Z$); write $\oneb_\lambda$ for the identity 1-morphism of the object $\lambda$;
\item For each weight $\lambda$ there are 1-morphisms
\begin{equation*}
\Ett\onebl:\lambda\to\lambda+2,\qquad
\onebl\Ftt:\lambda+2\to\lambda,\qquad
\Pi\onebl:\lambda\to\lambda.
\end{equation*}
We also assume that a right adjoint to $\Ett\onebl$ exists and that there are fixed adjunctions
\begin{equation}
\onebl\Ftt\ads{-\l-1}\dashv\Ett\onebl,\qquad
\Pi\onebl\dashv\Pi\onebl,
\end{equation}
and natural isomorphisms
\begin{equation}
\alpha_\Ett:\Ett\Pi\to\Pi\Ett,\qquad\alpha_\Ftt:\Ftt\Pi\to\Pi\Ftt.
\end{equation}
\end{itemize}
All 1-morphisms of $\Cc$ are generated from $\Ett,\Ftt,\Pi$ by taking direct sums, compositions, and grading shifts.  A \emph{strong supercategorical action} of $\Udotpi$ on $\Cc$ consists of the following data and conditions:
\begin{enumerate}
\item (Integrability) The object $\lambda+2r$ is isomorphic to the zero object for $r\ll0$ and for $r\gg0$.
\item   \label{co:hom} (Brick condition) $\Hom_\Cc(\onebl,\Pi^k\onebl\ads{\ell})=0$ if $\ell<0$ and is one-dimensional if $\ell=0$ and $k=0$.  Moreover, the space of 2-morphisms between any two 1-morphisms if finite dimensional.
\item (Covering isomorphisms) We are given isomorphisms in $\Cc$:
\begin{align} \label{eq:EF-rel}
&\Ftt\Ett\onebl\cong\Ett\Pi\Ftt\onebl\oplus
\bigoplus_{k=0}^{-\lambda-1}\Pi^{\l+1+k}\onebl\ads{-\lambda-1-2k} &\text{if }\lambda\leq0,\\
&\Ett\Ftt\onebl\cong\Ftt\Pi\Ett\onebl\oplus
\bigoplus_{k=0}^{\lambda-1}\Pi^k \onebl\ads{\lambda-1-2k} &\text{if }\lambda\geq0.\label{eq:FE-rel}
\end{align}
\item \label{co:oddNil} (Odd nilHecke action)There are 2-morphisms $X:\Ett\to\Pi\Ett$ and $T:\Ett^2\to\Pi\Ett^2$ such that for each $n\geq1$, the 2-morphisms
\begin{equation}\begin{split}
&X_i=\alpha_\Ett^{-i+1}(\oneb^{i-1}\otimes X\otimes\oneb^{n-i}):\Ett^n\to\Pi\Ett^n,\\
&T_i=\alpha_\Ett^{-i+1}(\oneb^{i-1}\otimes T\otimes\oneb^{n-i-1}):\Ett^n\to\Pi\Ett^n.
\end{split}\end{equation}
generate an action of $\onh_n$ on $\Pi\END(\Ett^n)$.
\end{enumerate}
\end{defn}
This last piece of data, the odd nilHecke action, is the key ingredient that makes the action ``strong''; the idea behind this observation goes back to the pioneering work of Chuang and Rouquier \cite{CR}.  The relations ensuring an action of the odd nilHecke algebra are given diagrammatically in equations \eqref{eq:oddnilquad}--\eqref{eq:onil-dot}.

{\bf{Important convention.}} The integrability condition above implies that ``most'' objects are isomorphic to the zero object. If $\l$ is the zero object then, by definition, $\Hom_{\Cc}(\onebl,\Pi^k\onebl \la l \ra) = 0$ for all $l$. So, to be precise, condition (\ref{co:hom}) above should say that $\Hom_{\Cc}(\onebl,\onebl)$ is one-dimensional if $\l$ is non-zero. There are many other such instances later in this paper. The convention is that any statement about a certain Hom being non-zero assumes that all objects involved are non-zero (otherwise the Hom space is automatically zero).
\medskip

For a 1-morphism $u$, let $u^L$ (respectively $u^R$) denote its left (respectively right) adjoint.  The requirement that $\onebl\Ftt \ads{-\l-1} \dashv\Ett\onebl$ implies that $(\Ett \onebl)^L = \onebl \Ftt \la -\lambda -1 \ra$ and $\left(\Ftt\onebl\right)^R =  \Ett\oneb_{\l-2} \la -\l+1\ra$.
In what follows we make use of the fact that  $(u^L)^R=u$ and $(v^R)^L=v$ for all 1-morphisms $u,v$ and that the adjunctions give rise to isomorphisms
\begin{alignat}{3}
 \Hom(ux,y) &\; \cong \;& \Hom (x, u^R y),
 &\qquad \Hom(x,uy)&\; \cong \;&  \Hom (u^Lx,  y), \nn\\
 \Hom(xv,y) &\; \cong \;& \Hom(x,yv^L),  &\qquad \Hom(x,yv) &\; \cong \;& \Hom(xv^R,y).
\end{alignat}

% -------------------------------------------------------------------------------
%
\subsubsection{Cancellation property}
%
% -------------------------------------------------------------------------------

The fact that the space of maps between any two 1-morphisms in a strong supercategorical action is finite dimensional means that the Krull-Schmidt property holds.   This means that any 1-morphism has a unique direct sum decomposition (see Section 2.2 of \cite{Rin}). In particular, this means that if $A,B,C$ are morphisms and $V$ is a $\Z$-graded vector space then we have the following cancellation laws (see Section 4 of \cite{CK3}):
\begin{eqnarray*}
A \oplus B \cong A \oplus C &\Rightarrow& B \cong C \\
A \otimes_\k V \cong B \otimes_\k V &\Rightarrow& A \cong B.
\end{eqnarray*}
A brick in a (graded) category is an indecomposable object $A$ such that $\End^k(A)=0$ for $k<0$ and  $\End(A)=\End^0(A)\cong \Bbbk$.
For example, by Lemma \ref{lem:E} below, $\Ett \oneb_{\mu}$ is a brick.

%---------------------------------------------------------------------
\subsection{Definition of the super-2-category $\Udotc$}\label{subsec-defn-odd-udot}
%---------------------------------------------------------------------

The whole of this subsection is a definition of the super-2-category $\dot{\mathcal{U}}_{q,\pi}(\sltwo)$, or $\Udotc$ for short.  This is entirely analogous to the definition in Section 5.2 of \cite{Lau1}.

\subsubsection{Objects and 1-morphisms}

The super-2-category $\Udotc$ will be the idempotent completion (Karoubi envelope) of a combinatorially defined graded $\Bbbk$-linear super-2-category $\Uc$.  The objects of $\Uc$ are indexed by integral weights for $\sltwo$ (as usual, identified with $\Z$) and the 1-morphisms are direct sums, compositions, and degree shifts of the following generating 1-morphisms:
\begin{equation}
\Ec\onebbl:\lambda\to\lambda+2,\qquad
\onebbl\Fc:\lambda+2\to\lambda,\qquad
\Pi\onebbl:\lambda\to\lambda,
\end{equation}
where $\Pi\onebbl=\Psi\onebbl$ is the parity shift 1-morphisms for the object $\lambda$.  These generating 1-morphisms are expressed diagrammatically as strands:
\begin{equation}\begin{split}
\hackcenter{
}
\;\;=\;\;\Pi\onebbl.
\end{split}\end{equation}

\subsubsection{Generating 2-morphisms}

We give $\Uc$ the following generating 2-morphisms:
\begin{equation} \label{eq_generators}
\begin{split}
\hackcenter{%
}
\quad=\quad\eta:\onebbl\ds{\lambda-1}\to\Ec\Fc\onebbl.
\end{split}\end{equation}
For each weight there are two more generating 2-morphisms
\begin{equation} \label{eq_generators_cont}
\hackcenter{%
}
\quad=\quad\eta':\onebbl\ds{-\lambda-1}\to\Fc\Pi^{\lambda+1}\Ec\onebbl,
\end{equation}
where double dashed lines are defined as in Section~\ref{subsec-superconventions}.

There are other 2-morphisms which are implicitly defined by virtue of $\Uc$ being a super-2-category.  A few of these are:
\begin{equation} \label{eq-dashed-solid-cross}
\begin{split}
\hackcenter{%
}
\quad=\quad \alpha_\Fc^{-1}:\onebbl\Pi\Fc\to\onebbl\Fc\Pi.\\
\end{split}\end{equation}

The following diagram constructed from $\epsilon'$ will play the role of counit in the adjunction $\Fc\Pi\dashv\Ec$:
\begin{equation}
\hackcenter{%
}\quad.
\end{equation}

\subsubsection{Defining relations among generating 2-morphisms}

The definition of a super-2-category ensures that all diagrams enjoy invariance for all planar isotopies of the dashed strands.  Thus we are free to write
\begin{equation*}
\hackcenter{%
}
\qquad\text{in place of}\qquad
\hackcenter{%
}\quad.
\end{equation*}

The relations we impose on the 2-morphisms of $\Uc$ are:
\begin{itemize}
\item relations expressing the adjunctions $\onebbl\Fc\dashv\Ec\onebbl\dashv\onebbl\Fc\Pi^{\lambda+1}$
\begin{equation} \label{eq_biadjoint1}
\hackcenter{
%
 }
\end{equation}
\item the odd cyclicity relations for dots and crossings:
\begin{equation}\label{eqn-dot-cyclicity}
(-1)^{\lambda} \;\;
\hackcenter{

}
\quad=\quad
\left\{
\begin{array}{ll}
\xy (0,-2)*{
\begin{tikzpicture}[scale=0.8]
  \draw[thick] (0.5,0) .. controls (0.5,0.8) and (-0.5,0.8) .. (-0.5,0);
  \draw[thick] (1.5,0) .. controls (1.5,-.8) and (.5,-.8) .. (.5,0)
    node[pos=1, shape=coordinate](DOT){};
  \draw[thick, ->-=0.15] (-.5,0) -- (-.5,-1.5);
  \draw[thick, ->] (1.5,1.5) -- (1.5,0);
  \draw[color=blue, thick, dashed]
   (DOT) .. controls ++(-1.4,-2) and ++(.2,-2) .. (2.25,0) to (2.25,1.5);
  \draw[line width=0mm] (1.5,0) .. controls (1.5,-.8) and (.5,-.8) .. (.5,0)
    node[pos=1](){\bbullet};
    \node at (.3,1.2) {$\l-1$};
\end{tikzpicture} }; \endxy
\quad+\quad2\quad
\hackcenter{\begin{tikzpicture}[scale=0.7]
    \draw[thick, <-] (-2,-1.5) -- (-2,1.5);
  \draw[thick, ->] (-0.5,0) .. controls (-0.5,0.9) and (0.5,0.9) .. (0.5,0)
      node[pos=0.5, shape=coordinate](X){}
      node[pos=0.1, shape=coordinate](Y){};
  \draw[thick] (-0.5,0) .. controls (-0.5,-0.9) and (0.5,-0.9) .. (0.5,0)
      node[pos=0.1, shape=coordinate](Z){};
  \draw[color=blue, thick, double distance=1pt, dashed] (X) .. controls++(0,.65) and ++(-.65,.3) .. (Y) node[pos=0.15,right]{$\scs \l$\;};
  \draw[color=blue, thick,  dashed] (Z) to[out=160, in=-90] (-1.25,1.5) ;
  \node[blue] at (-1.25,0.8){$\scs $\;};
  \draw[line width=0mm] (-0.5,0) .. controls (-0.5,0.9) and (0.5,0.9) .. (0.5,0)
    node[pos=0.1]{\bbullet};
  \draw[line width=0mm] (-0.5,0) (-0.5,0) .. controls (-0.5,-0.9) and (0.5,-0.9) .. (0.5,0)
        node[pos=0.1]{\bbullet};
  \node at (.8,-1) {$\lambda+1$};
\end{tikzpicture}} & \l>0, \\ & \\
\end{array}
\right.
\end{equation}
(Note that relations \eqref{eq_biadjoint1}, \eqref{eq_biadjoint2}, \eqref{eqn-dot-cyclicity} imply local cap and cup dot slide relations that can be found in Section~\ref{subsubsec-half_cyclic}.)

\begin{align}\label{eqn-crossing-cyclicity}
\hackcenter{
%
}
\end{align}

\item  The odd nilHecke algebra relations:
\begin{eqnarray} \label{eq:oddnilquad}
&\hackcenter{%
}
\quad,
\end{align}
as well as the corresponding relations on downwards-oriented strands.
\end{itemize}

We also define left- and right-crossing diagrams by using adjoints:
\begin{equation}
\hackcenter{%
}
\end{equation}

% - - - - - - - - - - - - - - - - - - - - - - - - - - - - - - - - - - - - - - - -
%
%
\subsubsection{Dotted bubble relations}
%
% - - - - - - - - - - - - - - - - - - - - - - - - - - - - - - - - - - - - - - - -
\begin{itemize}
\item Bubbles of negative degree are zero:
\begin{equation} \label{eq:positivity}
  \hackcenter{
 %
 }
  \;\; = 0 \quad \text{for $0 \leq m <-\l-1$}.
\end{equation}
\end{itemize}

The remaining relations in the 2-category $\Uc$ are conveniently expressed using
several useful conventions for depicting dotted bubbles.  The first convention is to depict dotted bubbles in a manner that emphasizes their degree.  For $m \geq 0$, the degree $m$ clockwise and counter-clockwise bubbles are drawn as follows:
\begin{align}
\hackcenter{
%
 }; \endxy
\end{align}
where we employ the conventions
\begin{align}
\hackcenter{
%
 }
  && \text{for $\l \geq 0$ and $-\l-1 \leq m$}.
 \end{align}
Notice that some labels for dots on the left-hand side appear to involve a negative number of dots, though when the two dots are combined their sum is positive.  The dot 2-morphism is not invertible, rather the equations above define formal symbols that will simplify the presentation of the graphical calculus.

As in the graphical calculus for the 2-category $\Ucev$ it is extremely convenient to extend this notation even further by defining formal symbols
\begin{align}
\hackcenter{
%
 }
  && \text{for $\l \geq 0$ and $0 \leq m \leq \l $},
 \end{align}
that are defined inductively by the following equations:
\begin{align}\label{eq:fake-bubble}
%\begin{gathered}
\sum_{f+g=m} (-1)^{g}
\hackcenter{
%
 }; \endxy
&\;\; = \;\; \onebb_{\onebb_{0}}
& \text{if $\l=0$}.
%\end{gathered}
\end{align}
These symbols are called fake bubbles because the labels of the dots are negative, though the total degree is positive as they are defined above in terms of nonzero real bubbles. Fake bubbles are explained in greater detail in Section~\ref{subsec:fake-bubbles}.

% - - - - - - - - - - - - - - - - - - - - - - - - - - - - - - - - - - - - - - - -
%
%
\subsubsection{More defining relations for $\lambda >0$}
%
% - - - - - - - - - - - - - - - - - - - - - - - - - - - - - - - - - - - - - - - -

\begin{equation}\label{eqn-mixed-R2-relation-1}
\hackcenter{%
} &\quad =\quad 0 .
\end{align}

% - - - - - - - - - - - - - - - - - - - - - - - - - - - - - - - - - - - - - - - -
%
%
\subsubsection{More defining relations for $\lambda <0$}
%
% - - - - - - - - - - - - - - - - - - - - - - - - - - - - - - - - - - - - - - - -
%
\begin{equation}
\hackcenter{%
} &\quad =\quad 0 .
\end{align}

% - - - - - - - - - - - - - - - - - - - - - - - - - - - - - - - - - - - - - - - -
%
%
\subsubsection{More defining relations for $\lambda =0$}
%
% - - - - - - - - - - - - - - - - - - - - - - - - - - - - - - - - - - - - - - - -

\[
\hackcenter{%
}
\]

At this point, the definition of $\Uc$ is complete.

%---------------------------------------------------------------------
%
\subsection{Implicit relations among generating 2-morphisms} \label{subsec:additional}
%
%---------------------------------------------------------------------

% - - - - - - - - - - - - - - - - - - - - - - - - - - - - - - - - - - - - - - - -
%
\subsubsection{Relations involving dashed strands}
%
% - - - - - - - - - - - - - - - - - - - - - - - - - - - - - - - - - - - - - - - -

The following relations are already present from the definition of a super-2-category, but we spell them out for completeness:
\begin{itemize}
\item $\xi$ and $\xi^{-1}$ are biadjoint:
\begin{equation}
\hackcenter{\begin{tikzpicture}
    \draw[thick, color=blue, dashed] (0,0) -- (0,.75) .. controls (0,1.15) and (.5,1.15) .. (.5,.75) .. controls (.5,.35) and (1,.35) .. (1,.75) -- (1,1.5);
\end{tikzpicture}}
\quad=\quad
\hackcenter{\begin{tikzpicture}
    \draw[thick, color=blue, dashed] (0,0) -- (0,1.5);
\end{tikzpicture}}
\quad=\quad
\hackcenter{\begin{tikzpicture}
    \draw[thick, color=blue, dashed] (1,0) -- (1,.75) .. controls (1,1.15) and (.5,1.15) .. (.5,.75) .. controls (.5,.35) and (0,.35) .. (0,.75) -- (0,1.5);
\end{tikzpicture}}
\end{equation}
\item the biadjunction $\xi\dashv\xi^{-1}\dashv\xi$ enjoys pitchfork-cyclicity for all 1-morphisms (by Lemma \ref{lem-dashed-cyclicity}, which is stated for $\SCat$ but holds in any super-2-category):
\begin{equation}
\hackcenter{\begin{tikzpicture}
    \draw[thick, color=blue, dashed] (0,0) .. controls (0,.8) and (1,.8) .. (1,0);
    \draw[thick] (.5,0) [out=90, in=-90] to (0,1) -- (0,1.5);
\end{tikzpicture}}
\quad=\quad
\hackcenter{\begin{tikzpicture}
    \draw[thick, color=blue, dashed] (0,0) .. controls (0,.8) and (1,.8) .. (1,0);
    \draw[thick] (.5,0) [out=90, in=-90] to (1,1) -- (1,1.5);
\end{tikzpicture}}
\quad,\qquad
\hackcenter{\begin{tikzpicture}
    \draw[thick, color=blue, dashed] (0,1.5) .. controls (0,.7) and (1,.7) .. (1,1.5);
    \draw[thick] (.5,1.5) [out=-90, in=90] to (0, .5) -- (0,0);
\end{tikzpicture}}
\quad=\quad
\hackcenter{\begin{tikzpicture}
    \draw[thick, color=blue, dashed] (0,1.5) .. controls (0,.7) and (1,.7) .. (1,1.5);
    \draw[thick] (.5,1.5) [out=-90, in=90] to (1,.5) -- (1,0);
\end{tikzpicture}}
\end{equation}
where in either equation both strands can be upwards-oriented or both strands can be downwards-oriented.
\item $\xi^{-1}$ is, in fact, the inverse of $\xi$:
\begin{equation}
\hackcenter{\begin{tikzpicture}
    \draw[thick, color=blue, dashed] (0,0) .. controls (0,.8) and (1,.8) .. (1,0);
    \draw[thick, color=blue, dashed] (0,2) .. controls (0,1.2) and (1,1.2) .. (1,2);
\end{tikzpicture}}
\quad=\quad
\hackcenter{\begin{tikzpicture}
    \draw[thick, color=blue, dashed] (0,0) -- (0, 2);
    \draw[thick, color=blue, dashed] (1,0) -- (1,2);
\end{tikzpicture}}
\quad,\qquad
\hackcenter{\begin{tikzpicture}
    \draw[thick, color=blue, dashed] (0,1) .. controls (0,1.8) and (1,1.8) .. (1,1);
    \draw[thick, color=blue, dashed] (0,1) .. controls (0,.2) and (1,.2) .. (1,1);
\end{tikzpicture}}
\quad=\varnothing.
\end{equation}
\item $\alpha_\Ec^{-1}$ is, in fact, the inverse of $\alpha_\Ec$:
\begin{equation}
\hackcenter{\begin{tikzpicture}
    \draw[thick, ->] (0,0) [out=90, in=-90] to (1,1) [out=90, in=-90] to (0,2);
    \draw[thick, color=blue, dashed] (1,0) [out=90, in=-90] to (0,1) [out=90, in=-90] to (1,2);
\end{tikzpicture}}
\quad=\quad
\hackcenter{\begin{tikzpicture}
    \draw[thick, ->] (0,0) -- (0,2);
    \draw[thick, color=blue, dashed] (1,0) -- (1,2);
\end{tikzpicture}}
\quad,\qquad
\hackcenter{\begin{tikzpicture}
    \draw[thick, color=blue, dashed] (0,0) [out=90, in=-90] to (1,1) [out=90, in=-90] to (0,2);
    \draw[thick, ->] (1,0) [out=90, in=-90] to (0,1) [out=90, in=-90] to (1,2);
\end{tikzpicture}}
\quad=\quad
\hackcenter{\begin{tikzpicture}
    \draw[thick, color=blue, dashed] (0,0) -- (0,2);
    \draw[thick, ->] (1,0) -- (1,2);
\end{tikzpicture}}\quad.
\end{equation}
And likewise for $\alpha_\Fc^{\pm1}$ (reflect the above equations about a horizontal axis).
\item Relations that involves pulling $\alpha$'s through generating 2-morphisms:
\begin{equation}
\hackcenter{\begin{tikzpicture}
    \draw[thick, ->] (0,0) .. controls (0,1) and (1,1) .. (1,2)
        node[pos=.5, shape=coordinate](DOT){};
    \draw[thick, color=blue, dashed] (DOT) -- (.5,2)
        node[pos=0](){\bbullet};
    \draw[thick, color=blue, dashed] (.5,0) .. controls (.5,.5) and (0,.5) .. (0,1) -- (0,2);
\end{tikzpicture}}
\quad=\quad
\hackcenter{\begin{tikzpicture}
    \draw[thick, ->] (0,0) .. controls (0,1) and (1,1) .. (1,2)
        node[pos=.5, shape=coordinate](DOT){};
    \draw[thick, color=blue, dashed] (DOT) .. controls ++(-0.75,.1) and ++(0,-.5).. (.5,2)
        node[pos=0](){\bbullet};
    \draw[thick, color=blue, dashed] (1,0) -- (1,.65) ..
       controls ++(0,1) and ++(0,-.5) .. (0,2);
\end{tikzpicture}}
\quad,\qquad
\hackcenter{\begin{tikzpicture}
  \draw[thick, ->] (0,0) .. controls ++(0,1) and ++(0,-0.5) ..(2,2);
    \draw[thick, ->] (2,1) .. controls ++(0,0.5) and ++(0,-0.5) .. (1,2)
        node[pos=.5, shape=coordinate](CROSSING){};
    \draw[thick, color=blue, dashed] (CROSSING) to[out=180, in=-90] (0.25,2);
   \draw[thick] (2,1) .. controls ++(0,-0.5) and ++(0,0.5) ..(1,0);
    \draw[thick, color=blue, dashed]
   (2,0) .. controls ++(0,.75) and ++(0,-.75) .. (-0.5,2);
\end{tikzpicture}}
\quad = \quad
\hackcenter{\begin{tikzpicture}
  \draw[thick, ->] (0,0) .. controls ++(0,0.5) and ++(0,-1) ..(2,2);
  \draw[thick, ->] (0,1) .. controls ++(0,0.5) and ++(0,-0.5) .. (1,2);
  \draw[thick] (0,1) .. controls ++(0,-0.5) and ++(0,0.5) ..(1,0)
        node[pos=.5, shape=coordinate](CROSSING){};
    \draw[thick, color=blue, dashed]
  (CROSSING) .. controls ++(-0.5,-.1) and ++(0,-.2) ..
    (-0.5,1) .. controls ++(0,.4) and ++(0,-.5) ..(0.25,2);
    \draw[thick, color=blue, dashed]
   (2,0) .. controls ++(0,.75) and ++(0,-.75) .. (-0.5,2);
\end{tikzpicture}}
\end{equation}
the analogues of the previous two for $\Fc$ (reflect about a horizontal axis),
\begin{equation}
\hackcenter{\begin{tikzpicture}
    \draw[thick, ->] (.5,0) -- (.5,.5) .. controls (.5,.85) and (0,.85) .. (0,.5) -- (0,0);
    \draw[thick, color=blue, dashed] (1,0) -- (1,1) .. controls (1,1.25) and (-.5,1.25) .. (-.5,1.5);
\end{tikzpicture}}
\quad=\quad
\hackcenter{\begin{tikzpicture}
    \draw[thick, ->] (.5,0) -- (.5,.5) .. controls (.5,.85) and (0,.85) .. (0,.5) -- (0,0);
    \draw[thick, color=blue, dashed] (1,0) .. controls (1,.5) and (-.5,.5) .. (-.5,1) -- (-.5,1.5);
\end{tikzpicture}}
\quad,\qquad
\hackcenter{\begin{tikzpicture}
    \draw[thick, <-] (.5,0) -- (.5,.5) .. controls (.5,.85) and (0,.85) .. (0,.5)
        node[pos=.5, shape=coordinate](TOPCAP){}
    -- (0,0);
    \draw[thick, color=blue, dashed] (1,0) -- (1,1) .. controls (1,1.25) and (-.5,1.25) .. (-.5,1.5);
    \draw[thick, color=blue, double distance=1pt, dashed] (TOPCAP) -- (.25,1.5)
        node[pos=1, above](){$\scs \lambda+1$};
\end{tikzpicture}}
\quad=\quad
\hackcenter{\begin{tikzpicture}
    \draw[thick, <-] (.5,0) -- (.5,.5) .. controls (.5,.85) and (0,.85) .. (0,.5)
        node[pos=.5, shape=coordinate](TOPCAP){}
    -- (0,0);
    \draw[thick, color=blue, dashed] (1,0) .. controls (1,.5) and (-.5,.5) .. (-.5,1) -- (-.5,1.5);
    \draw[thick, color=blue, double distance=1pt, dashed] (TOPCAP) -- (.25,1.5)
        node[pos=1, above](){$\scs \lambda+1$};
\end{tikzpicture}}
\quad,
\end{equation}
as well as the apparent analogues of these diagrams with $\alpha^{-1}$ in place of $\alpha$, with cups in place of caps, and both.
\end{itemize}

% - - - - - - - - - - - - - - - - - - - - - - - - - - - - - - - - - - - - - - - -
%
\subsubsection{Inductive dot slide }
%
% - - - - - - - - - - - - - - - - - - - - - - - - - - - - - - - - - - - - - - - -
We now deduce some helpful relations in $\Uc$, cf. Section 5.4 of \cite{Lau1}.

We refer to the equation
\begin{align} \label{eq:inductive-dot}
\hackcenter{\begin{tikzpicture}
    \draw[thick, ->] (0,0) .. controls (0,.75) and (.5,.75) .. (.5,1.5)
        node[pos=.5, shape=coordinate](CROSSING){}
        node[pos=.25, shape=coordinate](DOT){};
    \draw[thick, ->] (.5,0) .. controls (.5,.75) and (0,.75) .. (0,1.5);
    \draw[thick, color=blue, double distance=1pt, dashed] (DOT) [out=180, in=-90] to (-1,1.5);
    \draw[thick, color=blue, dashed] (CROSSING) [out=180, in=-90] to (-.5,1.5);
    \node() at (DOT) {\bbullet};
\node[blue] at (-1.3,1.35) {$\scs m$};
\end{tikzpicture}}
\quad-\;\;
\hackcenter{\begin{tikzpicture}
    \draw[thick, ->] (0,0) .. controls (0,.75) and (.5,.75) .. (.5,1.5)
        node[pos=.5, shape=coordinate](CROSSING){}
        node[pos=.75, shape=coordinate](DOT){};
    \draw[thick, ->] (.5,0) .. controls (.5,.75) and (0,.75) .. (0,1.5);
    \draw[thick, color=blue, double distance=1pt, dashed] (DOT) [out=180, in=-90] to (-1,1.5);
    \draw[thick, color=blue, dashed] (CROSSING) [out=180, in=-90] to (-.5,1.5);
    \node at (DOT) {\bbullet};
    \node[blue] at (-1.3,1.35) {$\scs m$};
\end{tikzpicture}}
&\quad =\quad
\sum_{f+g=m-1}
\hackcenter{\begin{tikzpicture}
    \draw[thick, ->] (0,0) -- (0,1.5)
      node[pos=.3, shape=coordinate](LD){};
    \draw[thick, ->] (.5,0) -- (.5,1.5)
        node[pos=.5, shape=coordinate](RD){};;
    \draw[thick, color=blue, dashed] (-1.25,1.5) .. controls ++(-0,-.4) and ++(0,-.4) .. (-.25,1.5);
    \draw[thick, color=blue, double distance=1pt, dashed]
        (LD) .. controls ++(-.5,.2) and ++(0,-1) .. (-1.5,1.5);
    \draw[thick, color=blue, double distance=1pt, dashed]
        (RD) .. controls ++(-.5,.2) and ++(0,-.7) .. (-.75,1.5);
    \node at (LD) {\bbullet};\node at (RD) {\bbullet};
    \node[blue] at (-1.75,1.35) {$\scs f$};
    \node[blue] at (-.8,.9) {$\scs g$};
\end{tikzpicture}}
 \\ \nn
\hackcenter{\begin{tikzpicture}
    \draw[thick, ->] (0,0) .. controls (0,.75) and (.5,.75) .. (.5,1.5);
    \draw[thick, ->] (.5,0) .. controls (.5,.75) and (0,.75) .. (0,1.5)
        node[pos=.5, shape=coordinate](CROSSING){}
        node[pos=.75, shape=coordinate](DOT){};
    \draw[thick, color=blue, double distance=1pt,dashed] (DOT) [out=180, in=-90] to (-1,1.5);
    \draw[thick, color=blue, dashed] (CROSSING) [out=180, in=-90] to (-.5,1.5);
    \node() at (DOT) {\bbullet};
\node[blue] at (-1.3,1.35) {$\scs m$};
\end{tikzpicture}}
\quad-\quad
\hackcenter{\begin{tikzpicture}
    \draw[thick, ->] (0,0) .. controls (0,.75) and (.5,.75) .. (.5,1.5);
    \draw[thick, ->] (.5,0) .. controls (.5,.75) and (0,.75) .. (0,1.5)
        node[pos=.5, shape=coordinate](CROSSING){}
        node[pos=.25, shape=coordinate](DOT){};
    \draw[thick, color=blue, double distance=1pt,dashed] (DOT) [out=180, in=-90] to (-1,1.5);
    \draw[thick, color=blue, dashed] (CROSSING) [out=180, in=-90] to (-.5,1.5);
    \node() at (DOT) {\bbullet};
\node[blue] at (-1.3,1.35) {$\scs m$};
\end{tikzpicture}}
&\quad=\quad
\sum_{f+g=m-1}
\hackcenter{\begin{tikzpicture}
    \draw[thick, ->] (0,0) -- (0,1.5)
      node[pos=.6, shape=coordinate](LD){};
    \draw[thick, ->] (.5,0) -- (.5,1.5)
        node[pos=.25, shape=coordinate](RD){};;
    \draw[thick, color=blue, dashed] (-1.25,1.5) .. controls ++(-0,-.4) and ++(0,-.4) .. (-.25,1.5);
    \draw[thick, color=blue, double distance=1pt, dashed]
        (LD) .. controls ++(-.5,.2) and ++(0,-.7) .. (-.75,1.5);
    \draw[thick, color=blue, double distance=1pt, dashed]
        (RD) .. controls ++(-.5,.2) and ++(0,-1) .. (-1.5,1.5);
    \node at (LD) {\bbullet};\node at (RD) {\bbullet};
    \node[blue] at (-1.75,1.35) {$\scs f$};
    \node[blue] at (-.8,.9) {$\scs g$};
\end{tikzpicture}}
\end{align}
as the inductive dot slide formula.   These equations follow by induction from \eqref{eq:onil-dot}.

% - - - - - - - - - - - - - - - - - - - - - - - - - - - - - - - - - - - - - - - -
%
%
\subsubsection{Deriving the other curl relation}
%
% - - - - - - - - - - - - - - - - - - - - - - - - - - - - - - - - - - - - - - - -

Notice that relations \eqref{eq:lgz-curl} and \eqref{eq:llz-curl} above only specify the value for one orientation of the curl depending on wether $\l$ is positive or negative.  Here we show that the other curl relation can be derived from the relations above.  Note that the relations below utilize fake bubbles.

\begin{prop}[Other curls] \label{prop:othercurl}
The following curl relations
\begin{align}
  \hackcenter{\begin{tikzpicture} [scale=0.8]
  \draw[thick] (0.5,1) -- (0.5,2);
  \draw[thick] (0.5,-.5) -- (0.5,0);
  \draw[thick] (-1.5,0) -- (-1.5,1);
  \draw[thick] (0.5,0) .. controls ++(-0,0.5) and ++(0,-0.5) .. (-0.5,1)
      node[pos=0.5, shape=coordinate](X){};
  \draw[thick, ->] (-0.5,0) .. controls ++(0,0.5) and ++(0,-0.5) .. (0.5,1);
  \draw[thick, ->] (-0.5,1) .. controls ++(0,0.6) and ++(0,0.6) .. (-1.5,1);
  \draw[thick, ->] (-1.5,0) .. controls ++(0,-0.6) and ++(0,-0.6) .. (-0.5,0)
     node[pos=0.5, shape=coordinate](CUP){};
  \draw[color=blue,  thick, dashed]
   (X) .. controls ++(-1.2,0) and ++(0,-.9) ..(0,2);
    \draw[color=blue,  thick, double distance=1pt,dashed]
   (CUP) .. controls ++(-0,1.5) and ++(0,-1) ..(-.5,2);
   \node at (0,-0.25) {$\l$};
   \node[blue] at (-1,1.8) {$\scs \l-1$};
\end{tikzpicture}}
&\;\; = \;\;
\sum_{f+g=\l} (-1)^{g} \;\;
\hackcenter{\begin{tikzpicture}
  \draw[thick, ->] (2.25,-1) -- (2.25,1.25)
        node[pos=0.5, shape=coordinate](D){};
  \draw[color=blue, thick,double distance=1pt,  dashed] (D) to[out=150,in=-90] (1.75,1.25);
  \draw[thick, ->] (1.1,0) .. controls ++(0,0.6) and ++(0,0.6) .. (.3,0)
      node[pos=0.05, shape=coordinate](C){};
  \draw[thick] (1.1,0) .. controls ++(0,-0.6) and ++(-0,-0.6) .. (.3,0)
      node[pos=0.5, shape=coordinate](A){}
      node[pos=0.2, shape=coordinate](B){};
  \draw[color=blue, thick, double distance=1pt, dashed]
    (A) .. controls++(-.1,.5) and ++(-.2,.3) .. (B)
         node[pos=0.9,right]{$\scs -\l-1$\;};
   \draw[color=blue, thick,double distance=1pt,  dashed]
    (C) .. controls ++(-.7,.2) and ++(0,-.8) .. (1.25,1.25)
    node[pos=0.9,left]{$\scs f$\;};
     \node at (B) {\bbullet};
     \node at (D) {\bbullet};
     \node at (C) {\bbullet};
     \node[blue] at (1.95,1) {$\scs g$};
  \node at (-0,-.5) {$\lambda$};
\end{tikzpicture} }
& \text{for $\l\geq 0$,}
\\
  \hackcenter{\begin{tikzpicture} [scale=0.8]
  \draw[thick] (-0.5,1) -- (-0.5,2);
  \draw[thick] (-0.5,-.5) -- (-0.5,0);
  \draw[thick] (1.5,0) -- (1.5,1);
  \draw[thick] (-0.5,0) .. controls ++(-0,0.5) and ++(0,-0.5) .. (0.5,1)
      node[pos=0.5, shape=coordinate](X){};
  \draw[thick, ->] (0.5,0) .. controls ++(0,0.5) and ++(0,-0.5) .. (-0.5,1);
  \draw[thick, ->] (0.5,1) .. controls ++(0,0.5) and ++(0,0.5) .. (1.5,1)
     node[pos=0.5, shape=coordinate](CUP){};
  \draw[thick, ->] (1.5,0) .. controls ++(0,-0.5) and ++(0,-0.5) .. (0.5,0);
  \draw[color=blue,  thick, dashed]
   (X) .. controls ++(-1,0) and ++(0,-.9) ..(-1,2);
    \draw[color=blue,  thick, double distance=1pt,dashed]
   (CUP) -- (1,2);
   \node at (0,-0.25) {$\l$};
   \node[blue] at (.5,1.8) {$\scs \l-1$};
\end{tikzpicture}}
&\;\; = \;\;
\sum_{f+g=\l} (-1)^{f} \;\;
\hackcenter{\begin{tikzpicture}
  \draw[thick, ->] (-1.75,-1) -- (-1.75,1.25)
        node[pos=0.5, shape=coordinate](D){};
  \draw[color=blue, thick,double distance=1pt,  dashed] (D) to[out=150,in=-90] (-2.25,1.25);
  \draw[thick, ->] (-0.4,0) .. controls ++(-0,0.6) and ++(0,0.6) .. (0.4,0)
      node[pos=0.5, shape=coordinate](X){}
      node[pos=0.1, shape=coordinate](Y){};
  \draw[thick] (-0.4,0) .. controls ++(0,-0.6) and ++(0,-0.6) .. (0.4,0)
      node[pos=0.1, shape=coordinate](Z){};
  \draw[color=blue, thick, double distance=1pt, dashed]
    (X) .. controls++(0,.65) and ++(-.65,.3) .. (Y) node[pos=0.15,right]{$\scs \l-1$\;};
  \draw[color=blue, thick, double distance=1pt, dashed] (Z) to[bend left] (-1,1.25);
  \node at (Y) {\bbullet};
  \node at (Z) {\bbullet};
  \node at (D) {\bbullet};
     \node[blue] at (-2.5,1) {$\scs f$};
     \node[blue] at (-1.25,1) {$\scs g$};
  \node at (-1.25,-.5) {$\lambda$};
\end{tikzpicture} }
& \text{for $\l\leq 0$,}
\end{align}
hold in $\Uc$.
\end{prop}

\begin{proof}
For $\l=0$ the above equations are just the definition of the fake bubbles in weight $\l=0$.
For $\l > 0$ we can add a cup with $\l$ dots to the bottom of relation \eqref{eqn-mixed-R2-relation-1} to get
 \begin{equation} \label{eq:lcurl1}
\hackcenter{\begin{tikzpicture} [scale=0.8]
  \draw[thick, ->] (-0.5,0) to (-0.5,2);
  \draw[thick, <-] (0.5,0) to (0.5,2);
  \node at (1,1.5) {$\l$};
  \draw[thick] (-0.5,0) .. controls ++(0,-0.6) and ++(0,-0.6) .. (0.5,0)
     node[pos=0, shape=coordinate](D){};
  \draw[color=blue, thick, double distance=1pt, dashed]
    (D) to[out=160, in=-90] (-1,2);
  \node at (D) {\bbullet};
    \node[blue] at (-1.3,1.8){$\scs \l$};
\end{tikzpicture}}
\quad = \quad -\;\;
\hackcenter{\begin{tikzpicture} [scale=0.8]
  \draw[thick, <-] (0.5,0) .. controls (0.5,0.4) and (-0.5,0.6) .. (-0.5,1)
      node[pos=0.5, shape=coordinate](X){};
    \draw[thick, ->] (-0.5,0) .. controls (-0.5,0.4) and (0.5,0.6) .. (0.5,1);
  \draw[thick, ->] (0.5,1) .. controls (0.5,1.4) and (-0.5,1.6) .. (-0.5,2)
      node[pos=0.5, shape=coordinate](Y){};
    \draw[thick, <-] (-0.5,1) .. controls (-0.5,1.4) and (0.5,1.6) .. (0.5,2);
  \draw[color=blue,  thick, dashed] (Y) -- (X);
  \node at (1,1.5) {$\l$};
    \draw[thick] (-0.5,0) .. controls ++(0,-0.6) and ++(0,-0.6) .. (0.5,0)
     node[pos=0, shape=coordinate](D){};
  \draw[color=blue, thick, double distance=1pt, dashed]
    (D) to[out=160, in=-90] (-1,2);
  \node at (D) {\bbullet};
    \node[blue] at (-1.3,1.8){$\scs \l$};
\end{tikzpicture} }
\quad + \quad
\sum_{
\xy (0,2)*{\scs f_1+f_2+f_3}; (0,-1)*{\scs = \l-1}; \endxy} (-1)^{f_3}
 \hackcenter{
\begin{tikzpicture} [scale=0.8]
  \draw[thick, ->] (-0.5,0) .. controls (-0.5,0.8) and (0.5,0.8) .. (0.5,0)
      node[pos=0.1, shape=coordinate](DOT){}
      node[pos=0.42, shape=coordinate](L){}
      node[pos=0.5, shape=coordinate](M){}
      node[pos=0.58, shape=coordinate](R){};
  \draw[thick, ->]
  (1.9,1) .. controls ++(0,0.6) and ++(0,0.6) .. (1.1,1)
      node[pos=0.05, shape=coordinate](Z){};
  \draw[thick] (1.9,1) .. controls ++(0,-0.6) and ++(-0,-0.6) .. (1.1,1)
      node[pos=0.5, shape=coordinate](X){}
      node[pos=0.2, shape=coordinate](Y){};
  \draw[color=blue, thick, double distance=1pt, dashed]
    (X) .. controls++(-.1,.5) and ++(-.2,.3) .. (Y)
         node[pos=0.9,right]{$\scs -\l-1$\;};
   \draw[color=blue, thick, double distance=1pt, dashed]
    (Z) .. controls ++(-.5,.4) and ++(.2,.8) .. (R) ;
   \node[blue] at (1.25,0.8){$\scs $\;};
     \node at (Y) {\bbullet};
     \node at (Z) {\bbullet};
 \draw[thick, <-] (-0.5,2.25) .. controls ++(0,-.8) and ++(0,-.8) .. (0.5,2.25)
      node[pos=0.2, shape=coordinate](tDOT){};
 \draw[color=blue, thick, double distance=1pt, dashed]
   (M) .. controls ++(.4,1.4) and ++(-.5,-1) .. (-1.1,1.8) to[out=90, in=140] (tDOT);
 \draw[color=blue, thick, double distance=1pt, dashed]
    (DOT) .. controls++(-.65,0) and ++(-.25,.3) .. (L);
 \node at (tDOT){\bbullet}; \node at (DOT){\bbullet};
   \node[blue] at (.6,1.45){$\scs f_3$};
   \node[blue] at (-1.05,1.1){$\scs f_1$};
   \node[blue] at (-.85,.65){$\scs f_2$};
   \node at (1,2) {$\l$};
   \draw[thick] (-0.5,0) .. controls ++(0,-0.6) and ++(0,-0.6) .. (0.5,0)
     node[pos=0.05, shape=coordinate](D){};
  \draw[color=blue, thick, double distance=1pt, dashed]
    (D) to[out=160, in=-90] (-1.75,2.25);
  \node at (D) {\bbullet};
   \node[blue] at (-2,1.8){$\scs \l$};
\end{tikzpicture} }
 \end{equation}
Then sliding one of the dots through the second term using \eqref{eq:rightdotslide} we get a sum of two terms \[
\hackcenter{\begin{tikzpicture} [scale=0.8]
  \draw[thick, <-] (0.5,0) .. controls (0.5,0.4) and (-0.5,0.6) .. (-0.5,1)
      node[pos=0.5, shape=coordinate](X){};
    \draw[thick, ->] (-0.5,0) .. controls (-0.5,0.4) and (0.5,0.6) .. (0.5,1);
  \draw[thick, ->] (0.5,1) .. controls (0.5,1.4) and (-0.5,1.6) .. (-0.5,2)
      node[pos=0.5, shape=coordinate](Y){};
    \draw[thick, <-] (-0.5,1) .. controls (-0.5,1.4) and (0.5,1.6) .. (0.5,2);
  \draw[color=blue,  thick, dashed] (Y) -- (X);
  \node at (1,1.5) {$\l$};
    \draw[thick] (-0.5,0) .. controls ++(0,-0.6) and ++(0,-0.6) .. (0.5,0)
     node[pos=0, shape=coordinate](D){};
  \draw[color=blue, thick, double distance=1pt, dashed]
    (D) to[out=160, in=-90] (-1,2);
  \node at (D) {\bbullet};
    \node[blue] at (-1.3,1.8){$\scs \l$};
\end{tikzpicture} }
\; = \;\;
\hackcenter{\begin{tikzpicture} [scale=0.8]
  \draw[thick, <-] (0.5,0) .. controls (0.5,0.4) and (-0.5,0.6) .. (-0.5,1)
      node[pos=0.5, shape=coordinate](X){};
    \draw[thick] (-0.5,0) .. controls (-0.5,0.4) and (0.5,0.6) .. (0.5,1)
        node[pos=1, shape=coordinate](tD){};
  \draw[thick, ->] (0.5,1) .. controls (0.5,1.4) and (-0.5,1.6) .. (-0.5,2)
      node[pos=0.5, shape=coordinate](Y){};
    \draw[thick, <-] (-0.5,1) .. controls (-0.5,1.4) and (0.5,1.6) .. (0.5,2);
  \draw[color=blue,  thick, dashed] (Y) -- (X);
  \node at (1,1.5) {$\l$};
    \draw[thick] (-0.5,0) .. controls ++(0,-0.6) and ++(0,-0.6) .. (0.5,0)
     node[pos=0, shape=coordinate](D){};
  \draw[color=blue, thick, double distance=1pt, dashed]
    (D) to[out=160, in=-90] (-1.25,2);
  \draw[color=blue, thick,  dashed]
    (tD) .. controls ++(-1.4,.3) and ++(0,-.6) .. (-1,2);
  \node at (D) {\bbullet}; \node at (tD) {\bbullet};
  \node[blue] at (-1.8,1.7){$\scs \l-1$};
\end{tikzpicture} }
\; + \;\;
\hackcenter{\begin{tikzpicture} [scale=0.8]
  \draw[thick, ->] (0.5,1) .. controls (0.5,1.4) and (-0.5,1.6) .. (-0.5,2)
      node[pos=0.5, shape=coordinate](Y){};
  \draw[thick, <-] (-0.5,1) .. controls (-0.5,1.4) and (0.5,1.6) .. (0.5,2);
    \draw[thick] (-0.5,1) .. controls ++(0,-0.5) and ++(0,-0.5) .. (0.5,1)
     node[pos=0.5, shape=coordinate](CUP){};
  \draw[color=blue, thick,  dashed]
    (Y) .. controls ++(.1,-.6) and ++(0,-.6) .. (-1,2);
  \draw[thick, ->] (-0.5,-.25) .. controls ++(-0,0.6) and ++(0,0.6) .. (0.5,-.25)
      node[pos=0.5, shape=coordinate](CAP){};
  \draw[thick] (-0.5,-.25) .. controls ++(0,-0.6) and ++(0,-0.6) .. (0.5,-.25)
      node[pos=0.05, shape=coordinate](D){};
  \draw[color=blue, thick, double distance=1pt, dashed]
    (CUP) .. controls++(0,.5) and ++(-0,.5) .. (.5,.7)
     .. controls ++(0,-.3) and ++(0,.3) .. (CAP) node[pos=0.15,right]{$\scs \l-1$\;};
    \draw[color=blue, thick, double distance=1pt, dashed]
    (D) to[out=160, in=-90] (-1.25,2);
  \node at (1,1.5) {$\l$};
  \node at (D) {\bbullet};
  \node[blue] at (-1.7,1.8){$\scs \l-1$};
\end{tikzpicture} }
\; = \;\;
(-1)^{\l-1}
\hackcenter{\begin{tikzpicture} [scale=0.8]
  \draw[thick, ->] (0.5,1) .. controls (0.5,1.4) and (-0.5,1.6) .. (-0.5,2)
      node[pos=0.5, shape=coordinate](Y){};
  \draw[thick] (-0.5,1) .. controls (-0.5,1.4) and (0.5,1.6) .. (0.5,2);
    \draw[thick,->] (-0.5,1) .. controls ++(0,-0.5) and ++(0,-0.5) .. (0.5,1)
     node[pos=0.5, shape=coordinate](CUP){};
  \draw[color=blue, thick,  dashed]
    (Y) .. controls ++(.1,-.6) and ++(0,-.6) .. (-1,2);
  \draw[color=blue, thick, double distance=1pt, dashed]
    (CUP) .. controls++(0,.5) and ++(-0,-.9) .. (-1.25,2);
  \node at (1,1.5) {$\l$};
  \node[blue] at (-1.7,1.8){$\scs \l-1$};
\end{tikzpicture} }
\]
where the first diagram is zero by the dot sliding relation, the positivity of bubbles axiom~\eqref{eq:positivity}, and the curl relation \eqref{eq:lgz-curl}. The second can be simplified using that degree zero bubbles are multiplication by 1.   The last term in \eqref{eq:lcurl1} can be rewritten as
\[
\sum_{\xy (0,2)*{\scs f_1+f_2+f_3}; (0,-1)*{\scs = \l-1}; \endxy}
\hackcenter{\begin{tikzpicture}
  \draw[thick, ->] (1.2,1.5) .. controls ++(0,-.8) and ++(0,-.8) .. (.4,1.5)
      node[pos=0.85, shape=coordinate](D){};
   \draw[color=blue, thick,double distance=1pt,  dashed] (D) to[out=160,in=-90] (0,1.5);
  \draw[thick, ->] (-1.1,-.25) .. controls ++(-0,0.6) and ++(0,0.6) .. (-.3,-.25)
      node[pos=0.5, shape=coordinate](X){}
      node[pos=0.1, shape=coordinate](Y){};
  \draw[thick] (-1.1,-.25) .. controls ++(0,-0.6) and ++(0,-0.6) .. (-0.3,-.25)
      node[pos=0.1, shape=coordinate](Z){};
  \draw[color=blue, thick, double distance=1pt, dashed] (X) .. controls++(0,.65) and ++(-.65,.3) .. (Y) node[pos=0.5,above]{$\scs \l-1$\;};
  \draw[color=blue, thick, dashed, double distance=1pt]
     (Z) .. controls ++(-.9,0) and ++(0,-.7) ..(-1.5,1.5)
     node[pos=0.9,left]{$\scs f_2+1$\;};;
  \node at (Y) {\bbullet};
  \node at (Z) {\bbullet};
  \draw[thick, ->] (1.1,0) .. controls ++(0,0.6) and ++(0,0.6) .. (.3,0)
      node[pos=0.05, shape=coordinate](C){};
  \draw[thick] (1.1,0) .. controls ++(0,-0.6) and ++(-0,-0.6) .. (.3,0)
      node[pos=0.5, shape=coordinate](A){}
      node[pos=0.2, shape=coordinate](B){};
  \draw[color=blue, thick, double distance=1pt, dashed]
    (A) .. controls++(-.1,.5) and ++(-.2,.3) .. (B)
         node[pos=0.9,right]{$\scs -\l-1$\;};
   \draw[color=blue, thick,double distance=1pt,  dashed]
    (C) .. controls ++(-.3,.4) and ++(-.1,-1) .. (-.75,1.5)
    node[pos=0.9,left]{$\scs f_3$\;};
   \node[blue] at (-.25,1.4){$\scs f_1$\;};
     \node at (B) {\bbullet};
     \node at (C) {\bbullet};\node at (D) {\bbullet};
  \node at (1.75,.8) {$\lambda$};
\end{tikzpicture} }
\quad = \quad -
\sum_{f=0}^{\l-1} (-1)^{\l-f}
\hackcenter{\begin{tikzpicture}
  \draw[thick, ->] (1.2,1.5) .. controls ++(0,-.8) and ++(0,-.8) .. (.4,1.5)
      node[pos=0.85, shape=coordinate](D){};
   \draw[color=blue, thick,double distance=1pt,  dashed] (D) to[out=160,in=-90] (0,1.5);
  \draw[thick, ->] (1.1,0) .. controls ++(0,0.6) and ++(0,0.6) .. (.3,0)
      node[pos=0.05, shape=coordinate](C){};
  \draw[thick] (1.1,0) .. controls ++(0,-0.6) and ++(-0,-0.6) .. (.3,0)
      node[pos=0.5, shape=coordinate](A){}
      node[pos=0.2, shape=coordinate](B){};
  \draw[color=blue, thick, double distance=1pt, dashed]
    (A) .. controls++(-.1,.5) and ++(-.2,.3) .. (B)
         node[pos=0.9,right]{$\scs -\l-1$\;};
   \draw[color=blue, thick,double distance=1pt,  dashed]
    (C) .. controls ++(-.3,.4) and ++(-.1,-1) .. (-.75,1.5)
    node[pos=0.9,left]{$\scs \l-f$\;};
   \node[blue] at (-.25,1.4){$\scs f$\;};
     \node at (B) {\bbullet};
     \node at (C) {\bbullet};\node at (D) {\bbullet};
  \node at (1.75,.8) {$\lambda$};
\end{tikzpicture} }
\]
using relation \eqref{eq:fake-bubble}.  Combining these bubble terms with the term on the right of \eqref{eq:lcurl1} and using the adjoint structure completes the claim. The case $\l<0$ is proven similarly.
\end{proof}

\begin{cor} (Dotted curls)\label{cor:dotcurl} For $m \geq 0$ the following dotted curl relations
\begin{align}
  \hackcenter{\begin{tikzpicture} [scale=0.8]
  \draw[thick] (0.5,1) -- (0.5,2);
  \draw[thick] (0.5,-.5) -- (0.5,0);
  \draw[thick] (-1.5,0) -- (-1.5,1);
  \draw[thick] (0.5,0) .. controls ++(-0,0.5) and ++(0,-0.5) .. (-0.5,1)
      node[pos=0.5, shape=coordinate](X){};
  \draw[thick, ->] (-0.5,0) .. controls ++(0,0.5) and ++(0,-0.5) .. (0.5,1);
  \draw[thick, ->] (-0.5,1) .. controls ++(0,0.6) and ++(0,0.6) .. (-1.5,1);
  \draw[thick] (-1.5,0) .. controls ++(0,-0.6) and ++(0,-0.6) .. (-0.5,0)
     node[pos=0.5, shape=coordinate](CUP){}
     node[pos=.9, shape=coordinate](DOT){};
  \draw[color=blue,  thick, dashed]
   (X) .. controls ++(-1.2,0) and ++(0,-1.2) ..(0,2);
    \draw[color=blue,  thick, double distance=1pt,dashed]
   (CUP) .. controls ++(-0,.5) and ++(0,-1.6) ..(-2,2);
  \draw[color=blue,  thick, double distance=1pt,dashed]
   (DOT) .. controls ++(-.3,.8) and ++(0,-.8) ..(-1.5,2);
   \node at (0,-0.25) {$\l$};
   \node at (DOT) {\bbullet};
   \node[blue] at (-2.4,1.8) {$\scs \l-1$};
   \node[blue] at (-1.2,1.8) {$\scs m$};
\end{tikzpicture}}
&\;\; = \;\;
\sum_{f+g=m+\l} (-1)^{g} \;\;
\hackcenter{\begin{tikzpicture} [scale=0.9]
  \draw[thick, ->] (2.25,-1) -- (2.25,1.25)
        node[pos=0.5, shape=coordinate](D){};
  \draw[color=blue, thick,double distance=1pt,  dashed] (D) to[out=150,in=-90] (1.75,1.25);
  \draw[thick, ->] (1.1,0) .. controls ++(0,0.6) and ++(0,0.6) .. (.3,0)
      node[pos=0.05, shape=coordinate](C){};
  \draw[thick] (1.1,0) .. controls ++(0,-0.6) and ++(-0,-0.6) .. (.3,0)
      node[pos=0.5, shape=coordinate](A){}
      node[pos=0.2, shape=coordinate](B){};
  \draw[color=blue, thick, double distance=1pt, dashed]
    (A) .. controls++(-.1,.5) and ++(-.2,.3) .. (B)
         node[pos=0.9,right]{$\scs -\l-1$\;};
   \draw[color=blue, thick,double distance=1pt,  dashed]
    (C) .. controls ++(-.7,.2) and ++(0,-.8) .. (1.25,1.25)
    node[pos=0.9,left]{$\scs f$\;};
     \node at (B) {\bbullet};
     \node at (D) {\bbullet};
     \node at (C) {\bbullet};
     \node[blue] at (1.95,1) {$\scs g$};
  \node at (-0,-.5) {$\lambda$};
\end{tikzpicture} }
& \text{for $\l\geq 0$,}
\\
  \hackcenter{\begin{tikzpicture} [scale=0.8]
  \draw[thick] (-0.5,1) .. controls ++(0,.3) and ++(0,-.3)..  (0.5,2);
  \draw[thick] (-0.5,-.5) -- (-0.5,0);
  \draw[thick] (1.5,0) -- (1.5,1);
  \draw[thick] (-0.5,0) .. controls ++(-0,0.5) and ++(0,-0.5) .. (0.5,1)
      node[pos=0.5, shape=coordinate](X){}
      node[pos=1, shape=coordinate](MD){};
  \draw[thick, ->] (0.5,0) .. controls ++(0,0.5) and ++(0,-0.5) .. (-0.5,1);
  \draw[thick, ->] (0.5,1) .. controls ++(0,0.5) and ++(0,0.5) .. (1.5,1)
     node[pos=0.5, shape=coordinate](CUP){};
  \draw[thick, ->] (1.5,0) .. controls ++(0,-0.5) and ++(0,-0.5) .. (0.5,0);
  \draw[color=blue,  thick, dashed]
   (X) .. controls ++(-1,0) and ++(0,-.9) ..(-1.5,2);
    \draw[color=blue,  thick, double distance=1pt,dashed]
   (CUP) .. controls ++(0,.3) and ++(0,-.3) .. (-0.5,2);
   \draw[color=blue,  thick, double distance=1pt,dashed]
   (MD) .. controls ++(-.3,.4) and ++(0,-.5) .. (-1,2);
   \node at (MD) {\bbullet};
   \node at (0,-0.25) {$\l$};
   \node[blue] at (1,1.8) {$\scs \l-1$};
   \node[blue] at (-1.2,1.8) {$\scs m$};
\end{tikzpicture}}
&\;\; = \;\;
\sum_{f+g=m-\l} (-1)^{f} \;\;
\hackcenter{\begin{tikzpicture} [scale=0.9]
  \draw[thick, ->] (-1.75,-1) -- (-1.75,1.25)
        node[pos=0.3, shape=coordinate](D){};
  \draw[color=blue, thick,double distance=1pt,  dashed] (D) to[out=150,in=-90] (-3,1.25);
  \draw[thick, ->] (-0.4,0) .. controls ++(-0,0.6) and ++(0,0.6) .. (0.4,0)
      node[pos=0.5, shape=coordinate](X){}
      node[pos=0.1, shape=coordinate](Y){};
  \draw[thick] (-0.4,0) .. controls ++(0,-0.6) and ++(0,-0.6) .. (0.4,0)
      node[pos=0.1, shape=coordinate](Z){};
  \draw[color=blue, thick, double distance=1pt, dashed]
    (X) .. controls++(0,.65) and ++(-.65,.3) .. (Y) node[pos=0.15,right]{$\scs \l-1$\;};
  \draw[color=blue, thick, double distance=1pt, dashed]
   (Z) .. controls ++(-1.5,.2) and ++(0,-.5) .. (-2.25,1.25)
    node[pos=0.9, left] {$\scs g$};
   \node[blue] at (-3.3,1.1) {$\scs f$};
  \node at (Y) {\bbullet};
  \node at (Z) {\bbullet};
  \node at (D) {\bbullet};
  \node at (-1.25,-.5) {$\lambda$};
\end{tikzpicture} }
& \text{for $\l\leq 0$,}
\end{align}
hold in $\Uc$.
\end{cor}

\begin{proof}
This follows immediately from Proposition~\ref{prop:othercurl} and the inductive dot slide formula \eqref{eq:inductive-dot}.
\end{proof}

The utility of fake bubbles is demonstrated by the previous corollary since the summation on the right-hand side involves both real and fake bubbles. Notice the similarity between Corollary \ref{cor:dotcurl} and the following Proposition.

\begin{prop}[Dotted curls] \label{prop-dotted-curl}The following dotted curl relations
\begin{align}
  \hackcenter{\begin{tikzpicture} [scale=0.8]
  \draw[thick] (0.5,1.5) -- (0.5,2);
  \draw[thick] (0.5,-.5) -- (0.5,.5);
  \draw[thick] (-0.5,0) -- (-0.5,.5)
    node[pos=.2, shape=coordinate](MD){}
    node[pos=1, shape=coordinate](TD){};
  \draw[thick] (-1.5,0) -- (-1.5,1.5);
  \draw[thick] (0.5,.5) .. controls ++(-0,0.5) and ++(0,-0.5) .. (-0.5,1.5)
      node[pos=0.5, shape=coordinate](X){};
  \draw[thick, ->] (-0.5,.5) .. controls ++(0,0.5) and ++(0,-0.5) .. (0.5,1.5);
  \draw[thick, ->] (-0.5,1.5) .. controls ++(0,0.6) and ++(0,0.6) .. (-1.5,1.5);
  \draw[thick] (-1.5,0) .. controls ++(0,-0.6) and ++(0,-0.6) .. (-0.5,0)
     node[pos=0.5, shape=coordinate](CUP){}
     node[pos=.8, shape=coordinate](DOT){};
  \draw[color=blue,  thick, dashed]
   (X) .. controls ++(-.8,.3) and ++(-.3,.4) ..(TD);
    \draw[color=blue,  thick, double distance=1pt,dashed]
   (CUP) .. controls ++(-0,.5) and ++(-.3,.2) ..(DOT);
   \draw[color=blue,  thick, double distance=1pt,dashed]
   (MD) .. controls ++(.4,0) and ++(0,-1.4) ..(-2,2);
   \node at (0,-0.25) {$\l$};
   \node at (DOT) {\bbullet};\node at (MD) {\bbullet};\node at (TD) {\bbullet};
   \node[blue] at (-2.5,1.8) {$\scs m+\l$};
\end{tikzpicture}}
&\;\; = \;\;
\sum_{f+g=m+\l} (-1)^{g} \;\;
\hackcenter{\begin{tikzpicture} [scale=0.9]
  \draw[thick, ->] (2.25,-1) -- (2.25,1.25)
        node[pos=0.5, shape=coordinate](D){};
  \draw[color=blue, thick,double distance=1pt,  dashed] (D) to[out=150,in=-90] (1.75,1.25);
  \draw[thick, ->] (1.1,0) .. controls ++(0,0.6) and ++(0,0.6) .. (.3,0)
      node[pos=0.05, shape=coordinate](C){};
  \draw[thick] (1.1,0) .. controls ++(0,-0.6) and ++(-0,-0.6) .. (.3,0)
      node[pos=0.5, shape=coordinate](A){}
      node[pos=0.2, shape=coordinate](B){};
  \draw[color=blue, thick, double distance=1pt, dashed]
    (A) .. controls++(-.1,.5) and ++(-.2,.3) .. (B)
         node[pos=0.9,right]{$\scs -\l-1$\;};
   \draw[color=blue, thick,double distance=1pt,  dashed]
    (C) .. controls ++(-.7,.2) and ++(0,-.8) .. (1.25,1.25)
    node[pos=0.9,left]{$\scs f$\;};
     \node at (B) {\bbullet};
     \node at (D) {\bbullet};
     \node at (C) {\bbullet};
     \node[blue] at (1.95,1) {$\scs g$};
  \node at (-0,-.5) {$\lambda$};
\end{tikzpicture} }
& \text{for $\l<0$,}
\\
  \hackcenter{\begin{tikzpicture} [scale=0.8]
  \draw[thick, ->] (-0.5,1) -- (-0.5,2);
  \draw[thick] (-0.5,-.5) -- (-0.5,0);
  \draw[thick] (1.5,0) -- (1.5,1);
  \draw[thick] (-0.5,0) .. controls ++(-0,0.5) and ++(0,-0.5) .. (0.5,1)
      node[pos=0.5, shape=coordinate](X){}
      node[pos=1, shape=coordinate](MD){}
      node[pos=0.7, shape=coordinate](BD){};
  \draw[thick] (0.5,0) .. controls ++(0,0.5) and ++(0,-0.5) .. (-0.5,1);
  \draw[thick, ->] (0.5,1) .. controls ++(-.2,0.8) and ++(.2,0.8) .. (1.5,1)
     node[pos=0.2, shape=coordinate](LC){}
     node[pos=0.58, shape=coordinate](C){};
  \draw[thick, ->] (1.5,0) .. controls ++(0,-0.5) and ++(0,-0.5) .. (0.5,0);
  \draw[color=blue,  thick, dashed]
   (X) .. controls ++(-1.5,.4) and ++(-.2,.5) ..(BD);
    \draw[color=blue,  thick, double distance=1pt,dashed]
   (LC) .. controls ++(-.2,.7) and ++(.2,.7) .. (C);
    \draw[color=blue,  thick, double distance=1pt,dashed]
   (MD) .. controls ++(-.3,.4) and ++(0,-.5) .. (-1,2)
    node[pos=0.85, left]{$\scs m-\l$};
   \node at (0,-0.25) {$\l$};
   \node at (LC) {\bbullet};\node at (MD) {\bbullet};\node at (BD) {\bbullet};
   \node[blue] at (1.54,1.8) {$\scs \l-1$};
\end{tikzpicture}}
&\;\; = \;\;
\sum_{f+g=m-\l} (-1)^{f} \;\;
\hackcenter{\begin{tikzpicture} [scale=0.9]
  \draw[thick, ->] (-1.75,-1) -- (-1.75,1.25)
        node[pos=0.3, shape=coordinate](D){};
  \draw[color=blue, thick,double distance=1pt,  dashed] (D) to[out=150,in=-90] (-3,1.25);
  \draw[thick, ->] (-0.4,0) .. controls ++(-0,0.6) and ++(0,0.6) .. (0.4,0)
      node[pos=0.5, shape=coordinate](X){}
      node[pos=0.1, shape=coordinate](Y){};
  \draw[thick] (-0.4,0) .. controls ++(0,-0.6) and ++(0,-0.6) .. (0.4,0)
      node[pos=0.1, shape=coordinate](Z){};
  \draw[color=blue, thick, double distance=1pt, dashed]
    (X) .. controls++(0,.65) and ++(-.65,.3) .. (Y) node[pos=0.15,right]{$\scs \l-1$\;};
  \draw[color=blue, thick, double distance=1pt, dashed]
   (Z) .. controls ++(-1.5,.2) and ++(0,-.5) .. (-2.25,1.25)
    node[pos=0.9, left] {$\scs g$};
   \node[blue] at (-3.3,1.1) {$\scs f$};
  \node at (Y) {\bbullet};
  \node at (Z) {\bbullet};
  \node at (D) {\bbullet};
  \node at (-1.25,-.5) {$\lambda$};
\end{tikzpicture} }
& \text{for $\l>0$,}
\end{align}
hold in $\Uc$.
\end{prop}

\begin{proof}
These relations follow immediately from the relations in $\Uc$, those in Proposition~\ref{prop:othercurl} above, and the inductive dot slide formula \eqref{eq:inductive-dot}.
\end{proof}

The dotted curl relations derived above imply that the equations in equation \eqref{eq:fake-bubble} defining the fake bubbles switch form as $m$ grows large.   More precisely, we have the following Proposition.

\begin{prop} The following relations
\begin{align}  \label{eq_biginfgrass1}
\sum_{f+g=m}
(-1)^{g} \xy (0,-2)*{
\begin{tikzpicture}[scale=0.9]
  \draw[thick, ->] (0.5,0) .. controls (0.5,0.8) and (-0.5,0.8) .. (-0.5,0)
      node[pos=0, shape=coordinate](Z){};
  \draw[thick] (0.5,0) .. controls (0.5,-0.8) and (-0.5,-0.8) .. (-0.5,0)
      node[pos=0.5, shape=coordinate](X){}
      node[pos=0.2, shape=coordinate](Y){};
  \draw[color=blue, thick, double distance=1pt, dashed] (X) .. controls++(-.1,.7) and ++(-.2,.4) .. (Y)
         node[pos=0.9,right]{$\scs -\l-1$\;};
   \draw[color=blue, thick, double distance=1pt, dashed]
    (Z) .. controls ++(-1,.7) and ++(.1,-1) .. (1,1.25) ;
   \node[blue] at (1.3,1.1){$\scs f$\;};
 \node at (Y) {\bbullet}; \node at (Z) {\bbullet};
  \node at (-.25,1.1) {$\lambda$};
\end{tikzpicture} }; \endxy
\hackcenter{
\begin{tikzpicture}[scale=0.9]
  \draw[thick, ->] (-0.5,0) .. controls (-0.5,0.8) and (0.5,0.8) .. (0.5,0)
      node[pos=0.5, shape=coordinate](X){}
      node[pos=0.1, shape=coordinate](Y){};
  \draw[thick] (-0.5,0) .. controls (-0.5,-0.8) and (0.5,-0.8) .. (0.5,0)
      node[pos=0.1, shape=coordinate](Z){};
  \draw[color=blue, thick, double distance=1pt, dashed] (X) .. controls++(0,.65) and ++(-.65,.3) .. (Y) node[pos=0.15,right]{$\scs \l-1$\;};
  \draw[color=blue, thick, double distance=1pt, dashed] (Z) to[out=180, in=90] (-1,1.25) ;
  \node[blue] at (-.6,1.2){$\scs g$\;};
  \node at (Y) {\bbullet}; \node at (Z) {\bbullet};
\end{tikzpicture} }
&\;\; = \;\; \delta_{m,0}\onebb_{\onebbl}
&
\text{for $\l>0$ and $\l <m $.}
\\ \label{eq_biginfgrass2}
\sum_{f+g=m} (-1)^{g}
\hackcenter{
\begin{tikzpicture}[scale=0.9]
  \draw[thick, ->] (-0.5,0) .. controls (-0.5,0.8) and (0.5,0.8) .. (0.5,0)
      node[pos=0.5, shape=coordinate](X){}
      node[pos=0.1, shape=coordinate](Y){};
  \draw[thick] (-0.5,0) .. controls (-0.5,-0.8) and (0.5,-0.8) .. (0.5,0)
      node[pos=0.1, shape=coordinate](Z){};
  \draw[color=blue, thick, double distance=1pt, dashed] (X) .. controls++(0,.65) and ++(-.65,.3) .. (Y) node[pos=0.15,right]{$\scs \l-1$\;};
  \draw[color=blue, thick, double distance=1pt, dashed] (Z) to[out=180, in=90] (-1,1.25) ;
  \node[blue] at (-.6,1.2){$\scs f$\;};
  \node at (Y) {\bbullet}; \node at (Z) {\bbullet};
\end{tikzpicture} }
  \xy (0,-2)*{
\begin{tikzpicture}[scale=0.9]
  \draw[thick, ->] (0.5,0) .. controls (0.5,0.8) and (-0.5,0.8) .. (-0.5,0)
      node[pos=0, shape=coordinate](Z){};
  \draw[thick] (0.5,0) .. controls (0.5,-0.8) and (-0.5,-0.8) .. (-0.5,0)
      node[pos=0.5, shape=coordinate](X){}
      node[pos=0.2, shape=coordinate](Y){};
  \draw[color=blue, thick, double distance=1pt, dashed] (X) .. controls++(-.1,.7) and ++(-.2,.4) .. (Y)
         node[pos=0.9,right]{$\scs -\l-1$\;};
   \draw[color=blue, thick, double distance=1pt, dashed]
    (Z) .. controls ++(-1,.7) and ++(.1,-1) .. (1,1.25) ;
   \node[blue] at (1.3,0.9){$\scs g$\;};
 \node at (Y) {\bbullet}; \node at (Z) {\bbullet};
  \node at (-.5,1.1) {$\lambda$};
\end{tikzpicture} }; \endxy\;\;
 &\;\; = \;\; \delta_{m,0}\onebb_{\onebbl}
 &\text{for $\l <0$ and $-\l < m$, }
\end{align}
hold in $\Uc$.
\end{prop}

\begin{proof}
Equation \eqref{eq_biginfgrass1} follows by simplifying the diagram
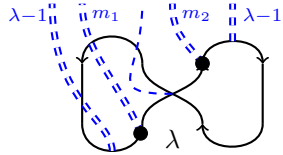
\[
  \hackcenter{\begin{tikzpicture} [scale=0.8]
  \draw[thick] (-1.5,0) -- (-1.5,1);
  \draw[thick] (0.5,0) .. controls ++(-0,0.5) and ++(0,-0.5) .. (-0.5,1)
      node[pos=0.5, shape=coordinate](X){};
  \draw[thick, ->] (-0.5,0) .. controls ++(0,0.5) and ++(0,-0.5) .. (0.5,1)
   node[pos=1, shape=coordinate](MD){};
  \draw[thick, ->] (-0.5,1) .. controls ++(0,0.6) and ++(0,0.6) .. (-1.5,1);
  \draw[thick] (-1.5,0) .. controls ++(0,-0.6) and ++(0,-0.6) .. (-0.5,0)
     node[pos=0.5, shape=coordinate](CUP){}
     node[pos=.9, shape=coordinate](DOT){};
    \draw[thick, ->] (0.5,1) .. controls ++(0,0.5) and ++(0,0.5) .. (1.5,1)
     node[pos=0.5, shape=coordinate](CUP2){};
    \draw[thick, ->] (1.5,0) .. controls ++(0,-0.5) and ++(0,-0.5) .. (0.5,0);
  \draw[color=blue,  thick, dashed]
   (X) .. controls ++(-1.2,0) and ++(0,-1.2) ..(-0.5,2);
    \draw[color=blue,  thick, double distance=1pt,dashed]
   (CUP) .. controls ++(-0,.5) and ++(0,-1.6) ..(-2,2);
   \draw[color=blue,  thick, double distance=1pt,dashed]
   (CUP2) .. controls ++(0,.3) and ++(0,-.3) .. (1,2);
   \draw[color=blue,  thick, double distance=1pt,dashed]
   (MD) .. controls ++(-.3,.4) and ++(0,-.5) .. (0,2);
  \draw[color=blue,  thick, double distance=1pt,dashed]
   (DOT) .. controls ++(-.3,.8) and ++(0,-.8) ..(-1.5,2);
  \draw[thick] (1.5,0) -- (1.5,1);
   \node at (0,-0.25) {$\l$};
   \node at (DOT) {\bbullet};
   \node at (MD) {\bbullet};
   \node[blue] at (-2.4,1.8) {$\scs \l-1$};
   \node[blue] at (1.5,1.8) {$\scs \l-1$};
   \node[blue] at (-1.1,1.8) {$\scs m_1$};
   \node[blue] at (.4,1.8) {$\scs m_2$};
\end{tikzpicture}}
\]
in two possible ways using the dotted curl relations from Corollary~\ref{cor:dotcurl} and Proposition ~\ref{prop-dotted-curl}.  The proof of \eqref{eq_biginfgrass2} is proven similarly.
\end{proof}

%---------------------------------------------------------------------
\subsection{An upper bound on the size of Homs in $\Uc$} \label{subsec:upper}
%---------------------------------------------------------------------

% - - - - - - - - - - - - - - - - - - - - - - - - - - - - - - - - - - - - - - - -
%
\subsubsection{Graded 2-hom-spaces}
%
% - - - - - - - - - - - - - - - - - - - - - - - - - - - - - - - - - - - - - - - -
For each pair of 1-morphisms $x$ and $y$ of $\Uc$ there is a $\Z$-graded vector space
\begin{equation}
\HOM_{\Uc}(x,y) := \bigoplus_{t \in \Z}  \Hom(x, y\la t\ra).
\end{equation}
Using the super-2-category structure of $\Uc$ we can form the $(\Z \times \Zt)$-graded vector space
\[
\Pi\HOM_{\Uc}(x,y) := \bigoplus_{t \in \Z}  \Hom(x, y\la t\ra)
 \oplus \bigoplus_{t \in \Z}\Hom(x, \Pi y\la t\ra)
\]
between 1-morphisms $x$ and $y$. The graded dimension of this graded vector space is defined by
\[
\dim_{q,\pi} \left(\Pi\HOM_{\Uc(x,y)} \right) =
\sum_{t \in \Z}
 q^t\dim \Hom(x, y \la t \ra ) +\pi\sum_{t \in \Z}
 q^t\dim \Hom(x, \Pi y\la t \ra ).
\]

% -------------------------------------------------------------------------------
%
\subsubsection{Spanning sets for endomorphisms of $\onebbl$}
%
% ----------------------------------------------------------------------------

Following the arguments in \cite[Section 8]{Lau1}  the relations from Section~\ref{subsec:additional} give rise to spanning sets for the space of homs between arbitrary 1-morphisms in $\Uc$.

Let
\[\Xi := \Bbbk \la z_1,z_2,\dots, z_a,\dots \ra / (z_a z_b - (-1)^{a \cdot b} z_b z_a \; \text{for all $a,b$})
\]
denote the $(\Z \otimes \Zt)$-graded superalgebra generated by symbols $z_a$ of degree $|z_a|=2a$ and $p(z_a)=a$.  The superalgebra $\Xi$ is supercommutative.  If $2\in\Bbbk^\times$ it is graded-commutative (tensor product of a polynomial algebra in the $z_a$'s for $a$ even with an exterior algebra in the $z_a$'s for $a$ odd).  If $2=0$ in $\Bbbk$ then it is a polynomial ring.

Any monomial $z_{a_1} \dots z_{a_k}$ in $\Xi$ can be identified with a 2-morphism in  $\Pi\HOM_{\Uc}(\onebbl,\onebbl)$ via the assignment sending $z_a$ to the 2-morphism
\[
\begin{array}{ccc}
  \hackcenter{
\begin{tikzpicture}[scale=0.75]
  \draw[thick, ->] (0.5,0) .. controls (0.5,0.8) and (-0.5,0.8) .. (-0.5,0)
      node[pos=0, shape=coordinate](Z){};
  \draw[thick] (0.5,0) .. controls (0.5,-0.8) and (-0.5,-0.8) .. (-0.5,0)
      node[pos=0.5, shape=coordinate](X){}
      node[pos=0.2, shape=coordinate](Y){};
 %% Draw double blue curvy line
  \draw[color=blue, thick, double distance=1pt, dashed] (X) .. controls++(-.1,.7) and ++(-.2,.4) .. (Y)
         node[pos=0.9,right]{$\scs -\l-1$\;};
   \draw[color=blue, thick, double distance=1pt, dashed]
    (Z) .. controls ++(-1,.7) and ++(.1,-1) .. (.75,1.25) ;
   \node[blue] at (1.1,1.1){$\scs a$\;};
 %% Draw the bullet last so it comes out on top
  \draw[line width=0mm] (0.5,0) .. controls (0.5,-0.8) and (-0.5,-0.8) .. (-0.5,0)
     node[pos=0.2]{\bbullet};
  \draw[line width=0mm] (0.5,0) .. controls (0.5,0.8) and (-0.5,0.8) .. (-0.5,0)
      node[pos=0.0]{\bbullet};
  \node at (-1,.3) {$\lambda$};
\end{tikzpicture} }
 \quad\text{for $\l\leq0 $,}&  \qquad \text{and}\qquad &
 \hackcenter{
\begin{tikzpicture}[scale=0.75]
  \draw[thick, ->] (-0.5,0) .. controls (-0.5,0.8) and (0.5,0.8) .. (0.5,0)
      node[pos=0.5, shape=coordinate](X){}
      node[pos=0.1, shape=coordinate](Y){};
  \draw[thick] (-0.5,0) .. controls (-0.5,-0.8) and (0.5,-0.8) .. (0.5,0)
      node[pos=0.1, shape=coordinate](Z){};
  \draw[color=blue, thick, double distance=1pt, dashed] (X) .. controls++(0,.65) and ++(-.65,.3) .. (Y) node[pos=0.15,right]{$\scs \l-1$\;};
  \draw[color=blue, thick, double distance=1pt, dashed] (Z) to[out=180, in=90] (-1,1.25) ;
  \node[blue] at (-.6,1.2){$\scs a$\;};
  \draw[line width=0mm] (-0.5,0) .. controls (-0.5,0.8) and (0.5,0.8) .. (0.5,0)
    node[pos=0.1]{\bbullet};
  \draw[line width=0mm] (-0.5,0) (-0.5,0) .. controls (-0.5,-0.8) and (0.5,-0.8) .. (0.5,0)
        node[pos=0.1]{\bbullet};
  \node at (-1.2,-.2) {$\l$};
\end{tikzpicture}} \quad\text{for $\l\geq 0$}.
\end{array}
\]
This assignment gives rise to a degree-preserving superalgebra homomorphism mapping multiplication in $\Xi$ to horizontal juxtaposition of diagrams coming from horizontal composition in $\Pi\HOM_{\Uc}(\onebbl,\onebbl)$.

Arguing as in \cite[Proposition 8.2]{Lau1} gives the following:
\begin{prop}
The map
\begin{equation} \label{eq_ltol}
 \Xi \to \Pi\HOM_{\Uc}(\onebbl,\onebbl).
\end{equation}
described above is a surjective homomorphism of graded superalgebras.
\end{prop}

In particular, any diagram representing a 2-morphism $\HOM_{\Uc}(\onebbl, \Pi^{s} \onebbl)$ can be reduced to a linear combination of products of non-nested dotted bubbles of the same orientation whose dashed lines never intersect.
\[
\begin{array}{ccc}
  \l\leq0 & \qquad \qquad &\l\geq 0 \\
  \hackcenter{
\begin{tikzpicture}[scale=0.75]
  \draw[thick, ->] (0.5,0) .. controls (0.5,0.8) and (-0.5,0.8) .. (-0.5,0)
      node[pos=0, shape=coordinate](Z){};
  \draw[thick] (0.5,0) .. controls (0.5,-0.8) and (-0.5,-0.8) .. (-0.5,0)
      node[pos=0.5, shape=coordinate](X){}
      node[pos=0.2, shape=coordinate](Y){};
 %% Draw double blue curvy line
  \draw[color=blue, thick, double distance=1pt, dashed] (X) .. controls++(-.1,.7) and ++(-.2,.4) .. (Y)
         node[pos=0.9,right]{$\scs -\l-1$\;};
   \draw[color=blue, thick, double distance=1pt, dashed]
    (Z) .. controls ++(-1,.7) and ++(.1,-1) .. (.75,1.25) ;
   \node[blue] at (1.1,1.1){$\scs a_1$\;};
 %% Draw the bullet last so it comes out on top
  \draw[line width=0mm] (0.5,0) .. controls (0.5,-0.8) and (-0.5,-0.8) .. (-0.5,0)
     node[pos=0.2]{\bbullet};
  \draw[line width=0mm] (0.5,0) .. controls (0.5,0.8) and (-0.5,0.8) .. (-0.5,0)
      node[pos=0.0]{\bbullet};
  \node at (-1,.3) {$\lambda$};
\end{tikzpicture} }
  \hackcenter{
\begin{tikzpicture}[scale=0.75]
  \draw[thick, ->] (0.5,0) .. controls (0.5,0.8) and (-0.5,0.8) .. (-0.5,0)
      node[pos=0, shape=coordinate](Z){};
  \draw[thick] (0.5,0) .. controls (0.5,-0.8) and (-0.5,-0.8) .. (-0.5,0)
      node[pos=0.5, shape=coordinate](X){}
      node[pos=0.2, shape=coordinate](Y){};
 %% Draw double blue curvy line
  \draw[color=blue, thick, double distance=1pt, dashed] (X) .. controls++(-.1,.7) and ++(-.2,.4) .. (Y)
         node[pos=0.9,right]{$\scs -\l-1$\;};
   \draw[color=blue, thick, double distance=1pt, dashed]
    (Z) .. controls ++(-1,.7) and ++(.1,-1) .. (.75,1.25) ;
   \node[blue] at (1.1,1.1){$\scs a_2$\;};
 %% Draw the bullet last so it comes out on top
  \draw[line width=0mm] (0.5,0) .. controls (0.5,-0.8) and (-0.5,-0.8) .. (-0.5,0)
     node[pos=0.2]{\bbullet};
  \draw[line width=0mm] (0.5,0) .. controls (0.5,0.8) and (-0.5,0.8) .. (-0.5,0)
      node[pos=0.0]{\bbullet};
\end{tikzpicture} } \dots \quad
  \hackcenter{
\begin{tikzpicture}[scale=0.75]
  \draw[thick, ->] (0.5,0) .. controls (0.5,0.8) and (-0.5,0.8) .. (-0.5,0)
      node[pos=0, shape=coordinate](Z){};
  \draw[thick] (0.5,0) .. controls (0.5,-0.8) and (-0.5,-0.8) .. (-0.5,0)
      node[pos=0.5, shape=coordinate](X){}
      node[pos=0.2, shape=coordinate](Y){};
 %% Draw double blue curvy line
  \draw[color=blue, thick, double distance=1pt, dashed] (X) .. controls++(-.1,.7) and ++(-.2,.4) .. (Y)
         node[pos=0.9,right]{$\scs -\l-1$\;};
   \draw[color=blue, thick, double distance=1pt, dashed]
    (Z) .. controls ++(-1,.7) and ++(.1,-1) .. (.75,1.25) ;
   \node[blue] at (1.1,1.1){$\scs a_k$\;};
 %% Draw the bullet last so it comes out on top
  \draw[line width=0mm] (0.5,0) .. controls (0.5,-0.8) and (-0.5,-0.8) .. (-0.5,0)
     node[pos=0.2]{\bbullet};
  \draw[line width=0mm] (0.5,0) .. controls (0.5,0.8) and (-0.5,0.8) .. (-0.5,0)
      node[pos=0.0]{\bbullet};
\end{tikzpicture} }
 &  \qquad  \qquad &
 \hackcenter{
\begin{tikzpicture}[scale=0.75]
  \draw[thick, ->] (-0.5,0) .. controls (-0.5,0.8) and (0.5,0.8) .. (0.5,0)
      node[pos=0.5, shape=coordinate](X){}
      node[pos=0.1, shape=coordinate](Y){};
  \draw[thick] (-0.5,0) .. controls (-0.5,-0.8) and (0.5,-0.8) .. (0.5,0)
      node[pos=0.1, shape=coordinate](Z){};
  \draw[color=blue, thick, double distance=1pt, dashed] (X) .. controls++(0,.65) and ++(-.65,.3) .. (Y) node[pos=0.15,right]{$\scs \l-1$\;};
  \draw[color=blue, thick, double distance=1pt, dashed] (Z) to[out=180, in=90] (-1,1.25) ;
  \node[blue] at (-.6,1.2){$\scs a_1$\;};
  \draw[line width=0mm] (-0.5,0) .. controls (-0.5,0.8) and (0.5,0.8) .. (0.5,0)
    node[pos=0.1]{\bbullet};
  \draw[line width=0mm] (-0.5,0) (-0.5,0) .. controls (-0.5,-0.8) and (0.5,-0.8) .. (0.5,0)
        node[pos=0.1]{\bbullet};
  \node at (-1.2,-.2) {$\l$};
\end{tikzpicture}}
 \hackcenter{
\begin{tikzpicture}[scale=0.75]
  \draw[thick, ->] (-0.5,0) .. controls (-0.5,0.8) and (0.5,0.8) .. (0.5,0)
      node[pos=0.5, shape=coordinate](X){}
      node[pos=0.1, shape=coordinate](Y){};
  \draw[thick] (-0.5,0) .. controls (-0.5,-0.8) and (0.5,-0.8) .. (0.5,0)
      node[pos=0.1, shape=coordinate](Z){};
  \draw[color=blue, thick, double distance=1pt, dashed] (X) .. controls++(0,.65) and ++(-.65,.3) .. (Y) node[pos=0.15,right]{$\scs \l-1$\;};
  \draw[color=blue, thick, double distance=1pt, dashed] (Z) to[out=180, in=90] (-1,1.25) ;
  \node[blue] at (-.6,1.2){$\scs a_2$\;};
  \draw[line width=0mm] (-0.5,0) .. controls (-0.5,0.8) and (0.5,0.8) .. (0.5,0)
    node[pos=0.1]{\bbullet};
  \draw[line width=0mm] (-0.5,0) (-0.5,0) .. controls (-0.5,-0.8) and (0.5,-0.8) .. (0.5,0)
        node[pos=0.1]{\bbullet};
\end{tikzpicture}}
\dots \quad
 \hackcenter{
\begin{tikzpicture}[scale=0.75]
  \draw[thick, ->] (-0.5,0) .. controls (-0.5,0.8) and (0.5,0.8) .. (0.5,0)
      node[pos=0.5, shape=coordinate](X){}
      node[pos=0.1, shape=coordinate](Y){};
  \draw[thick] (-0.5,0) .. controls (-0.5,-0.8) and (0.5,-0.8) .. (0.5,0)
      node[pos=0.1, shape=coordinate](Z){};
  \draw[color=blue, thick, double distance=1pt, dashed] (X) .. controls++(0,.65) and ++(-.65,.3) .. (Y) node[pos=0.15,right]{$\scs \l-1$\;};
  \draw[color=blue, thick, double distance=1pt, dashed] (Z) to[out=180, in=90] (-1,1.25) ;
  \node[blue] at (-.6,1.2){$\scs a_k$\;};
  \draw[line width=0mm] (-0.5,0) .. controls (-0.5,0.8) and (0.5,0.8) .. (0.5,0)
    node[pos=0.1]{\bbullet};
  \draw[line width=0mm] (-0.5,0) (-0.5,0) .. controls (-0.5,-0.8) and (0.5,-0.8) .. (0.5,0)
        node[pos=0.1]{\bbullet};
\end{tikzpicture}}
\end{array}
\]
The image of the monomial basis of $\Xi$ under the surjective homomorphism \eqref{eq_ltol} is a homogeneous spanning set of the graded $\Bbbk$-module $\Pi\HOM_{\Uc}(\onebbl, \onebbl)$. We call such a monomial in the image of $\Xi$ a {\em bubble monomial}.  Define the power series
\begin{equation} \label{eq-def-xi}\begin{split}
  \xi_\Z &:= \prod_{a=1}^{\infty} (1-(\pi q^2)^{2a})^{-1},\\
  \xi_0 &:= \prod_{a=1}^{\infty} (1-(\pi q^2)^{2a})^{-1}\cdot\prod_{a=1}^{\infty} (1+(\pi q^2)^{2a-1}),\\
  \xi_2 &:= \prod_{a=1}^{\infty} (1-(\pi q^2)^a)^{-1}.
\end{split}\end{equation}

\begin{cor} \label{cor:homonel} If $\Bbbk=\Z$ (respectively $\Bbbk$ is a field of characteristic $2$, respectively $\Bbbk$ is a field of characteristic other than $2$), then
\begin{equation*}
\dim_{q,\pi} \Pi\HOM_{\Uc}(\onebbl, \onebbl) \leq \xi
\end{equation*}
termwise and $\Pi\HOM_{\Uc}(\onebbl, \onebbl)$ is graded local.  Here $\xi=\xi_\Z$ (respectively $\xi=\xi_2$, respectively $\xi=\xi_0$).  In the case $\Bbbk=\Z$, by $\dim$ we mean free rank (and say nothing about torsion).
\end{cor}

% -------------------------------------------------------------------------------
%
\subsubsection{Other spanning sets} \label{subsubsec-otherspanning}
%
% ----------------------------------------------------------------------------

Let $\epsilon = (\epsilon_1,\dots, \epsilon_k)$ denote a {\em covering sequence} consisting of symbols $\epsilon_i$ in the alphabet $\{ -, +, \circ \}$.  Let ${\rm CSeq}$ denote the set of all such covering sequences.  We write ${\rm Seq}$ for those sequences $\epsilon$ with each $\epsilon_i$ in the sub-alphabet $\{ -,+\}$ and refer to such sequences as {\em signed sequences}.   Covering sequences index 1-morphisms in $\Uc$ by setting $\Ec_{+} := \Ec$,   $\Ec_{-} := \Fc$ and $\Ec_{\circ} := \Pi$, so that
\[
 \Ec_{\epsilon}\onebbl := \Ec_{\epsilon_1} \dots \Ec_{\epsilon_k}\onebbl.
\]
The empty sequence $\epsilon=\emptyset$ corresponds to the element $\Ec_{\emptyset}\onebbl = \onebbl$.

Every covering sequence $\ep$ determines a signed sequence denoted $\underline{\ep}$ since the super-2-category structure implies every 1-morphism of the form $\Ec_{\epsilon}\onebbl$ is canonically isomorphic to a  1-morphism $\Pi^s\Ec_{\underline{\epsilon}}\onebbl$ where $s$ is the number of $\circ$ symbols in $\epsilon$ and $\underline{\epsilon}$ is the signed sequence obtained from $\epsilon$ by removing all $\circ$ symbols.

A diagram representing a 2-morphism in $\HOM_{\Uc}(\Ec_{\epsilon}\onebbl, \Ec_{\epsilon'}\onebbl)$ has lower boundary labelled by the covering sequence $\epsilon$ and upper boundary labelled by $\epsilon'$ (where lines labelled $+$ are oriented up, lines labelled $-$ are oriented down, and lines labelled $\circ$ are dashed).
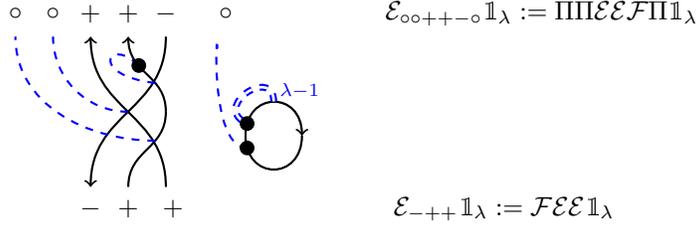
\[
\xy
(0,0)*{
\begin{tikzpicture}
    \draw[thick, <-] (0,0) .. controls (0,1) and (1,1) .. (1,2)
        node[pos=.5, shape=coordinate](CROSSING2){};
    \draw[thick, ->] (1,0) .. controls (1,1) and (0,1) .. (0,2);
    \draw[thick] (.5,0) .. controls (.5,.5) and (1,.5) .. (1,1)
        node[pos=.65, shape=coordinate](CROSSING1){};
    \draw[thick, ->] (1,1) .. controls (1,1.5) and (.5,1.5) .. (.5,2)
        node[pos=.35, shape=coordinate](CROSSING3){}
        node[pos=.65, shape=coordinate](DOT){};
    \draw[thick, color=blue, dashed] (CROSSING1) [out=180, in=-90] to (-1,2);
    \draw[thick, color=blue, dashed] (CROSSING2) [out=180, in=-90] to (-.5,2);
    \draw[thick, color=blue, dashed] (CROSSING3)  .. controls ++(-.9,.1) and ++(-.4,.4) .. (DOT);
\node() at (DOT) {\bbullet};
\end{tikzpicture}};
(24,2)*{\begin{tikzpicture}[scale=0.75]
  \draw[thick, ->] (-0.5,0) .. controls (-0.5,0.8) and (0.5,0.8) .. (0.5,0)
      node[pos=0.5, shape=coordinate](X){}
      node[pos=0.1, shape=coordinate](Y){};
  \draw[thick] (-0.5,0) .. controls (-0.5,-0.8) and (0.5,-0.8) .. (0.5,0)
      node[pos=0.1, shape=coordinate](Z){};
  \draw[color=blue, thick, double distance=1pt, dashed] (X) .. controls++(0,.65) and ++(-.65,.3) .. (Y) node[pos=0.15,right]{$\scs \l-1$\;};
  \draw[color=blue, thick, dashed] (Z) to[out=180, in=90] (-1,1.5) ;
  \draw[line width=0mm] (-0.5,0) .. controls (-0.5,0.8) and (0.5,0.8) .. (0.5,0)
    node[pos=0.1]{\bbullet};
  \draw[line width=0mm] (-0.5,0) (-0.5,0) .. controls (-0.5,-0.8) and (0.5,-0.8) .. (0.5,0)
        node[pos=0.1]{\bbullet};
\end{tikzpicture}};
(0,-13)*{-};
(5,-13)*{+};
(11,-13)*{+};
(-10,13)*{\circ};
(-5,13)*{\circ};
(0,13)*{+};
(5,13)*{+};
(10,13)*{-};
(18,13)*{\circ};
(55,-13)*{\Ec_{-++}\onebbl := \Fc\Ec\Ec\onebbl};
(60,13)*{\Ec_{\circ\circ++-\circ}\onebbl := \Pi \Pi\Ec\Ec\Fc\Pi\onebbl};
\endxy
\]

In order to define the spanning set we first recall the spanning sets from \cite[Section 3.2]{KL3} for the original 2-category $\Ucev$.

Mark $m$ points $1 \times \{0\},2 \times \{0\}, \ldots, m \times \{0\}$ on the
lower boundary $\R \times \{0\}$ of the strip $\R \times [0,1]$ and $k$ points $1
\times \{1\},2 \times \{1\}, \ldots, k \times \{1\}$ on the upper boundary $\R
\times \{1\}$. Assuming $m+k$ is even, choose an immersion of $\frac{m+k}{2}$
strands into $\R\times [0,1]$ with these $m+k$ points as the endpoints, and such that critical values for the height function are isolated, distinct, and nondegenerate.  Orient
each strand.  Then endpoints inherit orientations from the strands. Orientation at lower and upper endpoints define signed sequences
$\epsilon, \epsilon' \in {\rm Seq}$. We consider immersions modulo boundary-preserving homotopies and call them pairings between signed sequences $\epsilon$ and $\epsilon'$, or simply $(\epsilon,\epsilon')$-pairings.  Critical values are allowed to momentarily pass each other in homotopies.

There is a bijection between $(\epsilon,\epsilon')$-pairings and complete matchings of $m+k$ points such that the two points in each matching pair have compatible orientations.  A minimal diagram $D$ of a $(\epsilon,\epsilon')$-pairing is a generic immersion that realizes the pairing such that strands have no self--intersections and any two strands intersect at most once.  We consider minimal diagrams up to boundary--preserving isotopies. For each $(\epsilon,\epsilon')$-pairing fix a choice of minimal diagram $D$ and denote by $p(\epsilon,\epsilon')$ the set of the minimal diagram representatives of $(\epsilon,\epsilon')$-pairings.
For each diagram $D$ in $p(\epsilon,\ep')$ choose an interval on each arc, away from the intersections.  Let $B'_{\epsilon,\ep'}$ denote the union over all $D$, of diagrams built from $D$ by putting an arbitrary number of dots on each of the intervals.

Forgetting the dashed lines in a bubble monomial from $\Xi$ gives rise to a 2-endomorphisms of the 1-morphisms $\onebbl$ in the (even) 2-category $\Ucev$.
Let $B_{\ep,\ep',\lambda}$ be the set obtained from $B'_{\epsilon,\ep'}$ by labelling the rightmost region by the weight $\lambda$ and placing an arbitrary bubble monomial of $\Ucev$ in the rightmost region.
Proposition 3.11 of  \cite{KL3} shows that the sets $B_{\ep,\ep',\lambda}$ give a basis for the space of 2-morphisms $\HOM_{\Ucev}(\Ec_{\ep}\onebbl, \Ec_{\epsilon'}\onebbl)$ in the 2-category $\Ucev$.

To each $D \in B_{\epsilon,\ep',\lambda}$ we refer to each crossing and dot in $D$ as an {\em internal vertex} of the diagram $D$.

For the covering 2-category $\Uc$ we need some additional structure to describe spanning sets.

Given two covering sequences $\epsilon, \ep' \in {\rm CSeq}$ with associated signed sequences $\underline{\ep}, \underline{\ep'} \in {\rm Seq}$, let $D$ be an element of $B_{\underline{\ep},\underline{\ep'},\lambda}$.  An {\em assignment of $\Pi$-data} to $D$ relative to the covering sequences $\epsilon$ and $\ep'$ is a pairing $\zeta$ on the union of the set of internal vertices of $D$ and the set of points labelled by $\circ$ on the boundary.  Diagrammatically, pairs are joined by non-intersecting dashed arcs in generic position relative to the diagram $D$.  Dashed lines leaving an internal vertex that intersect solid strands connected to the same internal vertex are always assumed to intersect in a clock-wise fashion:
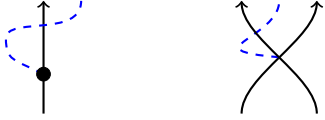
\[
\hackcenter{\begin{tikzpicture}
    \draw[thick, ->] (0,0) -- (0,1.5)
        node[pos=.35, shape=coordinate](DOT){};
    \draw[thick, color=blue, dashed]
   (DOT) .. controls ++(-.4,.2) and ++(0,-.3) .. (-.5,1)
 .. controls ++(0,.3) and ++(0,-.5) .. (.5,1.5);
    \node at (DOT) {\bbullet};
\end{tikzpicture}}
\qquad  \qquad
\hackcenter{\begin{tikzpicture}
    \draw[thick, ->] (0,0) .. controls (0,.5) and (1,1) .. (1,1.5)
        node[pos=.5, shape=coordinate](CROSSING){};
    \draw[thick, color=blue, dashed]
  (CROSSING).. controls ++(-1.2,.1) and ++(0,-.5) .. (.5,1.5);
    \draw[thick, ->] (1,0) .. controls (1,.5) and (0,1) .. (0,1.5);
\end{tikzpicture}}
\]
Such pairs $(D,\zeta)$ determine a 2-morphism, denoted $D^{\zeta}(\lambda)$, in $\Uc$. The above conventions ensure that any choice of $\Pi$-data on a fixed $D$ represents the same 2-morphism in $\Uc$.

\begin{example}
Several examples of $\Pi$-data are shown below for a diagram $D$ in $B_{-++,\circ\circ++-\circ, \lambda}$.
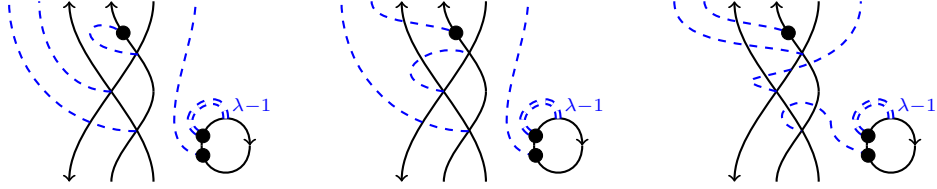
\[
\xy
(0,0)*{
\begin{tikzpicture}[scale=0.8]
    \draw[thick, <-] (0,0) .. controls ++(0,1) and ++(0,-1) .. (1.4,3)
        node[pos=.5, shape=coordinate](CROSSING2){};
    \draw[thick, ->] (1.4,0) .. controls ++(0,1) and ++(0,-1) .. (0,3);
    \draw[thick] (.7,0) .. controls ++(0,.5) and ++(0,-.5) .. (1.4,1.5)
        node[pos=.56, shape=coordinate](CROSSING1){};
    \draw[thick, ->] (1.4,1.5) .. controls ++(0,.5) and ++(0,-.5) .. (.7,3)
        node[pos=.41, shape=coordinate](CROSSING3){}
        node[pos=.65, shape=coordinate](DOT){};
    \draw[thick, color=blue, dashed] (CROSSING1) [out=180, in=-90] to (-1,3);
    \draw[thick, color=blue, dashed] (CROSSING2) [out=180, in=-90] to (-.5,3);
    \draw[thick, color=blue, dashed] (CROSSING3)  .. controls ++(-1.2,.1) and ++(-.6,.4) .. (DOT);
  \node() at (DOT) {\bbullet};
  \draw[thick, ->] (2.2,.6) .. controls ++(-0,0.6) and ++(0,0.6) .. (3,.6)
      node[pos=0.5, shape=coordinate](X){}
      node[pos=0.1, shape=coordinate](Y){};
  \draw[thick] (2.2,.6) .. controls ++(0,-0.6) and ++(0,-0.6) .. (3,.6)
      node[pos=0.1, shape=coordinate](Z){};
  \draw[color=blue, thick, double distance=1pt, dashed] (X) .. controls++(0,.65) and ++(-.65,.3) .. (Y) node[pos=0.15,right]{$\scs \l-1$\;};
  \draw[color=blue, thick, dashed] (Z) .. controls ++(-1,.1) and ++(0,-.8) .. (2.1,3) ;
  \node at (Y) {\bbullet};
  \node at (Z) {\bbullet};
\end{tikzpicture}};
\endxy
\qquad
\xy
(0,0)*{
\begin{tikzpicture}[scale=0.8]
    \draw[thick, <-] (0,0) .. controls ++(0,1) and ++(0,-1) .. (1.4,3)
        node[pos=.5, shape=coordinate](CROSSING2){};
    \draw[thick, ->] (1.4,0) .. controls ++(0,1) and ++(0,-1) .. (0,3);
    \draw[thick] (.7,0) .. controls ++(0,.5) and ++(0,-.5) .. (1.4,1.5)
        node[pos=.56, shape=coordinate](CROSSING1){};
    \draw[thick, ->] (1.4,1.5) .. controls ++(0,.5) and ++(0,-.5) .. (.7,3)
        node[pos=.41, shape=coordinate](CROSSING3){}
        node[pos=.65, shape=coordinate](DOT){};
    \draw[thick, color=blue, dashed] (CROSSING1) [out=180, in=-90] to (-1,3);
    \draw[thick, color=blue, dashed] (CROSSING2) .. controls ++(-1.2,.2) and ++(-.6,.3) .. (CROSSING3);
    \draw[thick, color=blue, dashed] (-.5,3)  .. controls ++(0,-.4) and ++(-.4,.2) .. (DOT);
  \node() at (DOT) {\bbullet};
  \draw[thick, ->] (2.2,.6) .. controls ++(-0,0.6) and ++(0,0.6) .. (3,.6)
      node[pos=0.5, shape=coordinate](X){}
      node[pos=0.1, shape=coordinate](Y){};
  \draw[thick] (2.2,.6) .. controls ++(0,-0.6) and ++(0,-0.6) .. (3,.6)
      node[pos=0.1, shape=coordinate](Z){};
  \draw[color=blue, thick, double distance=1pt, dashed] (X) .. controls++(0,.65) and ++(-.65,.3) .. (Y) node[pos=0.15,right]{$\scs \l-1$\;};
  \draw[color=blue, thick, dashed] (Z) .. controls ++(-1,.1) and ++(0,-.8) .. (2.1,3) ;
  \node at (Y) {\bbullet};
  \node at (Z) {\bbullet};
\end{tikzpicture}};
\endxy
\qquad
\xy
(0,0)*{
\begin{tikzpicture}[scale=0.8]
    \draw[thick, <-] (0,0) .. controls ++(0,1) and ++(0,-1) .. (1.4,3)
        node[pos=.5, shape=coordinate](CROSSING2){};
    \draw[thick, ->] (1.4,0) .. controls ++(0,1) and ++(0,-1) .. (0,3);
    \draw[thick] (.7,0) .. controls ++(0,.5) and ++(0,-.5) .. (1.4,1.5)
        node[pos=.56, shape=coordinate](CROSSING1){};
    \draw[thick, ->] (1.4,1.5) .. controls ++(0,.5) and ++(0,-.5) .. (.7,3)
        node[pos=.41, shape=coordinate](CROSSING3){}
        node[pos=.65, shape=coordinate](DOT){};
    \draw[thick, color=blue, dashed] (-1,3) .. controls ++(0,-.8) and ++(-.6,.2) .. (CROSSING3);
    \draw[thick, color=blue, dashed] (CROSSING2) .. controls ++(-1.5,.4) and ++(0,-1.5) ..(2.1,3);
    \draw[thick, color=blue, dashed] (-.5,3)  .. controls ++(0,-.4) and ++(-.4,.2) .. (DOT);
  \node() at (DOT) {\bbullet};
  \draw[thick, ->] (2.2,.6) .. controls ++(-0,0.6) and ++(0,0.6) .. (3,.6)
      node[pos=0.5, shape=coordinate](X){}
      node[pos=0.1, shape=coordinate](Y){};
  \draw[thick] (2.2,.6) .. controls ++(0,-0.6) and ++(0,-0.6) .. (3,.6)
      node[pos=0.1, shape=coordinate](Z){};
  \draw[color=blue, thick, double distance=1pt, dashed] (X) .. controls++(0,.65) and ++(-.65,.3) .. (Y) node[pos=0.15,right]{$\scs \l-1$\;};
  \draw[color=blue, thick, dashed] (Z) .. controls ++(-.2,.1) and ++(0,-.4) .. (1.6,1)
   .. controls ++(-.4,.7) and ++(-.8,.2) ..  (CROSSING1) ;
  \node at (Y) {\bbullet};
  \node at (Z) {\bbullet};
\end{tikzpicture}};
\endxy
\]
\end{example}

Denote by $B_{\epsilon,\ep',\lambda}^{\pi}$ the set obtained by fixing a choice of
$\Pi$-data $\zeta$ for each diagram $D$ in $B_{\epsilon,\ep',\lambda}$.
Arguing as in \cite[Section 3.2.3]{KL3} we have the following:

\begin{prop}
The set $B_{\ep,\ep',\lambda}^{\pi}$ is a homogeneous spanning set for the graded $\Bbbk$-module $\HOM_{\Uc}(\Ec_{\epsilon}\onebbl, \Ec_{\epsilon'}\onebbl)$.
\end{prop}

\begin{prop} \label{conj-spanning-2} If $\Bbbk$ is a field of characteristic $2$, then $B_{\ep,\ep',\lambda}^{\pi}$ is a homogeneous basis for the graded $\Bbbk$-vector space $\HOM_{\Uc}(\Ec_{\epsilon}\onebbl, \Ec_{\epsilon'}\onebbl)$.\end{prop}
\begin{proof} Given a 2-morphism $D^\zeta(\lambda)$, let $\text{ev}(D^\zeta(\lambda))$ be the 2-morphism in the (non-super) 2-category $\Ucev\otimes_Z\Bbbk$ obtained by deleting all dashed lines.  Since the 2-hom relations in $\Uc\otimes_\Z\Bbbk$ and $\Ucev\otimes_\Z\Bbbk$ are identical, it follows that their 2-hom-spaces have equal graded dimensions over $\Bbbk$.  Since graded 2-hom space dimensions in $\Ucev\otimes_\Z\Bbbk$ are in accord with \eqref{eqn-graded-dim-conj} (with the $\pi$'s removed), the claim follows.\end{proof}
It follows that over a field of characteristic $2$, given two signed sequences $\ep, \ep' \in {\rm Seq}$, the graded dimension of $\Pi\HOM_{\Uc}(\Ec_{\epsilon}\onebbl, \Ec_{\epsilon'}\onebbl)$ is given by
\begin{equation}\label{eqn-graded-dim-conj}
 \dim_{q,\pi}\Pi\HOM_{\Uc}(\Ec_{\epsilon}\onebbl, \Ec_{\epsilon'}\onebbl)
 = \sum_{D^{\zeta} \in B_{\ep,\ep',\lambda}} q^{\deg (D^{\zeta}(\l))} +
 \pi \sum_{D^{\zeta} \in B_{\ep,\circ\ep',\lambda}} q^{\deg (D^{\zeta}(\l))},
\end{equation}
where $\deg (D^{\zeta}(\l))$ denotes the $q$-degree of the 2-morphism $D^{\zeta}(\l)$ given by the sum of the degrees of each generator from \eqref{eq_generators} and \eqref{eq_generators_cont}.

\begin{conj} \label{conj-spanning} Let $\widetilde{B}_{\ep,\ep',\lambda}^{\pi}\subseteq B_{\ep,\ep',\lambda}^{\pi}$ be the subset consisting of diagrams in which all odd degree bubbles have distinct degrees (i.e., there are no ``repeated'' odd bubbles).
\begin{itemize}
\item If $\Bbbk$ is a field of characteristic not equal to $2$, then $\widetilde{B}_{\ep,\ep',\lambda}^{\pi}$ is a homogeneous basis for the graded $\Bbbk$-vector space $\HOM_{\Uc}(\Ec_{\epsilon}\onebbl, \Ec_{\epsilon'}\onebbl)$.
\item If $\Bbbk=\Z$, then the graded $\Z$-module $\HOM_{\Uc}(\Ec_{\epsilon}\onebbl, \Ec_{\epsilon'}\onebbl)$ has free part with homogeneous basis $\widetilde{B}_{\ep,\ep',\lambda}^{\pi}$ and torsion part the $\Z_2$-vector space on the homogeneous basis $B_{\ep,\ep',\lambda}^{\pi}\setminus\widetilde{B}_{\ep,\ep',\lambda}^{\pi}$.
\end{itemize}
\end{conj}
If Conjecture \ref{conj-spanning} holds, then equations analogous to \eqref{eqn-graded-dim-conj} hold in these cases.  Working over $\Z$, then, we have
\begin{equation*}
\Pi\HOM_{\Uc}(\Ec_{\epsilon}\onebbl, \Ec_{\epsilon'}\onebbl)\cong\Z^{f(q)}\oplus(\Zt)^{g(q)}
\end{equation*}
as graded $\Z$-modules, where $f,g$ are Laurent polynomials satisfying
\begin{equation*}
\dim_q\Pi\HOM_{\Ucev}(\text{ev}(\Ec_{\epsilon}\onebbl), \text{ev}(\Ec_{\epsilon'}\onebbl))=f(q)+g(q).
\end{equation*}
Every 2-hom space in $\Uc$ can be expressed as the tensor product of a part consisting of strands diagrams with no bubbles and several bubbles (non-nested with the same orientation).  It will follow from the results of Subsection \ref{subsec-consequences} that the strands part is a free $\Z$-module.

Thus Conjecture \ref{conj-spanning}, taken over $\Z$, is reduced to the following weaker conjecture.

\begin{conj} \label{conj-free}
The map $\Xi\to\Pi\HOM_{\Uc}(\onebb_\lambda,\onebb_\lambda)$ of \eqref{eq_ltol} is an isomorphism of graded $\Z$-modules.
\end{conj}

%---------------------------------------------------------------------
%
\subsection{Karoubi envelopes and 2-representations} \label{subsec:2reps}
%
%---------------------------------------------------------------------

%- - - - - - - - - - - - - - - - - - - - - - - - - - - - - - - - - - - - - - - -
%
\subsubsection{Karoubi envelopes}
%
% - - - - - - - - - - - - - - - - - - - - - - - - - - - - - - - - - - - - - - - -

Recall that the Karoubi envelope $Kar(\Cc)$ of a category $\Cc$ is an enlargement of
the category $\Cc$ in which all idempotents split (see \cite[Section 9]{Lau1}
and references therein). There is a fully faithful functor $\Cc \to
Kar(\Cc)$ that is universal with respect to functors which split idempotents
in $\Cc$. This means that if $F\maps \Cc \to \Dc$ is any functor
where all idempotents split in $\Dc$, then $F$ extends uniquely (up to
isomorphism) to a functor $\tilde{F} \maps Kar(\Cc) \to \Dc$ (see for
example \cite{Bor}, Proposition 6.5.9).

\begin{defn}
Define the additive $\Bbbk$-linear super-2-category $\Udotc$ to have the same objects
as $\Uc$ and hom additive $\Bbbk$-linear 1-hom-categories given by
$\Hom_{\Udotc}(\lambda,\lambda') = Kar\left(\Hom_{\Uc}(\lambda,\lambda')\right)$. The
fully-faithful additive $\Bbbk$-linear functors $\Hom_{\Uc}(\lambda,\lambda') \to
\Hom_{\Udotc}(\lambda,\lambda')$ combine to form an additive $\Bbbk$-linear 2-functor
$\Uc \to \Udotc$ universal with respect to splitting idempotents in the hom
categories $\Hom_{\Udotc}(\lambda,\lambda')$.  The composition functor
$\Hom_{\Udotc}(\lambda,\lambda') \times \Hom_{\Udotc}(\lambda',\lambda'') \to
\Hom_{\Udotc}(\lambda,\lambda'')$ is induced by the universal property of the Karoubi
envelope from the composition functor for $\Uc$. The 2-category $\Udotc$ has
graded 2-hom-spaces given by
\begin{equation}
\HOM_{\Udotc}(x,y) := \bigoplus_{t\in \Z}\Hom_{\Udotc}(x,y \la t\ra).
\end{equation}
\end{defn}

%- - - - - - - - - - - - - - - - - - - - - - - - - - - - - - - - - - - - - - - -
%
\subsubsection{Divided powers}
%
% - - - - - - - - - - - - - - - - - - - - - - - - - - - - - - - - - - - - - - - -
Consider the 2-morphism
\[
\overline{\partial}_{+,i} \quad := \quad
\hackcenter{
\begin{tikzpicture}[scale=0.6]
   \draw[thick, ->] (-1.5,0) -- (-1.5,1.5);
    \draw[thick, ->] (1.5,0) -- (1.5,1.5);
     \node at (-2.25,.75){$\cdots$};  \node at (2.25,.75){$\cdots$}; \node at (3.6,.75){$\l$};
    \draw[thick, ->] (3,0) -- (3,1.5);
    \draw[thick, ->] (-3,0) -- (-3,1.5);
  \draw[thick, ->] (-0.5,0) .. controls (-0.5,0.75) and (0.5,0.75) .. (0.5,1.5)
      node[pos=0.5, shape=coordinate](X){}
      node[pos=0.2, shape=coordinate](dot){};
  \draw[thick, ->] (0.5,0) .. controls (0.5,0.75) and (-0.5,0.75) .. (-0.5,1.5);
  \draw[color=blue,  thick, dashed] (X) .. controls++(-.5,0) and ++(-.65,.3) .. (dot);
  \draw[line width=0mm] (-0.5,0) .. controls (-0.5,0.75) and (0.5,0.75) .. (0.5,1.5)
      node[pos=0.2]{\tikz \draw[fill=black] circle (0.45ex);};
\end{tikzpicture} }
\]
in $\END_{\Uc}(\Ec^a\onebbl)$ (the dot is on the $i$-th strand). Let $w_0=s_{i_1} \dots s_{i_k}$ denote a reduced decomposition of the longest word in the symmetric group $S_a$.   It follows from \cite[Proposition 3.6]{EKL} that the element $e_{+,a}:=\overline{\partial}_{+,w_0}= \overline{\partial}_{+,i_1}\dots \overline{\partial}_{+,i_k}$ is an idempotent 2-morphism in $\Uc$.  Likewise, define 2-morphisms $\overline{\partial}_{-,i}$ in $\END_{\onebbl\Uc}(\Fc^a)$ by rotating the diagram for $\overline{\partial}_{+,i}$ 180 degrees. Let $e_{-,a}:= \overline{\partial}_{-,i}$ denote the corresponding idempotent in $\END_{\Uc}(\onebbl\Fc^a)$.

Following \cite{EKL} we introduce divided powers
\begin{align}
 \Ec^{(a)}\onebbl &:= \left(
 \Ec^a\onebbl \la \qbin{a}{2}\ra, e_{+,a}
 \right) \\
  \Fc^{(b)}\onebbl &:= \left(
 \Fc^b\onebbl \la \qbin{b}{2}\ra, e_{-,b}
 \right)
\end{align}
in the Karoubi envelope $\Udotc$.

Adding $\Pi$-data to diagrams in \cite[Section 4.4]{EKL} as explained in Section~\ref{subsubsec-otherspanning} prove the following proposition.

\begin{prop} The following relations holds in $\Udotc$.
\begin{align}
\label{eq_EaEb}
 \Ec^{(a)} \Ec^{(b)}\onebbl &= \bigoplus_{\qbins{a+b}{a}}\Ec^{(a+b)}\onebbl, &\\
 \label{eq_FaFb}
 \Fc^{(a)}\Fc^{(b)}\onebbl &= \bigoplus_{\qbins{a+b}{a}}\Fc^{(a+b)}\onebbl. &
\end{align}
\end{prop}

%- - - - - - - - - - - - - - - - - - - - - - - - - - - - - - - - - - - - - - - -
%
\subsubsection{2-representations}
%;8
% - - - - - - - - - - - - - - - - - - - - - - - - - - - - - - - - - - - - - - - -
\begin{defn}
A 2-representation of $\Udotpi(\sltwo)$ is a graded additive $\Bbbk$-linear super-2-functor $\Udotc\to \Cc$ for some graded, additive super-2-category $\Cc$.
\end{defn}

When all of the 1-hom-categories $\Hom_{\Cc}(x,y)$ between objects $x$ and $y$ of $\Cc$ are idempotent complete, in other words $Kar(\Cc) \cong \Cc$, any graded additive $\Bbbk$-linear super-2-functor $\Uc \to \Cc$ extends uniquely to a 2-representation of $\Udotc$.  All abelian categories are idempotent complete.

%---------------------------------------------------------------------
%
\subsection{Grothendieck groups of super-2-categories}
%
%---------------------------------------------------------------------

%- - - - - - - - - - - - - - - - - - - - - - - - - - - - - - - - - - - - - - - -
%
\subsubsection{Grothendieck groups}
%
% - - - - - - - - - - - - - - - - - - - - - - - - - - - - - - - - - - - - - - - -

The notion of Grothendieck group we will use is the split Grothendieck group of a $\Z$-graded super-2-category.  Let $\Cc$ be a graded additive super-2-category with translations, in the obvious sense generalizing \cite[Definition 5.1]{Lau1}; $\Uc$ is an example of such a category.  Elements of $K_0(\Cc)$ are spanned by the classes $[X]$ of 1-morphisms $X$ of $\Cc$.  Identity 1-morphisms $\onebb_{\l}:\l\to \l$ naturally become orthogonal idempotents $\one_{\l}=[\onebb_{\l}]$ of $K_0(\Cc)$.  So if $X: \mu \to \l$ and $Y:\mu'\to \l'$ are 1-morphisms in $\Cc$, then
\begin{equation*}
[X][Y]=[\onebb_{\l}X\onebb_{\mu}][\onebb_{\l'}Y\onebb_{\mu'}]=\one_{\l}[X]\one_{\mu}\one_{\l'}[Y]\one_{\mu'}
=\delta_{\mu \l'}\one_{\l}[X]\one_{\mu}[Y]\one_{\mu'}.
\end{equation*}
If $\mu=\l'$, that is if $X$ and $Y$ are composable, then this equals $\one_{\l}[XY]\one_{\mu'}$.  In this way, $K_0(\Cc)$ naturally becomes an idempotented algebra.  The relations in $K_0(\Cc)$, as in \cite{Lau1}, are determined by isomorphisms of 2-morphisms in $\Cc$:
\begin{equation*}
[\beta]=[\alpha]+[\gamma]\text{ if }\beta\cong\alpha\oplus\gamma
\end{equation*}
in any of the 1-hom-categories of $\Cc$.  The abelian group $K_0(\Cc)$ is made into a module over $\Ac_\pi=\Z[q,q^{-1},\pi]/(\pi^2-1)$ by having the $\Z$-grading act on the classes of 1-morphisms by multiplication by $q$,
\begin{equation*}
[X\la-1\ra]=q[X],
\end{equation*}
and having parity shifts act by multiplication by $\pi$,
\begin{equation*}
[\Pi X]=\pi[X].
\end{equation*}

The space of homs between any two objects in $\Hom_{\Udotc}(\lambda,\mu)$ is a
finite-dimensional $\Bbbk$-vector space.  In particular, the Krull-Schmidt
decomposition theorem holds, and an indecomposable object of
$\Hom_{\Udotc}(\lambda,\mu)$ has the form $(\Ec_{\ep}\onebbl\la t\ra,e)$ for some
minimal/primitive idempotent $e$.  Any presentation of $1=e_1+\dots+e_k$ into the
sum of minimal mutually-orthogonal idempotents gives rise to a decomposition
\begin{equation}
  \Ec_{\ep}\onebbl\la t\ra \cong \bigoplus_{r=1}^{k}(\Ec_{\ep}\onebbl\la t\ra,e_r)
\end{equation}
into a direct sum of indecomposable objects of $\Hom_{\Udotc}(\lambda,\mu)$. Any object
of $\Hom_{\Udotc}(\lambda,\mu)$ has a unique presentation, up to permutation of factors
and isomorphisms, as a direct sum of indecomposables. Choose one representative
${b}$ for each isomorphism class of indecomposables, up to grading shifts, and
denote by $\dot{\mathcal{B}}(\lambda,\mu)$ the set of these representatives. Then
$\{[{b}]\}_b$ is a basis of $K_0\big(\Hom_{\Udotc}(\lambda,\mu)\big)$, viewed as a free
$\Ac_\pi$-module. Composition functors
\begin{equation}
  \Hom_{\Udotc}(\lambda,\lambda') \times \Hom_{\Udotc}(\lambda',\lambda'') \longrightarrow
  \Hom_{\Udotc}(\lambda,\lambda'')
\end{equation}
induce $\Ac_\pi$-bilinear maps
\begin{equation}
  K_0\big(\Hom_{\Udotc}(\lambda,\lambda')\big) \otimes K_0\big(\Hom_{\Udotc}(\lambda',\lambda'')\big) \longrightarrow
  K_0\big(\Hom_{\Udotc}(\lambda,\lambda'')\big)
\end{equation}
turning $K_0(\Udotc)$ into a $\Ac_\pi$-linear additive category with objects
$\lambda\in \Z$.   Multiplication in this basis has structure coefficients in
$\N[q,q^{-1},\pi]/(\pi^2-1)$.
%---------------------------------------------------------------------
%
\subsubsection{$K_0(\Uc)$ admits a map to $\Udotpi$}\label{subsubsec-k0-map}
%
%---------------------------------------------------------------------

\begin{prop} \label{prop_coveringrelsU}
There are explicit 2-isomorphisms
\begin{align}
\Ec \Fc \onebbl &\cong \Fc\Pi \Ec \onebbl \bigoplus_{k=0}^{\lambda-1} \Pi^k\onebbl \la \l-1-2k \ra
 &\text{for $\l \geq 0$ } \\
\Ec\Pi\Fc \onebbl &\cong \Fc \Ec \onebbl\bigoplus_{k=0}^{-\l-1}\Pi^{\l+1+k} \onebbl \la -\l-1-2k \ra
 &\text{for $\l \leq 0$}
\end{align}
in the 2-category $\Uc$.
\end{prop}

\begin{proof}
The relations in the 2-category $\Uc$ imply the for $\l\geq 0$ the 2-morphism
\begin{equation}
\hackcenter{\begin{tikzpicture}[scale=0.8]
  \draw[semithick, <-] (-0.5,0) .. controls (-0.5,0.5) and (0.5,0.5) .. (0.5,1)
      node[pos=0.5, shape=coordinate](X){};
  \draw[semithick, ->] (0.5,0) .. controls (0.5,0.5) and (-0.5,0.5) .. (-0.5,1);
  \draw[color=blue,  thick, dashed] (X) to (0,0);
\end{tikzpicture}}\;\; \bigoplus_{k=0}^{\lambda-1}
\hackcenter{\begin{tikzpicture}[scale=0.8]
  \draw[thick, ->-=0.15, ->] (0.5,.2) .. controls (0.6,-0.8) and (-0.6,-0.8) .. (-0.5,.2)
      node[pos=0.85, shape=coordinate](Y){};
  \draw[color=blue, thick, double distance=1pt, dashed]
   (Y) .. controls++(-.5,.2) and ++(0,.4) .. (-1,-1)
         node[pos=0.75,left]{$\scs k$};
  \draw[line width=0mm] (0.5,.2) .. controls (0.5,-0.8) and (-0.5,-0.8) .. (-0.5,.2)
     node[pos=0.85]{\tikz \draw[fill=black] circle (0.4ex);};
\end{tikzpicture} }:\Fc\Pi \Ec \onebbl \bigoplus_{k=0}^{\lambda-1} \Pi^k\onebbl \la \l-1-2k \ra \rightarrow \Ec \Fc \onebbl
\end{equation}
is an isomorphism with inverse
\begin{equation}
-
  \;
\hackcenter{\begin{tikzpicture}[scale=0.8]
  \draw[thick, ->] (-0.5,1) .. controls (-0.5,1.4) and (0.5,1.6) .. (0.5,2)
      node[pos=0.5, shape=coordinate](Y){};
    \draw[thick, <-] (0.5,1) .. controls (0.5,1.4) and (-0.5,1.6) .. (-0.5,2);
  \draw[color=blue,  thick, dashed]
     (Y)  -- (0,2);
\end{tikzpicture} }\;\;
\;\; \bigoplus \;\;
\sum_{j=0}^{\l-1-k} (-1)^j
 \hackcenter{
\begin{tikzpicture}[scale=0.8]
  \draw[thick, ->] (-0.5,0) .. controls (-0.5,0.8) and (0.5,0.8) .. (0.5,0)
      node[pos=0.1, shape=coordinate](DOT){}
      node[pos=0.42, shape=coordinate](L){}
      node[pos=0.5, shape=coordinate](M){}
      node[pos=0.58, shape=coordinate](R){};
  \draw[thick, ->]
  (1.9,1) .. controls ++(0,0.6) and ++(0,0.6) .. (1.1,1)
      node[pos=0.05, shape=coordinate](Z){};
  \draw[thick] (1.9,1) .. controls ++(0,-0.6) and ++(-0,-0.6) .. (1.1,1)
      node[pos=0.5, shape=coordinate](X){}
      node[pos=0.2, shape=coordinate](Y){};
  \draw[color=blue, thick, double distance=1pt, dashed]
    (X) .. controls++(-.1,.5) and ++(-.2,.3) .. (Y)
         node[pos=0.9,right]{$\scs -\l-1$\;};
   \draw[color=blue, thick, double distance=1pt, dashed]
    (Z) .. controls ++(-.5,.4) and ++(.2,.8) .. (R) ;
   \node[blue] at (1.25,0.8){$\scs $\;};
     \node at (Y) {\bbullet};
     \node at (Z) {\bbullet};
 \draw[color=blue, thick, double distance=1pt, dashed] (M) -- (0,1.6);
 \draw[color=blue, thick, double distance=1pt, dashed]
    (DOT) .. controls++(-.65,0) and ++(-.25,.3) .. (L);
  \node at (DOT){\bbullet};
   \node[blue] at (.6,1.4){$\scs j$};
   \node[blue] at (-.3,1.5){$\scs k$};
   \node[blue] at (-1.25,.20){$\scs \l-1 -k-j$};
   \node at (-1,1.2) {$\l$};
\end{tikzpicture} }
\maps \Ec \Fc \onebbl\rightarrow \Fc\Pi \Ec \onebbl \bigoplus_{k=0}^{\lambda-1} \Pi^k\onebbl \la \l-1-2k \ra.
\end{equation}
Likewise, if $\l \le 0$ then the map
\begin{equation}
\hackcenter{\begin{tikzpicture}[scale=0.8]
  \draw[thick, ->] (-0.5,1) .. controls (-0.5,1.4) and (0.5,1.6) .. (0.5,2)
      node[pos=0.5, shape=coordinate](Y){};
    \draw[thick, <-] (0.5,1) .. controls (0.5,1.4) and (-0.5,1.6) .. (-0.5,2);
  \draw[color=blue,  thick, dashed]
     (Y) .. controls ++(.1,.4) and ++(.1,.4)  .. (-.6,1.6)
     .. controls ++(0,-.3) and ++(0,.3) ..(0,1);
\end{tikzpicture} }\;\; \bigoplus_{k=0}^{-\l-1}
\hackcenter{\begin{tikzpicture}[scale=0.8]
  \draw[thick, ->-=0.15, ->] (-0.7,.5) .. controls ++(-.1,-1) and ++(.1,-1) .. (0.7,.5)
      node[pos=0.85, shape=coordinate](Y){}
      node[pos=0.55, shape=coordinate](M){}
      node[pos=0.44, shape=coordinate](X){};
  \draw[color=blue, thick, double distance=1pt, dashed]
   (Y) .. controls++(-.5,.3) and ++(0,.5) .. (M)
         node[pos=0.15,above]{$\scs k$};
   \draw[color=blue, thick, double distance=1pt, dashed]
     (X) .. controls ++(0,.55) and ++(0,.55) ..
      (-.6,-.25) .. controls ++(0,-.3) and ++(0,.4) ..(0,-1);
   \node at (Y){\tikz \draw[fill=black] circle (0.4ex);};
\end{tikzpicture} }:\Ec\Pi\Fc \onebbl \bigoplus_{k=0}^{-\l-1}\Pi^{\l+1+k} \onebbl \la -\l-1-2k \ra \rightarrow \Fc \Ec \onebbl
\end{equation}
is an isomorphism with inverse
\begin{equation}
-
  \;
\hackcenter{\begin{tikzpicture}[scale=0.8]
  \draw[thick, <-] (-0.5,0) .. controls (-0.5,0.4) and (0.5,0.6) .. (0.5,1)
      node[pos=0.5, shape=coordinate](X){};
    \draw[thick, ->] (0.5,0) .. controls (0.5,0.4) and (-0.5,0.6) .. (-0.5,1);
  \draw[color=blue,  thick, dashed]
     (0,1)  .. controls ++(0,-.3) and ++(0,.3) .. (-.6,.4)
     .. controls ++(.1,-.4) and ++(.1,-.4)  .. (X);
\end{tikzpicture} }\;\;
\;\; \bigoplus \;\;
\sum_{j=0}^{-\l-1-k}(-1)^{j}
 \hackcenter{\begin{tikzpicture}[scale=0.8]
  \draw[thick,->-=0.8] (0.5,.25) -- (0.5,.5);
  \draw[thick,->-=0.55] (-0.5,.5) -- (-0.5,.25);
  \draw[thick] (0.5,.5) .. controls ++(.1,.8) and ++(-.1,.8) .. (-0.5,.5)
      node[pos=0.1, shape=coordinate](DOT){};
  \draw[color=blue, thick, double distance=1pt, dashed]
    (DOT) .. controls++(-.5,.4) and ++(0,-1) .. (-.75,1.75);
   \node at (DOT){\bbullet};
   \node[blue] at (0,1.60){$\scs \l-1 -k-j$};
   \node at (-1,.7) {$\l$};
\end{tikzpicture} }
\quad
\hackcenter{\begin{tikzpicture}
  \draw[thick, ->] (-0.4,0) .. controls ++(-0,0.6) and ++(0,0.6) .. (0.4,0)
      node[pos=0.5, shape=coordinate](X){}
      node[pos=0.1, shape=coordinate](Y){};
  \draw[thick] (-0.4,0) .. controls ++(0,-0.6) and ++(0,-0.6) .. (0.4,0)
      node[pos=0.1, shape=coordinate](Z){};
  \draw[color=blue, thick, double distance=1pt, dashed]
    (X) .. controls++(0,.65) and ++(-.65,.3) .. (Y) node[pos=0.15,right]{$\scs \l-1$\;};
  \draw[color=blue, thick, double distance=1pt, dashed] (Z) to[bend left] (-1,1);
  \node at (Y) {\bbullet};
  \node at (Z) {\bbullet};
  \node[blue] at (-.75,.9) {$\scs j$};
\end{tikzpicture} }
\maps\Fc \Ec \onebbl \rightarrow\Ec\Pi\Fc \onebbl \bigoplus_{k=0}^{-\l-1}\Pi^{\l+1+k} \onebbl \la -\l-1-2k \ra.
\end{equation}
\end{proof}

\begin{prop}\label{prop-gamma}
The assignment $E_{\ep}1_{\l} \mapsto [\Ec_{\ep}\onebbl]$ for each covering sequence $\ep \in {\rm CSeq}$ defines an $\Ac_\pi$-algebra homomorphism
\begin{equation}
\gamma \maps \AUdotpi \to K_0(\Udotc).
\end{equation}
\end{prop}

\begin{proof}
$K_0(\Udotc)$ is a free $\Ac_\pi$-module,
so it is enough to check that the assignment above extends to a homomorphism of
$\Q(q)^\pi$-algebras
\begin{equation}\label{gamma-field}
\gamma_{\Q(q)^\pi} \maps \Udotpi \lra K_0(\Udotc)\otimes_{\Ac_\pi}\Q(q)^\pi.
\end{equation}
Proposition \ref{prop_coveringrelsU} shows that defining relations of $\Udotpi$ lift to 2-isomorphisms of 1-morphisms in
$\Udotc$ and, therefore, descend to relations in the Grothendieck group
$K_0(\Udotc)$. Restricting $\gamma_{\Q(q)^\pi}$ to $\AUdotpi$ gives a homomorphism of
$\Ac_\pi$-algebras with the image of the homomorphism lying in
$K_0(\Udotc)$.
\end{proof}

\begin{prop} The following relations holds in $\Udotc$.
\begin{align}
\label{eq_FaEb} \Fc^{(a)}\Ec^{(b)}\onebbl&=
\bigoplus_{j=0}^{\min(a,b)} \bigoplus_{\qbins{a-b-n}{j}} \Ec^{(b-j)}\Fc^{(a-j)}\onebbl, & \text{if $\l < -2a+2$}\\
\label{eq_EaFb} \Ec^{(a)}\Fc^{(b)}\onebbl&=
\bigoplus_{j=0}^{\min(a,b)}\bigoplus_{\qbins{a-b+n}{j}}\Fc^{(b-j)}\Ec^{(a-j)}\onebbl & \text{if $\l > 2b-2$}.
\end{align}
\end{prop}

\begin{proof}
By the Krull-Schmidt theorem, the 1-morphisms $\Fc^b\Ec^a\onebbl$ and $\Ec^a\Fc^b\onebbl$ have unique decompositions into indecomposables.   These equations follow from the covering $\mf{sl}_2$ isomorphisms applied to $\Fc^b\Ec^a\onebbl$ and $\Ec^a\Fc^b\onebbl$.
\end{proof}

%#####################################################################
%
\section{Formal structures in strong supercategorical actions}\label{sec-formal}
%
%#####################################################################

The previous section introduced a super-2-category $\Udotc$ as well as the notion of a strong supercategorical action of $\sltwo$.  An action of $\Udotc$ determines a strong supercategorical action.  The goal of this section is to show that the converse is also true.

The key idea is to use the brick condition for endomorphisms of $\onebl$ in the definition of a strong supercategorical action together with the covering $\mf{sl}_2$-isomorphisms to control the sizes of homs between other 1-morphisms.  This control of hom spaces forces the covering $\mf{sl}_2$-isomorphisms \eqref{eq:EF-rel} and \eqref{eq:FE-rel} to take a rigid form compatible with the isomorphisms \eqref{prop_coveringrelsU} in the definition of $\Uc$.

Throughout Sections \ref{sec-formal}--\ref{sec:proofsl2}, suppose $\Ett,\Ftt$ are functors defining a strong supercategorical action on a 2-category $\Cc$.  All string diagrams in these sections depict 2-morphisms in the 2-category $\Cc$.  These are \emph{not} relations in $\Udotc$; our goal is to prove that the defining relations of $\Udotc$ hold.

%---------------------------------------------------------------------
\subsection{Consequences of $\mf{sl}_2$-relations} \label{subsec:consequences}
%---------------------------------------------------------------------
The structure of a strong action on $\Cc$ imposes strong conditions on the hom-spaces between various maps.  Units and counits for various adjoint structures are formally determined from this data.

For $\l \geq 0$ equation \eqref{eq:FE-rel} together with the adjoint structure formally fixes a choice of 2-morphism $\Ucapr \maps \Ett\Ftt\onebltwo\ads{\l+1} \to \Pi^{\l+1}\onebltwo$ via the adjoint pairing:
\begin{align*}
  &\Hom(\Ett \Ftt \onebltwo \la \lambda+1 \ra, \Pi^{\l+1} \onebltwo) \\
  &\qquad \cong
 \Hom(\Ftt \Pi\Ett \onebltwo \bigoplus_{[\lambda+2]}\onebltwo\la \lambda+1 \ra, \Pi^{\l+1}\onebltwo) \\
 &\qquad \cong
 \Hom((\Ett \onebltwo)^L \Pi\Ett \onebltwo \ads{2 \lambda+4},\Pi^{\l+1}\onebltwo)
  \bigoplus_{k=0}^{\lambda+1} \Hom(\Pi^{k} \onebltwo \la 2(\l+1)-2k \ra, \Pi^{\l+1}\onebltwo )\\
  & \qquad \cong
 \Hom(\Pi\Ett \onebltwo, \Pi^{\l+1}\Ett \onebltwo \la -2\lambda-4 \ra)
  \oplus \Hom(\Pi^{\l+1}\onebltwo, \Pi^{\l+1}\onebltwo),
\end{align*}
where all summands are zero in the second to last equation except when $k=\l+1$.  We take $\Ucapr$ to be the identity map in the second summand of the last equation. This corresponds to the projection out of the top degree summand of $\onebltwo$ in $\Ett\Ftt\onebltwo$
\[
 \xy
  (35,10)*+{\Pi^{\l+1}\onebltwo }="l1";
  (-5,20)*+{\Ftt\Pi\Ett\onebltwo \la \lambda+1 \ra }="t1";
  (-5,10)*+{\Pi^{\l+1} \onebltwo\la 0 \ra }="t2";
  (-5,0)*+{\Pi^{\l}\onebltwo\la 2\ra }="t3";
  (-5,-20)*+{\onebltwo\la 2(\l +1) \ra }="t4";
  (-5,15)*{\oplus}; (-5,5)*{\oplus};  (-5,-5)*{\oplus};  (-5,-15)*{\oplus}; (-5,-9)*{\vdots};
   {\ar "t2";"l1"};
 \endxy
 \]
Likewise, for $\l \geq 0$ define the map $\Ucupl: \onebltwo \la \lambda+1 \ra \rightarrow \Ett \Ftt \onebltwo$ as the inclusion into the lowest degree summand $\onebltwo$ in $\Ett \Ftt \onebltwo$.
\[
 \xy
  (-45,-20)*+{ \onebltwo \la \lambda+1 \ra}="l1";
  (-5,20)*+{\Ftt\Pi\Ett\onebltwo}="t1";
  (-5,10)*+{\Pi^{\l+1} \onebltwo\la -\l -1 \ra }="t2";
  (-5,0)*+{\Pi^{\l}\onebltwo\la -\l +1\ra }="t3";
  (-5,-20)*+{\onebltwo\la \l +1 \ra }="t4";
  (-5,15)*{\oplus}; (-5,5)*{\oplus};  (-5,-5)*{\oplus};  (-5,-15)*{\oplus}; (-5,-9)*{\vdots};
   {\ar "l1"; "t4"};
 \endxy
 \]

The map $\Ucapl \maps \Ftt \Ett \onebl \to \onebl\ads{\l+1}$ is defined by adjunction
\begin{align*}
  &\Hom( \Ftt\Ett \onebl \la -\lambda-1 \ra,  \onebl)
\\&\qquad \cong
\Hom( \Ftt (\Ftt\onebltwo)^R \onebl, \onebl) \cong
\Hom( \Ftt \onebltwo, \onebl \Ftt\onebltwo)
\\ &\qquad \cong
\Hom( (\Ett\onebl)^L\onebltwo \la \lambda+1 \ra,  \Ftt\onebltwo)
\cong
\Hom( \onebltwo \la \lambda+1 \ra,  \Ett \Ftt\onebltwo)
\end{align*}
as the mate under adjunction to the map $\onebltwo\la \l+1\ra \to \Ett\Ftt\onebltwo$ defined above. (For more on mates under adjunction see~\cite{ks1}.) The map $\Ucupr: \Pi^{\l+1}\onebl \la -\lambda-1 \ra \rightarrow  \Ftt\Ett \onebl$ is more difficult to define. Its definition requires an inductive procedure that utilizes the integrability assumption in the definition of a strong action.  This map will be defined in the process of proving proposition~\ref{prop:adjoints}.

Recall that $[-\l]=-\pi^{\l}[\l]$.  For negative weight spaces $(\l \leq 0)$ we have
\begin{equation}
\Ftt\Ett\onebl= \Ett\Pi \Ftt\onebl \oplus \bigoplus_{k=0}^{-\lambda-1}\Pi^{\lambda+1+k} \onebl \la -\lambda -1 -2k\ra .
\end{equation}
Thus a map $\Ucupr: \Pi^{\l+1}\onebl \la -\lambda-1 \ra \rightarrow  \Ftt\Ett \onebl$ is formally determined as the inclusion
\[
 \xy
  (-45,-20)*+{\Pi^{\l+1}\onebl \la -\lambda-1 \ra}="l1";
  (-5,20)*+{\Ett\Pi\Ftt\onebl}="t1";
  (-5,10)*+{\Pi^{(\l+1)+(-\l-1)}\onebl\la \l +1 \ra }="t2";
  (-5,0)*+{\Pi^{(\l+1)+(-\l-2)}\onebl\la \l +3\ra }="t3";
  (-5,-20)*+{\Pi^{\l+1}\onebl\la -\l -1 \ra .}="t4";
  (-5,15)*{\oplus}; (-5,5)*{\oplus};  (-5,-5)*{\oplus};  (-5,-15)*{\oplus}; (-5,-9)*{\vdots};
   {\ar "l1"; "t4"};
 \endxy
 \]
Likewise, define the map
 $\Ucapl: \Ftt\Ett \onebl\rightarrow \onebl \la \lambda+1 \ra$
as the projection:
\[
 \xy
  (45,10)*+{ \onebl \la \lambda+1 \ra .}="l1";
  (-5,20)*+{\Ett\Pi\Ftt\onebl}="t1";
  (-5,10)*+{\Pi^{(\l+1)+(-\l-1)}\onebl\la \l +1 \ra }="t2";
  (-5,0)*+{\Pi^{(\l+1)+(-\l-2)}\onebl\la \l +3\ra }="t3";
  (-5,-20)*+{\Pi^{\l+1}\onebl\la -\l -1 \ra . }="t4";
  (-5,15)*{\oplus}; (-5,5)*{\oplus};  (-5,-5)*{\oplus};  (-5,-15)*{\oplus}; (-5,-9)*{\vdots};
   {\ar "t2"; "l1"};
 \endxy
 \]
 The map $\Ucupl\maps \onebltwo \ads{\l+1} \to \Ett\Ftt\onebltwo$ is then determined by adjunction.  For negative weight space the final adjunction map  $\Ucapr \maps \Ett\Ftt\onebltwo\ads{\l+1} \to \Pi^{\l+1}\onebltwo$ will be defined using the integrability assumption below.

% -------------------------------------------------------------------------------
%
\subsection{Adjoint induction hypothesis and its consequences}
%
% -------------------------------------------------------------------------------

We now make use of the integrability assumption in the definition of a strong supercategorical action. For positive weight spaces $(\lambda \geq 0)$ we proceed by decreasing induction on $\lambda$ starting from the highest weight.  For negative weight spaces $(\l \leq 0)$ we perform increasing induction on $\l$ starting from the lowest weight.  To simplify the exposition we focus on the case $\lambda \geq 0$. Throughout this section we make the following assumption:

\medskip
\noindent\fbox {
   \parbox{\linewidth}{
{\bf Adjoint induction hypothesis:}  \newline Fix $\l \geq 0$.  By induction assume that if $\mu > \lambda$ we have
\begin{equation} \label{eq:ind_hyp}
(\Ett \oneb_{\mu})^R \cong \oneb_{\mu} \Ftt \Pi^{\mu+1}\la\mu+1 \ra \text{ and }
 \left(\oneb_{\mu}\Ftt\right)^L =  \Pi^{\mu+1}\oneb_{\mu}\Ett
  \la \mu+1\ra.
\end{equation}
    }
}

\medskip
\noindent For $\mu \gg 0$ or $\mu \ll 0$ the weight $\mu$ is zero by the integrability assumption so that
\eqref{eq:ind_hyp} vacuously holds since both 1-morphisms are zero.

In this section we derive a number of consequences of the adjoint induction hypothesis \eqref{eq:ind_hyp} culminating in the proof that \eqref{eq:ind_hyp} holds for $\mu=\l$. The reader should recall the `Important Convention' from Definition~\ref{def_strong}.

\begin{lem}\label{lem:E}
Assuming the adjoint induction hypothesis of \eqref{eq:ind_hyp}, if $\mu \ge \l$ then $\Hom(\Ett \oneb_{\mu}, \Pi^k\Ett \oneb_{\mu} \la \ell \ra)$ is zero for all $k$ if $\ell < 0$ and one-dimensional if $\ell=k=0$.  Likewise for $\Hom(\oneb_{\mu} \Ftt\Pi^k,\oneb_{\mu} \Ftt \la \ell \ra)$.
\end{lem}
\begin{proof}
We prove the result for $\Ett$ by (decreasing) induction on $\mu$ (the result for $\Ftt$ follows by adjunction). We have
\begin{align*}
 & \Hom(\Ett \oneb_{\mu}, \Pi^k\Ett \oneb_{\mu} \la \ell \ra) \\
&\qquad \cong
    \Hom( \oneb_{\mu+2}, \Pi^k \Ett\oneb_{\mu}(\Ett \oneb_{\mu})^L \la \ell \ra) \\
&\qquad \cong
    \Hom( \oneb_{\mu+2} , \Pi^{k}\Ett \Ftt\oneb_{\mu+2} \ads{\ell-\mu-1}) \\
&\qquad \cong
\Hom( \oneb_{\mu+2}, \Pi^{k}\Ftt \Pi \Ett \oneb_{\mu+2} \ads{\ell-\mu-1} ) \oplus
\Hom(\oneb_{\mu+2}, \bigoplus_{[\mu+2]}\Pi^{k} \oneb_{\mu+2} \la \ell -\mu-1\ra) \\
&\qquad \cong
\Hom( \oneb_{\mu+2},\Ftt \Pi^{k+1} \Ett \oneb_{\mu+2} \ads{\ell-\mu-1} ) \oplus
\Hom(\oneb_{\mu+2}, \bigoplus_{[\mu+2]}\Pi^{k} \oneb_{\mu+2} \la \ell -\mu-1\ra) \\
&\qquad \cong
\Hom((\oneb_{\mu+2}\Ftt)^L\oneb_{\mu+2},\Pi^{k+1}\Ett\oneb_{\mu+2} \ads{\ell-\mu-1}) \oplus
\Hom(\oneb_{\mu+2}, \bigoplus_{[\mu+2]}\Pi^{k} \oneb_{\mu+2} \la \ell -\mu-1\ra)) \\
&\qquad \cong
\Hom(\Pi^{\mu+3}\Ett\oneb_{\mu+2}\ads{\mu+3}, \Pi^{k+1}\Ett \oneb_{\mu+2} \ads{\ell-\mu-1}) \oplus
\bigoplus_{j=0}^{\mu+1}\Hom(\oneb_{\mu+2}, \Pi^{k+j} \oneb_{\mu+2} \la \ell -2j\ra) \\
&\qquad \cong
\Hom(\Ett\oneb_{\mu+2}, \Pi^{\mu+k+4}\Ett \oneb_{\mu+2} \ads{\ell-2\mu-4}) \oplus
\bigoplus_{j=0}^{\mu+1}\Hom(\oneb_{\mu+2}, \Pi^{k+j} \oneb_{\mu+2} \la \ell -2j\ra),
\end{align*}
where we used the isomorphisms \eqref{eq-dashed-solid-cross} in the fourth line and the adjoint induction hypothesis \eqref{eq:ind_hyp} on the fifth line.  By induction the first term above is zero and, by condition (\ref{co:hom}) of Definition \ref{def_strong}, all the terms in the direct sum are zero unless $k=0$ and $\ell=0$. In that case we get $\Hom(\oneb_{\mu+2}, \oneb_{\mu+2}) \cong \k$ and we are done.
\end{proof}

\begin{lem}\label{lem:EE}
Assuming the adjoint induction hypothesis \eqref{eq:ind_hyp}, if $\mu \ge \l-2$ then $\Hom(\Ett \Ett \oneb_{\mu}, \Pi^k\Ett \Ett \oneb_{\mu} \la \ell \ra)$ is zero for all $k$ and $\ell \leq -2$ unless $\ell=-2$ and $k=1$ in which case it is one dimensional.
\end{lem}

\begin{proof}
The proof is by (decreasing) induction on $\mu$. We have
\begin{align*}
 & \Hom(\Ett \Ett \oneb_{\mu}, \Pi^k\Ett \Ett \oneb_{\mu}) \\
&\quad \cong \Hom(\Ett \oneb_{\mu+2}, \Pi^k\Ett \Ett \oneb_{\mu}(\Ett \oneb_{\mu})^L)
\\
&\quad \cong \Hom(\Ett\oneb_{\mu+2}, \Pi^k\Ett \Ett \Ftt\oneb_{\mu+2}\la -\mu-1 \ra)
\\
&\quad \cong \Hom(\Ett \oneb_{\mu+2}, \Pi^k\Ett \Ftt \Pi \Ett \oneb_{\mu+2} \la -\mu-1 \ra)
 \bigoplus_{j=0}^{\mu+1} \Hom(\Ett \oneb_{\mu+2}, \Pi^{k}\Ett\Pi^j \oneb_{\mu+2} \la \mu+1-2j-\mu-1 \ra)
 \\
&\quad \cong \Hom( \Ett \oneb_{\mu+2}, \Pi^k\Ftt \Pi\Ett\Pi\Ett \oneb_{\mu+2} \la -\mu-1 \ra)
    \bigoplus_{j=0}^{\mu+3} \Hom(\Ett \oneb_{\mu+2}, \Pi^{k+j+1}\Ett \oneb_{\mu+2} \la \mu+3-2j-\mu-1 \ra) \\
& \qquad \qquad \bigoplus_{j=0}^{\mu+1} \Hom(\Ett \oneb_{\mu+2}, \Pi^{k+j} \Ett \oneb_{\mu+2} \la -2j \ra)
\\
&\quad \cong \Hom(\Ett \Ett \oneb_{\mu+2},\Pi^k \Ett \Ett \oneb_{\mu+2} \la -2\mu-6 \ra)
 \bigoplus_{j=0}^{\mu+3} \Hom(\Ett \oneb_{\mu+2}, \Pi^{k+j+1}\Ett \oneb_{\mu+2} \la -2j+2 \ra) \\
& \qquad \qquad \bigoplus_{j=0}^{\mu+1} \Hom(\Ett \oneb_{\mu+2}, \Pi^{k+j}\Ett \oneb_{\mu+2} \la -2j \ra).
\end{align*}
Shifting by $\la \ell \ra$ where $\ell < -2$ we find that the first term is zero by induction and the others are zero by Lemma \ref{lem:E}. If $\ell=-2$ the same vanishing holds with the exception of the term in the middle summation when $j=0$ and $k=1$ which yields $\End(\Ett \oneb_{\mu}) \cong \k$.
\end{proof}

\begin{lem} \label{lem:FEtEF}
Assuming the adjoint induction hypothesis \eqref{eq:ind_hyp}, if $\mu \ge \l$ then $\Hom(\Ftt\Pi\Ett\oneb_{\l}, \Pi^{k}\Ett\Ftt \oneb_{\l} \ads{\ell})$ is zero for all $k$ if $\ell <0$ and one-dimensional if $k=0$ and $\ell=0$.
\end{lem}

\begin{proof}
Use the adjoint induction hypothesis to reduce to Lemma~\ref{lem:EE}.
\end{proof}

Lemma~\ref{lem:FEtEF} above implies that the map $\Ftt \Pi\Ett\onebl \to \Ett\Ftt \onebl \cong\Ftt \Pi\Ett \onebl \oplus_{[\l]} \onebl $ which includes $\Ftt \Pi\Ett\onebl$ into the  $\Ftt \Pi\Ett \onebl$ summand of $\Ett \Ftt \onebl$  is unique up to scalar multiple so that it must induce an isomorphism between the $\Ftt \Pi\Ett \onebl$ summand.  This inclusion must also induce the zero map from the $\Ftt \Pi\Ett \onebl$ summand on the right to any summand $\onebl \la \l-1-2k \ra$ on the right.

\begin{lem} \label{lem:EEF}
Assuming the adjoint induction hypothesis \eqref{eq:ind_hyp}, if $\mu \ge \l$ then $$\Hom(\Ett \Ett \Ftt \oneb_{\mu}\la \mu+1 \ra, \Ett\Pi^{\mu}\oneb_{\mu} ) \cong \k.$$
\end{lem}

\begin{proof}
Moving the $\Ftt$ past the $\Ett$'s using the covering ${\mathfrak{sl}}_2$ relation \eqref{eqn-covering-sl2-relation} we get
\begin{align*}
 & \Hom(\Ett \Ett \Ftt\oneb_{\mu}\la \mu+1 \ra, \Ett\Pi^{\mu} \oneb_{\mu} )  \\
 &\qquad \cong \Hom( \Ftt\Pi \Ett \Pi\Ett \oneb_{\mu}\la \mu+1 \ra, \Pi^{\mu}  \Ett \oneb_{\mu} )
 \bigoplus_{k=0}^{\mu+1} \Hom(\Pi^{k+1} \Ett \oneb_{\mu}\la 2(\mu+1)-2k\ra, \Ett \Pi^{\mu} \oneb_{\mu})
  \\ \nn  & \qquad \qquad
  \bigoplus_{k=0}^{\mu-1} \Hom(\Ett\Pi^{k} \oneb_{\mu} \la 2\mu-2k\ra,  \Ett\Pi^{\mu} \oneb_{\mu} ).
\end{align*}
By Lemma \ref{lem:E} all the terms in middle and last summations are zero except in the middle summation when $k=\mu+1$ in which case we get $\Hom(\Pi^{\mu+2}\Ett\oneb_{\mu},\Ett\Pi^{\mu}\oneb_{\mu}) \cong \End(\Pi^{\mu}\Ett \oneb_{\mu}) \cong \k$ spanned by the identity. The first summand is isomorphic to
\begin{equation*}\begin{split}
 \Hom(\Ftt \Ett \Ett \oneb_{\mu}\la \mu+1 \ra, \Pi^{\mu+2} \Ett \oneb_{\mu} ) &\cong \Hom(\Ett\Ett \oneb_{\mu} \la \mu+1 \ra, \Pi^{\mu+2}(\oneb_{\mu+2} \Ftt)^R  \Ett \oneb_{\mu} ) \nn \\
 &\cong \Hom(\Ett \Ett \oneb_{\mu}\la \mu +1\ra, \Pi^{\mu+2}\Ett \Ett \oneb_{\mu} \la -\mu-3 \ra)\nn \\
 &\cong \Hom(\Ett \Ett \oneb_{\mu}, \Pi^{\mu+2}\Ett \Ett \oneb_{\mu} \la -2\mu -4 \ra)\nn
\end{split}\end{equation*}
which vanishes by Lemma~\ref{lem:EE} since $\mu \geq \l \ge 0$.
\end{proof}

% - - - - - - - - - - - - - - - - - - - - - - - - - - - - - - - - - - - - - - - -
%
%
\subsubsection{Induced maps}
%
% - - - - - - - - - - - - - - - - - - - - - - - - - - - - - - - - - - - - - - - -
%

The data of a strong supercategorical action formally determines properties of various 2-morphisms. Recall \cite{EKL} that in the Karoubi envelope of the odd nilHecke algebra we define $\Ett^{(2)}\onebl:= (\Ett^2\onebl \la 1 \ra, e_2)$  where
\[
e_2:= \xy (0,0)*{
\begin{tikzpicture}[scale=0.6]
  \draw[thick, ->] (-0.5,0) .. controls (-0.5,0.75) and (0.5,0.75) .. (0.5,1.5)
      node[pos=0.5, shape=coordinate](X){}
      node[pos=0.2, shape=coordinate](dot){};
  \draw[thick, ->] (0.5,0) .. controls (0.5,0.75) and (-0.5,0.75) .. (-0.5,1.5);
  \draw[color=blue,  thick, dashed] (X) .. controls++(-.5,0) and ++(-.65,.3) .. (dot);
  \draw[line width=0mm] (-0.5,0) .. controls (-0.5,0.75) and (0.5,0.75) .. (0.5,1.5)
      node[pos=0.2]{\tikz \draw[fill=black] circle (0.45ex);};
\end{tikzpicture} };
\endxy
\]
is idempotent. From the definition of the thick calculus in \cite{EKL} we have the following isomorphism in the Karoubi envelope:
\[
\xy
 (-25,0)*+{\Ett^{(2)}\onebl \la 1 \ra}="L";
 (25,0)*+{\Pi\Ett^{(2)}\oneb \la -1 \ra}="R";
 (0,20)*+{(\Ett^2\oneb, \Id)}="T";
 (0,-20)*+{(\Ett^2\oneb, \Id)}="B";
 (0,0)*{\bigoplus};
 {\ar^{\begin{tikzpicture}[scale=0.6]
  \draw[thick, ->] (-0.5,0) .. controls (-0.5,0.75) and (0.5,0.75) .. (0.5,1.5)
      node[pos=0.5, shape=coordinate](X){}
      node[pos=0.2, shape=coordinate](dot){};
  \draw[thick, ->] (0.5,0) .. controls (0.5,0.75) and (-0.5,0.75) .. (-0.5,1.5);
  \draw[color=blue,  thick, dashed] (X) .. controls++(-.65,.2) and ++(-.65,.3) .. (dot);
  \draw[line width=0mm] (-0.5,0) .. controls (-0.5,0.75) and (0.5,0.75) .. (0.5,1.5)
      node[pos=0.2]{\tikz \draw[fill=black] circle (0.4ex);};
\end{tikzpicture}} "L";"T"};
 {\ar_{\begin{tikzpicture}[scale=0.6]
  \draw[thick, ->] (-0.5,0) .. controls (-0.5,0.75) and (0.5,0.75) .. (0.5,1.5)
   %% add a label in the middle of the crossing
      node[pos=0.5, shape=coordinate](X){}
      node[pos=0.75, shape=coordinate](dot){}
      node[pos=0.2, shape=coordinate](bldot){};
  \draw[thick, ->] (0.5,0) .. controls (0.5,0.75) and (-0.5,0.75) .. (-0.5,1.5);
  \draw[color=blue,  thick, dashed] (X) .. controls++(-.5,00.1) and ++(-1.95,.35) .. (dot);
  \draw[color=blue,  thick, dashed] (bldot) to[out=100, in=90] (-1,0);
  %% Draw the bullet last so it shows up on top, use an invisible line
    \draw[line width=0mm] (-0.5,0) .. controls (-0.5,0.75) and (0.5,0.75) .. (0.5,1.5)
      node[pos=0.75]{\tikz \draw[fill=black] circle (0.4ex);}
      node[pos=0.2]{\tikz \draw[fill=black] circle (0.4ex);};
\end{tikzpicture}} "R";"T"};
  {\ar_{\begin{tikzpicture}[scale=0.6]
  \draw[thick, ->] (-0.5,0) .. controls (-0.5,0.75) and (0.5,0.75) .. (0.5,1.5)
      node[pos=0.5, shape=coordinate](X){};
  \draw[thick, ->] (0.5,0) .. controls (0.5,0.75) and (-0.5,0.75) .. (-0.5,1.5);
  \draw[color=blue,  thick, dashed] (X) to [out=190,in=-90](-1,1.5) ;
\end{tikzpicture}
  } "B";"R"};
 {\ar^{\begin{tikzpicture}[scale=0.6]
  \draw[thick, ->] (-0.5,0) .. controls (-0.5,0.75) and (0.5,0.75) .. (0.5,1.5)
      node[pos=0.5, shape=coordinate](X){}
      node[pos=0.2, shape=coordinate](dot){};
  \draw[thick, ->] (0.5,0) .. controls (0.5,0.75) and (-0.5,0.75) .. (-0.5,1.5);
  \draw[color=blue,  thick, dashed] (X) .. controls++(-.65,.2) and ++(-.65,.3) .. (dot);
  \draw[line width=0mm] (-0.5,0) .. controls (-0.5,0.75) and (0.5,0.75) .. (0.5,1.5)
      node[pos=0.2]{\tikz \draw[fill=black] circle (0.4ex);};
\end{tikzpicture}} "B";"L"};
\endxy
\]

\begin{lem} \label{lem_niliso}
The map  $\sUupdot \sUup: \Ett \Ett \onebl \rightarrow \Pi\Ett \Ett \onebl \la 2 \ra$ induces an isomorphism $\phi$
 \[
 \xy
  (-5,20)*+{\Pi\Ett^{(2)} \onebl \la -1 \ra}="t1";
  (-5,10)*+{\Ett^{(2)} \onebl \la 1 \ra}="t2";
  (-5,15)*{\oplus};
  (35,10)*+{\Pi^2\Ett^{(2)} \onebl \la 1 \ra}="r1";
  (35,00)*+{\Pi\Ett^{(2)} \onebl \la 3 \ra }="r2";
   (35,5)*{\oplus};
   {\ar^{\phi} "t2"; "r1"};
 \endxy
 \]
 on the $\Ett^{(2)}\onebl$ summand.
\end{lem}

\begin{proof}
The claim follows immediately from the axioms of the odd nilHecke algebra.
\end{proof}

The following lemma is the key technical result that allows us to define the other adjunction map.

\begin{lem}\label{lemXind} If $\l > 1$, then the map $\sUupdot \sUdown: \Ett \Ftt \onebl \rightarrow  \Pi\Ett \Ftt \onebl \la 2 \ra$ induces isomorphisms $\phi_k$ shown below for all $0 \leq k \leq \lambda-2$. Similarly, if $\lambda < -1$, then $\sUdown \sUupdot: \Ftt \Ett \onebl \rightarrow\Ftt \Pi\Ett \onebl \la 2 \ra$ induces isomorphisms $\psi_k$ shown below for all $0 \leq k \leq -\lambda-2$.
 \[
  \xy
  (-5,20)*+{\Ftt\Pi\Ett\onebl}="t1";
  (-5,10)*+{\Pi^{\lambda-1}\onebl\la -\l +1 \ra }="t2";
  (-5,0)*+{\Pi^{\lambda-2}\onebl\la -\l +3\ra }="t3";
  (-5,-10)*+{\Pi^{\lambda-3}\onebl\la -\l +5 \ra }="t4";
  (-5,-30)*+{\Pi\onebl\la \l -3 \ra }="t5";
   (-5,-40)*+{\onebl\la \l -1 \ra }="t6"; (-5,-35)*{\oplus};
  (-5,15)*{\oplus}; (-5,5)*{\oplus};  (-5,-5)*{\oplus};  (-5,-15)*{\oplus}; (-5,-25)*{\oplus};
   (-5,-18)*{\vdots};
  (35,10)*+{\Pi\Ftt\Pi\Ett\onebl\la 2\ra}="r1";
  (35,00)*+{\Pi^{\lambda}\onebl\la -\l +3 \ra }="r2";
  (35,-10)*+{\Pi^{\lambda-1}\onebl\la -\l +5\ra }="r3";
  (35,-20)*+{\Pi^{\lambda-3}\onebl\la -\l +5 \ra }="r4";
  (35,-40)*+{\Pi^2\onebl\la \l -1 \ra }="r5";
  (35,-50)*+{\Pi\onebl\la \l +1 \ra }="r6"; (35,-45)*{\oplus};
  (35,5)*{\oplus}; (25,-5)*{\oplus};  (35,-15)*{\oplus};  (35,-25)*{\oplus}; (35,-35)*{\oplus};
   (35,-28)*{\vdots};
   {\ar^{\phi_{\l-2}} "t3"; "r2"}; {\ar^{\phi_{\l-3}} "t4"; "r3"}; {\ar^{\phi_0} "t6"; "r5"};
 \endxy
 \qquad \quad
  \xy
  (-5,20)*+{\Ett\Pi\Ftt\onebl}="t1";
  (-5,10)*+{\onebl\la \l +1 \ra }="t2";
  (-5,0)*+{\Pi\onebl\la \l +3\ra }="t3";
  (-5,-10)*+{\Pi^2\onebl\la \l +5 \ra }="t4";
  (-5,-30)*+{\Pi^{\lambda-2}\onebl\la -\l -3 \ra }="t5";
   (-5,-40)*+{\Pi^{\lambda+1}\onebl\la -\l -1 \ra }="t6"; (-5,-35)*{\oplus};
  (-5,15)*{\oplus}; (-5,5)*{\oplus};  (-5,-5)*{\oplus};  (-5,-15)*{\oplus}; (-5,-25)*{\oplus};
   (-5,-18)*{\vdots};
  (35,10)*+{\Pi\Ett\Pi\Ftt\onebl\la 2\ra}="r1";
  (35,00)*+{\Pi\onebl\la \l +3 \ra }="r2";
  (35,-10)*+{\Pi^{2}\onebl\la \l +5\ra }="r3";
  (35,-20)*+{\Pi^{3}\onebl\la \l +5 \ra }="r4";
  (35,-40)*+{\Pi^{\lambda+1}\onebl\la -\l -1 \ra }="r5";
  (35,-50)*+{\Pi^{\lambda+2}\onebl\la -\l +1 \ra }="r6"; (35,-45)*{\oplus};
  (35,5)*{\oplus}; (25,-5)*{\oplus};  (35,-15)*{\oplus};  (35,-25)*{\oplus}; (35,-35)*{\oplus};
   (35,-28)*{\vdots};
   {\ar^{\psi_{-\l-2}} "t3"; "r2"}; {\ar^{\psi_{-\l-3}} "t4"; "r3"}; {\ar^{\psi_0} "t6"; "r5"};
 \endxy
 \]
Note that if $\lambda = -1, 0 ,1$ the statement above is vacuous (which is why we only consider $\lambda > 1$ and $\lambda < -1$).
\end{lem}

\begin{proof}
It is easy to see that the morphisms $\phi_k \maps \Pi^k \onebl \ads{\l-1-2k}\to\Pi^k \onebl \ads{\l-1-2k}$ coming from the map $\sUupdot \sUdown$ are the only possible isomorphisms on summands of $\Ett\Ftt\onebl$ isomorphic to $\Pi^k \onebl \ads{\l-1-2k}$.  This follows because such summands cannot exist in $\Ftt\Pi\Ett\onebl$ since the space of projections onto a summand
\begin{equation*}\begin{split}
\Hom(\Ftt\Pi\Ett\onebl, \Pi^k\onebl \ads{\l-1-2k}) &\cong
\Hom(\Pi\Ett\onebl, (\onebl\Ftt)^R\Pi^{k}\onebl \ads{\l-1-2k})\\
&\cong
\Hom(\Ett\onebl, \Pi^{k+1}\Ett\onebl \ads{-2-2k})
\end{split}\end{equation*}
is zero dimensional by Lemma~\ref{lem:E} for all $0 \leq k \leq \l-2$.

Temporarily assume that $\Ett\oneb_{\l-2}$ is indecomposable.
Observe that
\begin{align}
\Ett \Ftt \Ett \oneb_{\lambda-2} &\cong \Ftt \Pi\Ett \Ett \oneb_{\lambda-2} \bigoplus_{k=0}^{\lambda-1} \Pi^k\Ett \oneb_{\lambda-2} \la \lambda-1-2k \ra .
\end{align}
We argue that $\Ftt\Pi\Ett\Ett\oneb_{\l-2}$ contains no summands isomorphic to $\Pi^k\Ett\oneb_{\l-2}\ads{\l-1-2k}$ for $0 < k \leq \l-2$.  Using the cancellation property for indecomposable 1-morphisms the number of isomorphisms $\phi_k$ above is the same as the number of isomorphisms
$$\phi'_{k} \maps \Pi^k\Ett\oneb_{\l-2} \ads{\l-1-2k} \to \Pi^k\Ett\oneb_{\l-2} \ads{\l-1-2k}$$  induced by the map
\begin{equation} \label{eq_EFE}
\sUupdot\sUdown\sUup: \Ett \Ftt \Ett \oneb_{\lambda-2} \rightarrow \Pi \Ett \Ftt \Ett \oneb_{\lambda-2} \la 2 \ra.
\end{equation}

If $\Ftt\Pi\Ett\Ett\oneb_{\l-2}$  contained a summand isomorphic to $\Pi^k\Ett\oneb_{\l-2}\ads{\l-1-2k}$, then there would be a projection
\begin{equation*}\begin{split}
 \Hom(\Ftt\Pi\Ett\Ett\oneb_{\l-2}, \Ett\oneb_{\l-2}\ads{l})
 &\cong   \Hom(\Ett\Ett\oneb_{\l-2}, (\oneb_{\l}\Ftt)^R \Ett\oneb_{\l-2}\ads{\ell})\\
 &\cong    \Hom(\Ett\Ett\oneb_{\l-2}, \Ett\Ett\oneb_{\l-2}\ads{\ell-\l-1}),
\end{split}\end{equation*}
which can only exist if $\ell=\l-1$ by Lemma~\ref{lem:EE}.  By the remarks following Lemma~\ref{lem:FEtEF}, the inclusion map $\Ftt\Pi\Ett \Ett \oneb_{\l-2} \to \Ett\Ftt\Ett\oneb_{\l-2}$ induces the zero map on all summands $(\onebl\ads{\l-1-2k})\Ett\oneb_{\l-2}$. Thus \eqref{eq_EFE} cannot induce any isomorphisms $\phi'_k$ from the $\Ftt\Pi\Ett\Ett\oneb_{\l-2}$ summand.

We show the maps $\phi'_k$ are isomorphisms by relating them to different maps. By Lemma~\ref{lem_niliso}, the map
\begin{equation}\label{eq_EEF}
\sUupdot \sUup \sUdown: \Ett \Ett \Ftt \oneb_{\lambda-2} \rightarrow \Pi\Ett \Ett \Ftt \oneb_{\lambda-2} \la 2 \ra
\end{equation}
induces an isomorphism on the summand of the form $\Ett^{(2)}\Ftt\oneb_{\lambda-2}\la 1\ra$.  But by repeatedly applying \eqref{eq:EF-rel} and using the unique decomposition property it follows that for $\lambda-2 \geq 0$ there is an isomorphism
\begin{equation}
\Ett^{(2)} \Ftt \oneb_{\lambda-2}\ads{1} \cong \Ftt \Ett^{(2)} \oneb_{\lambda-2} \ads{1}
\bigoplus_{k=0}^{\lambda-2} \Pi^k \Ett \oneb_{\lambda-2} \la \lambda-1-2k \ra
\end{equation}
so that the number of isomorphisms
$$\phi''_k \maps \Pi^k \Ett \oneb_{\lambda-2} \ads{\lambda-1-2k} \to
\Pi^k \Ett \oneb_{\lambda-2} \ads{\lambda-1-2k}
$$
induced by \eqref{eq_EEF} is at least one for each $0 \leq k \leq \l-2$.

On the other hand, by decomposing $\Ett\Ftt\oneb_{\lambda-2}$, the map in \eqref{eq_EEF} induces the map
\[
\sUupdot\sUdown\sUup \bigoplus_{i=0}^{\lambda-3} \sUupdot: \Ett \Ftt \Pi\Ett \oneb_{\lambda-2} \bigoplus_{k=0}^{\lambda-3} \Pi^k\Ett \oneb_{\lambda-2} \la \lambda-3-2k \ra \rightarrow \Pi\Ett \Ftt \Pi\Ett \oneb_{\lambda-2} \la 2 \ra \bigoplus_{k=0}^{\lambda-3} \Pi^{k+1}\Ett \oneb_{\lambda-2} \la \lambda-5-2k \ra.
\]
For each $0 \leq k \leq \lambda-2$ the map
\[
 \sUupdot: \Pi^k\Ett \oneb_{\lambda-2} \la \lambda-3-2k \ra \rightarrow \Pi^{k+1}\Ett \oneb_{\lambda-2} \la \lambda-5-2k \ra
\]
is clearly never an isomorphism since the parities and gradings do not match.  Hence the only contribution to the number of isomorphisms $\phi''_k$ comes from the map \eqref{eq_EFE}, showing that $\phi''_k$ is an isomorphism if and only if $\phi'_k$ is an isomorphism.

If $\Ett\oneb_{\l-2}$ is not indecomposable then suppose that $X$ is an indecomposable summand of $\Ett\oneb_{\l-2}$ with multiplicity $m$.  The same argument above shows that the number of isomorphisms $\phi_k$ is equal to $m$ times the number of isomorphisms induced on summands of $X$ by $\phi'_k$. The arguments above show this number must be at least $m$.
\end{proof}

From condition \eqref{co:hom} in the definition of a strong supercategorical action it follows that negative degree diagrams with only dashed lines for endpoints must be equal to the zero 2-morphism.  In particular, the diagrams below are negative degree dotted bubbles.
\[
  \hackcenter{
 \begin{tikzpicture}[scale=0.8]
  \draw[thick, ->] (-0.5,0) .. controls (-0.5,0.8) and (0.5,0.8) .. (0.5,0)
      node[pos=0.5, shape=coordinate](X){};
  \draw[thick] (-0.5,0) .. controls (-0.5,-0.8) and (0.5,-0.8) .. (0.5,0)
      node[pos=0.1, shape=coordinate](Z){};
  \draw[color=blue, thick, double distance=1pt, dashed] (X) -- (0,1.25);
  \draw[color=blue, thick, double distance=1pt, dashed] (Z) to[out=180, in=90] (-1,1.25) ;
  \node[blue] at (-.6,1.1){$\scs  m$\;};
  \node[blue] at (.5,1.1){$\scs  \l-1$\;};
  \draw[line width=0mm] (-0.5,0) (-0.5,0) .. controls (-0.5,-0.8) and (0.5,-0.8) .. (0.5,0)
        node[pos=0.1]{\bbullet};
  \node at (1.3,0) {$\l$};
\end{tikzpicture} } \;\; = 0 \quad \text{for $0 \leq m <\l-1$},
\qquad
\qquad
\hackcenter{
\begin{tikzpicture}[scale=0.8]
  \draw[thick, ->] (0.5,0) .. controls (0.5,0.8) and (-0.5,0.8) .. (-0.5,0);
  \draw[thick] (0.5,0) .. controls (0.5,-0.8) and (-0.5,-0.8) .. (-0.5,0)
      node[pos=0.5, shape=coordinate](X){}
      node[pos=0.2, shape=coordinate](Y){};
 %% Draw double blue curvy line
  \draw[color=blue, thick, double distance=1pt, dashed] (X) -- (0, -1.25)
         node[pos=0.85,right]{$\scs -\l-1$\;};
   \draw[color=blue, thick, double distance=1pt, dashed] (Y) to[out=00, in=-90] (1,1) ;
   \node[blue] at (1.45,0.8){$\scs m$\;};
  \draw[line width=0mm] (0.5,0) .. controls (0.5,-0.8) and (-0.5,-0.8) .. (-0.5,0)
      node[pos=0.2]{\bbullet};
  \node at (-1,.3) {$\lambda$};
\end{tikzpicture} }
  \;\; = 0 \quad \text{for $0 \leq m <-\l-1$}.
\]
As explained in Section~\ref{subsec:consequences}, the $\mf{sl}_2$-relations formally determine a map $\Pi^{\l+1}\onebl\la -\l-1\ra \to \Ftt\Ett\onebl$ when $\l<0$.  It more convenient to work with a related map $\onebl\la -\l-1\ra \to \Ftt\Pi^{\l+1}\Ett\onebl$ defined as follows:
\[
\hackcenter{
\begin{tikzpicture}[scale=0.8]
  \draw[thick, <-] (0.5,0) .. controls ++(0,-0.8) and ++(-0,-0.8) .. (-0.5,0)
      node[pos=0.5, shape=coordinate](X){};
  \draw[color=blue, thick, double distance=1pt, dashed] (X) -- (0, 0)
         node[pos=0.95,above]{$\scs \l+1$\;};
  \node at (-1,-.6) {$\lambda$};
\end{tikzpicture} }
\quad := \quad
\xy (0,-2)*{
\begin{tikzpicture}[scale=0.8]
  \draw[thick, <-] (0.5,0) .. controls (0.5,-0.8) and (-0.5,-0.8) .. (-0.5,0)
      node[pos=0.5, shape=coordinate](X){};
  \draw[color=blue, thick, double distance=1pt, dashed]
    (X).. controls ++(.1,-.5) and ++(0,-.5) .. (-.5, -.5) ..
     controls ++(0,.3) and ++(0,-.3) .. (0,0);
  \node at (-1,-.8) {$\lambda$};
  \node[blue] at (.5,-.9) {$\scs \lambda+1$};
\end{tikzpicture}}; \endxy
\]
The next Corollary shows that degree zero dotted bubbles are non-zero multiples of the identity.

\begin{cor} \label{cor:degz-bubbles}
The maps
\begin{align}
\xy (0,5)*{
\begin{tikzpicture}[scale=0.8]
  \draw[thick, ->] (-0.5,0) .. controls (-0.5,0.8) and (0.5,0.8) .. (0.5,0)
      node[pos=0.5, shape=coordinate](X){}
      node[pos=0.1, shape=coordinate](Y){};
  \draw[thick] (-0.5,0) .. controls (-0.5,-0.8) and (0.5,-0.8) .. (0.5,0);
 %% Draw double blue curvy line
  \draw[color=blue, thick, double distance=1pt, dashed] (X) .. controls++(0,.65) and ++(-.65,.3) .. (Y)
         node[pos=0.15,right]{$\scs \l-1$\;};
 %% Draw the bullet last so it comes out on top
  \draw[line width=0mm] (-0.5,0) .. controls (-0.5,0.8) and (0.5,0.8) .. (0.5,0)
     node[pos=0.1]{\bbullet};
  \node at (1,0.5) {$\lambda$};
\end{tikzpicture}
 };
 \endxy
  &\maps \onebl \to \onebl  \qquad \text{for $\l >0$,}
& \qquad \qquad
 \hackcenter{
\begin{tikzpicture}[scale=0.8]
  \draw[thick, ->] (0.5,0) .. controls (0.5,0.8) and (-0.5,0.8) .. (-0.5,0);
  \draw[thick] (0.5,0) .. controls (0.5,-0.8) and (-0.5,-0.8) .. (-0.5,0)
      node[pos=0.5, shape=coordinate](X){}
      node[pos=0.1, shape=coordinate](Y){};
 %% Draw double blue curvy line
  \draw[color=blue, thick, double distance=1pt, dashed] (X) .. controls++(-.1,1) and ++(-.2,.4) .. (Y)
         node[pos=0.9,right]{$\scs -\l-1$\;};
 %% Draw the bullet last so it comes out on top
  \draw[line width=0mm
  ] (0.5,0) .. controls (0.5,-0.8) and (-0.5,-0.8) .. (-0.5,0)
     node[pos=0.1]{\bbullet};
  \node at (1,0.6) {$\lambda$};
\end{tikzpicture} }&\maps \onebl \to \onebl  \qquad \text{for $\l <0$,}
\end{align}
are both equal to non-zero multiples of $\oneb_{\onebl}$.
\end{cor}

\begin{proof}
As usual, we prove the case $\l > 0$ (the case $\l < 0$ follows similarly).

By construction, the map $\Ucupl\maps \onebl \rightarrow \Ett \Ftt \onebl \la -\l+1 \ra$ is an isomorphism between the summand $\onebl$ on the left and the corresponding $\onebl$ on the right hand side. Then by Lemma \ref{lemXind}, applying $^{\textcolor[rgb]{0.00,0.00,1.00}{\lambda-1}}\Uupdots\Udown: \Ett \Ftt \onebl \la -\l+1 \ra \rightarrow \Pi^{\l-1}\Ett \Ftt \onebl \la \l-1 \ra$ induces an isomorphism between the summands $\onebl$ on either side.

Finally, again by construction, the map $\Ucapr \maps \Ett \Ftt \onebl \la \l-1 \ra \rightarrow \Pi^{\l-1}\onebl$ is an isomorphism between the summand $\Pi^{\l-1}\onebl$ on the right side and the corresponding top degree summand  $\Pi^{\l-1}\onebl$ on the left hand side. Thus the composition
\[\onebl \xrightarrow{ \Ucupl} \Ett \Ftt  \onebl \la -\l+1 \ra \xrightarrow{\textcolor[rgb]{0.00,0.00,1.00}{\lambda-1} \Uupdots \Udown} \Pi^{\l-1}\Ett \Ftt \onebl \la \l-1 \ra \xrightarrow{
  \begin{tikzpicture}[scale=0.5]
  \draw[semithick, ->] (-0.5,0) .. controls (-0.5,0.65) and (0.5,0.65) .. (0.5,0)
      node[pos=0.5, shape=coordinate](X){};
  \draw[color=blue, thick, double distance=1pt, dashed] (X) .. controls ++(.1,.4) and ++(0,1).. (-1,0);
\end{tikzpicture}
} \onebl
 \]
is an isomorphism and this completes the proof since $\Hom(\onebl, \onebl) \cong \k$.
\end{proof}

% - - - - - - - - - - - - - - - - - - - - - - - - - - - - - - - - - - - - - - - -
%
%
\subsubsection{Defining the last adjunction}
%
% - - - - - - - - - - - - - - - - - - - - - - - - - - - - - - - - - - - - - - - -
%

We are still assuming, for simplicity, that $\l \ge 0$. Recall that we defined all but one of the adjunction maps, namely the 2-morphism $\Ucupr$ in (\ref{eq:B}). This map cannot be defined formally by adjunction.  And since $\l \ge 0$, the 1-morphism $\Ftt \Ett \onebl$ is indecomposable so one cannot define it as an inclusion. To overcome this problem we construct this 2-morphism by defining an up-down crossing and composing it with the $\Ucupl$.

Let $\Ucrossr$ denote the map $ \Ett\Ftt \onebl \cong\Ftt \Pi\Ett \onebl \oplus_{[\l]} \onebl
\rightarrow\Ftt \Pi\Ett\onebl $ which projects onto the  $\Ftt\Pi \Ett \onebl$ summand of $\Ett \Ftt \onebl$. Note that this map is not unique because there exist non-zero maps $\oplus_{[\l]} \onebl \rightarrow\Ftt \Ett \onebl$ but this ambiguity will not matter.

Then we define a map $\Ucuprm \maps \onebl \to \Ftt\Pi^{\l+1}\Ett\onebl\la \l+1\ra$ as the composite
\begin{equation}\label{eq:C}
\xy
 (-25,0)*+{\scs \onebl}="1";
 (-5,0)*+{\scs\Ett \Ftt \onebl \la -\l+1 \ra }="2";
 (25,0)*+{\scs \Pi^{\l}\Ett \Ftt \onebl \la \l +1 \ra}="3";
 (60,0)*+{\scs \Pi^{\l}\Ftt \Pi \Ett \onebl\la \l +1\ra}="4";
 (105,0)*+{\scs \Ftt \Pi^{\l+1} \Ett \onebl\la \l +1\ra.}="5";
  {\ar^-{ \vcenter{\xy (0,0)*{\begin{tikzpicture}[scale=0.5]
  \draw[color=blue, thick, double distance=1pt, dashed] (-0.5,0) .. controls (-0.5,0.5) and (0.5,0.5) .. (0.5,1);
  \draw[semithick, <-] (0.5,0) .. controls (0.5,0.5) and (-0.5,0.5) .. (-0.5,1);
\end{tikzpicture}}; \endxy}\;
\sUupbb
\sUup} "4"; "5"};
 {\ar^-{\xy (0,0)*{\sUupb\Ucrossr}; \endxy\;} "3"; "4"};
 {\ar^-{\l\Uupdots\Udown} "2"; "3"};
 {\ar^-{\Ucupl} "1"; "2"};
\endxy
\end{equation}

\begin{rem}
In the even case, it is implicit in the ``curl-relations" \cite[Proposition 5.4]{Lau1} that given one of the adjunction maps, the other adjunction map is fixed as above.  This approach to the second adjunctions later appears in \cite[Section 4.1.4]{Rou2}.   This form of the second adjunction was verified in the context of cyclotomic quotients by Kashiwara~\cite{Kash}.
\end{rem}

\begin{lem}\label{lem:A'} If $\mu \ge \l$, then the two maps
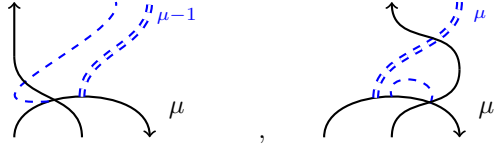
\[
\begin{tikzpicture}[scale=0.9]
  \draw[thick, ->] (-0.5,0) .. controls (-0.5,0.8) and (1.5,0.8) .. (1.5,0)
      node[pos=0.5, shape=coordinate](L){}
      node[pos=0.33, shape=coordinate](I){};
\draw[color=blue, thick, double distance=1pt, dashed] (L) to [out=90, in=-90] (1.5,2);
\draw[color=blue, thick,  dashed] (I) .. controls++(-1.6,-.1) and ++(.15,-.5) .. (1,2);
\draw[thick] (.5,0) .. controls (.5,.7) and (-.5,.5).. (-.5,1.2);
\draw[thick, ->] (-.5,1.2) -- (-.5,2);
 %% Draw a bunch of labels
   \node[blue] at (1.9,1.8){$\scs \mu-1$};
    \node at (1.9,0.4) {$\mu$};
\end{tikzpicture}
\qquad, \qquad
\begin{tikzpicture}[scale=0.9]
  \draw[thick, ->] (-0.5,0) .. controls (-0.5,0.8) and (1.5,0.8) .. (1.5,0)
      node[pos=0.42, shape=coordinate](L){}
      node[pos=0.5, shape=coordinate](M){}
      node[pos=0.7, shape=coordinate](R){};
  \draw[color=blue, thick, dashed] (R) .. controls ++(.2,.4) and ++(-.15,.4) .. (M);
\draw[color=blue, thick, double distance=1pt, dashed] (L) to [out=90, in=-90] (1.5,2);
\draw[thick] (.5,0) .. controls (.5,.7) and (1.5,.3).. (1.5,1);
\draw[thick, ->] (1.5,1) .. controls (1.5,1.7) and (.5,1.3).. (.5,2);
 %% Draw a bunch of labels
     %%
   \node[blue] at (1.8,1.8){$\scs \mu$};
    \node at (1.9,0.4) {$\mu$};
\end{tikzpicture}
\]
are non-zero multiples of a 2-morphism spanning the 1-dimensional 2-hom-space $ \Hom(\Ett \Ett \Ftt \oneb_{\mu} \la \mu+1 \ra, \Ett \Pi^{\mu} \oneb_{\mu})$ from Lemma~\ref{lem:EEF}.
\end{lem}

\begin{proof}
We just need to show that both maps are non-zero. Suppose the map on the left is zero. Adding a dot at the top of the middle upward pointing strand and sliding it past the crossing using the odd nilHecke relation~\eqref{eq:onil-dot} one gets two terms.
\[
\xy (0,0)*{
\begin{tikzpicture}
  \draw[thick, ->] (-0.5,0) .. controls (-0.5,0.8) and (1.5,0.8) .. (1.5,0)
      node[pos=0.5, shape=coordinate](L){}
      node[pos=0.33, shape=coordinate](I){};
\draw[color=blue, thick, double distance=1pt, dashed] (L) to [out=90, in=-90] (1.5,2);
\draw[color=blue, thick,  dashed] (I) .. controls++(-1.6,-.2) and ++(.15,-.5) .. (1,2);
\draw[thick] (.5,0) .. controls (.5,.7) and (-.5,.5).. (-.5,1);
\draw[thick, ->] (-.5,1) -- (-.5,2)
    node[pos=0.32, shape=coordinate](DOT){};
 \draw[color=blue, thick, dashed] (DOT) to [out=180, in=-90] (-1,2);
 \draw[, line width=0mm] (-.5,1) -- (-.5,2)
     node[pos=0.32](){\bbullet};
 %% Draw a bunch of labels
     %%
   \node[blue] at (1.8,1.8){$\scs \mu-1$};
    \node at (1.9,0.4) {$\mu$};
\end{tikzpicture}};
\endxy
\quad
 =
\quad
\xy
(0,0)*{
\begin{tikzpicture}
  \draw[thick, ->] (-0.5,0) .. controls (-0.5,0.8) and (1.5,0.8) .. (1.5,0)
      node[pos=0.5, shape=coordinate](L){}
      node[pos=0.33, shape=coordinate](I){};
\draw[color=blue, thick, double distance=1pt, dashed] (L) to [out=90, in=-90] (1.5,2);
\draw[color=blue, thick,  dashed] (I) .. controls++(-1.5,-.2) and ++(.15,-.5) .. (1,2);
\draw[thick] (.5,0) .. controls (.5,.7) and (-.5,.5).. (-.5,1)
    node[pos=0.15, shape=coordinate](DOT){};
\draw[thick, ->] (-.5,1) -- (-.5,2);
 \draw[color=blue, thick, dashed] (DOT) .. controls ++(-1.7,-0.2) and ++(0,-1).. (-1,2);
 \draw[line width=0mm] (.5,0) .. controls (.5,.7) and (-.5,.5).. (-.5,1)
     node[pos=0.15](){\bbullet};
 %% Draw a bunch of labels
     %%
   \node[blue] at (1.8,1.8){$\scs \mu-1$};
    \node at (1.9,0.4) {$\mu$};
\end{tikzpicture}}; \endxy
\quad + \quad
\xy
(0,0)*{
\begin{tikzpicture}
  \draw[thick, ->] (0.5,0) .. controls (0.5,0.8) and (1.5,0.8) .. (1.5,0)
      node[pos=0.5, shape=coordinate](L){};
 \draw[color=blue, thick, double distance=1pt, dashed] (L) to [out=90, in=-90] (1.5,2);
  \draw[color=blue, thick,  dashed] (-1,2) to [out=-90, in=-90] (1,2);
 \draw[thick, ->] (-.5,0) -- (-.5,2);
 %% Draw a bunch of labels
     %%
   \node[blue] at (1.8,1.8){$\scs \mu-1$};
    \node at (1.9,0.4) {$\mu$};
\end{tikzpicture}}; \endxy
\]
One term is again zero (because it is the composition of a dot and the original map) and the other is just the composite of several isomorphisms and the adjoint map which cannot be zero because it is the projection of $\Ett \oneb_{\mu}$ out of the highest  degree summand inside $\Ett \Ett \Ftt \oneb_{\mu} \la \mu+1 \ra$. Thus the map is non-zero.

On the other hand, the map on the right is the composition
\[
\xy
(40,0)*+{\Pi^{\mu}\Ett \oneb_{\mu}}="1";
(0,0)*+{\Ett \Ftt\Pi \Ett \oneb_{\mu} \la \mu+1 \ra}="2";
(-50,0)*+{ \Ett \Ett \Ftt \oneb_{\mu} \la \mu+1 \ra}="3";
 {\ar^-{\begin{tikzpicture}[scale=0.5]
  \draw[semithick, ->] (-0.5,0) .. controls (-0.5,0.65) and (0.5,0.65) .. (0.5,0)
      node[pos=0.5, shape=coordinate](X){};
  \draw[color=blue, thick, double distance=1pt, dashed] (X) -- (0,1);
  \draw[semithick,->] (1.5,0) -- (1.5,1);
  \draw[thick, dashed, blue] (1,0) -- (1,1);
\end{tikzpicture}} "2";"1"};
 {\ar^-{\sUup\Ucrossr} "3";"2"};
\endxy
\]
together with several isomorphisms.  The first map in the above composite is induced by the projection of $\Ett \Ftt \oneb_{\mu}\cong\Ftt\Pi\Ett \oneb_{\mu} \oplus_{[\mu]} \oneb_{\mu}$ into $\Ftt \Pi\Ett \oneb_{\mu} $ and the second map is induced by the is the projection of $\Ett \oneb_{\mu}$ out of the top degree summand of $\Ett \oneb_{\mu}$ in $(\Ett \Ftt)\Ett \oneb_{\mu}$.  Since the domain of the second projection is in the image of the first, this map is also non-zero.
\end{proof}

\begin{prop} \label{prop:adjoints}
The 2-morphisms
\begin{equation}
\Ucapr: \Ett \Ftt \onebltwo \la \lambda+1 \ra \rightarrow \Pi^{\l+1}\onebltwo, \hspace{1.0cm}
\Ucupl: \onebltwo \la \lambda+1 \ra \rightarrow  \Ett \Ftt \onebltwo,
\end{equation}
\begin{equation}
\Ucapl: \Ftt \Ett \onebl \rightarrow \onebl \la \lambda+1 \ra, \hspace{1.0cm}
\Ucuprm:
\onebl \rightarrow\Ftt\Pi^{\lambda+1} \Ett \onebl \la \lambda+1 \ra
\end{equation}
defined above satisfy the adjunction relations (\ref{eq_biadjoint1}) and (\ref{eq_biadjoint2}) up to non-zero multiples.
\end{prop}

\begin{proof}
We prove one of the adjunction axioms (the second one follows formally).  Since $\End(\Ett\onebl) \cong \k$ it suffices to show that the left side of (\ref{eq_biadjoint2}) is non-zero.  Now, $\Hom( \onebltwo\la \l+1 \ra,\Ett\Ftt\onebltwo) \cong \Hom(\Ett \onebl, \Ett \onebl) \cong \k$ by Lemma \ref{lem:E}. So the map $\Ucupl: \onebltwo \la \l+1 \ra \rightarrow \Ett \Ftt \onebltwo$ must be equal to the adjunction map (up to a multiple). Since $\Ucapl: \Ftt \Ett\onebl \rightarrow \onebl\la \l+1 \ra$ is defined to be the adjunction map their composition must be non-zero.

To show that the left side of (\ref{eq_biadjoint1}) is non-zero we use the definition of the cup given in \eqref{eq:C}. It follows from Lemma~\ref{lem:A'} that
\[
\hackcenter{
\begin{tikzpicture}[scale=0.7]
  \draw[thick] (-0.5,0) .. controls (-0.5,-0.8) and (0.5,-0.8) .. (0.5,0)
     node[pos=0.5, shape=coordinate](X){};
  \draw[thick,->] (-1.5,0) .. controls (-1.5,.8) and (-.5,.8) .. (-.5,0)
      node[pos=0.5, shape=coordinate](Y){};
  \draw[thick, ->-=0.15] (.5,0) -- (.5,1.25);
  \draw[thick, ->] (-1.5,-1.25) -- (-1.5,0);
  \draw[color=blue, thick, double distance=1pt, dashed]
   (X) .. controls ++(.2,1.5) and ++(0,1).. (Y);
    \node at (.7,-1) {$\l$};
    \node[blue] at (-1.4,1) {$\scs \l+1$};
\end{tikzpicture} }
\quad := \quad
 \hackcenter{\begin{tikzpicture}[scale=0.7]
  \draw[thick] (-0.5,0) .. controls (-0.5,0.5) and (0.5,0.5) .. (0.5,1)
      node[pos=0.5, shape=coordinate](X){};
  \draw[thick] (-1.5,1) .. controls (-1.5,1.6) and (-.5,1.6) .. (-.5,1)
      node[pos=0.45, shape=coordinate](LY){}
      node[pos=0.55, shape=coordinate](Y){};
  \draw[thick ] (-0.5,0) .. controls (-0.5,-0.5) and (0.5,-0.5) .. (0.5,0);
  \draw[color=blue, thick, double distance=1pt, dashed]
     (Y) .. controls ++(0,.5) and ++(.1,.6).. (-.2,1)
      .. controls ++(0,-.6) and ++(-1,.4) .. (-0.5,0);
  \draw[color=blue, thick, dashed]
     (LY) .. controls ++(-.1,1.2) and ++(.2,1).. (.2,1)
      .. controls ++(0,.-.3) and ++(0,.4) .. (X);
  \node at (-0.5,0) {\bbullet};
  \draw[thick, <-] (0.5,0) .. controls (0.5,0.5) and (-0.5,0.5) .. (-0.5,1);
  \draw[thick, ->-=0.3] (.5,1) -- (.5,2);
  \draw[thick, ->-=0.8] (-1.5,-1) -- (-1.5,1);
    \node[blue] at (-1,0){$\scs \l$};
    \node at (1.4,0.4) {$\l$};
\end{tikzpicture} }
\quad = \quad
\hackcenter{
\begin{tikzpicture}[scale=0.7]
  \draw[thick, ->] (-0.5,0) .. controls (-0.5,0.5) and (0.5,0.5) .. (0.5,1)
      node[pos=0.5, shape=coordinate](X){};
  \draw[thick] (0.5,0) .. controls (0.5,0.5) and (-0.5,0.5) .. (-0.5,1);
  \draw[thick] (0.5,1) .. controls (0.5,1.6) and (1.5,1.6) .. (1.5,1)
      node[pos=0.45, shape=coordinate](Y){}
      node[pos=0.56, shape=coordinate](RY){};
  \draw[thick] (0.5,0) .. controls (0.5,-.6) and (1.5,-.6) .. (1.5,0)
      node[pos=0.2, shape=coordinate](Z){};
  \draw[thick, ->-=-0.5] (1.5,1) -- (1.5,0);
  \draw[thick] (-0.5,-.75) -- (-0.5,0);
  \draw[thick, ->-=0.5] (-0.5,1) -- (-0.5,2.5);
  \draw[color=blue,  thick, dashed]
     (X) .. controls ++(-1.4,-0.1) and ++(0,-.4) .. (0,1.5)
       ..  controls ++(0,.5) and ++(0,.6) .. (Y);
  \draw[color=blue,  thick, dashed,double distance=1pt,]
    (Z) .. controls ++(-2,0.3) and ++(0,-.5) .. (-1,1.5)
     .. controls ++(0,1) and ++(.1,1.3) .. (RY);
   \node at (Z) {\bbullet};
   \node[blue] at (.25,0){$\scs \l$};
   \node at (2,0.4) {$\l$};
   \node at (-2,1) {$\kappa$};
\end{tikzpicture} }
\]
for some non-zero scalar $\kappa$.
Moving one of the dots through the crossing using the odd nilHecke relation~\eqref{eq:onil-dot} gives a sum of two diagrams.  One is zero since it is the composite of a dot and an endomorphism of $\Ett \onebl$ of degree $-2$.  The other is a scalar multiple of the identity on $\Ett\onebl$ and the non-zero degree zero bubble from Corollary \ref{cor:degz-bubbles}.
\end{proof}

Thus we get that \newline
\medskip
\noindent\fbox {
   \parbox{\linewidth}{
\[
(\Ett \onebl)^R \cong \onebl \Ftt\Pi^{\l+1} \la \l+1 \ra \text{ and } (\Ett \onebl)^L \cong \onebl \Ftt \la -\l-1 \ra,
\]
\[
 \left(\Ftt\onebl\right)^L =  \Pi^{\l-1}\onebl\Ett \la \l-1\ra\text{ and }  \left(\Ftt\onebl\right)^R =  \onebl\Ett \la -\l+1\ra,
\]
    }
}
\noindent which completes the induction step for the adjoint induction hypothesis~\ref{eq:ind_hyp}.

We show that these adjunction maps are unique up to a multiple.

\begin{cor}\label{cor:EFidhoms}
Given a strong supercategorical action of $\mf{sl}_2$ we have
$$\Hom(\Ett \Ftt \onebltwo \la \l+1 \ra, \Pi^{\l+1}\onebltwo) \cong \k, \hspace{1.0cm}
\Hom(\onebltwo \la \l+1 \ra, \Ett \Ftt \onebltwo) \cong \k,$$
$$\Hom(\Ftt \Ett \onebl, \onebl \la \l+1 \ra) \cong \k, \hspace{1.0cm}
\Hom(\onebl, \Pi^{\l+1}\Ftt \Ett \onebl \la \l+1 \ra) \cong \k.$$
\end{cor}
\begin{proof}
We calculate the first space (the other three are similar). We have
\begin{eqnarray*}
\Hom(\Ett \Ftt \onebltwo \la n+1 \ra, \Pi^{\l+1}\onebltwo)
&\cong& \Hom(\Ftt \onebltwo \la \l+1 \ra, \Pi^{\l+1}(\Ett \onebl)_R \onebltwo) \\
&\cong& \Hom(\Ftt \onebltwo, \Pi^{2(\l+1)}\Ftt \onebltwo) \\
&\cong& \Hom(\Ftt \onebltwo, \Ftt \onebltwo) \cong \k.
\end{eqnarray*}
\end{proof}

We summarize the results of this subsection with the following:

\begin{thm}\label{thm:lradj}
Given a strong supercategorical action of $\mf{sl}_2$ the left and right adjoints of $\Ett$ are isomorphic up to specified shifts and parity. More precisely, the units and counits of these adjunctions are given by the 2-morphisms below which are uniquely determined up to a scalar and induce the adjunction maps.
\begin{equation}\label{eq:A}
\Ucapr: \Ett \Ftt \onebltwo \la \lambda+1 \ra \rightarrow \Pi^{\l+1}\onebltwo \hspace{1.0cm}
\Ucupl: \onebltwo \la \lambda+1 \ra \rightarrow  \Ett \Ftt \onebltwo
\end{equation}
\begin{equation}\label{eq:B}
\Ucapl: \Ftt \Ett \onebl \rightarrow \onebl \la \lambda+1 \ra \hspace{1.0cm}
\Ucuprm:
\onebl \rightarrow\Ftt \Pi^{\lambda+1}\Ett \onebl \la \lambda+1 \ra.
\end{equation}
\end{thm}

% -------------------------------------------------------------------------------
%
\subsection{A specific form for $\mf{sl}_2$ isomorphisms}
%
% -------------------------------------------------------------------------------
At this point we know that both the left and right adjoint of $\Ftt$ is $\Ett$ (up to a specified shift and parity). Consequently we can prove the following.

\begin{lem}\label{lem:homs}
For any $n \in \Z$ we have
$$\Hom(\Ett \Pi\Ftt \onebl,\Ftt\Ett \onebl) \cong \k \cong \Hom(\Ftt \Pi \Ett \onebl, \Ett \Ftt \onebl).$$
\end{lem}
\begin{proof}
We prove that $\Hom(\Ett \Pi\Ftt \onebl,\Ftt\Ett \onebl) \cong \k$ (the other case follows similarly). One has
\begin{eqnarray*}
\Hom(\Ett \Pi\Ftt \onebl,\Ftt\Ett \onebl)
&\cong& \Hom(\Pi\Ftt(\Ett \onebl)^R \onebltwo, (\Ett \oneb_{\lambda-2})^R \Ftt\onebltwo) \\
&\cong& \Hom(\Pi\Ftt\Ftt\Pi^{\l+1} \onebltwo \la \l+1 \ra,\Ftt \Pi^{\l-1}\Ftt \onebltwo \la \l-1 \ra) \cong \k
\end{eqnarray*}
where the last isomorphism follows from (the adjoint of) Lemma~\ref{lem:EE} together with the parity isomorphisms.
\end{proof}

Subsequently, we denote these maps as
\begin{align} \label{eq:sideways}
 & \hackcenter{\begin{tikzpicture}[scale=0.6]
  \draw[semithick, <-] (-0.5,0) .. controls (-0.5,0.5) and (0.5,0.5) .. (0.5,1)
      node[pos=0.5, shape=coordinate](X){};
  \draw[semithick, ->] (0.5,0) .. controls (0.5,0.5) and (-0.5,0.5) .. (-0.5,1);
  \draw[color=blue,  thick, dashed] (X) to (0,0);
\end{tikzpicture}}\;\;
 \maps \Ftt \Pi\Ett \onebl \rightarrow \Ett \Ftt \onebl,
  & \hackcenter{\begin{tikzpicture}[scale=0.6]
  \draw[semithick, ->] (-0.5,0) .. controls (-0.5,0.5) and (0.5,0.5) .. (0.5,1)
      node[pos=0.5, shape=coordinate](X){};
    \draw[semithick, <-] (0.5,0) .. controls (0.5,0.5) and (-0.5,0.5) .. (-0.5,1);
  \draw[color=blue,  thick, dashed] (X) .. controls ++(.1,.5) and ++(0,.5) .. (-.5,.5)
  .. controls ++(0,-.3) and ++(0,.3) .. (0,0);
\end{tikzpicture}} \;\;
 \maps \Ett \Pi \Ftt \onebl \rightarrow\Ftt\Ett \onebl.
\end{align}
For the moment these maps are uniquely defined only up to a non-zero scalar.

\begin{cor} \label{cor:1} If $\l \ge 0$ then the map
\begin{equation}\label{eq:iso1}
\zeta\;\;:=\;\;
\hackcenter{\begin{tikzpicture}[scale=0.6]
  \draw[semithick, <-] (-0.5,0) .. controls (-0.5,0.5) and (0.5,0.5) .. (0.5,1)
      node[pos=0.5, shape=coordinate](X){};
  \draw[semithick, ->] (0.5,0) .. controls (0.5,0.5) and (-0.5,0.5) .. (-0.5,1);
  \draw[color=blue,  thick, dashed] (X) to (0,0);
\end{tikzpicture}}\;\; \bigoplus_{k=0}^{\lambda-1}
\hackcenter{\begin{tikzpicture}[scale=0.6]
  \draw[thick, ->-=0.15, ->] (0.5,.2) .. controls (0.6,-0.8) and (-0.6,-0.8) .. (-0.5,.2)
      node[pos=0.85, shape=coordinate](Y){};
  \draw[color=blue, thick, double distance=1pt, dashed]
   (Y) .. controls++(-.5,.2) and ++(0,.4) .. (-1,-1)
         node[pos=0.75,left]{$\scs k$};
  \draw[line width=0mm] (0.5,.2) .. controls (0.5,-0.8) and (-0.5,-0.8) .. (-0.5,.2)
     node[pos=0.85]{\tikz \draw[fill=black] circle (0.4ex);};
\end{tikzpicture} }:\Ftt\Pi \Ett \onebl \bigoplus_{k=0}^{\lambda-1} \Pi^k\onebl \la \l-1-2k \ra \rightarrow \Ett \Ftt \onebl
\end{equation}
is an isomorphism. Likewise, if $\l \le 0$ then the map
\begin{equation}\label{eq:iso2}
\zeta\;\;:=\;\;
\hackcenter{\begin{tikzpicture}[scale=0.6]
  \draw[semithick, ->] (-0.5,0) .. controls (-0.5,0.5) and (0.5,0.5) .. (0.5,1)
      node[pos=0.5, shape=coordinate](X){};
    \draw[semithick, <-] (0.5,0) .. controls (0.5,0.5) and (-0.5,0.5) .. (-0.5,1);
  \draw[color=blue,  thick, dashed] (X) .. controls ++(.1,.5) and ++(0,.5) .. (-.5,.5)
  .. controls ++(0,-.3) and ++(0,.3) .. (0,0);
\end{tikzpicture}} \;\; \bigoplus_{k=0}^{-\l-1}
\hackcenter{\begin{tikzpicture}[scale=0.6]
  \draw[thick, ->-=0.15, ->] (-0.7,.5) .. controls ++(-.1,-1) and ++(.1,-1) .. (0.7,.5)
      node[pos=0.85, shape=coordinate](Y){}
      node[pos=0.55, shape=coordinate](M){}
      node[pos=0.44, shape=coordinate](X){};
  \draw[color=blue, thick, double distance=1pt, dashed]
   (Y) .. controls++(-.5,.3) and ++(0,.5) .. (M)
         node[pos=0.15,above]{$\scs k$};
   \draw[color=blue, thick, double distance=1pt, dashed]
     (X) .. controls ++(0,.55) and ++(0,.55) ..
      (-.6,-.25) .. controls ++(0,-.3) and ++(0,.4) ..(0,-1);
   \node at (Y){\tikz \draw[fill=black] circle (0.4ex);};
\end{tikzpicture} }:\Ett\Pi\Ftt \onebl \bigoplus_{k=0}^{-\l-1}\Pi^{\l+1+k} \onebl \la -\l-1-2k \ra \rightarrow \Ftt \Ett \onebl
\end{equation}
is an isomorphism.
\end{cor}

\begin{proof}
We prove the case $\l \ge 0$ (the case $\l \le 0$ is proved similarly).

We know that
$$\Ett \Ftt \onebl \cong\Ftt \Pi\Ett \onebl \bigoplus_{k=0}^{\lambda-1}\Pi^k \onebl \la \l-1-2k \ra$$
and by Lemma \ref{lem:homs} the map $\Ucrossl $  must induce an isomorphism between the $\Ftt \Pi\Ett \onebl$ summands and must induce the zero map from the $\Ftt \Pi\Ett \onebl$ summand on the left to any summand $\Pi^k\onebl \la \l-1-2k \ra$ on the right.

It remains to show that $\bigoplus_{k=0}^{\lambda-1} \hackcenter{\begin{tikzpicture}[scale=0.6]
  \draw[thick, ->-=0.15, ->] (0.5,.2) .. controls (0.6,-0.8) and (-0.6,-0.8) .. (-0.5,.2)
      node[pos=0.85, shape=coordinate](Y){};
  \draw[color=blue, thick, double distance=1pt, dashed]
   (Y) .. controls++(-.5,.2) and ++(0,.4) .. (-1,-1)
         node[pos=0.75,left]{$\scs k$};
  \draw[line width=0mm] (0.5,.2) .. controls (0.5,-0.8) and (-0.5,-0.8) .. (-0.5,.2)
     node[pos=0.85]{\tikz \draw[fill=black] circle (0.4ex);};
\end{tikzpicture} }$ induces an isomorphism between the summands $\Pi^k\onebl \la \l-1-2k \ra$ on either side. Since $\Hom(\onebl, \onebl \la \ell \ra) = 0$ if $\ell < 0$ it follows that the induced map
$$\bigoplus_{k=0}^{\lambda-1}\Pi^k \onebl \la \l-1-2k \ra \rightarrow \bigoplus_{k=0}^{\lambda-1}\Pi^k  \onebl \la \l-1-2k \ra$$
is upper triangular (when expressed as a matrix). We show that the maps on the diagonal are isomorphisms between the summands $\Pi^k \onebl \la \l-1-2k \ra$ on either side.

Now, by construction, the map $\;\Ucupl: \onebl \ads{\l-1} \rightarrow \Ett \Ftt \onebl$ is an isomorphism onto the summand  $\onebl \ads{\l-1}$ on the right side. Consequently, by Lemma \ref{lemXind}, the composition
$$\onebl \la \l-1-2k \ra \xrightarrow{\Ucupl} \Ett \Ftt \onebl \la -2k \ra \xrightarrow{\textcolor[rgb]{0.00,0.00,1.00}{k}\Uupdots\Udown} \Pi^k\Ett \Ftt \onebl $$
must also induce an isomorphism between the summands $\onebl \la \l-1-2k \ra$ on either side. This proves that all the diagonal entries are isomorphisms so we are done.
\end{proof}

% -------------------------------------------------------------------------------
%
\subsection{Endomorphisms of $\Ett \onebl$}
%
% -------------------------------------------------------------------------------

In an arbitrary strong supercategorical action on a 2-category $\Cc$ there may be many additional 2-morphisms in $\Cc$ that are not composites of caps, cups, dots or crossings.  The Lemma below is an important technical result that limits the form of 2-endomorphisms of $\Ett \onebl$.

\begin{lem}\label{lem:main}
Suppose $\mu < |\l+2|$ (or $\mu=1$ and $\l=-1$) and $f \in \Hom^{2\mu}(\Ett \onebl, \Pi^{\mu}\Ett \onebl)$. If $\l \ge -1$ then $f$ is of the form
\begin{equation}\label{eq:main1}
\sum_{i=0}^{\mu} \;\;
\hackcenter{\begin{tikzpicture}[scale=0.7]
  \draw[thick, ->] (1,0) -- (1,2)
     node[pos=0.5, shape=coordinate](X){};
  \draw[color=blue, thick, double distance=1pt, dashed]
   (X) to [out=180, in=-90] (0,2);
    \draw[line width=0mm] (1,0) -- (1,2)
     node[pos=0.5](){\bbullet};
    %% Make a coupon filled with a given label place at coordinate (-2,0.75) name the node Fj
   \node[draw, thick, fill=blue!20,rounded corners=4pt,inner sep=3pt]
     (fi) at (-1,.5) {$f_i$};
   \draw[color=blue, thick, double distance=1pt, dashed] (fi) to  (-1,2);
    \node at (1.8,1.5) {$ \l$};
     \node[blue] at (.3,1.8) {$\scs i$};
      \node[blue] at (-1.7,1.8) {$\scs\mu-i$};
\end{tikzpicture}  }
\end{equation}
where  $f_i \in \Hom^{2\mu-2i}(\onebltwo, \Pi^{\mu-i}\onebltwo)$. Similarly, if $\l \le -1$ (or $\mu=2$ and $\l=-1$) then $f$ is of the form
\begin{equation}\label{eq:main2}
\sum_{i=0}^{\mu} \;\;
\hackcenter{\begin{tikzpicture}[scale=0.7]
  \draw[thick, ->] (1,-.25) -- (1,2)
     node[pos=0.7, shape=coordinate](X){};
  \draw[color=blue, thick, double distance=1pt, dashed]
   (X) to [out=180, in=-90] (0,2);
    \draw[line width=0mm] (1,-.25) -- (1,2)
     node[pos=0.7](){\bbullet};
    %% Make a coupon filled with a given label place at coordinate (-2,0.75) name the node Fj
   \node[draw, thick, fill=blue!20,rounded corners=4pt,inner sep=3pt]
     (fi) at (2,.25) {$f_i$};
   \draw[color=blue, thick, double distance=1pt, dashed]
    (fi) .. controls ++(0,1) and ++(0,-1.8) ..  (-1,2);
    \node at (3,.5) {$\l$};
     \node[blue] at (.3,1.8) {$\scs i$};
      \node[blue] at (-1.5,1.8) {$\scs\mu-i$};
\end{tikzpicture}  }
\end{equation}
where $f_i \in \Hom^{2\mu-2i}(\onebl, \Pi^{\mu-i}\onebl)$. By adjunction there are analogous results for $f \in \Hom^{2\mu}(\Ftt \Pi^{\mu}\onebl,\Ftt \onebl)$.
\end{lem}

\begin{proof}
We prove the case $\l \ge -1$ (the case $\l \le -1$ is proved similarly). We will deal with the special case $\mu=1$ and $\l=-1$ at the end. We have
\begin{eqnarray*}
& &\Hom^{2\mu}(\Ett \onebl, \Pi^{\mu}\Ett \onebl) \\
&\cong&
  \Hom^{2\mu}(\onebltwo, \Pi^{\mu}\Ett (\Ett \onebl)^L \onebltwo) \\
&\cong&
  \Hom^{2\mu}(\onebltwo, \Pi^{\mu}\Ett \Ftt \onebltwo \la -(\l+1) \ra) \\
&\cong&
  \Hom^{2\mu}\left(\onebltwo, \Pi^{\mu} \left(\bigoplus_{[\l+2]} \onebltwo \la -(\l+1) \ra
   \oplus\Ftt\Pi \Ett \onebltwo \la -(\l+1) \ra\right) \right) \\
&\cong&
 \Hom^{2\mu-\l-1}(\onebltwo, \Pi^{\mu}
 \bigoplus_{[\l+2]} \onebltwo)
 \oplus \Hom^{2\mu}((\Ftt \oneb_{\l+4})^L, \Pi^{\mu+1} \Ett \onebltwo \la -(\l+1) \ra) \\
&\cong&
 \Hom^{2\mu-\l-1}(\onebltwo,\Pi^{\mu}
 \bigoplus_{[\l+2]} \onebltwo)
 \oplus \Hom^{2\mu}(\Pi^{\l+3}\Ett \onebltwo \la \l+3 \ra, \Pi^{\mu+1}\Ett \onebltwo \la -(\l+1) \ra) \\
&\cong&
 \Hom^{2\mu-\l-1}(\onebltwo, \bigoplus_{[\l+2]}\Pi^{\mu} \onebltwo)
\end{eqnarray*}
where the last line follows since $\mu<(\l+2)$ meaning $\Hom^{2\mu}(\Pi^{\l+3}\Ett \onebltwo, \Pi^{\mu+1} \Ett \onebltwo \la -2(\l+2) \ra) = 0.$ Note that in the third isomorphism above we use the inverse of the isomorphism in \eqref{eq:iso1}. Keeping track of degrees, we find that
\begin{equation}\label{eq:2}
\Hom^{2\mu}(\Ett \onebl, \Pi^{\mu}\Ett \onebl) \cong \bigoplus_{k = 0}^{\l+1} \Hom(\onebltwo, \Pi^{k+\mu}\onebltwo \la 2(\mu-k) \ra).
\end{equation}

If $f \in \Hom^{2\mu}(\Ett \onebl, \Pi^{\mu}\Ett \onebl)$ then we denote the map induced by adjunction
\[
f' \in \Hom^{2\mu}(\onebltwo, \Pi^{\mu}\Ett \Ftt \onebltwo \la -(\l+1) \ra)
\]
and the induced maps on the right side of (\ref{eq:2}) by $f_k' \in \Hom(\onebltwo, \Pi^{k+\mu} \onebltwo \la 2(\mu-k) \ra)$.

Now let us trace through the series of isomorphisms above in order to explicitly identify $f'$ with $f_k'$. The critical isomorphism is the third one where one uses the isomorphism \eqref{eq:iso1}
from Corollary \ref{cor:1}. Thus we find that $f'$ corresponds to the composite
\[
\hackcenter{
\begin{tikzpicture}[scale=0.8]
 \draw[thick, <-] (-0.5,2) .. controls (-0.5,1) and (0.5,1) .. (0.5,2)
      node[pos=0.2, shape=coordinate](tDOT){};
 \draw[color=blue, thick, double distance=1pt, dashed]
   (-1.4,.8) .. controls ++(0,1) and ++(-.5,.5) .. (tDOT);
 \draw[color=blue, thick, double distance=1pt, dashed] (-1.6,.8) -- (-1.6,2);
  \node[draw, thick, fill=blue!20,rounded corners=4pt,inner sep=3pt]
    (Fj) at (-1.5,0.75) {$\scs \;\; f_k' \;\;$};
  \draw[line width=0mm] (-0.5,2) .. controls (-0.5,1) and (0.5,1) .. (0.5,2)
      node[pos=0.2](){\bbullet};
   \node[blue] at (-1,1.45){$\scs k$};
   \node[blue] at (-1.8,1.8){$\scs \mu$};
   \node at (.6,0.7) {$\l$};
\end{tikzpicture} } \maps \onebltwo \to \Pi^{\mu}\Ett \Ftt \onebltwo \la 2\mu-(\l+1) \ra.
\]
Setting
\[
\hackcenter{
\begin{tikzpicture}
 \draw[color=blue, thick, double distance=1pt, dashed] (-1.5,.8) -- (-1.5,2);
  \node[draw, thick, fill=blue!20,rounded corners=4pt,inner sep=3pt]
    (Fj) at (-1.5,0.75) {$\scs  f_k $};
   \node[blue] at (-2,1.8){$\scs \mu-k$};
   \node at (-.6,1) {$\l$};
\end{tikzpicture} }
\quad :=\quad
\hackcenter{
\begin{tikzpicture}
 \draw[color=blue, thick, double distance=1pt, dashed]
   (-1.6,.8) .. controls ++(0,.5) and ++(0,.5) .. (-1.2,.8);
 \draw[color=blue, thick, double distance=1pt, dashed] (-1.8,.8) -- (-1.8,2);
  \node[draw, thick, fill=blue!20,rounded corners=4pt,inner sep=3pt]
    (Fj) at (-1.5,0.75) {$\scs \quad f_k' \quad$};
   \node[blue] at (-1.2,1.35){$\scs k$};
   \node[blue] at (-1.5,1.8){$\scs \mu-k$};
   \node at (-.6,1.5) {$\l$};
\end{tikzpicture} }
\]
and using the adjunction which relates $f$ and $f'$ completes the proof.

If $\mu=2$ and $\l=-1$ then the long calculation above yields
$$\Hom^2(\Ett \oneb_{-1}, \Pi \Ett \oneb_{-1}) \cong \Hom^2(\oneb_1,\Pi  \oneb_1) \oplus \Hom^2(\Ett \oneb_1, \Pi \Ett \oneb_1).$$
The space of maps on the right is one-dimensional and is induced by the dot. The result follows.
\end{proof}

% ==============================================================================
%
\section{Odd cyclic biadjointness}\label{sec:proofcycbiadjoint}
%
% ==============================================================================

Recall that at this point, adjunction maps are only determined up to a scalar. In this section we rescale them so that caps and cups are adjoint to each other and so that
\begin{align}
\xy (0,5)*{
\begin{tikzpicture}[scale=0.8]
  \draw[thick, ->] (-0.5,0) .. controls (-0.5,0.8) and (0.5,0.8) .. (0.5,0)
      node[pos=0.5, shape=coordinate](X){}
      node[pos=0.1, shape=coordinate](Y){};
  \draw[thick] (-0.5,0) .. controls (-0.5,-0.8) and (0.5,-0.8) .. (0.5,0);
 %% Draw double blue curvy line
  \draw[color=blue, thick, double distance=1pt, dashed] (X) .. controls++(0,.65) and ++(-.65,.3) .. (Y)
         node[pos=0.15,right]{$\scs \l-1$\;};
 %% Draw the bullet last so it comes out on top
  \draw[line width=0mm] (-0.5,0) .. controls (-0.5,0.8) and (0.5,0.8) .. (0.5,0)
     node[pos=0.1]{\bbullet};
  \node at (1,0.5) {$\lambda$};
\end{tikzpicture}
 };
 \endxy
  &= \oneb_{\onebl}  \qquad \text{for $\l >0$,}
& \qquad \qquad
 \xy (0,0)*{
\begin{tikzpicture}[scale=0.8]
  \draw[thick, ->] (0.5,0) .. controls (0.5,0.8) and (-0.5,0.8) .. (-0.5,0);
  \draw[thick] (0.5,0) .. controls (0.5,-0.8) and (-0.5,-0.8) .. (-0.5,0)
      node[pos=0.5, shape=coordinate](X){}
      node[pos=0.1, shape=coordinate](Y){};
 %% Draw double blue curvy line
  \draw[color=blue, thick, double distance=1pt, dashed] (X) .. controls++(-.1,1) and ++(-.2,.4) .. (Y)
         node[pos=0.9,right]{$\scs -\l-1$\;};
 %% Draw the bullet last so it comes out on top
  \draw[line width=0mm
  ] (0.5,0) .. controls (0.5,-0.8) and (-0.5,-0.8) .. (-0.5,0)
     node[pos=0.1]{\bbullet};
  \node at (1,0.6) {$\lambda$};
\end{tikzpicture}
}; \endxy &= \oneb_{\onebl}  \qquad \text{for $\l <0$.}
\end{align}
Recall that negative degree dotted bubbles are zero while a dotted bubble of degree zero must be a non-zero multiple of the identity map by Corollary~\ref{cor:degz-bubbles} (note that for $\l=0$ there are no degree zero dotted bubbles). By rescaling the adjunction maps we can ensure the degree zero bubbles satisfy the conditions above. More precisely, we rescale in the following order:
\[
\begin{tabular}{|l|c|c|c|c|}
 \hline
 $ \l \geq 0$ &
  \xy (0,-3)*{\begin{tikzpicture}[scale=0.8]
  \draw[thick, ->] (-.75,0) .. controls ++(0,-1) and ++(0,-1) .. (.75,0)
      node[pos=0.5, shape=coordinate](X){};
  \draw[color=blue, thick, double distance=1pt, dashed]
   (X) -- (0,-1.4);
   \node at (1.2,-.7) {$\l-1$};
   \node at (0,-.25) {$\l+1$};
  \end{tikzpicture} }; \endxy
  &
    \xy (0,-3)*{\begin{tikzpicture}[scale=0.8]
  \draw[thick, ->] (-.75,0) .. controls ++(0,1) and ++(0,1) .. (.75,0)
      node[pos=0.5, shape=coordinate](X){};
  \draw[color=blue, thick, double distance=1pt, dashed]
   (X) -- (0,1.4);
   \node at (1.2,.7) {$\l+1$};
   \node at (0,.25) {$\l-1$};
  \end{tikzpicture} }; \endxy
  &
  \xy (0,-3)*{\begin{tikzpicture}[scale=0.8]
  \draw[thick, ->] (.75,0) .. controls ++(0,-1) and ++(0,-1) .. (-.75,0);
   \node at (1.2,-.7) {$\l+1$};
   \node at (0,-.25) {$\l-1$};
  \end{tikzpicture} }; \endxy
  &
  \xy (0,-3)*{\begin{tikzpicture}[scale=0.8]
  \draw[thick, ->] (.75,0) .. controls ++(0,1) and ++(0,1) .. (-.75,0);
   \node at (1.2,.7) {$\l-1$};
   \node at (0,.25) {$\l+1$};
  \end{tikzpicture} }; \endxy
    \\& & &  &\\ \hline
 & \;\;\txt{ fixed arbitrarily} \;\;
 & \;\;\txt{ determined by \\adjunction}\;\;
 & \;\;\txt{ fixed by value\\ of bubble}\;\;
 & \;\;\txt{ determined by\\ adjunction}\;\;
 \\
 \hline
\end{tabular}
\]
\[
\begin{tabular}{|l|c|c|c|c|}
 \hline
 $ \l < 0$ &
   \xy (0,-3)*{\begin{tikzpicture}[scale=0.8]
  \draw[thick, ->] (.75,0) .. controls ++(0,-1) and ++(0,-1) .. (-.75,0);
   \node at (1.2,-.7) {$\l+1$};
   \node at (0,-.25) {$\l-1$};
  \end{tikzpicture} }; \endxy
  &
   \xy (0,-3)*{\begin{tikzpicture}[scale=0.8]
  \draw[thick, ->] (.75,0) .. controls ++(0,1) and ++(0,1) .. (-.75,0);
   \node at (1.2,.7) {$\l-1$};
   \node at (0,.25) {$\l+1$};
  \end{tikzpicture} }; \endxy
  &
   \xy (0,-3)*{\begin{tikzpicture}[scale=0.8]
  \draw[thick, ->] (-.75,0) .. controls ++(0,-1) and ++(0,-1) .. (.75,0)
      node[pos=0.5, shape=coordinate](X){};
  \draw[color=blue, thick, double distance=1pt, dashed]
   (X) -- (0,-1.4);
   \node at (1.2,-.7) {$\l-1$};
   \node at (0,-.25) {$\l+1$};
  \end{tikzpicture} }; \endxy
  &
    \xy (0,-3)*{\begin{tikzpicture}[scale=0.8]
  \draw[thick, ->] (-.75,0) .. controls ++(0,1) and ++(0,1) .. (.75,0)
      node[pos=0.5, shape=coordinate](X){};
  \draw[color=blue, thick, double distance=1pt, dashed]
   (X) -- (0,1.4);
   \node at (1.2,.7) {$\l+1$};
   \node at (0,.25) {$\l-1$};
  \end{tikzpicture} }; \endxy
    \\& & &  &\\ \hline
 & \;\;\txt{ fixed arbitrarily}\;\;
 & \;\;\txt{ determined by \\adjunction}\;\;
 & \;\;\txt{ fixed by value\\ of bubble}\;\;
 & \;\;\txt{ determined by\\ adjunction}\;\;
 \\
 \hline
\end{tabular}
\]
The only time this rescaling fails is when $\l=0$ above. In that case, the rescaling for $\l \geq 0$ fixes the value of the positive bubble in weight $\l=1$, so that
\begin{equation} \label{eq_bub1}
 \hackcenter{\begin{tikzpicture} [scale=0.6]
  \draw[thick, ->, , ->-=0.03] (-0.5,0) .. controls (-0.5,0.8) and (0.5,0.8) .. (0.5,0);
  \draw[thick] (0.5,0) .. controls (0.5,-0.8) and (-0.5,-0.8) .. (-0.5,0);
  \node at (1,0.5) {$+1$};
\end{tikzpicture} }= \oneb_{\onebl},
\end{equation}
where $+1$ denotes the outside region of the bubble. However, the rescaling for $\l=0$ fails to rescale the positive degree bubble in region $\l=-1$.  This bubble is multiplication by some arbitrary scalar $c_{-1}$:
\begin{equation} \label{eq_defcmone}
 \hackcenter{\begin{tikzpicture} [scale=0.6]
  \draw[thick, ->, , ->-=0.03] (0.5,0) .. controls (0.5,0.8) and (-0.5,0.8) .. (-0.5,0);
  \draw[thick] (0.5,0) .. controls (0.5,-0.8) and (-0.5,-0.8) .. (-0.5,0);
  \node at (1,0.5) {$-1$};
\end{tikzpicture} }
  =  c_{-1} \oneb_{\onebl}
\end{equation}
where $-1$ denotes the outside region of the bubble. We will show in Proposition~\ref{prop_free_param} that in fact $c_{-1} = 1$.

% -------------------------------------------------------------------------------
%
\subsection{Odd cyclicity for dots}
%
% -------------------------------------------------------------------------------

% - - - - - - - - - - - - - - - - - - - - - - - - - - - - - - - - - - - - - - - -
%
%
\subsubsection{Defining downward oriented dots}
%
% - - - - - - - - - - - - - - - - - - - - - - - - - - - - - - - - - - - - - - - -
%

We will make repeated use of the 2-morphism $x_\Fc'$, which is defined diagrammatically by
\begin{equation}\label{eqn-x-fc-prime}
\hackcenter{
 }; \endxy .
\end{equation}
Above, we defined a dot on a down-strand using the left-oriented adjunction.  The next lemma computes the difference between this and the analogous right-oriented picture.

\begin{lem}\label{lem-x-fc}
\begin{equation}\label{eqn-x-fc-right}
(-1)^{\lambda} \;\;
\hackcenter{

}
\quad=\quad
\left\{
\begin{array}{ll}
\hackcenter{
 \begin{tikzpicture} [scale=0.7]
	\draw[thick, <-] (0,0) -- (0,3)
		node[pos=.5, shape=coordinate](DOT){};
	\draw[color=blue, thick, dashed]
  (DOT) .. controls ++(.3,-1) and ++(0,-2.5) .. (.75,3);
	\draw (DOT) -- (DOT)
		node[pos=0](){\bbullet};
     \node at (-1.2,2) {$\l-1$};
\end{tikzpicture}}
\quad+\quad2\quad
\hackcenter{\begin{tikzpicture}[scale=0.7]
	\draw[thick, <-] (-2,-1.5) -- (-2,1.5);
  \draw[thick, ->] (-0.5,0) .. controls (-0.5,0.9) and (0.5,0.9) .. (0.5,0)
      node[pos=0.5, shape=coordinate](X){}
      node[pos=0.1, shape=coordinate](Y){};
  \draw[thick] (-0.5,0) .. controls (-0.5,-0.9) and (0.5,-0.9) .. (0.5,0)
      node[pos=0.1, shape=coordinate](Z){};
  \draw[color=blue, thick, double distance=1pt, dashed] (X) .. controls++(0,.65) and ++(-.65,.3) .. (Y) node[pos=0.15,right]{$\scs \l$\;};
  \draw[color=blue, thick,  dashed] (Z) to[out=160, in=-90] (-1.25,1.5) ;
  \node[blue] at (-1.25,0.8){$\scs $\;};
  \draw[line width=0mm] (-0.5,0) .. controls (-0.5,0.9) and (0.5,0.9) .. (0.5,0)
    node[pos=0.1]{\bbullet};
  \draw[line width=0mm] (-0.5,0) (-0.5,0) .. controls (-0.5,-0.9) and (0.5,-0.9) .. (0.5,0)
        node[pos=0.1]{\bbullet};
  \node at (.8,-1) {$\lambda+1$};
\end{tikzpicture}} & \l>0, \\ & \\
\end{array}
\right.
\end{equation}
\end{lem}

\begin{proof}
By Lemma~\ref{lem:main}, there is an equation of the form
\begin{equation}\label{eqn-x-fc-right-gamma}
\hackcenter{
%
}
\end{equation}
for all $\mu \geq -1$.   Given any 2-morphism $\beta:\oneb_{\mu}\Ftt\to\oneb_{\mu}\Ftt\Pi\la 2\ra$, we can form the diagram
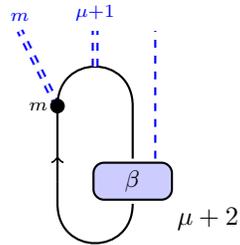
\begin{equation*}
%
\end{equation*}
by closing the 2-morphism off with $m \geq 0$ dots. Applied to equation \eqref{eqn-x-fc-right-gamma}, this yields
\begin{equation}\label{eqn-bumpy-bubbles}
\hackcenter{%
}
\quad.
\end{equation}
On the left-hand side, resolve the $\mu+1$ crossings (accruing a factor of $(-1)^{\mu+1}$), move the caps to the left $\mu+1$ strands, and use the right-oriented adjunction to get rid of the left cap and the cup in the middle; what remains is a bubble with $m+1$ dots on the left.  The first term on the right-hand side is more interesting:
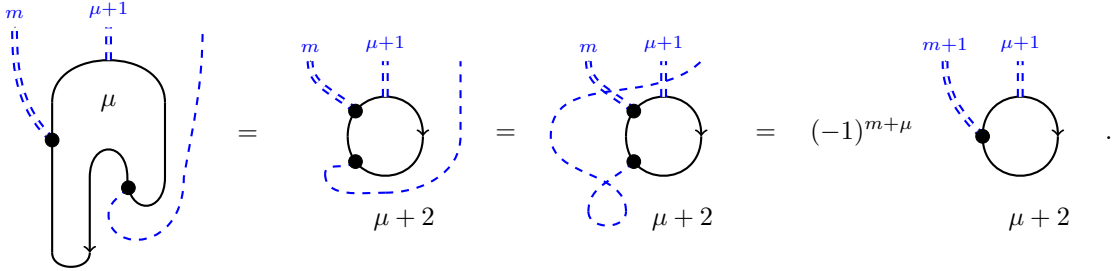
\begin{equation*}
\hackcenter{%
}
\quad.
\end{equation*}
Plugging this back into \eqref{eqn-bumpy-bubbles},
\begin{equation} \label{eq:bubble}
-(-1)^\mu(1+\gamma_0(-1)^m)
\hackcenter{%
}
\quad.
\end{equation}
Taking $m=\mu\geq 0$, the bubble on the right-hand side is zero.  Thus $\gamma_0=(-1)^{\mu+1}$.  Taking $m=\mu+1$, the bubble on the right-hand side equals $1$, so
\begin{equation}
\hackcenter{%
},
\end{equation}
completing the proof of \eqref{eqn-x-fc-right} for $\l>0$.  The proof for $\l<0$ is similar.

The case when $\l=0$ requires more care.  Inserting equation~\eqref{eqn-x-fc-right-gamma} for $\mu=-1$ into $\beta$ in the diagram
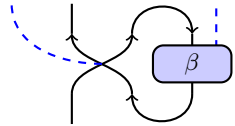
\[
%
\]
and simplifying using zig-zag identities  gives the equation:
\[
\hackcenter{%
}.
\]
The right-curl diagram on the right is zero by Lemma~\ref{lem:E} since it is a 2-morphism of degree $-2$.  Simplifying the remaining equation using the odd nilHecke dot slide equation and removing terms containing a dot composed with a degree $-2$ right twist curl we find that $\gamma_0=-1$. Substituting this into \eqref{eq:bubble} with $m=\mu+1=0$ shows that $\gamma_1$ is the zero 2-morphism.
\end{proof}

% - - - - - - - - - - - - - - - - - - - - - - - - - - - - - - - - - - - - - - - -
%
%
\subsubsection{Half cyclicity relations} \label{subsubsec-half_cyclic}
%
% - - - - - - - - - - - - - - - - - - - - - - - - - - - - - - - - - - - - - - - -
It is often more convenient to use Lemma~\ref{lem-x-fc} in the following form.
\begin{cor}\label{cor-dot-cyclicity}
The following local relations hold in any strong supercategorical action.
\begin{equation}\label{eqn-dot-cyclicity-left-cap}
\hackcenter{\begin{tikzpicture} [scale=0.9]
	\draw[thick, ->] (1,0) .. controls (1,.8) and (0,.8) .. (0,0)
		node[pos=.8, shape=coordinate](DOT){};
	\draw[thick, color=blue, dashed] (DOT) [out=-30, in=90] to (.5,0);
		\node at (DOT){\bbullet};
  \node at (1.2,0.8) {$\l-1$};
\end{tikzpicture}}
\quad = \quad
\hackcenter{\begin{tikzpicture}[scale=0.9]
	\draw[thick, ->] (1,0) .. controls (1,.8) and (0,.8) .. (0,0)
		node[pos=.2, shape=coordinate](DOT){};
	\draw[thick, color=blue, dashed] (DOT) [out=--140, in=90] to (.5,0);
		\node at (DOT){\bbullet};
  \node at (1.2,0.8) {$\l-1$};
\end{tikzpicture}}
\qquad \qquad
\hackcenter{\begin{tikzpicture}[scale=0.9]
	\draw[thick, ->] (1,0) .. controls (1,-.8) and (0,-.8) .. (0,0)
		node[pos=.8, shape=coordinate](DOT){};
	\draw[thick, color=blue, dashed]
        (DOT) .. controls ++(-.6,.5) and ++(0,.7) ..(-.25,-1);
		\node at (DOT){\bbullet};
  \node at (1.2,-0.8) {$\l+1$};
\end{tikzpicture}}
\quad = \quad
\hackcenter{\begin{tikzpicture}[scale=0.9]
	\draw[thick, ->] (1,0) .. controls (1,-.8) and (0,-.8) .. (0,0)
		node[pos=.2, shape=coordinate](DOT){};
	\draw[thick, color=blue, dashed] (DOT) to[out=-30, in=90] (1.25,-1);
		\node at (DOT){\bbullet};
  \node at (-.3,-0.8) {$\l+1$};
\end{tikzpicture}}
\end{equation}

\begin{equation}\label{eqn-dot-cyclicity-right-cap}
\hackcenter{\begin{tikzpicture}[scale=0.9]
	\draw[thick, ->] (0,0) .. controls (0,.8) and (1,.8) .. (1,0)
		node[pos=.5, shape=coordinate](TOPCAP){}
		node[pos=.15, shape=coordinate](DOT){};
	\draw[thick, color=blue, dashed, double distance=1pt] (TOPCAP) -- (.5,1.5)
		node[pos=0.8,blue, right](){$\scs \lambda$};
	\draw[thick, color=blue, dashed] (DOT) [out=135, in=-90] to (-.25,1.5);
		\node at (DOT){\bbullet};
  \node at (1.4,0.8) {$\l+1$};
\end{tikzpicture}}
\quad=\quad
\left\{
\begin{array}{ll}
  \hackcenter{\begin{tikzpicture}[scale=0.9]
	\draw[thick, ->] (0,0) .. controls (0,.8) and (1,.8) .. (1,0)
		node[pos=.5, shape=coordinate](TOPCAP){}
		node[pos=.75, shape=coordinate](DOT){};
	\draw[thick, color=blue, dashed, double distance=1pt] (TOPCAP) -- (.5,1.5)
		node[pos=0.8,blue, left](){$\scs \lambda$};
	\draw[thick, color=blue, dashed]
     (DOT) .. controls ++(.6,-.5) and ++(0,-1) ..  (1.25,1.5);
	\draw (DOT) -- (DOT)
		node(){\bbullet};
   \node at (-.7,0.8) {$\l+1$};
\end{tikzpicture}}
\quad+\quad 2\;
\hackcenter{\begin{tikzpicture}[scale=0.9]
	\draw[thick, ->] (-2,-.5) .. controls ++(0,.7) and ++(0,.7) .. (-1,-.5)
		node[pos=.5, shape=coordinate](TOPCAP){}
		node[pos=.25, shape=coordinate](DOT){};
 \draw[thick, color=blue, dashed, double distance=1pt] (TOPCAP) -- (-1.5,1)
		node[pos=0.8,blue, left](){$\scs \lambda$};
  \draw[thick, ->] (-0.4,0) .. controls ++(-0,0.6) and ++(0,0.6) .. (0.4,0)
      node[pos=0.5, shape=coordinate](X){}
      node[pos=0.1, shape=coordinate](Y){};
  \draw[thick] (-0.4,0) .. controls ++(0,-0.6) and ++(0,-0.6) .. (0.4,0)
      node[pos=0.1, shape=coordinate](Z){};
  \draw[color=blue, thick, double distance=1pt, dashed] (X) .. controls++(0,.65) and ++(-.65,.3) .. (Y) node[pos=0.15,right]{$\scs \l$\;};
  \draw[color=blue, thick, dashed] (Z) to[bend left] (-1,1) ;
  \node at (Y) {\bbullet};
  \node at (Z) {\bbullet};
  \node at (1,-.5) {$\l+1$};
\end{tikzpicture}} & \l>0 \\ & \\
  \hackcenter{\begin{tikzpicture}[scale=0.9]
	\draw[thick, ->] (0,0) .. controls (0,.8) and (1,.8) .. (1,0)
		node[pos=.5, shape=coordinate](TOPCAP){}
		node[pos=.75, shape=coordinate](DOT){};
	\draw[thick, color=blue, dashed, double distance=1pt] (TOPCAP) -- (.5,1.5)
		node[pos=0.8,blue, left](){$\scs \lambda$};
	\draw[thick, color=blue, dashed]
     (DOT) .. controls ++(.6,-.5) and ++(0,-1) ..  (1.25,1.5);
	\draw (DOT) -- (DOT)
		node(){\bbullet};
   \node at (-.7,0.8) {$\l+1$};
\end{tikzpicture}}
\quad+\quad 2\;
\hackcenter{\begin{tikzpicture}[scale=0.9]
  \draw[thick, ->] (-1.4,-0.5) .. controls ++(-.1,1.5) and ++(.1,1.5) .. (1.4,-0.5)
		node[pos=.5, shape=coordinate](TOPCAP){};
	\draw[thick, color=blue, dashed, double distance=1pt] (TOPCAP) -- (0,1.25)
		node[pos=0.8,blue, left](){$\scs \lambda$};
  \draw[thick, ->] (0.4,0) .. controls ++(0,0.6) and ++(0,0.6) .. (-0.4,0)
      node[pos=0.05, shape=coordinate](Z){};
  \draw[thick] (0.4,0) .. controls ++(0,-0.6) and ++(-0,-0.6) .. (-0.4,0)
      node[pos=0.5, shape=coordinate](X){}
      node[pos=0.2, shape=coordinate](Y){};
  \draw[color=blue, thick, double distance=1pt, dashed]
    (X) .. controls++(-.1,.5) and ++(-.2,.3) .. (Y)
         node[pos=0.9,right]{$\scs -\l$\;};
   \draw[color=blue, thick,  dashed]
    (Z) .. controls ++(-1,.4) and ++(.1,-1) .. (1,1) ;
   \node[blue] at (1.25,0.8){$\scs $\;};
     \node at (Y) {\bbullet};
     \node at (Z) {\bbullet};
  \node at (2,.3) {$\lambda+1$};
\end{tikzpicture} }
   & \l<0 \\ & \\-\;\;
     \hackcenter{\begin{tikzpicture}[scale=0.9]
	\draw[thick, ->] (0,0) .. controls (0,.8) and (1,.8) .. (1,0)
		node[pos=.5, shape=coordinate](TOPCAP){}
		node[pos=.75, shape=coordinate](DOT){};
	\draw[thick, color=blue, dashed, double distance=1pt] (TOPCAP) -- (.5,1.5)
		node[pos=0.8,blue, left](){$\scs \lambda$};
	\draw[thick, color=blue, dashed]
     (DOT) .. controls ++(.6,-.5) and ++(0,-1) ..  (1.25,1.5);
	\draw (DOT) -- (DOT)
		node(){\bbullet};
   \node at (-.7,0.8) {$\l+1$};
\end{tikzpicture}}& \l=0
\end{array}
\right.
\end{equation}

\begin{equation}\label{eqn-dot-cyclicity-right-cup}
\hackcenter{\begin{tikzpicture}[scale=0.9]
	\draw[thick] (0,1) .. controls (0,.3) and (1,.3) .. (1,1)
		node[pos=.5, shape=coordinate](BOTTOMCUP){}
		node[pos=.85, shape=coordinate](DOT){};
	\draw[thick, ->] (1,1) -- (1,1.5);
	\draw[thick] (0,1) -- (0,1.5);
	\draw[thick, color=blue, dashed] (DOT) [out=135, in=-90] to (.75,1.5);
	\draw[thick, color=blue, double distance=1pt, dashed] (BOTTOMCUP) -- (.5,1.5)
		node[pos=.8, left, blue](){$\scs \lambda$};
	\node at (DOT){\bbullet};;
    \node at (1.5,.3) {$\lambda-1$};
\end{tikzpicture}}
\quad=\quad
\left\{
\begin{array}{ll}
  \hackcenter{\begin{tikzpicture}[scale=0.9]
	\draw[thick] (0,1) .. controls (0,.3) and (1,.3) .. (1,1)
		node[pos=.5, shape=coordinate](BOTTOMCUP){}
		node[pos=.15, shape=coordinate](DOT){};
	\draw[thick, ->] (1,1) -- (1,1.5);
	\draw[thick] (0,1) -- (0,1.5);
	\draw[thick, color=blue, dashed] (DOT) [out=45, in=-90] to (.25,1.5);
	\draw[thick, color=blue, double distance=1pt, dashed] (BOTTOMCUP) -- (.5,1.5)
		node[pos=.8, right, blue](){$\scs \lambda$};
	\node at (DOT){\bbullet};
    \node at (1.5,.3) {$\lambda-1$};
\end{tikzpicture}} \quad + \quad 2 \;\;
\hackcenter{\begin{tikzpicture}[scale=0.9]
  \draw[thick]  (-1.6,0.5) -- (-1.6,1);
  \draw[thick, ->]  (1.6,0.5) -- (1.6,1);
  \draw[thick] (-1.6,0.5) .. controls ++(-.1,-2) and ++(.1,-2) .. (1.6,0.5)
		node[pos=.5, shape=coordinate](TOPCAP){};
	\draw[thick, color=blue, dashed, double distance=1pt]
    (TOPCAP) .. controls ++(0,.5) and ++(0,-1.5) ..  (1,1)
		node[pos=0.8,blue, right](){$\scs \lambda$};
  \draw[thick, ->] (-0.4,0) .. controls ++(-0,0.6) and ++(0,0.6) .. (0.4,0)
      node[pos=0.5, shape=coordinate](X){}
      node[pos=0.1, shape=coordinate](Y){};
  \draw[thick] (-0.4,0) .. controls ++(0,-0.6) and ++(0,-0.6) .. (0.4,0)
      node[pos=0.1, shape=coordinate](Z){};
  \draw[color=blue, thick, double distance=1pt, dashed] (X) .. controls++(0,.65) and ++(-.65,.3) .. (Y) node[pos=0.15,right]{$\scs \l$\;};
  \draw[color=blue, thick, dashed] (Z) to[bend left] (-1,1) ;
  \node at (Y) {\bbullet};
  \node at (Z) {\bbullet};
  \node at (2,-.6) {$\lambda-1$};
\end{tikzpicture} }& \l>0
 \\ & \\
  \hackcenter{\begin{tikzpicture}[scale=0.9]
	\draw[thick] (0,1) .. controls (0,.3) and (1,.3) .. (1,1)
		node[pos=.5, shape=coordinate](BOTTOMCUP){}
		node[pos=.15, shape=coordinate](DOT){};
	\draw[thick, ->] (1,1) -- (1,1.5);
	\draw[thick] (0,1) -- (0,1.5);
	\draw[thick, color=blue, dashed] (DOT) [out=45, in=-90] to (.25,1.5);
	\draw[thick, color=blue, double distance=1pt, dashed] (BOTTOMCUP) -- (.5,1.5)
		node[pos=.8, right, blue](){$\scs \lambda$};
	\node at (DOT){\bbullet};
    \node at (1.5,.3) {$\lambda-1$};
\end{tikzpicture}}
\quad + \quad 2\;\;
\hackcenter{\begin{tikzpicture}
	\draw[thick, ->] (.5,1) .. controls ++(0,-.8) and ++(0,-.8) .. (2,1)
		node[pos=.5, shape=coordinate](BOTTOMCUP){};
	\draw[thick, color=blue, double distance=1pt, dashed] (BOTTOMCUP) to[out=90, in=-90] (1.5,1);
		\node[ blue] at (1.65,.8){$\scs \lambda$};
    \draw[thick, ->] (0.4,0) .. controls ++(0,0.6) and ++(0,0.6) .. (-0.4,0)
      node[pos=0.05, shape=coordinate](Z){};
  \draw[thick] (0.4,0) .. controls ++(0,-0.6) and ++(-0,-0.6) .. (-0.4,0)
      node[pos=0.5, shape=coordinate](X){}
      node[pos=0.2, shape=coordinate](Y){};
  \draw[color=blue, thick, double distance=1pt, dashed]
    (X) .. controls++(-.1,.5) and ++(-.2,.3) .. (Y)
         node[pos=0.9,right]{$\scs -\l$\;};
   \draw[color=blue, thick,  dashed]
    (Z) .. controls ++(-1,.4) and ++(.1,-1) .. (1,1) ;
   \node[blue] at (1.25,0.8){$\scs $\;};
     \node at (Y) {\bbullet};
     \node at (Z) {\bbullet};
    \node at (1.8,-.3) {$\lambda-1$};
\end{tikzpicture}}
   & \l<0
   \\ & \\
 - \;\; \hackcenter{\begin{tikzpicture}[scale=0.9]
	\draw[thick] (0,1) .. controls (0,.3) and (1,.3) .. (1,1)
		node[pos=.5, shape=coordinate](BOTTOMCUP){}
		node[pos=.15, shape=coordinate](DOT){};
	\draw[thick, ->] (1,1) -- (1,1.5);
	\draw[thick] (0,1) -- (0,1.5);
	\draw[thick, color=blue, dashed] (DOT) [out=45, in=-90] to (.25,1.5);
	\draw[thick, color=blue, double distance=1pt, dashed] (BOTTOMCUP) -- (.5,1.5)
		node[pos=.8, right, blue](){$\scs \lambda$};
	\node at (DOT){\bbullet};
    \node at (1.5,.3) {$\lambda-1$};
\end{tikzpicture}}& \l=0
\end{array}
\right.
\end{equation}
\end{cor}

\begin{proof}
Applying a right cap to equation \eqref{eqn-x-fc-right} gives
\begin{equation}
(-1)^{\lambda}\;\;
\hackcenter{\begin{tikzpicture}
	\draw[color=blue, thick, double distance=1pt, dashed] (.25,.6) .. controls (.25,1.8) and (.75,1.8) .. (.75,1.4);
	\draw[thick, <-] (1,0) [out=90, in=-90] to (1,1) .. controls (1,1.5) and (.5,1.5) .. (.5,1) .. controls (.5,.5) and (0,.5) .. (0,1) [out=90, in=-90] to (0,2)
		node[pos=0, shape=coordinate](DOT){};
	\draw[color=blue, thick, dashed] (DOT) [out=135, in=-90] to (.2, 2.5);
	\draw[thick] (0,2) .. controls (0,2.35) and (-.5,2.35) .. (-.5,2)
		node[pos=.5, shape=coordinate](TOPCAP2){};
	\draw[thick] (-.5,0) -- (-.5,2);
	\draw[thick, color=blue, double distance=1pt, dashed] (TOPCAP2) -- (-.25,2.5)
		node[pos=0.9, left, blue](){$\scs \lambda$};
	\draw (DOT) -- (DOT)
		node[pos=0](){\bbullet};
  \node at (1.5,2) {$\l+1$};
\end{tikzpicture}}
\quad=\quad
\hackcenter{\begin{tikzpicture}
	\draw[thick, <-] (0,0) [out=90, in=-90] to (0,1) .. controls (0,1.5) and (.5,1.5) .. (.5,1) .. controls (.5,.5) and (1,.5) .. (1,1) [out=90, in=-90] to (1,1.25);
		node[pos=.1, shape=coordinate](DOT){};
	\draw[thick] (-.5,1.25) .. controls (-.5,2.3) and (1,2.3) .. (1,1.25)
		node[pos=.5, shape=coordinate](TOPCAP2){};
	\draw[thick] (-.5,0) -- (-.5,1.25);
	\draw[thick, color=blue, double distance=1pt, dashed] (TOPCAP2) -- (.25,2.5)
		node[pos=0.9,left, blue](){$\scs \lambda$};
	\draw[color=blue, thick, dashed] (DOT) [out=-135, in=90] to (.25, .5) .. controls (.25,0) and (1.25,0) .. (1.25,1) [out=80, in=-90] to (1.5,2.5);
	\draw (DOT) -- (DOT)
		node[pos=0](){\bbullet};
  \node at (1.8,0.2) {$\l+1$};
\end{tikzpicture}}
\quad+\quad 2 \;\;
\hackcenter{\begin{tikzpicture}
	\draw[thick, ->] (-2,-.5) .. controls ++(0,.7) and ++(0,.7) .. (-1,-.5)
		node[pos=.5, shape=coordinate](TOPCAP){}
		node[pos=.25, shape=coordinate](DOT){};
 \draw[thick, color=blue, dashed, double distance=1pt] (TOPCAP) -- (-1.5,1)
		node[pos=0.8,blue, left](){$\scs \lambda$};
  \draw[thick, ->] (-0.4,0) .. controls ++(-0,0.6) and ++(0,0.6) .. (0.4,0)
      node[pos=0.5, shape=coordinate](X){}
      node[pos=0.1, shape=coordinate](Y){};
  \draw[thick] (-0.4,0) .. controls ++(0,-0.6) and ++(0,-0.6) .. (0.4,0)
      node[pos=0.1, shape=coordinate](Z){};
  \draw[color=blue, thick, double distance=1pt, dashed] (X) .. controls++(0,.65) and ++(-.65,.3) .. (Y) node[pos=0.15,right]{$\scs \l$\;};
  \draw[color=blue, thick, dashed] (Z) to[bend left] (-1,1) ;
  \node at (Y) {\bbullet};
  \node at (Z) {\bbullet};
  \node at (1,-.5) {$\l+1$};
\end{tikzpicture}}
\end{equation}
for $\l>0$. Using isotopy and resolving $\lambda$ dashed-dashed crossings on the left-hand side, this becomes exactly \eqref{eqn-dot-cyclicity-right-cap}. Equation~\eqref{eqn-dot-cyclicity-right-cap} for other values of $\l$ is proven similarly.

The left cap equation \eqref{eqn-x-fc-right} can be proven in a similar manner using \eqref{eqn-x-fc-right} and right cap equation proven above.  The cup equations are proven similarly.
\end{proof}

% -------------------------------------------------------------------------------
%
\subsection{Odd cyclicity for crossings}
%
% -------------------------------------------------------------------------------

% - - - - - - - - - - - - - - - - - - - - - - - - - - - - - - - - - - - - - - - -
%
\subsubsection{Defining downward oriented crossings}
%
% - - - - - - - - - - - - - - - - - - - - - - - - - - - - - - - - - - - - - - - -

In any strong supercategorical action we can define a downward oriented crossing using the easy adjunctions as follows.
\begin{equation}
\hackcenter{\begin{tikzpicture}
\begin{scope}[shift={(0,0)},rotate=180]
	\draw[thick, ->] (0,0) .. controls (0,.75) and (.5,.75) .. (.5,1.5);
	\draw[thick, ->] (.5,0) .. controls (.5,.75) and (0,.75) .. (0,1.5)
		node[pos=.5, shape=coordinate](CROSSING){};
	\draw[thick, color=blue, dashed] (CROSSING) [out=180, in=-90] to (-.5,1.5);
 \node at (.8,.75) {$\l$};
\end{scope}
\end{tikzpicture} }
\quad := \quad
\hackcenter{
\begin{tikzpicture}[scale=0.6]
  \draw[thick, ->] (-0.5,0) .. controls ++(-0,0.5) and ++(0,-0.5) .. (0.5,1)
      node[pos=0.5, shape=coordinate](X){};
  \draw[thick, ->] (0.5,0) .. controls ++(0,0.5) and ++(0,-0.5) .. (-0.5,1);
  \draw[thick] (0.5,0) .. controls ++(0,-0.5) and ++(0,-0.5) .. (1.5,0);
  \draw[thick] (-0.5,0) .. controls ++(0,-1.5) and ++(0,-1.5) .. (2.5,0);
  \draw[thick, ->-=0.5] (2.5,2) -- (2.5,0);
  \draw[thick, ->-=0.5] (1.5,2) -- (1.5,0);
  \draw[thick] (-0.5,1) .. controls ++(0,0.5) and ++(0,0.5) .. (-1.5,1);
  \draw[thick] (0.5,1) .. controls ++(0,1.5) and ++(0,1.5) .. (-2.5,1);
  \draw[thick, ->-=0.5] (-2.5,1) -- (-2.5,-1.5);
  \draw[thick, ->-=0.5] (-1.5,1) -- (-1.5,-1.5);
  \draw[color=blue,  thick, dashed]
   (X) .. controls ++(-1,.4) and ++(0,1).. (-.75,-1.5);
      \node at (4,-1) {$\l+4$};
\end{tikzpicture}}
\end{equation}

It follows from Corollary~\ref{cor-dot-cyclicity} and the definition of the downward oriented dot from \eqref{eqn-x-fc-prime} that with this definition of the downward oriented crossing we have the downward oriented nilHecke axioms.  In particular, the equations
\[
\hackcenter{\begin{tikzpicture}\begin{scope}[shift={(0,0)},rotate=180]
	\draw[thick, ->] (0,0) .. controls (0,.75) and (.5,.75) .. (.5,1.5)
		node[pos=.5, shape=coordinate](CROSSING){}
		node[pos=.25, shape=coordinate](DOT){};
	\draw[thick, ->] (.5,0) .. controls (.5,.75) and (0,.75) .. (0,1.5);
	\draw[thick, color=blue, dashed] (DOT) [out=180, in=-90] to (-1,1.5);
	\draw[thick, color=blue, dashed] (CROSSING) [out=180, in=-90] to (-.5,1.5);
	\node() at (DOT) {\bbullet};\end{scope}
\end{tikzpicture}}
\quad-\quad
\hackcenter{\begin{tikzpicture}\begin{scope}[shift={(0,0)},rotate=180]
	\draw[thick, ->] (0,0) .. controls (0,.75) and (.5,.75) .. (.5,1.5)
		node[pos=.5, shape=coordinate](CROSSING){}
		node[pos=.75, shape=coordinate](DOT){};
	\draw[thick, ->] (.5,0) .. controls (.5,.75) and (0,.75) .. (0,1.5);
	\draw[thick, color=blue, dashed] (DOT) [out=180, in=-90] to (-1,1.5);
	\draw[thick, color=blue, dashed] (CROSSING) [out=180, in=-90] to (-.5,1.5);
	\node() at (DOT) {\bbullet}; \end{scope}
\end{tikzpicture}}
\quad = \quad
\hackcenter{\begin{tikzpicture} \begin{scope}[shift={(0,0)},rotate=180]
	\draw[thick, ->] (0,0) .. controls (0,.75) and (.5,.75) .. (.5,1.5);
	\draw[thick, ->] (.5,0) .. controls (.5,.75) and (0,.75) .. (0,1.5)
		node[pos=.5, shape=coordinate](CROSSING){}
		node[pos=.75, shape=coordinate](DOT){};
	\draw[thick, color=blue, dashed] (DOT) [out=180, in=-90] to (-.5,1.5);
	\draw[thick, color=blue, dashed] (CROSSING) [out=180, in=-90] to (-1,1.5);
	\node() at (DOT) {\bbullet};
\end{scope}
\end{tikzpicture}}
\quad-\quad
\hackcenter{\begin{tikzpicture} \begin{scope}[shift={(0,0)},rotate=180]
	\draw[thick, ->] (0,0) .. controls (0,.75) and (.5,.75) .. (.5,1.5);
	\draw[thick, ->] (.5,0) .. controls (.5,.75) and (0,.75) .. (0,1.5)
		node[pos=.5, shape=coordinate](CROSSING){}
		node[pos=.25, shape=coordinate](DOT){};
	\draw[thick, color=blue, dashed] (DOT) [out=180, in=-90] to (-.5,1.5);
	\draw[thick, color=blue, dashed] (CROSSING) [out=180, in=-90] to (-1,1.5);
	\node() at (DOT) {\bbullet};
\end{scope}
\end{tikzpicture}}
\quad=\quad
\hackcenter{\begin{tikzpicture}
\begin{scope}[shift={(0,0)},rotate=180]
	\draw[thick, ->] (0,0) -- (0,1.5);
	\draw[thick, ->] (.5,0) -- (.5,1.5);
	\draw[thick, color=blue, dashed] (-.75,1.5) .. controls (-.75,1.15) and (-.25,1.15) .. (-.25,1.5);
\end{scope}
\end{tikzpicture}}
\]
holds in any strong supercategorical action.

% - - - - - - - - - - - - - - - - - - - - - - - - - - - - - - - - - - - - - - - -
%
\subsubsection{Half cyclicity for crossings}
%
% - - - - - - - - - - - - - - - - - - - - - - - - - - - - - - - - - - - - - - - -

\begin{prop}
The following relations hold in any strong supercategorical action.
\begin{alignat}{2}\label{eq-half-crossing-cycl}
\hackcenter{
\begin{tikzpicture}[scale=0.7]
  \draw[thick, ->] (-0.5,0) .. controls ++(-0,0.5) and ++(0,-0.5) .. (0.5,1)
      node[pos=0.5, shape=coordinate](X){};
  \draw[thick, ->] (0.5,0) .. controls ++(0,0.5) and ++(0,-0.5) .. (-0.5,1);
  \draw[thick] (0.5,0) .. controls ++(0,-0.5) and ++(0,-0.5) .. (1.5,0);
  \draw[thick] (-0.5,0) .. controls ++(0,-1.5) and ++(0,-1.5) .. (2.5,0);
  \draw[thick, ->-=0.5] (2.5,1) -- (2.5,0);
  \draw[thick, ->-=0.5] (1.5,1) -- (1.5,0);
  \draw[color=blue,  thick, dashed] (X) to [out=180, in=-90](-1.5,1);
  \node at (1.5,-1.5) {$\l$};
\end{tikzpicture}}
\;\; &=
\hackcenter{ \begin{tikzpicture} [scale=0.7]
  \draw[thick,->] (-0.5,1) .. controls ++(0,-0.5) and ++(0,0.5) .. (0.5,0)
      node[pos=0.5, shape=coordinate](X){};
  \draw[thick, ->] (0.5,1) .. controls ++(0,-0.5) and ++(0,0.5) .. (-0.5,0);
  \draw[thick] (-0.5,0) .. controls ++(0,-0.5) and ++(0,-0.5) .. (-1.5,0);
  \draw[thick] (0.5,0) .. controls ++(0,-1.5) and ++(0,-1.5) .. (-2.5,0);
  \draw[thick, ->-=0.5] (-2.5,0) -- (-2.5,1);
  \draw[thick, ->-=0.5] (-1.5,0) -- (-1.5,1);
  \draw[color=blue,  thick, dashed]
    (X) .. controls  ++(1.5,0) and ++(1,0) .. (0,-1.5)
    .. controls ++(-3.5,-.5) and ++(0,-1) .. (-3,1);
  \node at (1.5,-1) {$\l$};
\end{tikzpicture}}
&\qquad
\hackcenter{
\begin{tikzpicture}[scale=0.7]
  \draw[thick] (-0.5,1) .. controls ++(-0,-0.5) and ++(0,0.5) .. (0.5,0)
      node[pos=0.5, shape=coordinate](X){};
  \draw[thick, ->] (0.5,1) .. controls ++(0,-0.5) and ++(0,0.5) .. (-0.5,0);
  \draw[thick] (0.5,0) .. controls ++(0,-0.5) and ++(0,-0.5) .. (1.5,0)
    node[pos=0.5, shape=coordinate](inCAP){};
  \draw[thick] (-0.5,0) .. controls ++(0,-1.5) and ++(0,-1.5) .. (2.5,0)
   node[pos=0.44, shape=coordinate](L){}
   node[pos=0.51, shape=coordinate](R){}
   node[pos=0.5, shape=coordinate](outCAP){};
  \draw[thick, ->-=0.5] (2.5,0) -- (2.5,1);
  \draw[thick, ->-=0.5] (1.5,0) -- (1.5,1);
  \draw[color=blue,  thick, dashed]
    (X) .. controls ++(.4,-.2) and ++(0,-.5) .. (1,1);
  \draw[color=blue, thick, double distance=1pt, dashed]
    (inCAP) .. controls++(0,.7) and ++(0,.7) .. (.5,-0.5)
     .. controls ++(0,-0.3) and ++(.1,.6) .. (outCAP);
       \node at (2,-1.5) {$\l$};
\end{tikzpicture}}
\;\; &= \;\; (-1)^{\l+1} \;\;
\hackcenter{
\begin{tikzpicture}[scale=0.7]
  \draw[thick, ->] (0.5,0) .. controls ++(-0,0.5) and ++(0,-0.5) .. (-0.5,1)
      node[pos=0.5, shape=coordinate](X){};
  \draw[thick, ->] (-0.5,0) .. controls ++(0,0.5) and ++(0,-0.5) .. (0.5,1);
  \draw[thick] (-0.5,0) .. controls ++(0,-0.5) and ++(0,-0.5) .. (-1.5,0)
     node[pos=0.5, shape=coordinate](inCAP){};
  \draw[thick] (0.5,0) .. controls ++(0,-1.5) and ++(0,-1.5) .. (-2.5,0)
   node[pos=0.44, shape=coordinate](L){}
   node[pos=0.51, shape=coordinate](R){}
     node[pos=0.55, shape=coordinate](outCAP){};;
  \draw[thick, ->-=0.5] (-2.5,1) -- (-2.5,0);
  \draw[thick, ->-=0.5] (-1.5,1) -- (-1.5,0);
    \draw[color=blue, thick, double distance=1pt, dashed]
    (inCAP) .. controls++(0,.7) and ++(0,.7) .. (-1.5,-0.5)
     .. controls ++(0,-0.3) and ++(0,.5) .. (outCAP);
  \draw[color=blue,  thick, dashed] (X) to [out=180, in=-90](-1,1);
       \node at (0,-1.5) {$\l$};
\end{tikzpicture}} \nn
\\
\hackcenter{
\begin{tikzpicture} [scale=0.7]
  \draw[thick, ->] (-0.5,0) .. controls ++(-0,-0.5) and ++(0,0.5) .. (0.5,-1)
      node[pos=0.5, shape=coordinate](X){};
  \draw[thick, ->] (0.5,0) .. controls ++(0,-0.5) and ++(0,0.5) .. (-0.5,-1);
  \draw[thick] (0.5,0) .. controls ++(0,0.5) and ++(0,0.5) .. (1.5,0);
  \draw[thick] (-0.5,0) .. controls ++(0,1.5) and ++(0,1.5) .. (2.5,0);
  \draw[thick, ->-=0.5] (2.5,-1) -- (2.5,0);
  \draw[thick, ->-=0.5] (1.5,-1) -- (1.5,0);
  \draw[color=blue,  thick, dashed] (X) to [out=0, in=90](1,-1);
  \node at (2.7,1) {$\l$};
\end{tikzpicture}}
\;\; &= \;\; \;\;
\hackcenter{ \begin{tikzpicture} [scale=0.7]
  \draw[thick,->] (-0.5,-1) .. controls ++(0,0.5) and ++(0,-0.5) .. (0.5,0)
      node[pos=0.5, shape=coordinate](X){};
  \draw[thick, ->] (0.5,-1) .. controls ++(0,0.5) and ++(0,-0.5) .. (-0.5,0);
  \draw[thick] (-0.5,0) .. controls ++(0,0.5) and ++(0,0.5) .. (-1.5,0);
  \draw[thick] (0.5,0) .. controls ++(0,1.5) and ++(0,1.5) .. (-2.5,0);
  \draw[thick, ->-=0.5] (-2.5,0) -- (-2.5,-1);
  \draw[thick, ->-=0.5] (-1.5,0) -- (-1.5,-1);
  \draw[color=blue,  thick, dashed]  (X) .. controls  ++(-.4,0.2) and ++(0,0.5) .. (-1,-1);
  \node at (1,1) {$\l$};
\end{tikzpicture} }
& \qquad
\hackcenter{ \begin{tikzpicture}[scale=0.7]
  \draw[thick,->] (0.5,-1) .. controls ++(0,0.5) and ++(0,-0.5) .. (-0.5,0)
      node[pos=0.5, shape=coordinate](X){};
  \draw[thick, ->] (-0.5,-1) .. controls ++(0,0.5) and ++(0,-0.5) .. (0.5,0);
  \draw[thick] (0.5,0) .. controls ++(0,0.5) and ++(0,0.5) .. (1.5,0)
    node[pos=0.5, shape=coordinate](inCAP){};;
  \draw[thick] (-0.5,0) .. controls ++(0,1.5) and ++(0,1.5) .. (2.5,0)
     node[pos=0.51, shape=coordinate](L){}
   node[pos=0.56, shape=coordinate](R){}
     node[pos=0.5, shape=coordinate](outCAP){};;
  \draw[thick, ->-=0.5] (2.5,0) -- (2.5,-1);
  \draw[thick, ->-=0.5] (1.5,0) -- (1.5,-1);
  \draw[color=blue,  thick, dashed]
    (X) .. controls  ++(-1.5,0) and ++(0,-1) .. (-0.5,1.75);
  \draw[color=blue, thick, double distance=1pt, dashed]
    (inCAP) .. controls++(0,.7) and ++(0,-.7) .. (0,1)
     .. controls ++(0,.5) and ++(0,.5) .. (outCAP);
  \node at (2.3,1.5) {$\l$};
\end{tikzpicture}}
\;\; &= \;\; (-1)^{\l-1} \;\;
\hackcenter{ \begin{tikzpicture}[scale=0.7]
  \draw[thick,->] (-0.5,0) .. controls ++(0,-0.5) and ++(0,0.5) .. (0.5,-1)
      node[pos=0.5, shape=coordinate](X){};
  \draw[thick, ->] (0.5,0) .. controls ++(0,-0.5) and ++(0,0.5) .. (-0.5,-1);
  \draw[thick] (-0.5,0) .. controls ++(0,0.5) and ++(0,0.5) .. (-1.5,0)
    node[pos=0.5, shape=coordinate](inCAP){};
  \draw[thick] (0.5,0) .. controls ++(0,1.5) and ++(0,1.5) .. (-2.5,0)
     node[pos=0.51, shape=coordinate](L){}
   node[pos=0.56, shape=coordinate](R){}
     node[pos=0.5, shape=coordinate](outCAP){};
  \draw[thick, ->-=0.5] (-2.5,-1) -- (-2.5,0);
  \draw[thick, ->-=0.5] (-1.5,-1) -- (-1.5,0);
  \draw[color=blue,  thick, dashed]
    (X) .. controls  ++(1.5,-.2) and ++(0,-1) .. (.5,1.75);
  \draw[color=blue, thick, double distance=1pt, dashed]
    (inCAP) .. controls++(0,.7) and ++(0,-.7) .. (-2,1)
     .. controls ++(0,.5) and ++(0,.5) .. (outCAP);
  \node at (-.25,1.5) {$\l$};
\end{tikzpicture}}
\end{alignat}
\end{prop}

\begin{proof}
Using Lemma~\ref{lem:EE} together with the biadjoint structure, it follows that each of the equations above must hold up to a scalar multiple since each of the hom-spaces are non-zero and each map involved is adjoint to a crossing which is assumed to be non-zero.  For example,
\begin{align*}
 &\Hom(\oneb_{\l+4}, \Pi \Ett^2 \Ftt^2 \oneb_{\l+4} \la -2(\l +3)\ra)  \\ &=
  \Hom(\oneb_{\l+4}(\Ftt\oneb_{\l+4})^R, \Pi \Ett^2 \onebl  \la -2(\l +2)\ra)
  =
  \Hom(\Ett^2\onebl, \Pi \Ett^2 \onebl  \la -2\ra) \cong \Bbbk.
\end{align*}
To solve for the exact scalar we (pre)compose with a dot, or a dot and a crossing, and simplify using the odd nilHecke relations together with the dot cyclicity relations from Corollary~\ref{cor-dot-cyclicity}. For example, to prove the last equality in \eqref{eq-half-crossing-cycl} assume that
\begin{equation} \label{eq-c2}
\hackcenter{ \begin{tikzpicture}[scale=0.7]
  \draw[thick,->] (0.5,-1) .. controls ++(0,0.5) and ++(0,-0.5) .. (-0.5,0)
      node[pos=0.5, shape=coordinate](X){};
  \draw[thick, ->] (-0.5,-1) .. controls ++(0,0.5) and ++(0,-0.5) .. (0.5,0);
  \draw[thick] (0.5,0) .. controls ++(0,0.5) and ++(0,0.5) .. (1.5,0)
    node[pos=0.5, shape=coordinate](inCAP){};;
  \draw[thick] (-0.5,0) .. controls ++(0,1.5) and ++(0,1.5) .. (2.5,0)
     node[pos=0.51, shape=coordinate](L){}
   node[pos=0.56, shape=coordinate](R){}
     node[pos=0.5, shape=coordinate](outCAP){};;
  \draw[thick, ->-=0.5] (2.5,0) -- (2.5,-1);
  \draw[thick, ->-=0.5] (1.5,0) -- (1.5,-1);
  \draw[color=blue,  thick, dashed]
    (X) .. controls  ++(-1.5,0) and ++(0,-1) .. (-0.5,1.75);
  \draw[color=blue, thick, double distance=1pt, dashed]
    (inCAP) .. controls++(0,.7) and ++(0,-.7) .. (0,1)
     .. controls ++(0,.5) and ++(0,.5) .. (outCAP);
  \node at (2.3,1.5) {$\l$};
\end{tikzpicture}}
\quad = \quad \kappa \;\;
\hackcenter{ \begin{tikzpicture}[scale=0.7]
  \draw[thick,->] (-0.5,0) .. controls ++(0,-0.5) and ++(0,0.5) .. (0.5,-1)
      node[pos=0.5, shape=coordinate](X){};
  \draw[thick, ->] (0.5,0) .. controls ++(0,-0.5) and ++(0,0.5) .. (-0.5,-1);
  \draw[thick] (-0.5,0) .. controls ++(0,0.5) and ++(0,0.5) .. (-1.5,0)
    node[pos=0.5, shape=coordinate](inCAP){};
  \draw[thick] (0.5,0) .. controls ++(0,1.5) and ++(0,1.5) .. (-2.5,0)
     node[pos=0.51, shape=coordinate](L){}
   node[pos=0.56, shape=coordinate](R){}
     node[pos=0.5, shape=coordinate](outCAP){};
  \draw[thick, ->-=0.5] (-2.5,-1) -- (-2.5,0);
  \draw[thick, ->-=0.5] (-1.5,-1) -- (-1.5,0);
  \draw[color=blue,  thick, dashed]
    (X) .. controls  ++(1.5,-.2) and ++(0,-1) .. (.5,1.75);
  \draw[color=blue, thick, double distance=1pt, dashed]
    (inCAP) .. controls++(0,.7) and ++(0,-.7) .. (-2,1)
     .. controls ++(0,.5) and ++(0,.5) .. (outCAP);
  \node at (-2.5,1.5) {$\l$};
\end{tikzpicture}}
\end{equation}
for some non-zero scalar $\kappa$.  This implies
\[
\hackcenter{ \begin{tikzpicture}[scale=0.7]
  \draw[thick,<-] (2.5,-2) -- (2.5,-1);
  \draw[thick, <-] (1.5,-2) -- (1.5,-1);
    \draw[thick] (0.5,-2) .. controls ++(0,0.5) and ++(0,-0.5) .. (-0.5,-1)
      node[pos=1, shape=coordinate](DOT){}
      node[pos=0.5, shape=coordinate](Y){};
  \draw[thick, ->] (-0.5,-2) .. controls ++(0,0.5) and ++(0,-0.5) .. (0.5,-1);
  \draw[thick,->] (0.5,-1) .. controls ++(0,0.5) and ++(0,-0.5) .. (-0.5,0)
      node[pos=0.5, shape=coordinate](X){};
  \draw[thick, ->] (-0.5,-1) .. controls ++(0,0.5) and ++(0,-0.5) .. (0.5,0);
  \draw[thick] (0.5,0) .. controls ++(0,0.5) and ++(0,0.5) .. (1.5,0)
    node[pos=0.5, shape=coordinate](inCAP){};;
  \draw[thick] (-0.5,0) .. controls ++(0,1.5) and ++(0,1.5) .. (2.5,0)
     node[pos=0.51, shape=coordinate](L){}
   node[pos=0.56, shape=coordinate](R){}
     node[pos=0.5, shape=coordinate](outCAP){};;
  \draw[thick] (2.5,0) -- (2.5,-1);
  \draw[thick] (1.5,0) -- (1.5,-1);
  \draw[color=blue,  thick, dashed]
    (X) .. controls  ++(-1.5,0) and ++(0,-1) .. (-0.5,1.75);
  \draw[color=blue, thick, double distance=1pt, dashed]
    (inCAP) .. controls++(0,.7) and ++(0,-.7) .. (0,1)
     .. controls ++(0,.5) and ++(0,.5) .. (outCAP);
  \draw[thick, color=blue, dashed] (DOT) .. controls ++(-1.2,.5) and ++(0,-.5) .. (-1,1.75);
  \node at (2.3,1.5) {$\l$};
  \node at (DOT){\bbullet};
\end{tikzpicture}}
\quad = \quad \kappa \;\;
\hackcenter{ \begin{tikzpicture}[scale=0.7]
   \draw[thick,<-] (0.5,-2) -- (0.5,-1);
  \draw[thick, <-] (-0.5,-2) -- (-0.5,-1);
    \draw[thick] (-1.5,-2) .. controls ++(0,0.5) and ++(0,-0.5) .. (-2.5,-1)
    node[pos=1, shape=coordinate](DOT){}
      node[pos=0.5, shape=coordinate](Y){};
  \draw[thick] (-2.5,-2) .. controls ++(0,0.5) and ++(0,-0.5) .. (-1.5,-1);
  \draw[thick] (-0.5,0) .. controls ++(0,-0.5) and ++(0,0.5) .. (0.5,-1)
      node[pos=0.5, shape=coordinate](X){};
  \draw[thick] (0.5,0) .. controls ++(0,-0.5) and ++(0,0.5) .. (-0.5,-1);
  \draw[thick] (-0.5,0) .. controls ++(0,0.5) and ++(0,0.5) .. (-1.5,0)
    node[pos=0.5, shape=coordinate](inCAP){};
  \draw[thick] (0.5,0) .. controls ++(0,1.5) and ++(0,1.5) .. (-2.5,0)
     node[pos=0.51, shape=coordinate](L){}
   node[pos=0.56, shape=coordinate](R){}
     node[pos=0.5, shape=coordinate](outCAP){};
  \draw[thick, ->-=0.5] (-2.5,-1) -- (-2.5,0);
  \draw[thick, ->-=0.5] (-1.5,-1) -- (-1.5,0);
  \draw[color=blue,  thick, dashed]
    (X) .. controls  ++(1.5,-.2) and ++(0,-1) .. (.5,1.75);
  \draw[color=blue, thick, double distance=1pt, dashed]
    (inCAP) .. controls++(0,.7) and ++(0,-.7) .. (-2,1)
     .. controls ++(0,.5) and ++(0,.5) .. (outCAP);
  \draw[thick, color=blue, dashed] (DOT) .. controls ++(-1,.3) and ++(0,-.5) .. (-3.25,1.75);
  \node at (-2.5,1.5) {$\l$};
   \node at (DOT){\bbullet};
\end{tikzpicture}}
\]
which after simplification implies that $\kappa = (-1)^{\l-1}$.
\end{proof}
% - - - - - - - - - - - - - - - - - - - - - - - - - - - - - - - - - - - - - - - -
%
\subsubsection{Cyclicity for sideways crossings} \label{subsec-half-sideways}
%
% - - - - - - - - - - - - - - - - - - - - - - - - - - - - - - - - - - - - - - - -

We now fix the free scalar multiple in the definition of sideways crossings from \eqref{eq:sideways}.  Define sideways crossings  using the adjoint structure
\begin{equation}
\hackcenter{\begin{tikzpicture}[scale=0.7]
	\draw[thick, ->] (.5,0) .. controls ++(0,.75) and ++(0,-.75) .. (-.5,1.5);
	\draw[thick, ->] (.5,1.5) .. controls ++(0,-.75) and ++(0,.75)..(-.5,0)
		node[pos=.5, shape=coordinate](CROSSING){};
	\draw[thick, color=blue, dashed] (CROSSING)  to (0,-0);
 \node at (1,.75) {$\l$};
\end{tikzpicture} }
\quad := \quad
\hackcenter{
\begin{tikzpicture}[scale=0.7]
  \draw[thick, ->] (-0.5,0) .. controls ++(-0,0.5) and ++(0,-0.5) .. (0.5,1)
      node[pos=0.5, shape=coordinate](X){};
  \draw[thick, ->] (0.5,0) .. controls ++(0,0.5) and ++(0,-0.5) .. (-0.5,1);
  \draw[thick] (0.5,0) .. controls ++(0,-0.5) and ++(0,-0.5) .. (1.5,0);
  \draw[thick] (.5,2) -- (.5,1);
  \draw[thick, ->-=0.5] (1.5,2) -- (1.5,0);
  \draw[thick] (-0.5,1) .. controls ++(0,0.5) and ++(0,0.5) .. (-1.5,1);
  \draw[thick] (-.5,0) -- (-.5,-1);
  \draw[thick, ->-=0.5] (-1.5,1) -- (-1.5,-1);
  \draw[color=blue,  thick, dashed]
   (X) .. controls ++(-1,.4) and ++(0,1).. (-1,-1);
      \node at (1.8,-0.8) {$\l$};
\end{tikzpicture}}
\qquad
\qquad
\hackcenter{\begin{tikzpicture}[scale=0.7]
	\draw[thick, ->] (-.5,0) .. controls ++(0,.75) and ++(0,-.75) .. (.5,1.5);
	\draw[thick, ->] (-.5,1.5) .. controls ++(0,-.75) and ++(0,.75)..(.5,0)
		node[pos=.5, shape=coordinate](CROSSING){};
	\draw[thick, color=blue, dashed] (CROSSING)  to (0,1.5);
    \node at (1,.75) {$\l$};
\end{tikzpicture} }
\quad := \quad
\hackcenter{
\begin{tikzpicture}[scale=0.7]
  \draw[thick, ->] (0.5,0) .. controls ++(-0,0.5) and ++(0,-0.5) .. (-0.5,1)
      node[pos=0.5, shape=coordinate](X){};
  \draw[thick, ->] (-0.5,0) .. controls ++(0,0.5) and ++(0,-0.5) .. (0.5,1);
  \draw[thick] (-0.5,0) .. controls ++(0,-0.5) and ++(0,-0.5) .. (-1.5,0)
      node[pos=.48, shape=coordinate](RD){}
      node[pos=.48, shape=coordinate](LD){}
      node[pos=.5, shape=coordinate](bCROSS){};
  \draw[thick] (-.5,2) -- (-.5,1);
  \draw[thick, ->-=0.5] (-1.5,2) -- (-1.5,0);
  \draw[thick] (0.5,1) .. controls ++(0,0.5) and ++(0,0.5) .. (1.5,1)
     node[pos=.5, shape=coordinate](tCROSS){};;
  \draw[thick] (.5,0) -- (.5,-1);
  \draw[thick, ->-=0.5] (1.5,1) -- (1.5,-1);
  \draw[color=blue,  thick, dashed]
   (X) .. controls ++(-1.3,0) and ++(0,-1).. (-1,2);
  \draw[color=blue, thick, double distance=1pt, dashed]
    (bCROSS) .. controls ++(-.1,3) and ++(0.1,.75) ..(tCROSS);
      \node at (2.1,.5) {$\l$};
   \node[blue] at (1.5,1.8) {$\scs \l-1$};
\end{tikzpicture}}
\end{equation}
for all weights $\lambda$.  It follows immediately from these definitions that the following relations hold.
\begin{align}
\hackcenter{  \begin{tikzpicture}[scale=0.8]
  \draw[thick, ->] (-0.5,0).. controls ++(-0,-0.5) and ++(0,0.5) ..(0.5,-1)
      node[pos=0.5, shape=coordinate](X){};
  \draw[thick] (0.5,0).. controls ++(0,-0.5) and ++(0,0.5) ..(-0.5,-1);
  \draw[thick] (-0.5,0) .. controls ++(0,0.5) and ++(0,0.5) .. (-1.5,0)
      node[pos=.5, shape=coordinate](CUP){};
  \draw[thick, ->-=0.5] (-1.5,-1) -- (-1.5,0);
  \draw[thick, ->] (.5,0) -- (.5,0.75);
  \draw[color=blue,  thick, dashed]   (X) -- (0,.75);
   \draw[color=blue, thick, double distance=1pt, dashed]   (CUP) -- (-1,.75);
      \node at (1,-.5) {$\l$};
\end{tikzpicture} }
\quad &= \quad
\hackcenter{
  \begin{tikzpicture}[scale=0.8]
  \draw[thick, ->](-0.5,-1) .. controls ++(-0,0.5) and ++(0,-0.5) ..(0.5,-.25)
      node[pos=0.5, shape=coordinate](X){};
  \draw[thick] (-0.5,-.25) .. controls ++(0,-0.5) and ++(0,0.5) .. (0.5,-1);
  \draw[thick] (0.5,-.25) .. controls ++(0,0.3) and ++(0,0.3) .. (1.5,-.25)
      node[pos=.5, shape=coordinate](CUP){};
  \draw[thick, ->-=0.5] (1.5,-.25) -- (1.5,-1);
  \draw[thick,, ->] (-.5,-.25) to[out=90, in=-90] (0.5,0.75);
  \draw[color=blue,  thick, dashed]
    (X) .. controls ++(-.5,0) and ++(0,-.5) .. (-.75,0)
    .. controls ++(0,.3) and ++(0,-.4) .. (0,.75);
   \draw[color=blue, thick, double distance=1pt, dashed]
   (CUP) .. controls ++(0,.3) and ++(0,-.6).. (-.75,.75);
      \node at (1.3,.5) {$\l$};
\end{tikzpicture} }
\quad &
\hackcenter{
  \begin{tikzpicture}[scale=0.8]
  \draw[thick, ->](-0.5,1.25) .. controls ++(-0,-0.5) and ++(0,0.5) ..(0.5,0)
      node[pos=0.5, shape=coordinate](X){};
  \draw[thick, ->] (-0.5,0) .. controls ++(0,0.5) and ++(0,-0.5) .. (0.5,1.25);
  \draw[thick] (0.5,0) .. controls ++(0,-0.5) and ++(0,-0.5) .. (1.5,0)
      node[pos=.5, shape=coordinate](CUP){};
  \draw[thick, ->-=0.5] (1.5,0 ) -- (1.5,1.25);
  \draw[thick] (-.5,0) -- (-.5,-0.5);
  \draw[color=blue, thick, double distance=1pt, dashed]
   (CUP) -- (1,1.25);
  \draw[color=blue,  thick, dashed]
   (X) to (0,1.25);
      \node at (0,-.25) {$\l$};
\end{tikzpicture} }
\quad &= \quad
\hackcenter{  \begin{tikzpicture}[scale=0.8]
  \draw[thick, ->] (-0.5,0) .. controls ++(-0,0.5) and ++(0,-0.5) .. (0.5,1.25)
      node[pos=0.5, shape=coordinate](X){};
  \draw[thick, ->] (0.5,0) .. controls ++(0,0.5) and ++(0,-0.5) .. (-0.5,1.25);
  \draw[thick] (-0.5,0) .. controls ++(0,-0.5) and ++(0,-0.5) .. (-1.5,0)
      node[pos=.5, shape=coordinate](CUP){};
  \draw[thick, ->-=0.5] (-1.5,1.25) -- (-1.5,0);
  \draw[thick] (.5,0) -- (.5,-0.5);
  \draw[color=blue, thick, double distance=1pt, dashed]
   (CUP) .. controls ++(0,1.25) and ++(0,-.5) .. (0,1.25);
  \draw[color=blue,  thick, dashed]
   (X) .. controls ++(-1.3,0) and ++(0,-.3).. (-1,1.25);
      \node at (1,.5) {$\l$};
\end{tikzpicture} }
\\
\hackcenter{
  \begin{tikzpicture}[scale=0.8]
  \draw[thick, ->] (0.5,0).. controls ++(-0,-0.5) and ++(0,0.5) ..(-0.5,-1.25)
      node[pos=0.5, shape=coordinate](X){};
  \draw[thick] (-0.5,0) .. controls ++(0,-0.5) and ++(0,0.5) .. (0.5,-1.25);
  \draw[thick] (0.5,0) .. controls ++(0,0.5) and ++(0,0.5) .. (1.5,0)
      node[pos=.5, shape=coordinate](CUP){};
  \draw[thick, ->-=0.5] (1.5,-1.25) -- (1.5,0);
  \draw[thick,, ->] (-.5,0) -- (-.5,0.5);
  \draw[color=blue,  thick, dashed]
   (X) -- (0,-1.25);
      \node at (-1,-.5) {$\l+2$};
\end{tikzpicture} }
\quad &= \quad
\hackcenter{  \begin{tikzpicture}[scale=0.8]
  \draw[thick, ->] (0.5,-1.25) .. controls ++(-0,0.5) and ++(0,-0.5) ..(-0.5,0)
      node[pos=0.5, shape=coordinate](X){};
  \draw[thick] (-0.5,-1.25).. controls ++(0,0.5) and ++(0,-0.5) ..(0.5,0);
  \draw[thick] (-0.5,0) .. controls ++(0,0.5) and ++(0,0.5) .. (-1.5,0)
      node[pos=.5, shape=coordinate](CUP){};
  \draw[thick, ->] (-1.5,0) -- (-1.5,-1.25);
  \draw[thick, ->] (.5,0) -- (.5,0.5);
  \draw[color=blue,  thick, dashed]
   (X) .. controls ++(-.3,.1) and ++(0,.5) .. (-1,-1.25);
      \node at (1,-.5) {$\l$};
\end{tikzpicture} }
\quad &
\hackcenter{
  \begin{tikzpicture}[scale=0.8]
  \draw[thick, <-](-0.5,1.25) .. controls ++(-0,-0.5) and ++(0,0.5) ..(0.5,0)
      node[pos=0.5, shape=coordinate](X){};
  \draw[thick, ->] (-0.5,0) .. controls ++(0,0.5) and ++(0,-0.5) .. (0.5,1.25);
  \draw[thick] (0.5,0) .. controls ++(0,-0.5) and ++(0,-0.5) .. (1.5,0)
      node[pos=.5, shape=coordinate](CUP){};
  \draw[thick, ->-=0.5] (1.5,1.25) -- (1.5, 0);
  \draw[thick] (-.5,0) -- (-.5,-0.5);
  \draw[color=blue,  thick, dashed]
   (X) .. controls ++(-.5,.3) and ++(0,.5)..(-1,-.5);
      \node at (1,.5) {$\l$};
\end{tikzpicture} }
\quad &= \quad
\hackcenter{  \begin{tikzpicture}[scale=0.8]
  \draw[thick, ->] (0.5,1.25) .. controls ++(-0,-0.5) and ++(0,0.5) ..(-0.5,0)
      node[pos=0.5, shape=coordinate](X){};
  \draw[thick, ->] (0.5,0).. controls ++(0,0.5) and ++(0,-0.5) ..(-0.5,1.25);
  \draw[thick] (-0.5,0) .. controls ++(0,-0.5) and ++(0,-0.5) .. (-1.5,0)
      node[pos=.5, shape=coordinate](CUP){};
  \draw[thick, ->-=0.5] (-1.5,0) -- (-1.5,1.25);
  \draw[thick] (.5,0) -- (.5,-0.5);
  \draw[color=blue,  thick, dashed]
   (X) -- (0,-.5);
      \node at (1,.5) {$\l$};
\end{tikzpicture} }
\end{align}
Similar equations hold for a downward oriented line in the middle of a cap and cup.

It is straightforward to derive directly from the definitions that the following equalities hold.
\begin{equation}\label{eq:leftdotslide}
\hackcenter{\begin{tikzpicture}
	\draw[thick, ->] (.5,1.25).. controls ++(0,-.75) and ++(0,.75) .. (-.5,0)
		node[pos=.5, shape=coordinate](CROSSING){}
		node[pos=.75, shape=coordinate](DOT){};
	\draw[thick, ->] (.5,0) .. controls ++(0,.75) and ++(0,-.75) .. (-.5,1.25);
	\draw[thick, color=blue, dashed]
   (DOT) to[out=-20, in=90] (-.2,0);
	\draw[thick, color=blue, dashed] (CROSSING)  to[out=-90, in=90] (.2,0);
	\node() at (DOT) {\bbullet};
\end{tikzpicture}}
\quad-\quad
\hackcenter{\begin{tikzpicture}
	\draw[thick, ->] (.5,1.25).. controls ++(0,-.75) and ++(0,.75) .. (-.5,0)
		node[pos=.5, shape=coordinate](CROSSING){}
		node[pos=.25, shape=coordinate](DOT){};
	\draw[thick, ->] (.5,0) .. controls ++(0,.75) and ++(0,-.75) .. (-.5,1.25);
	\draw[thick, color=blue, dashed]
   (DOT) .. controls ++(.3,-.3) and ++(0,.4) .. (-.2,0);
	\draw[thick, color=blue, dashed] (CROSSING)  to[out=-90, in=90] (.2,0);
	\node() at (DOT) {\bbullet};
\end{tikzpicture}}
\quad = \quad
\hackcenter{\begin{tikzpicture}
	\draw[thick, ->] (.5,1.25).. controls ++(0,-.75) and ++(0,.75) .. (-.5,0);
	\draw[thick, ->] (.5,0) .. controls ++(0,.75) and ++(0,-.75) .. (-.5,1.25)
        node[pos=.5, shape=coordinate](CROSSING){}
		node[pos=.2, shape=coordinate](DOT){};
	\draw[thick, color=blue, dashed]
   (DOT) .. controls ++(-.1,.2) and ++(0,.4) .. (0,0);
	\draw[thick, color=blue, dashed] (CROSSING)  to[out=-90, in=90] (-.25,0);
	\node() at (DOT) {\bbullet};
\end{tikzpicture}}
\quad-
\hackcenter{\begin{tikzpicture}
	\draw[thick, ->] (.5,1.25).. controls ++(0,-.75) and ++(0,.75) .. (-.5,0);
	\draw[thick, ->] (.5,0) .. controls ++(0,.75) and ++(0,-.75) .. (-.5,1.25)
        node[pos=.5, shape=coordinate](CROSSING){}
		node[pos=.75, shape=coordinate](DOT){};
	\draw[thick, color=blue, dashed]
   (DOT) .. controls ++(-1.3,.5) and ++(0,.4) .. (.2,0);
	\draw[thick, color=blue, dashed] (CROSSING)  to[out=-90, in=90] (-.25,0);
	\node() at (DOT) {\bbullet};
\end{tikzpicture}} \quad = \quad
\hackcenter{\begin{tikzpicture}
	\draw[thick, ->] (.5,1.25).. controls ++(0,-.6) and ++(0,-.6) .. (-.5,1.25);
	\draw[thick, ->] (.5,0) .. controls ++(0,.6) and ++(0,.6) .. (-.5,0);
	\draw[thick, color=blue, dashed]
   (.2,0) .. controls ++(0,.3) and ++(0,.3) .. (-.2,0);
\end{tikzpicture}}
\end{equation}

\begin{equation}\label{eq:rightdotslide}
\hackcenter{\begin{tikzpicture}
\begin{scope}[shift={(0,0)},rotate=180]
	\draw[thick, ->] (.5,1.25).. controls ++(0,-.75) and ++(0,.75) .. (-.5,0)
		node[pos=.5, shape=coordinate](CROSSING){}
		node[pos=.75, shape=coordinate](DOT){};
	\draw[thick, ->] (.5,0) .. controls ++(0,.75) and ++(0,-.75) .. (-.5,1.25);
	\draw[thick, color=blue, dashed]
   (DOT) to[out=-20, in=90] (-.2,0);
	\draw[thick, color=blue, dashed] (CROSSING)  to[out=-90, in=90] (.2,0);
	\node() at (DOT) {\bbullet};
\end{scope}
\end{tikzpicture}}
\quad-\quad
\hackcenter{\begin{tikzpicture}
\begin{scope}[shift={(0,0)},rotate=180]
	\draw[thick, ->] (.5,1.25).. controls ++(0,-.75) and ++(0,.75) .. (-.5,0)
		node[pos=.5, shape=coordinate](CROSSING){}
		node[pos=.25, shape=coordinate](DOT){};
	\draw[thick, ->] (.5,0) .. controls ++(0,.75) and ++(0,-.75) .. (-.5,1.25);
	\draw[thick, color=blue, dashed]
   (DOT) to[out=-20, in=90] (-.2,0);
	\draw[thick, color=blue, dashed] (CROSSING)  to[out=-90, in=90] (.2,0);
	\node() at (DOT) {\bbullet};\end{scope}
\end{tikzpicture}}
\quad = \quad
\hackcenter{\begin{tikzpicture}
	\draw[thick, ->] (-.5,1.5).. controls ++(0,-.8) and ++(0,-.8) .. (.5,1.5)
        node[pos=.5, shape=coordinate](tCUP){};
	\draw[thick, ->] (-.5,0) .. controls ++(0,.4) and ++(0,.4) .. (.5,0)
        node[pos=.5, shape=coordinate](bCAP){};
 \draw[color=blue, thick, double distance=1pt, dashed]
    (tCUP) .. controls ++(0,.4) and ++(0,.4)..(.5,.7)
     .. controls ++(0,-.2) and ++(0,.3) .. (bCAP) ;
	\draw[thick, color=blue, dashed]
   (.2,1.5) .. controls ++(0,-.25) and ++(0,-.25) .. (-.2,1.5);%\end{scope}
\end{tikzpicture}}
\end{equation}
Note that the other version of \eqref{eq:rightdotslide} is more complicated and the formula for sliding a dot on the downward oriented strand has additional terms.
% ==============================================================================
%
\section{Proof of the extended ${\mf{sl}}_2$ relations}\label{sec:proofsl2}
%
% ==============================================================================

In this section we show that the formal inverses of the $\mf{sl}_2$ isomorphisms defined in Corollary ~\ref{cor:1} agree with the version of the inverses given by the 2-category $\Uc$.

% -------------------------------------------------------------------------------
%
\subsection{Fake bubbles} \label{subsec:fake-bubbles}
%
% -------------------------------------------------------------------------------

We introduce a shorthand notation for representing certain 2-morphisms in $\Uc$.   These 2-morphisms are defined inductively by an equation analogous to the equation
\[
 \sum_{r=0}^m (-1)^re_r h_{m-r} = \delta_{m,r}
\]
relating elementary and complete symmetric functions.

% - - - - - - - - - - - - - - - - - - - - - - - - - - - - - - - - - - - - - - - -
%
%
\subsubsection{The case of $\lambda >0$}
%
% - - - - - - - - - - - - - - - - - - - - - - - - - - - - - - - - - - - - - - - -
%
For all $\l>0$ the equation
\begin{equation} \label{eq:fake-bubble-p}
\sum_{j=0}^{m}
(-1)^j
\xy
(0,0)*{
\begin{tikzpicture}[scale=0.9]
  %% Make a coupon filled with a given label place at coordinate (-2,0.75) name the node Fj
 \node[draw, thick, fill=blue!20,rounded corners=4pt,inner sep=3pt]
   (Fj) at (1.5,0) {$B_j$};
 \draw[color=blue, thick, double distance=1pt, dashed] (Fj) to  (1.5,1.25);
    \node[blue] at (1.8,0.8){$\scs j$\;};
  \draw[thick, ->] (-0.5,0) .. controls (-0.5,0.8) and (0.5,0.8) .. (0.5,0)
      node[pos=0.5, shape=coordinate](X){}
      node[pos=0.1, shape=coordinate](Y){};
  \draw[thick] (-0.5,0) .. controls (-0.5,-0.8) and (0.5,-0.8) .. (0.5,0)
      node[pos=0.1, shape=coordinate](Z){};
  \draw[color=blue, thick, double distance=1pt, dashed] (X) .. controls++(0,.65) and ++(-.65,.3) .. (Y) node[pos=0.15,right]{$\scs \l-1$\;};
  \draw[color=blue, thick, double distance=1pt, dashed] (Z) to[out=180, in=90] (-1,1.25) ;
  \node[blue] at (-.4,1.2){$\scs m-j$\;};
  \node at (Y) {\bbullet};\node at (Z) {\bbullet};
  \node at (2.3,0) {$\l$};
\end{tikzpicture}
};
\endxy \;\; = \;\; (-1)^m  \delta_{m,0}\oneb_{\onebl}
\end{equation}
inductively defines 2-morphisms $B_j \maps \onebl \to \Pi^j \onebl \la 2j\ra$ for $0 \leq j \leq m$.
It is clear that this definition is independent of the weight $\l$ for all $\lambda>0$.

\begin{example} Several examples are given below.
\begin{align*}
 \hackcenter{
 \begin{tikzpicture}[scale=0.8]
 \node[draw, thick, fill=blue!20,rounded corners=4pt,inner sep=3pt]
   (Fj) at (-0,0) {$B_0$};
  \node at (.7,0) {$\l$};
\end{tikzpicture} }
& \;\; =\;\;
\oneb_{\onebl}  \\
 \hackcenter{
 \begin{tikzpicture}[scale=0.9]
 \node[draw, thick, fill=blue!20,rounded corners=4pt,inner sep=3pt]
   (Fj) at (-0,0) {$B_1$};
 \draw[color=blue, thick, dashed] (Fj) to  (-0,1.25);
    \node[blue] at (.3,0.8){};
  \node at (.7,0) {$\l$};
\end{tikzpicture} }
& \;\; =\;\;
\hackcenter{\begin{tikzpicture}
  \draw[thick, ->] (-0.4,0) .. controls ++(-0,0.6) and ++(0,0.6) .. (0.4,0)
      node[pos=0.5, shape=coordinate](X){}
      node[pos=0.1, shape=coordinate](Y){};
  \draw[thick] (-0.4,0) .. controls ++(0,-0.6) and ++(0,-0.6) .. (0.4,0)
      node[pos=0.1, shape=coordinate](Z){};
  \draw[color=blue, thick, double distance=1pt, dashed] (X) .. controls++(0,.65) and ++(-.65,.3) .. (Y) node[pos=0.15,right]{$\scs \l-1$\;};
  \draw[color=blue, thick, dashed] (Z) to[bend left] (-1,1) ;
  \node at (Y) {\bbullet};
  \node at (Z) {\bbullet};
  \node at (.8,-.25) {$\l$};
\end{tikzpicture}}  \\
 \hackcenter{
 \begin{tikzpicture}[scale=0.9]
 \node[draw, thick, fill=blue!20,rounded corners=4pt,inner sep=3pt]
   (Fj) at (-0,0) {$B_2$};
 \draw[color=blue, thick, dashed, double distance=1pt,] (Fj) to  (-0,1.25);
    \node[blue] at (.3,0.8){$\scs 2$\;};
  \node at (.7,0) {$\l$};
\end{tikzpicture} }
& \;\; =\;\; \;
\hackcenter{\begin{tikzpicture}
  \draw[thick, ->] (-0.4,0) .. controls ++(-0,0.6) and ++(0,0.6) .. (0.4,0)
      node[pos=0.5, shape=coordinate](X){}
      node[pos=0.1, shape=coordinate](Y){};
  \draw[thick] (-0.4,0) .. controls ++(0,-0.6) and ++(0,-0.6) .. (0.4,0)
      node[pos=0.1, shape=coordinate](Z){};
  \draw[color=blue, thick, double distance=1pt, dashed] (X) .. controls++(0,.65) and ++(-.65,.3) .. (Y) node[pos=0.15,right]{$\scs \l-1$\;};
  \draw[color=blue, thick, dashed] (Z) to[bend left] (-1,1) ;
  \node at (Y) {\bbullet};
  \node at (Z) {\bbullet};
\end{tikzpicture}}
\hackcenter{\begin{tikzpicture}
  \draw[thick, ->] (-0.4,0) .. controls ++(-0,0.6) and ++(0,0.6) .. (0.4,0)
      node[pos=0.5, shape=coordinate](X){}
      node[pos=0.1, shape=coordinate](Y){};
  \draw[thick] (-0.4,0) .. controls ++(0,-0.6) and ++(0,-0.6) .. (0.4,0)
      node[pos=0.1, shape=coordinate](Z){};
  \draw[color=blue, thick, double distance=1pt, dashed] (X) .. controls++(0,.65) and ++(-.65,.3) .. (Y) node[pos=0.15,right]{$\scs \l-1$\;};
  \draw[color=blue, thick, dashed] (Z) to[bend left] (-1,1) ;
  \node at (Y) {\bbullet};
  \node at (Z) {\bbullet};
  \node at (.8,-.25) {$\l$};
\end{tikzpicture}}
\;\; - \;\;
\hackcenter{\begin{tikzpicture}
  \draw[thick, ->] (-0.4,0) .. controls ++(-0,0.6) and ++(0,0.6) .. (0.4,0)
      node[pos=0.5, shape=coordinate](X){}
      node[pos=0.1, shape=coordinate](Y){};
  \draw[thick] (-0.4,0) .. controls ++(0,-0.6) and ++(0,-0.6) .. (0.4,0)
      node[pos=0.1, shape=coordinate](Z){};
  \draw[color=blue, thick, double distance=1pt, dashed] (X) .. controls++(0,.65) and ++(-.65,.3) .. (Y) node[pos=0.15,right]{$\scs \l-1$\;};
  \draw[color=blue, thick, double distance=1pt,dashed] (Z) to[bend left] (-1,1) ;
  \node at (Y) {\bbullet};
  \node at (Z) {\bbullet};
  \node at (.8,-.25) {$\l$};
\end{tikzpicture}}
\end{align*}
\end{example}

% - - - - - - - - - - - - - - - - - - - - - - - - - - - - - - - - - - - - - - - -
%
%
\subsubsection{The case of $\lambda <0$}
%
% - - - - - - - - - - - - - - - - - - - - - - - - - - - - - - - - - - - - - - - -
%
For $\l<0$ inductively define 2-morphisms $\overline{B_j} \maps \onebl \to \Pi^j \onebl\la 2j\ra$ by the equation
\begin{equation} \label{eq:fake-bubble-n}
\sum_{j=0}^{m}
(-1)^{j}\;\;
\hackcenter{
\begin{tikzpicture}[scale=0.9]
 \node[draw, thick, fill=blue!20,rounded corners=4pt,inner sep=3pt]
   (Fj) at (2,0) {$\overline{B_j}$};
 \draw[color=blue, thick, double distance=1pt, dashed] (Fj) to  (2,1.25);
    \node[blue] at (2.3,0.8){$\scs j$\;};
  \draw[thick, ->] (0.5,0) .. controls (0.5,0.8) and (-0.5,0.8) .. (-0.5,0)
      node[pos=0, shape=coordinate](Z){};
  \draw[thick] (0.5,0) .. controls (0.5,-0.8) and (-0.5,-0.8) .. (-0.5,0)
      node[pos=0.5, shape=coordinate](X){}
      node[pos=0.2, shape=coordinate](Y){};
  \draw[color=blue, thick, double distance=1pt, dashed] (X) .. controls++(-.1,.7) and ++(-.2,.4) .. (Y)
         node[pos=0.9,right]{$\scs -\l-1$\;};
   \draw[color=blue, thick, double distance=1pt, dashed]
    (Z) .. controls ++(-1,.7) and ++(.1,-1) .. (1,1.25) ;
   \node[blue] at (.4,0.9){$\scs m-j$\;};
  \node at (Y) {\bbullet};\node at (X) {\bbullet};
  \node at (-1,-.3) {$\l$};
\end{tikzpicture} }
 \;\; = \;\; \delta_{m,0}\oneb_{\onebl}.
\end{equation}
Again it is clear that these 2-morphisms are independent of the weight $\l$ for all $\l<0$, except for the case when $\l=-1$.  In this case, the definition of the 2-morphisms $\overline{B_j}$ depends on the free parameter $c_{-1}$ from \eqref{eq_defcmone} corresponding to the degree zero bubble.

\begin{example} Several examples for $\l<-1$ are given below.
\begin{align*}
 \hackcenter{
 \begin{tikzpicture}[scale=0.8]
  %% Make a coupon filled with a given label place at coordinate (-2,0.75) name the node Fj
 \node[draw, thick, fill=blue!20,rounded corners=4pt,inner sep=3pt]
   (Fj) at (-0,0) {$\overline{B_0}$};
  \node at (.7,0) {$\l$};
\end{tikzpicture} }
& \;\; =\;\;
\oneb_{\onebl}  \\
 \hackcenter{
 \begin{tikzpicture}[scale=0.9]
  %% Make a coupon filled with a given label place at coordinate (-2,0.75) name the node Fj
 \node[draw, thick, fill=blue!20,rounded corners=4pt,inner sep=3pt]
   (Fj) at (-0,0) {$\overline{B_1}$};
 \draw[color=blue, thick, dashed] (Fj) to  (-0,1.25);
    \node[blue] at (.3,0.8){};
  \node at (.7,0) {$\l$};
\end{tikzpicture} }
& \;\; =\;\; \;\;
\hackcenter{\begin{tikzpicture}
  \draw[thick, ->] (0.4,0) .. controls ++(0,0.6) and ++(0,0.6) .. (-0.4,0)
      node[pos=0.05, shape=coordinate](Z){};
  \draw[thick] (0.4,0) .. controls ++(0,-0.6) and ++(-0,-0.6) .. (-0.4,0)
      node[pos=0.5, shape=coordinate](X){}
      node[pos=0.2, shape=coordinate](Y){};
  \draw[color=blue, thick, double distance=1pt, dashed]
    (X) .. controls++(-.1,.5) and ++(-.2,.3) .. (Y)
         node[pos=0.9,right]{$\scs -\l-1$\;};
   \draw[color=blue, thick,  dashed]
    (Z) .. controls ++(-1,.4) and ++(.1,-1) .. (1,1) ;
   \node[blue] at (1.25,0.8){$\scs $\;};
     \node at (Y) {\bbullet};
     \node at (Z) {\bbullet};
  \node at (1.5,.3) {$\lambda$};
\end{tikzpicture}}  \\
 \hackcenter{
 \begin{tikzpicture}[scale=0.9]
  %% Make a coupon filled with a given label place at coordinate (-2,0.75) name the node Fj
 \node[draw, thick, fill=blue!20,rounded corners=4pt,inner sep=3pt]
   (Fj) at (-0,0) {$\overline{B_2}$};
 \draw[color=blue, thick, dashed, double distance=1pt,] (Fj) to  (-0,1.25);
    \node[blue] at (.3,0.8){$\scs 2$\;};
  \node at (.7,0) {$\l$};
\end{tikzpicture} }
& \;\; =\;\; \;
\hackcenter{\begin{tikzpicture}
  \draw[thick, ->] (0.4,0) .. controls ++(0,0.6) and ++(0,0.6) .. (-0.4,0)
      node[pos=0.05, shape=coordinate](Z){};
  \draw[thick] (0.4,0) .. controls ++(0,-0.6) and ++(-0,-0.6) .. (-0.4,0)
      node[pos=0.5, shape=coordinate](X){}
      node[pos=0.2, shape=coordinate](Y){};
  \draw[color=blue, thick, double distance=1pt, dashed]
    (X) .. controls++(-.1,.5) and ++(-.2,.3) .. (Y)
         node[pos=0.9,right]{$\scs -\l-1$\;};
   \draw[color=blue, thick,  dashed]
    (Z) .. controls ++(-1,.4) and ++(.1,-1) .. (1,1) ;
   \node[blue] at (1.25,0.8){$\scs $\;};
     \node at (Y) {\bbullet};
     \node at (Z) {\bbullet};
\end{tikzpicture}}
\hackcenter{\begin{tikzpicture}
  \draw[thick, ->] (0.4,0) .. controls ++(0,0.6) and ++(0,0.6) .. (-0.4,0)
      node[pos=0.05, shape=coordinate](Z){};
  \draw[thick] (0.4,0) .. controls ++(0,-0.6) and ++(-0,-0.6) .. (-0.4,0)
      node[pos=0.5, shape=coordinate](X){}
      node[pos=0.2, shape=coordinate](Y){};
  \draw[color=blue, thick, double distance=1pt, dashed]
    (X) .. controls++(-.1,.5) and ++(-.2,.3) .. (Y)
         node[pos=0.9,right]{$\scs -\l-1$\;};
   \draw[color=blue, thick,  dashed]
    (Z) .. controls ++(-1,.4) and ++(.1,-1) .. (1,1) ;
   \node[blue] at (1.25,0.8){$\scs $\;};
     \node at (Y) {\bbullet};
     \node at (Z) {\bbullet};
  \node at (1.5,.3) {$\lambda$};
\end{tikzpicture}}
\;\; - \;\;
\hackcenter{\begin{tikzpicture}
  \draw[thick, ->] (0.4,0) .. controls ++(0,0.6) and ++(0,0.6) .. (-0.4,0)
      node[pos=0.05, shape=coordinate](Z){};
  \draw[thick] (0.4,0) .. controls ++(0,-0.6) and ++(-0,-0.6) .. (-0.4,0)
      node[pos=0.5, shape=coordinate](X){}
      node[pos=0.2, shape=coordinate](Y){};
  \draw[color=blue, thick, double distance=1pt, dashed]
    (X) .. controls++(-.1,.5) and ++(-.2,.3) .. (Y)
         node[pos=0.9,right]{$\scs -\l-1$\;};
   \draw[color=blue, thick,  dashed,double distance=1pt,]
    (Z) .. controls ++(-1,.4) and ++(.1,-1) .. (1,1)
     node[pos=0.9, left] {$\scs 2$} ;
   \node[blue] at (1.25,0.8){$\scs $\;};
     \node at (Y) {\bbullet};
     \node at (Z) {\bbullet};
  \node at (1.5,.3) {$\lambda$};
\end{tikzpicture}}
\end{align*}
\end{example}

% - - - - - - - - - - - - - - - - - - - - - - - - - - - - - - - - - - - - - - - -
%
%
\subsubsection{The case of $\lambda =0$}
%
% - - - - - - - - - - - - - - - - - - - - - - - - - - - - - - - - - - - - - - - -
By Lemma~\ref{lem:E} the degree zero curls in weight $\l=0$ must be scalar multiples of the identity map. We define parameters by the equations
\begin{equation} \label{eq:c0}
  \hackcenter{\begin{tikzpicture} [scale=0.8]
  \draw[thick] (0.5,1) -- (0.5,1.5);
  \draw[thick] (0.5,-.5) -- (0.5,0);
  \draw[thick] (-1.5,0) -- (-1.5,1);
  \draw[thick,->] (0.5,0) .. controls ++(-0,0.5) and ++(0,-0.5) .. (-0.5,1)
      node[pos=0.5, shape=coordinate](X){};
  \draw[thick, ->] (-0.5,0) .. controls ++(0,0.5) and ++(0,-0.5) .. (0.5,1);
  \draw[thick, ->] (-0.5,1) .. controls ++(0,0.6) and ++(0,0.6) .. (-1.5,1);
  \draw[thick, ->] (-1.5,0) .. controls ++(0,-0.6) and ++(0,-0.6) .. (-0.5,0)
     node[pos=0.5, shape=coordinate](CUP){};
  \draw[color=blue,  thick, dashed]
   (X) .. controls ++(-1.2,0) and ++(0,.9) ..(CUP);
   \node at (0,-0.25) {$0$};
   \node at (1,-0.25) {$+2$};
\end{tikzpicture}}
\quad = \quad c^-_0  \;\;
  \hackcenter{\begin{tikzpicture} [scale=0.8]
  \draw[thick,->] (0,-.5) -- (0,1.5);
   \node at (-.5,-0.25) {$0$};
   \node at (.5,-0.25) {$+2$};
\end{tikzpicture}}, \qquad \quad
  \hackcenter{\begin{tikzpicture} [scale=0.8]
  \draw[thick] (-0.5,1) -- (-0.5,1.5);
  \draw[thick] (-0.5,-.5) -- (-0.5,0);
  \draw[thick] (1.5,0) -- (1.5,1);
  \draw[thick,->] (-0.5,0) .. controls ++(-0,0.5) and ++(0,-0.5) .. (0.5,1)
      node[pos=0.5, shape=coordinate](X){};
  \draw[thick, ->] (0.5,0) .. controls ++(0,0.5) and ++(0,-0.5) .. (-0.5,1);
  \draw[thick, ->] (0.5,1) .. controls ++(0,0.6) and ++(0,0.6) .. (1.5,1)
     node[pos=0.5, shape=coordinate](CUP){};
  \draw[thick, ->] (1.5,0) .. controls ++(0,-0.6) and ++(0,-0.6) .. (0.5,0);
  \draw[color=blue,  thick, dashed]
   (X) .. controls ++(-2,.3) and ++(.1,.7) ..(CUP);
   \node at (0,-0.25) {$0$};
   \node at (-1,-0.25) {$+2$};
\end{tikzpicture}}
\quad = \quad -c^+_0 \;\;
  \hackcenter{\begin{tikzpicture} [scale=0.8]
  \draw[thick,->] (0,-.5) -- (0,1.5);
   \node at (-.5,-0.25) {$-2$};
   \node at (.5,-0.25) {$0$};
\end{tikzpicture}},
\end{equation}
for some coefficients $c_0^+$ and $c_0^-$.  It is convenient to introduce fake bubbles in weight $\l=0$ defined by setting
\begin{align}
  \hackcenter{
\begin{tikzpicture}[scale=0.9]
  \draw[thick, ->] (-0.5,0) .. controls (-0.5,0.8) and (0.5,0.8) .. (0.5,0)
      node[pos=0.5, shape=coordinate](X){}
      node[pos=0.1, shape=coordinate](Y){};
  \draw[thick] (-0.5,0) .. controls (-0.5,-0.8) and (0.5,-0.8) .. (0.5,0)
      node[pos=0.1, shape=coordinate](Z){};
  \draw[color=blue, thick,  dashed] (X) .. controls++(0,.65) and ++(-.65,.3) .. (Y) node[pos=0.15,right]{$\scs -1$\;};
  \node at (Y) {\bbullet}; \node at (1,.75) {$\lambda$};
\end{tikzpicture} }
&\;\; = \;\; c_0^+\oneb_{\oneb_{0}} &
\xy (0,-2)*{
\begin{tikzpicture}[scale=0.9]
  \draw[thick, ->] (0.5,0) .. controls (0.5,0.8) and (-0.5,0.8) .. (-0.5,0)
      node[pos=0, shape=coordinate](Z){};
  \draw[thick] (0.5,0) .. controls (0.5,-0.8) and (-0.5,-0.8) .. (-0.5,0)
      node[pos=0.5, shape=coordinate](X){}
      node[pos=0.2, shape=coordinate](Y){};
  \draw[color=blue, thick,  dashed] (X) .. controls++(-.1,.7) and ++(-.2,.4) .. (Y)
         node[pos=0.9,right]{$\scs -1$\;};
 \node at (Y) {\bbullet};
  \node at (1,.65) {$\lambda$};
\end{tikzpicture} }; \endxy
&\;\; = \;\; c_0^-\oneb_{\oneb_{0}}.
\end{align}

% - - - - - - - - - - - - - - - - - - - - - - - - - - - - - - - - - - - - - - - -
%
%
\subsubsection{Notation for all $\lambda$}
%
% - - - - - - - - - - - - - - - - - - - - - - - - - - - - - - - - - - - - - - - -
%

We refer to the maps $B_j$ for $0 \leq j \leq \l-1$ and $\overline{B_j}$ for $0 \leq j \leq -\l-1$ as {\em odd fake bubbles} because they are analogues of fake bubbles from the even case.  It is convenient to introduce a different notation for these odd fake bubbles.  This new notation makes it possible to express various equations in a uniform manner independent of whether the weight $\l$ is positive or negative.  We write
\[
\begin{array}{ccc}
  \l>0 & \qquad \qquad &\l<0 \\
  \hackcenter{
\begin{tikzpicture}[scale=0.9]
  \draw[thick, ->] (0.5,0) .. controls (0.5,0.8) and (-0.5,0.8) .. (-0.5,0)
      node[pos=0, shape=coordinate](Z){};
  \draw[thick] (0.5,0) .. controls (0.5,-0.8) and (-0.5,-0.8) .. (-0.5,0)
      node[pos=0.5, shape=coordinate](X){}
      node[pos=0.2, shape=coordinate](Y){};
 %% Draw double blue curvy line
  \draw[color=blue, thick, double distance=1pt, dashed] (X) .. controls++(-.1,.7) and ++(-.2,.4) .. (Y)
         node[pos=0.9,right]{$\scs -\l-1$\;};
   \draw[color=blue, thick, double distance=1pt, dashed]
    (Z) .. controls ++(-1,.7) and ++(.1,-1) .. (1,1) ;
   \node[blue] at (1.3,0.8){$\scs j$\;};
 %% Draw the bullet last so it comes out on top
  \draw[line width=0mm] (0.5,0) .. controls (0.5,-0.8) and (-0.5,-0.8) .. (-0.5,0)
     node[pos=0.2]{\bbullet};
  \draw[line width=0mm] (0.5,0) .. controls (0.5,0.8) and (-0.5,0.8) .. (-0.5,0)
      node[pos=0.0]{\bbullet};
  \node at (-1,.3) {$\lambda$};
\end{tikzpicture} }
\quad := \quad
\hackcenter{
 \begin{tikzpicture}[scale=0.9]
 \node[draw, thick, fill=blue!20,rounded corners=4pt,inner sep=3pt]
   (Fj) at (-0,0) {$B_j$};
  \draw[color=blue, thick, double distance=1pt, dashed] (Fj) to  (0,1.25);
  \node at (.7,.5) {$\l$};
  \node[blue] at (.3,1.1){$\scs j$\;};
\end{tikzpicture}  }
 &  \qquad \qquad &
 \hackcenter{
\begin{tikzpicture}[scale=0.9]
  \draw[thick, ->] (-0.5,0) .. controls (-0.5,0.8) and (0.5,0.8) .. (0.5,0)
      node[pos=0.5, shape=coordinate](X){}
      node[pos=0.1, shape=coordinate](Y){};
  \draw[thick] (-0.5,0) .. controls (-0.5,-0.8) and (0.5,-0.8) .. (0.5,0)
      node[pos=0.1, shape=coordinate](Z){};
  \draw[color=blue, thick, double distance=1pt, dashed] (X) .. controls++(0,.65) and ++(-.65,.3) .. (Y) node[pos=0.15,right]{$\scs \l-1$\;};
  \draw[color=blue, thick, double distance=1pt, dashed] (Z) to[out=180, in=90] (-1,1.25) ;
  \node[blue] at (-.65,1.2){$\scs j$\;};
  \draw[line width=0mm] (-0.5,0) .. controls (-0.5,0.8) and (0.5,0.8) .. (0.5,0)
    node[pos=0.1]{\bbullet};
  \draw[line width=0mm] (-0.5,0) (-0.5,0) .. controls (-0.5,-0.8) and (0.5,-0.8) .. (0.5,0)
        node[pos=0.1]{\bbullet};
  \node at (1.2,0) {$\l$};
\end{tikzpicture}}
\quad := \quad
\hackcenter{
 \begin{tikzpicture}[scale=0.9]
 \node[draw, thick, fill=blue!20,rounded corners=4pt,inner sep=3pt]
   (Fj) at (-0,0) {$\overline{B_j}$};
  \draw[color=blue, thick, double distance=1pt, dashed] (Fj) to  (0,1.25);
  \node at (.7,.5) {$\l$};
  \node[blue] at (.3,1.1){$\scs j$\;};
\end{tikzpicture}  }
\end{array}
\]
for all $0 \leq j\leq |\l|-1$.  The drawback of this notation is that it appears to involve a negative number of dots.  We never allow negative dots in diagrams involving 2-morphisms.  Whenever a negative number of dots is encountered in a dotted bubble diagram the dotted bubble is interpreted as an odd fake bubble defined above.

The advantage of introducing this notation for odd fake bubbles is that equations \eqref{eq:fake-bubble-p} and \eqref{eq:fake-bubble-n} can be expressed as:
\begin{align}\label{eq:fake-bubble-last}
\sum_{f+g=m} (-1)^{g}
\hackcenter{
\begin{tikzpicture}[scale=0.9]
  \draw[thick, ->] (-0.5,0) .. controls (-0.5,0.8) and (0.5,0.8) .. (0.5,0)
      node[pos=0.5, shape=coordinate](X){}
      node[pos=0.1, shape=coordinate](Y){};
  \draw[thick] (-0.5,0) .. controls (-0.5,-0.8) and (0.5,-0.8) .. (0.5,0)
      node[pos=0.1, shape=coordinate](Z){};
  \draw[color=blue, thick, double distance=1pt, dashed] (X) .. controls++(0,.65) and ++(-.65,.3) .. (Y) node[pos=0.15,right]{$\scs \l-1$\;};
  \draw[color=blue, thick, double distance=1pt, dashed] (Z) to[out=180, in=90] (-1,1.25) ;
  \node[blue] at (-.6,1.2){$\scs f$\;};
  \node at (Y) {\bbullet}; \node at (Z) {\bbullet};
\end{tikzpicture} }
  \xy (0,-2)*{
\begin{tikzpicture}[scale=0.9]
  \draw[thick, ->] (0.5,0) .. controls (0.5,0.8) and (-0.5,0.8) .. (-0.5,0)
      node[pos=0, shape=coordinate](Z){};
  \draw[thick] (0.5,0) .. controls (0.5,-0.8) and (-0.5,-0.8) .. (-0.5,0)
      node[pos=0.5, shape=coordinate](X){}
      node[pos=0.2, shape=coordinate](Y){};
  \draw[color=blue, thick, double distance=1pt, dashed] (X) .. controls++(-.1,.7) and ++(-.2,.4) .. (Y)
         node[pos=0.9,right]{$\scs -\l-1$\;};
   \draw[color=blue, thick, double distance=1pt, dashed]
    (Z) .. controls ++(-1,.7) and ++(.1,-1) .. (1,1.25) ;
   \node[blue] at (1.3,0.9){$\scs g$\;};
 \node at (Y) {\bbullet}; \node at (Z) {\bbullet};
  \node at (-.5,1.1) {$\lambda$};
\end{tikzpicture} }; \endxy\;\;
 &\;\; = \;\; \delta_{m,0}\oneb_{\onebl}
 &\text{for $\l >0$,}
\\
\sum_{f+g=m}
(-1)^{g} \xy (0,-2)*{
\begin{tikzpicture}[scale=0.9]
  \draw[thick, ->] (0.5,0) .. controls (0.5,0.8) and (-0.5,0.8) .. (-0.5,0)
      node[pos=0, shape=coordinate](Z){};
  \draw[thick] (0.5,0) .. controls (0.5,-0.8) and (-0.5,-0.8) .. (-0.5,0)
      node[pos=0.5, shape=coordinate](X){}
      node[pos=0.2, shape=coordinate](Y){};
  \draw[color=blue, thick, double distance=1pt, dashed] (X) .. controls++(-.1,.7) and ++(-.2,.4) .. (Y)
         node[pos=0.9,right]{$\scs -\l-1$\;};
   \draw[color=blue, thick, double distance=1pt, dashed]
    (Z) .. controls ++(-1,.7) and ++(.1,-1) .. (1,1.25) ;
   \node[blue] at (1.3,1.1){$\scs f$\;};
 \node at (Y) {\bbullet}; \node at (Z) {\bbullet};
  \node at (-.25,1.1) {$\lambda$};
\end{tikzpicture} }; \endxy
\hackcenter{
\begin{tikzpicture}[scale=0.9]
  \draw[thick, ->] (-0.5,0) .. controls (-0.5,0.8) and (0.5,0.8) .. (0.5,0)
      node[pos=0.5, shape=coordinate](X){}
      node[pos=0.1, shape=coordinate](Y){};
  \draw[thick] (-0.5,0) .. controls (-0.5,-0.8) and (0.5,-0.8) .. (0.5,0)
      node[pos=0.1, shape=coordinate](Z){};
  \draw[color=blue, thick, double distance=1pt, dashed] (X) .. controls++(0,.65) and ++(-.65,.3) .. (Y) node[pos=0.15,right]{$\scs \l-1$\;};
  \draw[color=blue, thick, double distance=1pt, dashed] (Z) to[out=180, in=90] (-1,1.25) ;
  \node[blue] at (-.6,1.2){$\scs g$\;};
  \node at (Y) {\bbullet}; \node at (Z) {\bbullet};
\end{tikzpicture} }
&\;\; = \;\; \delta_{m,0}\oneb_{\onebl}
&
\text{for $\l<0$,}
\end{align}
for $0 \leq m \leq |\l|-1$.  The relations in the 2-category $\Uc$ are also conveniently expressed in terms of fake bubbles.

% -------------------------------------------------------------------------------
%
\subsection{A general form for the inverse map}
%
% -------------------------------------------------------------------------------

By Corollary~\ref{cor:1} the map
\begin{equation}
\zeta\;\;:=\;\;
\hackcenter{\begin{tikzpicture}[scale=0.6]
  \draw[semithick, <-] (-0.5,0) .. controls (-0.5,0.5) and (0.5,0.5) .. (0.5,1)
      node[pos=0.5, shape=coordinate](X){};
  \draw[semithick, ->] (0.5,0) .. controls (0.5,0.5) and (-0.5,0.5) .. (-0.5,1);
  \draw[color=blue,  thick, dashed] (X) to (0,0);
\end{tikzpicture}}\;\; \bigoplus_{k=0}^{\lambda-1}
\hackcenter{\begin{tikzpicture}[scale=0.6]
  \draw[thick, ->-=0.15, ->] (0.5,.2) .. controls (0.6,-0.8) and (-0.6,-0.8) .. (-0.5,.2)
      node[pos=0.85, shape=coordinate](Y){};
  \draw[color=blue, thick, double distance=1pt, dashed]
   (Y) .. controls++(-.5,.2) and ++(0,.4) .. (-1,-1)
         node[pos=0.75,left]{$\scs k$};
  \draw[line width=0mm] (0.5,.2) .. controls (0.5,-0.8) and (-0.5,-0.8) .. (-0.5,.2)
     node[pos=0.85]{\tikz \draw[fill=black] circle (0.4ex);};
\end{tikzpicture} }:\Ftt\Pi \Ett \onebl \bigoplus_{k=0}^{\lambda-1} \Pi^k\onebl \la \l-1-2k \ra \rightarrow \Ett \Ftt \onebl
\end{equation}
is invertible. We describe its inverse $\zeta^{-1}$ diagrammatically as follows:
\begin{equation}
\hackcenter{\begin{tikzpicture}[scale=0.55]
  \draw[thick, ->-=0.12, ->-=0.95] (-0.6,-0.1) .. controls ++(-0,0.75) and ++(0,-0.75) .. (0.6,2.1);
  \draw[thick,  ->-=0.12, ->-=0.95] (-0.6,2.1).. controls ++(0,-0.75) and ++(0,0.75) ..(0.6,-0.1);
  \draw[color=blue,  thick, dashed] (0,1) -- (0,2.1) ;
   %% Make a coupon filled with a given label place at coordinate (-2,0.75) name the node Fj
   \node[draw, thick, fill=blue!20,rounded corners=4pt,inner sep=3pt]
     (fi) at (0,1) {$\zeta(\l)$};
\end{tikzpicture} }\;\; \bigoplus_{k=0}^{\l-1}
\left( \quad\;\;
 \hackcenter{
 \begin{tikzpicture}[scale=0.55]
  \draw[thick, ->-=0.12, ->-=0.95] (-0.6,-0.1) .. controls ++(-0.1,1.3) and ++(0.1,1.3) .. (0.6,-0.1);
  \draw[color=blue,  thick, dashed, double distance=1pt] (0,1) -- (0,2.4) ;
   %% Make a coupon filled with a given label place at coordinate (-2,0.75) name the node Fj
   \node[draw, thick, fill=blue!20,rounded corners=4pt,inner sep=3pt]
     (fi) at (0,1) {$\zeta(\l-1-k)$};
   \node[blue] at (.3,2.1) {$\scs k$};
\end{tikzpicture} } \;\; \;\;\right)
 : \Ett \Ftt \onebl \rightarrow\Ftt\Pi \Ett \onebl \bigoplus_{k=0}^{\lambda-1} \Pi^k\onebl \la \l-1-2k \ra .
\end{equation}
Likewise, for $\l \leq 0$ the inverse of
\begin{equation}
\zeta\;\;:=\;\;
\hackcenter{\begin{tikzpicture}[scale=0.6]
  \draw[semithick, ->] (-0.5,0) .. controls (-0.5,0.5) and (0.5,0.5) .. (0.5,1)
      node[pos=0.5, shape=coordinate](X){};
    \draw[semithick, <-] (0.5,0) .. controls (0.5,0.5) and (-0.5,0.5) .. (-0.5,1);
  \draw[color=blue,  thick, dashed] (X) .. controls ++(.1,.5) and ++(0,.5) .. (-.5,.5)
  .. controls ++(0,-.3) and ++(0,.3) .. (0,0);
\end{tikzpicture}} \;\; \bigoplus_{k=0}^{-\l-1}
\hackcenter{\begin{tikzpicture}[scale=0.6]
  \draw[thick, ->-=0.15, ->] (-0.7,.5) .. controls ++(-.1,-1) and ++(.1,-1) .. (0.7,.5)
      node[pos=0.85, shape=coordinate](Y){}
      node[pos=0.55, shape=coordinate](M){}
      node[pos=0.44, shape=coordinate](X){};
  \draw[color=blue, thick, double distance=1pt, dashed]
   (Y) .. controls++(-.5,.3) and ++(0,.5) .. (M)
         node[pos=0.15,above]{$\scs k$};
   \draw[color=blue, thick, double distance=1pt, dashed]
     (X) .. controls ++(0,.55) and ++(0,.55) ..
      (-.6,-.25) .. controls ++(0,-.3) and ++(0,.4) ..(0,-1);
   \node at (Y){\tikz \draw[fill=black] circle (0.4ex);};
\end{tikzpicture} }:\Ett\Pi\Ftt \onebl \bigoplus_{k=0}^{-\l-1}\Pi^{\l+1+k} \onebl \la -\l-1-2k \ra \rightarrow \Ftt \Ett \onebl
\end{equation}
can be expressed as
\begin{equation}
\hackcenter{\begin{tikzpicture}[scale=0.5]
  \draw[thick, ->-=0.12, ->-=0.95] (0.6,-0.1) .. controls ++(0,0.75) and ++(0,-0.75) .. (-0.6,2.1);
  \draw[thick, ->-=0.12, ->-=0.95]
   (0.6,2.1).. controls ++(0,-0.75) and ++(0,0.75) ..(-0.6,-0.1);
  \draw[color=blue,  thick, dashed] (0,1) -- (0,2.1) ;
  \node[draw, thick, fill=blue!20,rounded corners=4pt,inner sep=3pt]
     (fi) at (0,1) {$\zeta(\l)$};
\end{tikzpicture} }\;\; \bigoplus_{k=0}^{-\l-1}
\left( \quad\;\;
 \hackcenter{
 \begin{tikzpicture}[scale=0.5]
  \draw[thick, ->-=0.12, ->-=0.95] (0.6,-0.1) .. controls ++(0.1,1.3) and ++(-0.1,1.3) .. (-0.6,-0.1);
  \draw[color=blue,  thick, dashed, double distance=1pt] (0,1) -- (0,2.4) ;
   %% Make a coupon filled with a given label place at coordinate (-2,0.75) name the node Fj
   \node[draw, thick, fill=blue!20,rounded corners=4pt,inner sep=3pt]
     (fi) at (0,1) {$\zeta(-\l-1-k)$};
   \node[blue] at (1.1,2.1) {$\scs \l+1+k$};
\end{tikzpicture} } \;\; \;\;\right)
 : \Ftt \Ett \onebl \rightarrow\Ett\Pi\Ftt \onebl \bigoplus_{k=0}^{-\l-1}\Pi^{\l+1+k} \onebl \la -\l-1-2k \ra .
\end{equation}

Condition (3) of Definition~\ref{def_strong} only requires the {\it{existence}} of isomorphisms between the two 1-morphisms on either side. However, the space of 2-morphisms between a pair of 1-morphisms in $\Cc$ could contain maps that cannot be expressed using 2-morphisms from the strong supercategorical action of $\mf{sl}_2$, i.e. using dots, crossings, caps, and cups. In the next proposition we show that this is not the case for the 2-morphisms giving the isomorphism $\zeta^{-1}$.

\begin{prop} \label{prop_form-of-inv}
The isomorphism $\zeta^{-1}$ for $\l \geq 0$ has the form
\begin{equation} \label{eq_phi-inverses-U}
\hackcenter{\begin{tikzpicture}[scale=0.5]
  \draw[thick, ->-=0.12, ->-=0.95] (-0.6,-0.1) .. controls ++(-0,0.75) and ++(0,-0.75) .. (0.6,2.1);
  \draw[thick,  ->-=0.12, ->-=0.95] (-0.6,2.1).. controls ++(0,-0.75) and ++(0,0.75) ..(0.6,-0.1);
  \draw[color=blue,  thick, dashed] (0,1) -- (0,2.1);
   \node[draw, thick, fill=blue!20,rounded corners=4pt,inner sep=3pt]
     (fi) at (0,1) {$\zeta(\l)$};
   \node at (1.3,0.5) {$\l$};
\end{tikzpicture} }
\;\;=\;\; \beta_{\l}
  \;
\hackcenter{\begin{tikzpicture}[scale=0.5]
  \draw[thick, ->-=0.12, ->-=0.95] (-0.6,-0.1) .. controls ++(-0,0.75) and ++(0,-0.75) .. (0.6,2.1);
  \draw[thick,  ->-=0.12, ->-=0.95] (-0.6,2.1).. controls ++(0,-0.75) and ++(0,0.75) ..(0.6,-0.1);
  \draw[color=blue,  thick, dashed] (0,1) -- (0,2.1) ;
  \node at (1,1) {$\l$};
\end{tikzpicture} }\;\;
\qquad  \qquad
\hackcenter{
 \begin{tikzpicture}[scale=0.5]
  \draw[thick, ->-=0.12, ->-=0.95] (-0.6,-0.1) .. controls ++(-0.1,1.3) and ++(0.1,1.3) .. (0.6,-0.1);
  \draw[color=blue,  thick, dashed, double distance=1pt] (0,1) -- (0,2.4) ;
   \node[draw, thick, fill=blue!20,rounded corners=4pt,inner sep=3pt]
     (fi) at (0,1) {$\zeta(\l-1-k)$};
   \node[blue] at (.3,2.1) {$\scs k$};
   \node at (2.3,0) {$\l$};
\end{tikzpicture} }
\;\; = \;\;
\sum_{j=0}^{\l-1-k} (-1)^j
 \xy
 (0,0)*{
\begin{tikzpicture}
  \draw[thick, ->] (-0.5,0) .. controls (-0.5,0.8) and (0.5,0.8) .. (0.5,0)
      node[pos=0.1, shape=coordinate](DOT){}
      node[pos=0.42, shape=coordinate](L){}
      node[pos=0.5, shape=coordinate](M){}
      node[pos=0.58, shape=coordinate](R){};
 \node[draw, thick, fill=blue!20,rounded corners=4pt,inner sep=3pt]
    (Fj) at (1.25,0.65) {$\scs B_j$};
 \draw[color=blue, thick, double distance=1pt, dashed] (Fj) to [out=90, in=90] (R);
 \draw[color=blue, thick, double distance=1pt, dashed] (M) to (0,1.5);
 \draw[color=blue, thick, double distance=1pt, dashed]
    (DOT) .. controls++(-.65,0) and ++(-.25,.3) .. (L);
 \draw[line width=0mm] (-0.5,0) .. controls (-0.5,0.8) and (0.5,0.8) .. (0.5,0)
        node[pos=0.1]{\bbullet};
   \node[blue] at (-.3,1.5){$\scs k$};
   \node[blue] at (1.15,1.4){$\scs j$};
   \node[blue] at (-1.45,.20){$\scs \l-1 -k-j$};
   \node at (-1,1.2) {$\l$};
\end{tikzpicture} };
\endxy
\end{equation}
for some $\lambda$-dependent coefficients $\beta_{\lambda} \in \Bbbk^{\times}$. For $\ell \leq 0$ the isomorphism $\zeta^{-1}$ has the form
\begin{equation}
\hackcenter{\begin{tikzpicture}[scale=0.5]
  \draw[thick, ->-=0.12, ->-=0.95] (0.6,-0.1) .. controls ++(0,0.75) and ++(0,-0.75) .. (-0.6,2.1);
  \draw[thick, ->-=0.12, ->-=0.95]
   (0.6,2.1).. controls ++(0,-0.75) and ++(0,0.75) ..(-0.6,-0.1);
  \draw[color=blue,  thick, dashed] (0,1) -- (0,2.1) ;
  \node[draw, thick, fill=blue!20,rounded corners=4pt,inner sep=3pt]
     (fi) at (0,1) {$\zeta(\l)$};
\end{tikzpicture} }
\;\;=\;\; \beta_{\l}
  \;
  \hackcenter{\begin{tikzpicture}[scale=0.8]
  \draw[thick, <-] (-0.5,0) .. controls (-0.5,0.4) and (0.5,0.6) .. (0.5,1)
      node[pos=0.5, shape=coordinate](X){};
    \draw[thick, ->] (0.5,0) .. controls (0.5,0.4) and (-0.5,0.6) .. (-0.5,1);
  \draw[color=blue,  thick, dashed]
     (0,1)  .. controls ++(0,-.3) and ++(0,.3) .. (-.6,.4)
     .. controls ++(.1,-.4) and ++(.1,-.4)  .. (X);
\end{tikzpicture} }
\;\;
\qquad  \qquad
\hackcenter{
 \begin{tikzpicture}[scale=0.5]
  \draw[thick, ->-=0.12, ->-=0.95] (0.6,-0.1) .. controls ++(0.1,1.3) and ++(-0.1,1.3) .. (-0.6,-0.1);
  \draw[color=blue,  thick, dashed, double distance=1pt] (0,1) -- (0,2.4) ;
   \node[draw, thick, fill=blue!20,rounded corners=4pt,inner sep=3pt]
     (fi) at (0,1) {$\zeta(-\l-1-k)$};
   \node[blue] at (.3,2.1) {$\scs k$};
   \node at (2.3,0) {$\l$};
\end{tikzpicture} }
\;\; = \;\;
\sum_{j=0}^{-\l-1-k}(-1)^{j}
 \hackcenter{\begin{tikzpicture}
  \draw[thick,->-=0.8] (0.5,.25) -- (0.5,.5);
  \draw[thick,->-=0.55] (-0.5,.5) -- (-0.5,.25);
  \draw[thick] (0.5,.5) .. controls ++(.1,.8) and ++(-.1,.8) .. (-0.5,.5)
      node[pos=0.1, shape=coordinate](DOT){};
 \node[draw, thick, fill=blue!20,rounded corners=4pt,inner sep=3pt]
    (Fj) at (1.5,0.75) {$\scs \overline{B_j}$};
 \draw[color=blue, thick, double distance=1pt, dashed]
    (Fj) .. controls ++(0,.4) and ++(0,-.6) .. (1.25,1.75);
  \draw[color=blue, thick, double distance=1pt, dashed]
    (DOT) .. controls++(-.5,.4) and ++(0,-1) .. (-.75,1.75);
   \node at (DOT){\bbullet};
   \node[blue] at (1.5,1.6){$\scs j$};
   \node[blue] at (0,1.60){$\scs \l-1 -k-j$};
   \node at (-1,.7) {$\l$};
\end{tikzpicture} }
\end{equation}
for $\beta_{\lambda} \in \Bbbk^{\times}$.
\end{prop}

\begin{proof}
The first equation in \eqref{eq_phi-inverses-U} follows immediately from Lemma~\ref{lem:homs}.  For the second claim  take adjoints in Lemma~\ref{lem:main} equation \eqref{eq:main1} so that
\begin{equation} \label{eq_phi-inverses}
\hackcenter{
 \begin{tikzpicture}[scale=0.55]
  \draw[thick, ->-=0.12, ->-=0.95] (-0.6,-0.1) .. controls ++(-0.1,1.3) and ++(0.1,1.3) .. (0.6,-0.1);
  \draw[color=blue,  thick, dashed, double distance=1pt] (0,1) -- (0,2.4) ;
   \node[draw, thick, fill=blue!20,rounded corners=4pt,inner sep=3pt]
     (fi) at (0,1) {$\zeta(\l-1-k)$};
   \node[blue] at (.3,2.1) {$\scs k$};
   \node at (2.3,0) {$\l$};
\end{tikzpicture} }
\;\; = \;\;
\sum_{j=0}^{\l-1-k}
 \xy
 (0,3)*{
\begin{tikzpicture}
  \draw[thick, ->] (-0.5,0) .. controls (-0.5,0.8) and (0.5,0.8) .. (0.5,0)
      node[pos=0.1, shape=coordinate](DOT){}
      node[pos=0.42, shape=coordinate](L){}
      node[pos=0.5, shape=coordinate](M){}
      node[pos=0.58, shape=coordinate](R){};
 \node[draw, thick, fill=blue!20,rounded corners=4pt,inner sep=3pt]
  (Fj) at (-1.75,0.75) {$\scs f_j'(\l -1 -k)$};
 \draw[color=blue, thick, double distance=1pt, dashed] (Fj) to [out=90, in=90] (M);
 \draw[color=blue, thick, double distance=1pt, dashed] (R) to[bend right] (.5,1.5);
 \draw[color=blue, thick, double distance=1pt, dashed]
   (DOT) .. controls++(-.65,0) and ++(-.25,.3) .. (L);
   \draw[line width=0mm] (-0.5,0) .. controls (-0.5,0.8) and (0.5,0.8) .. (0.5,0)
        node[pos=0.1]{\bbullet};
   \node[blue] at (.7,1.5){$\scs k$};
   \node[blue] at (-1.55,1.6){$\scs j$};
   \node[blue] at (-1.45,.20){$\scs \l-1 -k-j$};
\end{tikzpicture}
};
\endxy
\;\; :\Ett\Ftt \onebl \rightarrow \onebl \la \ell-1-2k \ra
\end{equation}
for some 2-morphisms $f_j'(\l-1-k) \in \Hom_{\Uc}(\onebl,\Pi^k\onebl\la 2j \ra)$.
After absorbing additional scalars we can rewrite this as
\begin{equation}
 \hackcenter{
 \begin{tikzpicture}[scale=0.5]
  \draw[thick, ->-=0.12, ->-=0.95] (-0.6,-0.1) .. controls ++(-0.1,1.3) and ++(0.1,1.3) .. (0.6,-0.1);
  \draw[color=blue,  thick, dashed, double distance=1pt] (0,1) -- (0,2.4) ;
   \node[draw, thick, fill=blue!20,rounded corners=4pt,inner sep=3pt]
     (fi) at (0,1) {$\zeta(\l-1-k)$};
   \node[blue] at (.3,2.1) {$\scs k$};
   \node at (2.3,0) {$\l$};
\end{tikzpicture} }
\;\; = \;\;
\sum_{j=0}^{\l-1-k} (-1)^{j}  \xy
 (0,3)*{
\begin{tikzpicture}
  \draw[thick, ->] (-0.5,0) .. controls (-0.5,0.8) and (0.5,0.8) .. (0.5,0)
      node[pos=0.1, shape=coordinate](DOT){}
      node[pos=0.42, shape=coordinate](L){}
      node[pos=0.5, shape=coordinate](M){}
      node[pos=0.58, shape=coordinate](R){};
 \node[draw, thick, fill=blue!20,rounded corners=4pt,inner sep=3pt]
  (Fj) at (1.55,0.5) {$\scs f_j(\l -1 -k)$};
 \draw[color=blue, thick, double distance=1pt, dashed] (Fj) to [out=90, in=80] (R);
 \draw[color=blue, thick, double distance=1pt, dashed] (M) to  (0,1.5);
 \draw[color=blue, thick, double distance=1pt, dashed]
   (DOT) .. controls++(-.65,0) and ++(-.25,.3) .. (L);
   \draw[line width=0mm] (-0.5,0) .. controls (-0.5,0.8) and (0.5,0.8) .. (0.5,0)
        node[pos=0.1]{\bbullet};
   \node[blue] at (-.35,1.4){$\scs k$};
   \node[blue] at (1.35,1.4){$\scs j$};
   \node[blue] at (-1.45,.20){$\scs \l-1 -k-j$};
\end{tikzpicture}
};
\endxy
\end{equation}
for some 2-morphisms $f_j(\l-1-k) \in \Hom_{\Uc}(\onebl,\Pi^k\onebl\la 2j \ra)$.

The component of $\zeta^{-1} \zeta$ mapping the summand $\Pi^{\ell}\onebl\la \ell-1-2\ell\ra$ to the summand $\Pi^k\onebl\la \l-1-2k\ra$ is given by the composite
\begin{equation}
\xy
 (-45,0)*+{\Pi^{\ell}\onebl\la \l-1-2\ell\ra}="1";
 (0,0)*+{\Ett\Ftt\onebl}="2";
 (75,0)*+{\Pi^{k}\onebl\la \l-1-2k \ra .}="3";
 {\ar^{
  \hackcenter{\begin{tikzpicture}[scale=0.9]
  \draw[thick, ->-=0.15, ->] (0.5,.2) .. controls (0.6,-0.8) and (-0.6,-0.8) .. (-0.5,.2)
      node[pos=0.85, shape=coordinate](Y){};
  \draw[color=blue, thick, double distance=1pt, dashed]
   (Y) .. controls++(-.5,.2) and ++(0,.4) .. (-1,-1)
         node[pos=0.75,left]{$\scs \ell$};
  \draw[line width=0mm] (0.5,.2) .. controls (0.5,-0.8) and (-0.5,-0.8) .. (-0.5,.2)
     node[pos=0.85]{\tikz \draw[fill=black] circle (0.4ex);};
\end{tikzpicture} }
} "1";"2"};
 {\ar^-{
\xy
 (0,3)*{
\begin{tikzpicture}
  \draw[thick, ->] (-0.5,0) .. controls (-0.5,0.8) and (0.5,0.8) .. (0.5,0)
      node[pos=0.1, shape=coordinate](DOT){}
      node[pos=0.42, shape=coordinate](L){}
      node[pos=0.5, shape=coordinate](M){}
      node[pos=0.58, shape=coordinate](R){};
 \node[draw, thick, fill=blue!20,rounded corners=4pt,inner sep=3pt]
  (Fj) at (1.55,0.5) {$\scs f_j(\l -1 -k)$};
 \draw[color=blue, thick, double distance=1pt, dashed] (Fj) to [out=90, in=80] (R);
 \draw[color=blue, thick, double distance=1pt, dashed] (M) to  (0,1.5);
 \draw[color=blue, thick, double distance=1pt, dashed]
   (DOT) .. controls++(-.65,0) and ++(-.25,.3) .. (L);
   \draw[line width=0mm] (-0.5,0) .. controls (-0.5,0.8) and (0.5,0.8) .. (0.5,0)
        node[pos=0.1]{\bbullet};
   \node[blue] at (-.35,1.4){$\scs k$};
   \node[blue] at (1.35,1.4){$\scs j$};
   \node[blue] at (-1.45,.20){$\scs \l-1 -k-j$};
   \node at (-1.5,.80){$\sum_j$};
\end{tikzpicture}
};
\endxy
 } "2";"3"};
\endxy
\end{equation}
The condition $\zeta^{-1} \zeta = \oneb$ implies that this composite must equal $\delta_{\ell,k} \oneb_{\Pi^k \onebl \la \l-1-2k\ra}$.  Assume that $\ell \geq k+j \geq k$ or else the composite will contain a negative degree bubble which always equal to zero.  Then closing off $k$ of the dashed lines and bending up the remaining $\ell-k$ dashed lines coming out of the first diagram we get
\[
\sum_{j=0}^{\l-1-k} (-1)^j
 \xy
 (0,3)*{
\begin{tikzpicture}
  \draw[thick, ->] (-0.5,0) .. controls (-0.5,0.8) and (0.5,0.8) .. (0.5,0)
      node[pos=0.1, shape=coordinate](DOT){}
      node[pos=0.41, shape=coordinate](L){}
      node[pos=0.5, shape=coordinate](M){}
      node[pos=0.58, shape=coordinate](R){};
  \draw[thick] (-0.5,0) .. controls (-0.5,-1) and (0.5,-1) .. (0.5,0)
      node[pos=0.05, shape=coordinate](mDOT){}
      node[pos=0.2, shape=coordinate](lDOT){};
 \node[draw, thick, fill=blue!20,rounded corners=4pt,inner sep=3pt]
    (Fj) at (1.75,0.5) {$\scs f_j(\l -1 -k)$};
 \draw[color=blue, thick, double distance=1pt, dashed] (Fj) to [out=90, in=70] (R);
 \draw[color=blue, thick, double distance=1pt, dashed]
    (mDOT) .. controls++(-1.5,.3) and ++(-0.1,1) .. (M);
 \draw[color=blue, thick, double distance=1pt, dashed]
    (lDOT) .. controls++(-1.5,.3) and ++(-0,-1) .. (-1,1.75);
 \draw[color=blue, thick, double distance=1pt, dashed]
     (DOT) .. controls++(-.65,0) and ++(-.25,.3) .. (L);
 \draw[line width=0mm] (-0.5,0) .. controls (-0.5,0.8) and (0.5,0.8) .. (0.5,0)
        node[pos=0.1]{\bbullet};
 \draw[line width=0mm] (-0.5,0) .. controls (-0.5,-1) and (0.5,-1) .. (0.5,0)
      node[pos=0.05]{\bbullet}
      node[pos=0.2]{\bbullet};
   \node[blue] at (-.1,1.1){$\scs k$};
   \node[blue] at (1.75,1.3){$\scs j$};
   \node[blue] at (-.65,1.6){$\scs \ell-k$};
   \node at (1.3,-.25){$\l$};
\end{tikzpicture}
};
\endxy \quad = \quad \delta_{\ell,k}  \oneb_{\onebl \la \l-1-2k\ra}.
\]
By varying $\ell$ for fixed $k$, it is possible to rewrite each of the 2-morphisms $f_j(\ell-1-k)$ as products of bubbles.  For example, by setting $\ell=k$ it follows that the degree zero 2-morphism $f_{0}(\ell-1-k)$ is multiplication by the scalar $1$.  Continuing by induction, decreasing $\ell$ shows that all the $f_j(\ell-1-k)$ can be rewritten as a linear combination of 2-morphisms in the image of the generating 2-morphisms of the 2-category $\Uc$. Notice that the above diagram contains a negative degree bubble if $j>(\ell-k)$, so we can restrict the sum for $0 \leq j \leq (\ell-k)$. Comparing this equation to the defining equations \eqref{eq:fake-bubble-last} for the fake bubbles we must have $f_j(\l-1-k) = B_j$ for $0 \leq \ell, k \leq \l-1$.

For $\l \leq 0$ the adjoint of  Lemma~\ref{lem:main} equation \eqref{eq:main2} implies
\begin{equation}
\hackcenter{ \begin{tikzpicture}[scale=0.5]
  \draw[thick, ->-=0.12, ->-=0.95]
     (0.6,-0.1) .. controls ++(0.1,1.3) and ++(-0.1,1.3) .. (-0.6,-0.1);
  \draw[color=blue,  thick, dashed, double distance=1pt] (0,1) -- (0,2.4) ;
   \node[draw, thick, fill=blue!20,rounded corners=4pt,inner sep=3pt]
     (fi) at (0,1) {$\zeta(-\l-1-k)$};
   \node[blue] at (.3,2.1) {$\scs k$};
   \node at (2.3,0) {$\l$};
\end{tikzpicture} }
\quad = \quad
\sum_{j=0}^{-\l-1-k}
 \hackcenter{\begin{tikzpicture}
  \draw[thick,->-=0.8] (0.5,.25) -- (0.5,.5);
  \draw[thick,->-=0.55] (-0.5,.5) -- (-0.5,.25);
  \draw[thick] (0.5,.5) .. controls ++(.1,.8) and ++(-.1,.8) .. (-0.5,.5)
      node[pos=0.1, shape=coordinate](DOT){};
 \node[draw, thick, fill=blue!20,rounded corners=4pt,inner sep=3pt]
    (Fj) at (1.75,0.75) {$\scs f_j'(-\l-1-k)$};
 \draw[color=blue, thick, double distance=1pt, dashed]
    (Fj) .. controls ++(0,.4) and ++(0,-.6) .. (1.25,1.75);
  \draw[color=blue, thick, double distance=1pt, dashed]
    (DOT) .. controls++(-.5,.4) and ++(0,-1) .. (-.75,1.75);
   \node at (DOT){\bbullet};
   \node[blue] at (1.5,1.6){$\scs j$};
   \node[blue] at (0,1.60){$\scs \l-1 -k-j$};
   \node at (-1,.7) {$\l$};
\end{tikzpicture} }
\end{equation}
for some 2-morphisms $f_j'(-\l-1-k)\maps \Ftt\Ett\onebl \to \Pi^j\onebl\la 2j \ra$.
Again, by considering the composite map $\pi^{\l+1+\ell}\onebl\la -\l-1-2\ell\ra \to \Ftt\Ett\onebl \to \pi^{\l+1+k}\onebl\la -l-1-2k\ra$ the 2-morphisms $f_j'(-\l-1-k)$ can be related to the fake bubbles $\overline{B_j}$.
\end{proof}

% -------------------------------------------------------------------------------
%
\subsection{Relations resulting from the ${\mf{sl}}_2$ commutator relation}
%
% -------------------------------------------------------------------------------

In this section we collect a set of relations that follow from the general form of the inverse $\zeta^{-1}$ of $\zeta$.  We also uniquely solve for the coefficients $\beta_{\l}$ and $c_{-1}$.
Observe that $\zeta^{-1}$ in Proposition~\ref{prop_form-of-inv} is the inverse of $\zeta$  if and only if the following relations hold in $\Cc$:

% - - - - - - - - - - - - - - - - - - - - - - - - - - - - - - - - - - - - - - - -
%
%
\subsubsection{Relations for $\lambda >0$}
%
% - - - - - - - - - - - - - - - - - - - - - - - - - - - - - - - - - - - - - - - -
%
In addition to the equations relating part of the inverse to fake bubbles, we have the following relations.
\[
\hackcenter{\begin{tikzpicture}
  \draw[thick, ->] (-0.5,0) to (-0.5,2);
  \draw[thick, <-] (0.5,0) to (0.5,2);
  \node at (1,1.5) {$\l$};
\end{tikzpicture}}
\quad = \quad
\beta_{\l} \;\;
\hackcenter{\begin{tikzpicture}
  \draw[thick, <-] (0.5,0) .. controls (0.5,0.4) and (-0.5,0.6) .. (-0.5,1)
      node[pos=0.5, shape=coordinate](X){};
    \draw[thick, ->] (-0.5,0) .. controls (-0.5,0.4) and (0.5,0.6) .. (0.5,1);
  \draw[thick, ->] (0.5,1) .. controls (0.5,1.4) and (-0.5,1.6) .. (-0.5,2)
      node[pos=0.5, shape=coordinate](Y){};
    \draw[thick, <-] (-0.5,1) .. controls (-0.5,1.4) and (0.5,1.6) .. (0.5,2);
  \draw[color=blue,  thick, dashed] (Y) -- (X);
  \node at (1,1.5) {$\l$};
\end{tikzpicture} }
\quad + \quad
\sum_{
\xy (0,2)*{\scs f_1+f_2+f_3}; (0,-1)*{\scs = \l-1}; \endxy} (-1)^{f_3} \;\;
 \hackcenter{
\begin{tikzpicture}
  \draw[thick, ->] (-0.5,0) .. controls (-0.5,0.8) and (0.5,0.8) .. (0.5,0)
      node[pos=0.1, shape=coordinate](DOT){}
      node[pos=0.42, shape=coordinate](L){}
      node[pos=0.5, shape=coordinate](M){}
      node[pos=0.58, shape=coordinate](R){};
  \draw[thick, ->]
  (1.9,1) .. controls ++(0,0.6) and ++(0,0.6) .. (1.1,1)
      node[pos=0.05, shape=coordinate](Z){};
  \draw[thick] (1.9,1) .. controls ++(0,-0.6) and ++(-0,-0.6) .. (1.1,1)
      node[pos=0.5, shape=coordinate](X){}
      node[pos=0.2, shape=coordinate](Y){};
  \draw[color=blue, thick, double distance=1pt, dashed]
    (X) .. controls++(-.1,.5) and ++(-.2,.3) .. (Y)
         node[pos=0.9,right]{$\scs -\l-1$\;};
   \draw[color=blue, thick, double distance=1pt, dashed]
    (Z) .. controls ++(-.5,.4) and ++(.2,.8) .. (R) ;
   \node[blue] at (1.25,0.8){$\scs $\;};
     \node at (Y) {\bbullet};
     \node at (Z) {\bbullet};
 \draw[thick, <-] (-0.5,2.25) .. controls ++(0,-.8) and ++(0,-.8) .. (0.5,2.25)
      node[pos=0.2, shape=coordinate](tDOT){};
 \draw[color=blue, thick, double distance=1pt, dashed]
   (M) .. controls ++(.4,1.4) and ++(-.5,-1) .. (-1.1,1.8) to[out=90, in=140] (tDOT);
 \draw[color=blue, thick, double distance=1pt, dashed]
    (DOT) .. controls++(-.65,0) and ++(-.25,.3) .. (L);
 \node at (tDOT){\bbullet}; \node at (DOT){\bbullet};
   \node[blue] at (.6,1.4){$\scs f_3$};
   \node[blue] at (-1.35,1.45){$\scs f_1$};
   \node[blue] at (-1.0,.30){$\scs f_2$};
   \node at (1,2) {$\l$};
\end{tikzpicture} }
\]
\begin{equation} \label{eq:FEtEF-beta}
\beta_{\l} \;\;
\hackcenter{\begin{tikzpicture}
  \draw[thick, <-] (-0.5,0) .. controls (-0.5,0.4) and (0.5,0.6) .. (0.5,1)
      node[pos=0.5, shape=coordinate](X){};
    \draw[thick, ->] (0.5,0) .. controls (0.5,0.4) and (-0.5,0.6) .. (-0.5,1);
  \draw[thick, ->] (-0.5,1) .. controls (-0.5,1.4) and (0.5,1.6) .. (0.5,2)
      node[pos=0.5, shape=coordinate](Y){};
    \draw[thick, <-] (0.5,1) .. controls (0.5,1.4) and (-0.5,1.6) .. (-0.5,2);
  \draw[color=blue,  thick, dashed] (X) -- (0,0);
  \draw[color=blue,  thick, dashed] (Y) -- (0,2);
  \node at (1,0.5) {$\l$};
\end{tikzpicture}} \quad = \quad
\hackcenter{\begin{tikzpicture}
  \draw[thick, ->] (-0.5,2) to (-0.5,0);
  \draw[thick, <-] (0.5,2) to (0.5,0);
  \draw[color=blue,  thick, dashed] (0,0) -- (0,2);
  \node at (1,0.5) {$\l$};
\end{tikzpicture}}
\end{equation}
We simplify the remaining relations omitting several relations that follow from those below using odd nilHecke relations. For all $0 \leq m < \l$
\begin{align} \label{eq:curldiep}
\hackcenter{\begin{tikzpicture}
  \draw[thick, <-] (-0.5,0) .. controls (-0.5,0.4) and (0.5,0.6) .. (0.5,1)
      node[pos=0.5, shape=coordinate](X){};
  \draw[thick] (0.5,0) .. controls (0.5,0.4) and (-0.5,0.6) .. (-0.5,1)
       node[pos=1, shape=coordinate](DOT){};;
  \draw[thick, ->] (-0.5,1) .. controls ++(0,.6) and ++(0,.6) .. (0.5,1)
      node[pos=0.5, shape=coordinate](Y){};
  \draw[color=blue,  thick, dashed] (X) -- (0,0);
  \draw[color=blue,  thick, double distance=1pt,dashed] (Y) -- (0,2);
   \draw[color=blue,  thick, double distance=1pt,dashed] (DOT) to[bend left](-1,2);
  \node at (1,0.5) {$\l$};
  \node[blue] at (-.7,1.8) {$\scs m$};
  \node at (DOT) {\bbullet};
\end{tikzpicture}} &\quad =\quad 0 &
\hackcenter{\begin{tikzpicture}
  \draw[thick] (-0.5,1) .. controls ++(0,-0.6) and ++(0,-0.6) .. (0.5,1)
      node[pos=0.5, shape=coordinate](X){}
      node[pos=0, shape=coordinate](DOT){};;
  \draw[thick, ->] (-0.5,1) .. controls (-0.5,1.4) and (0.5,1.6) .. (0.5,2)
      node[pos=0.5, shape=coordinate](Y){};
  \draw[thick, <-] (0.5,1) .. controls (0.5,1.4) and (-0.5,1.6) .. (-0.5,2);
  \draw[color=blue,  thick, dashed] (Y) -- (0,2);
  \draw[color=blue,  thick, double distance=1pt,dashed] (DOT) to[bend left](-1,2);
  \node[blue] at (-.7,1.8) {$\scs m$};
  \node at (1,0.5) {$\l$};
  \node at (DOT) {\bbullet};
\end{tikzpicture}} &\quad =\quad 0 .
\end{align}
Note that the two equations above already follow from Lemma~\ref{lem:E} using the adjunctions.

% - - - - - - - - - - - - - - - - - - - - - - - - - - - - - - - - - - - - - - - -
%
%
\subsubsection{Relations for $\lambda <0$}
%
% - - - - - - - - - - - - - - - - - - - - - - - - - - - - - - - - - - - - - - - -
%
In addition to the equations relating part of the inverse to fake bubbles in weights $\l<0$, we have the following relations.
\begin{equation} \label{eq:EFp}
\hackcenter{\begin{tikzpicture}
  \draw[thick, <-] (-0.5,0) to (-0.5,2);
  \draw[thick, ->] (0.5,0) to (0.5,2);
  \node at (1,0.5) {$\l$};
\end{tikzpicture}}
 \quad = \quad \beta_{\l} \;\;
\hackcenter{\begin{tikzpicture}
  \draw[thick, <-] (-0.5,0) .. controls (-0.5,0.4) and (0.5,0.6) .. (0.5,1)
      node[pos=0.5, shape=coordinate](X){};
    \draw[thick, ->] (0.5,0) .. controls (0.5,0.4) and (-0.5,0.6) .. (-0.5,1);
  \draw[thick, ->] (-0.5,1) .. controls (-0.5,1.4) and (0.5,1.6) .. (0.5,2)
      node[pos=0.5, shape=coordinate](Y){};
    \draw[thick, <-] (0.5,1) .. controls (0.5,1.4) and (-0.5,1.6) .. (-0.5,2);
  \draw[color=blue,  thick, dashed]
     (Y) .. controls ++(.1,.4) and ++(.1,.4)  .. (-.6,1.6)
     .. controls ++(0,-.3) and ++(0,.3) ..(0,1)
     .. controls ++(0,-.3) and ++(0,.3) .. (-.6,.4)
     .. controls ++(.1,-.4) and ++(.1,-.4)  .. (X);
\end{tikzpicture} }
\quad + \quad
\sum_{
\xy (0,2)*{\scs f_1+f_2+f_3}; (0,-1)*{\scs = -\l-1}; \endxy} (-1)^{f_3} \;\;
 \hackcenter{
\begin{tikzpicture}
% BOTTOM CAP
  \draw[thick, <-] (0.5,0) .. controls ++(0,0.8) and ++(0,0.8) .. (-0.5,0)
      node[pos=0.15, shape=coordinate](DOT){};
  % BUBBLE
  \draw[thick, ->] (1.1,.75) .. controls ++(-0,0.6) and ++(0,0.6) .. (1.9,.75)
      node[pos=0.5, shape=coordinate](X){}
      node[pos=0.1, shape=coordinate](Y){};
  \draw[thick] (1.1,.75) .. controls ++(0,-0.6) and ++(0,-0.6) .. (1.9,.75)
      node[pos=0.1, shape=coordinate](Z){};
  \draw[color=blue, thick, double distance=1pt, dashed]
    (X) .. controls++(0,.65) and ++(-.65,.3) .. (Y) node[pos=0.15,right]{$\scs \l-1$\;};
    %% TOP CUP
  \draw[thick, ->] (0.8,2.25) -- (0.8,2.5);
 \draw[thick] (-0.8,2.25) -- (-0.8,2.5);
 \draw[thick] (0.8,2.25) .. controls ++(0,-.8) and ++(0,-.8) .. (-0.8,2.25)
      node[pos=0.15, shape=coordinate](tDOT){}
      node[pos=0.42, shape=coordinate](RCUP){}
      node[pos=0.5, shape=coordinate](MCUP){}
      node[pos=0.58, shape=coordinate](LCUP){};
 \draw[color=blue, thick, double distance=1pt, dashed]
    (tDOT) ..controls ++(-.3,.3) and ++(0,.4) .. (RCUP) ;
 \draw[color=blue, thick, double distance=1pt, dashed]
    (DOT) .. controls ++(-.3,.2) and ++(0,-.5) .. (-1,1)
     .. controls ++(0,1.7) and ++(.1,.7) .. (MCUP) ;
 \draw[color=blue, thick, double distance=1pt, dashed]
    (Z) .. controls ++(-.3,.4) and ++(0,-.4) .. (-.75,1.5)
    .. controls ++(0,.5) and ++(0,.4) .. (LCUP) ;
 \node at (tDOT){\bbullet};
 \node at (DOT){\bbullet};
 \node at (Y) {\bbullet};
 \node at (Z) {\bbullet};
   \node[blue] at (.5,1.25){$\scs f_3$};
   \node[blue] at (.6,2.35){$\scs f_1$};
   \node[blue] at (-1.0,.40){$\scs f_2$};
   \node at (1,2) {$\l$};
\end{tikzpicture} }
\end{equation}

\begin{equation}
\beta_{\l} \;\;
\hackcenter{\begin{tikzpicture}
\draw[thick, <-] (0.5,0) .. controls (0.5,0.4) and (-0.5,0.6) .. (-0.5,1)
      node[pos=0.5, shape=coordinate](X){};
    \draw[thick, ->] (-0.5,0) .. controls (-0.5,0.4) and (0.5,0.6) .. (0.5,1);
  \draw[thick, ->] (0.5,1) .. controls (0.5,1.4) and (-0.5,1.6) .. (-0.5,2)
      node[pos=0.5, shape=coordinate](Y){};
    \draw[thick, <-] (-0.5,1) .. controls (-0.5,1.4) and (0.5,1.6) .. (0.5,2);
   \draw[color=blue,  thick, dashed]
     (Y) .. controls ++(.1,-.5) and ++(-.1,-.5)  .. (-.6,1.5)
     .. controls ++(0,.3) and ++(0,-.4) ..(0,2);
     \draw[color=blue,  thick, dashed]
     (X) .. controls ++(.1,.4) and ++(-.1,.4)  .. (-.6,.5)
     .. controls ++(0,-.3) and ++(0,.4) ..(0,0);
  \node at (1,0.5) {$\l$};
\end{tikzpicture}}
\quad = \quad
\hackcenter{\begin{tikzpicture}
  \draw[thick, ->] (-0.5,0) to (-0.5,2);
  \draw[thick, <-] (0.5,0) to (0.5,2);
  \draw[color=blue,  thick, dashed] (0,0) -- (0,2);
\end{tikzpicture}}
\end{equation}
For $0 \leq m < -\l$ the curl relations
\begin{align} \label{eq:curldien}
\hackcenter{\begin{tikzpicture}
  \draw[thick, <-] (0.5,0) .. controls (0.5,0.4) and (-0.5,0.6) .. (-0.5,1)
      node[pos=0.5, shape=coordinate](X){};
  \draw[thick] (-0.5,0) .. controls (-0.5,0.4) and (0.5,0.6) .. (0.5,1)
       node[pos=1, shape=coordinate](DOT){};;
  \draw[thick, ->] (0.5,1) .. controls ++(0,.6) and ++(0,.6) .. (-0.5,1)
      node[pos=0.5, shape=coordinate](Y){};
  \draw[color=blue,  thick, dashed]
     (X) .. controls ++(.1,.4) and ++(-.1,.4)  .. (-.6,.5)
     .. controls ++(0,-.3) and ++(0,.4) ..(0,0);
   \draw[color=blue,  thick, double distance=1pt,dashed] (DOT) to[bend right](1,2);
  \node at (-1,0.5) {$\l$};
  \node[blue] at (.7,1.8) {$\scs m$};
  \node at (DOT) {\bbullet};
\end{tikzpicture}} &\quad =\quad 0 &
\hackcenter{\begin{tikzpicture}
  \draw[thick] (-0.5,1) .. controls ++(0,-0.6) and ++(0,-0.6) .. (0.5,1)
      node[pos=0.5, shape=coordinate](X){}
      node[pos=0, shape=coordinate](DOT){};;
  \draw[thick, ->] (-0.5,1) .. controls (-0.5,1.4) and (0.5,1.6) .. (0.5,2)
      node[pos=0.5, shape=coordinate](Y){};
  \draw[thick, <-] (0.5,1) .. controls (0.5,1.4) and (-0.5,1.6) .. (-0.5,2);
  \draw[color=blue,  thick, double distance=1pt,dashed] (X) .. controls ++(0,1.25) and ++(0,-.75) .. (-1,2);
  \draw[color=blue,  thick, dashed] (Y) -- (0,2);
  \draw[color=blue,  thick, double distance=1pt,dashed] (DOT) .. controls ++(-.75,0) and ++(0,-1) ..(-1.5,2);
  \node[blue] at (-1.8,1.8) {$\scs m$};
  \node at (1,0.5) {$\l$};
  \node at (DOT) {\bbullet};
\end{tikzpicture}} &\quad =\quad 0 .
\end{align}
hold in any strong supercategorical action.

% - - - - - - - - - - - - - - - - - - - - - - - - - - - - - - - - - - - - - - - -
%
%
\subsubsection{Relations for $\lambda =0$}
%
% - - - - - - - - - - - - - - - - - - - - - - - - - - - - - - - - - - - - - - - -

Invertibility of $\zeta$ for $\l=0$ implies that in any strong supercategorical action the relations
\begin{equation} \label{eq:EFtoFEzero}
\hackcenter{\begin{tikzpicture}
  \draw[thick, ->] (-0.5,0) to (-0.5,2);
  \draw[thick, <-] (0.5,0) to (0.5,2);
  \node at (1,1.5) {$\l$};
\end{tikzpicture}}
\quad = \quad
\beta_{0} \;\;
\hackcenter{\begin{tikzpicture}
  \draw[thick, <-] (0.5,0) .. controls (0.5,0.4) and (-0.5,0.6) .. (-0.5,1)
      node[pos=0.5, shape=coordinate](X){};
    \draw[thick, ->] (-0.5,0) .. controls (-0.5,0.4) and (0.5,0.6) .. (0.5,1);
  \draw[thick, ->] (0.5,1) .. controls (0.5,1.4) and (-0.5,1.6) .. (-0.5,2)
      node[pos=0.5, shape=coordinate](Y){};
    \draw[thick, <-] (-0.5,1) .. controls (-0.5,1.4) and (0.5,1.6) .. (0.5,2);
  \draw[color=blue,  thick, dashed] (Y) -- (X);
  \node at (1,1.5) {$\l$};
\end{tikzpicture} }
\qquad
\hackcenter{\begin{tikzpicture}
  \draw[thick, ->] (-0.5,2) to (-0.5,0);
  \draw[thick, <-] (0.5,2) to (0.5,0);
  \draw[color=blue,  thick, dashed] (0,0) -- (0,2);
  \node at (1,0.5) {$\l$};
\end{tikzpicture}} \quad = \quad
\beta_{0} \;\;
\hackcenter{\begin{tikzpicture}
  \draw[thick, <-] (-0.5,0) .. controls (-0.5,0.4) and (0.5,0.6) .. (0.5,1)
      node[pos=0.5, shape=coordinate](X){};
    \draw[thick, ->] (0.5,0) .. controls (0.5,0.4) and (-0.5,0.6) .. (-0.5,1);
  \draw[thick, ->] (-0.5,1) .. controls (-0.5,1.4) and (0.5,1.6) .. (0.5,2)
      node[pos=0.5, shape=coordinate](Y){};
    \draw[thick, <-] (0.5,1) .. controls (0.5,1.4) and (-0.5,1.6) .. (-0.5,2);
  \draw[color=blue,  thick, dashed] (X) -- (0,0);
  \draw[color=blue,  thick, dashed] (Y) -- (0,2);
  \node at (1,0.5) {$\l$};
\end{tikzpicture}}
\end{equation}
holds.  Note this is equation is consistent with the alternative equation
\begin{equation}
\hackcenter{\begin{tikzpicture}
  \draw[thick, ->] (-0.5,0) to (-0.5,2);
  \draw[thick, <-] (0.5,0) to (0.5,2);
  \draw[color=blue,  thick, dashed] (0,0) -- (0,2);
\end{tikzpicture}}
\quad = \quad
\beta_{0} \;\;
\hackcenter{\begin{tikzpicture}
\draw[thick, <-] (0.5,0) .. controls (0.5,0.4) and (-0.5,0.6) .. (-0.5,1)
      node[pos=0.5, shape=coordinate](X){};
    \draw[thick, ->] (-0.5,0) .. controls (-0.5,0.4) and (0.5,0.6) .. (0.5,1);
  \draw[thick, ->] (0.5,1) .. controls (0.5,1.4) and (-0.5,1.6) .. (-0.5,2)
      node[pos=0.5, shape=coordinate](Y){};
    \draw[thick, <-] (-0.5,1) .. controls (-0.5,1.4) and (0.5,1.6) .. (0.5,2);
   \draw[color=blue,  thick, dashed]
     (Y) .. controls ++(.1,-.5) and ++(-.1,-.5)  .. (-.6,1.5)
     .. controls ++(0,.3) and ++(0,-.4) ..(0,2);
     \draw[color=blue,  thick, dashed]
     (X) .. controls ++(.1,.4) and ++(-.1,.4)  .. (-.6,.5)
     .. controls ++(0,-.3) and ++(0,.4) ..(0,0);
  \node at (1,0.5) {$\l$};
\end{tikzpicture}}
\qquad
\qquad\hackcenter{\begin{tikzpicture}
  \draw[thick, <-] (-0.5,0) to (-0.5,2);
  \draw[thick, ->] (0.5,0) to (0.5,2);
  \node at (1,0.5) {$\l$};
\end{tikzpicture}}
 \quad = \quad \beta_{0} \;\;
\hackcenter{\begin{tikzpicture}
  \draw[thick, <-] (-0.5,0) .. controls (-0.5,0.4) and (0.5,0.6) .. (0.5,1)
      node[pos=0.5, shape=coordinate](X){};
    \draw[thick, ->] (0.5,0) .. controls (0.5,0.4) and (-0.5,0.6) .. (-0.5,1);
  \draw[thick, ->] (-0.5,1) .. controls (-0.5,1.4) and (0.5,1.6) .. (0.5,2)
      node[pos=0.5, shape=coordinate](Y){};
    \draw[thick, <-] (0.5,1) .. controls (0.5,1.4) and (-0.5,1.6) .. (-0.5,2);
  \draw[color=blue,  thick, dashed]
     (Y) .. controls ++(.1,.4) and ++(.1,.4)  .. (-.6,1.6)
     .. controls ++(0,-.3) and ++(0,.3) ..(0,1)
     .. controls ++(0,-.3) and ++(0,.3) .. (-.6,.4)
     .. controls ++(.1,-.4) and ++(.1,-.4)  .. (X);
\end{tikzpicture} }
\end{equation}

% -------------------------------------------------------------------------------
%
\subsection{Finding the free parameters}
%
% -------------------------------------------------------------------------------

We can now solve for the remaining free parameters  using the relations derived so far.

\begin{prop} \label{prop_free_param}
In any strong supercategorical action the free parameters must be fixed as follows.
\begin{enumerate}
  \item The coefficient $c_{-1}$ from \eqref{eq_defcmone} is equal to 1.
  \item For all values of $\l$ where $\onebl$ is not zero the coefficients from Proposition ~\ref{prop_form-of-inv} satisfy $\beta_{\l}=-1$ for all $\l$.
  \item The coefficients $c_0^+$ and $c_0^-$ from \eqref{eq:c0} are both equal to one.
\end{enumerate}
\end{prop}

\begin{proof}
To find the coefficient $c_{-1}$ note that
\[
 0 \;\; =\;\;
 \hackcenter{\begin{tikzpicture}[scale=0.8]
  \draw[thick, ->] (-0.5,0) .. controls (-0.5,0.4) and (0.5,0.6) .. (0.5,1)
      node[pos=0.5, shape=coordinate](X){};
    \draw[thick, ->] (0.5,0) .. controls (0.5,0.4) and (-0.5,0.6) .. (-0.5,1);
  \draw[thick] (-0.5,1) .. controls (-0.5,1.4) and (0.5,1.6) .. (0.5,2)
      node[pos=0.5, shape=coordinate](Y){};
  \draw[thick, ->] (0.5,1) .. controls (0.5,1.4) and (-0.5,1.6) .. (-0.5,2);
  \draw[thick] (-0.5,2) .. controls ++(0,.4) and ++(0,.4) .. (-1.5,2)
    node[pos=0.5, shape=coordinate](CAP){};
  \draw[thick] (-0.5,0) .. controls ++(0,-.4) and ++(0,-.4) .. (-1.5,0);
  \draw[thick, ->-=0.5] (-1.5,2) -- (-1.5,0);
  \draw[thick, ->] (.5,2) -- (.5,2.5);
  \draw[thick] (.5,0) -- (.5,-.5);
  \draw[color=blue,  thick, dashed] (X) .. controls ++(-1.2,-.2) and ++(-1.2,.2) .. (Y);
  \node at (1,0.5) {$-1$};
\end{tikzpicture}} \quad = \quad
 \hackcenter{\begin{tikzpicture}[scale=0.8]
  \draw[thick, <-] (0.5,0) .. controls (0.5,0.4) and (-0.5,0.6) .. (-0.5,1)
      node[pos=0.5, shape=coordinate](X){};
    \draw[thick, ->] (-0.5,0) .. controls (-0.5,0.4) and (0.5,0.6) .. (0.5,1);
  \draw[thick] (0.5,1) .. controls (0.5,1.4) and (-0.5,1.6) .. (-0.5,2)
      node[pos=0.5, shape=coordinate](Y){};
  \draw[thick, <-] (-0.5,1) .. controls (-0.5,1.4) and (0.5,1.6) .. (0.5,2);
  \draw[thick] (0.5,2) .. controls ++(0,.4) and ++(0,.4) .. (1.5,2)
    node[pos=0.5, shape=coordinate](CAP){};
  \draw[thick] (0.5,0) .. controls ++(0,-.4) and ++(0,-.4) .. (1.5,0);
  \draw[thick, ->-=0.5] (1.5,0) -- (1.5,2)
     node[pos=0.85, shape=coordinate](DOT){};
  \draw[thick, ->] (-.5,2) -- (-.5,2.5);
  \draw[thick] (-.5,0) -- (-.5,-.5);
  \draw[color=blue,  thick, dashed] (X) -- (Y);
  \node at (-1,0.5) {$+1$};
\end{tikzpicture}}
\quad = \quad
\frac{1}{\beta_1} \;\; \left(
(c_- -1)
\hackcenter{\begin{tikzpicture}[scale=0.8]
  \draw[thick, ->] (0,-1) -- (0,2);
  \node at (-.4,0.5) {$+1$};
 \end{tikzpicture}} \quad
 \right)
\]
where the last equality follows from \eqref{eq:EFp}.

Using \eqref{eq:FEtEF-beta} and the relations from Section~\ref{subsec-half-sideways} we have
\begin{eqnarray}\label{eqn-one-over-betal}
\frac{1}{\beta_{\l}} \;
\hackcenter{\begin{tikzpicture}
  \draw[thick, ->] (1.5,-1) -- (1.5,1);
  \draw[color=blue,  thick, dashed] (1,-1) -- (1,1);
  \draw[thick, ->] (-0.4,0) .. controls ++(0,0.6) and ++(0,0.6) .. (0.4,0)
      node[pos=0.5, shape=coordinate](X){}
      node[pos=0.1, shape=coordinate](Y){};
  \draw[thick] (-0.4,0) .. controls ++(0,-0.6) and ++(0,-0.6) .. (0.4,0);
 %% Draw double blue curvy line
  \draw[color=blue, thick, double distance=1pt, dashed] (X) .. controls++(0,.65) and ++(-.65,.3) .. (Y)
         node[pos=0.15,right]{$\scs \l+1$\;};
  \node at (.25,-0.75) {$\lambda+2$};
  \node at (Y) {\bbullet};
\end{tikzpicture}}
\quad = \;\;
\hackcenter{\begin{tikzpicture}[scale=0.8]
  \draw[thick, <-] (-0.5,0) .. controls (-0.5,0.4) and (0.5,0.6) .. (0.5,1)
      node[pos=0.5, shape=coordinate](X){};
    \draw[thick, ->] (0.5,0) .. controls (0.5,0.4) and (-0.5,0.6) .. (-0.5,1);
  \draw[thick] (-0.5,1) .. controls (-0.5,1.4) and (0.5,1.6) .. (0.5,2)
      node[pos=0.5, shape=coordinate](Y){};
  \draw[thick, <-] (0.5,1) .. controls (0.5,1.4) and (-0.5,1.6) .. (-0.5,2);
  \draw[thick] (-0.5,2) .. controls ++(0,.4) and ++(0,.4) .. (-1.5,2)
    node[pos=0.5, shape=coordinate](CAP){};
  \draw[thick] (-0.5,0) .. controls ++(0,-.4) and ++(0,-.4) .. (-1.5,0);
  \draw[thick, ->-=0.5] (-1.5,0) -- (-1.5,2)
     node[pos=0.85, shape=coordinate](DOT){};
  \draw[thick, ->] (.5,2) -- (.5,2.5);
  \draw[thick] (.5,0) -- (.5,-.5);
  \draw[color=blue,  thick, dashed] (X) -- (0,-.5);
  \draw[color=blue,  thick, dashed] (Y) -- (0,2.5);
  \draw[color=blue, thick, double distance=1pt, dashed]
   (DOT) .. controls ++(-.9,.3) and ++(.1,.7) .. (CAP)
   node[pos=0.35,left]{$\scs \l+1$\;};
  \node at (1,0.5) {$\l$};
  \node at (DOT) {\bbullet};
\end{tikzpicture}} \quad = \quad
\hackcenter{\begin{tikzpicture}[scale=0.8]
  \draw[thick] (0.5,-.25) .. controls ++(0,0.5) and ++(0,-0.5) .. (-0.5,.75)
      node[pos=0.5, shape=coordinate](X){}
      node[pos=0.9, shape=coordinate](LDOT){};
  \draw[thick] (-0.5,-.25) .. controls ++(0,0.5) and ++(0,-0.5) .. (0.5,.75);
  \draw[thick] (0.5,1.25) .. controls ++(0,.5) and ++(0,-.5) .. (-0.5,2.25)
      node[pos=0.5, shape=coordinate](Y){};
  \draw[thick, ->] (-0.5,1.25) .. controls ++(0,.5) and ++(0,-.5) .. (0.5,2.25)
     node[pos=0.15, shape=coordinate](DOT){};
  \draw[thick] (0.5,2.25) .. controls ++(0,.4) and ++(0,.4) .. (1.5,2.25)
    node[pos=0.5, shape=coordinate](CAP){};
  \draw[thick] (0.5,-.25) .. controls ++(0,-.4) and ++(0,-.4) .. (1.5,-.25);
  \draw[thick, ->-=0.5] (1.5,2.25) -- (1.5,-.25);
  \draw[thick, ->] (-.5,2.25) -- (-.5,3);
  \draw[thick] (-.5,-.25) -- (-.5,-.5);
  \draw[thick] (-.5,.75) -- (-.5,1.25)
    node[pos=0.6, shape=coordinate](MDOT){};
  \draw[thick] (.5,.75) -- (.5,1.25);
  \draw[color=blue,  thick, dashed] (X) to[out=180, in=90] (-1,-.5);
  \draw[color=blue,  thick, dashed] (Y) to[out=190, in=-100] (-1,3);
  \draw[color=blue, thick, double distance=1pt, dashed]
   (DOT) .. controls ++(-1,.8) and ++(.1,.7) .. (CAP)
   node[pos=0.85,above]{$\scs \l-1$\;};
   \draw[color=blue, thick, dashed] (LDOT) .. controls ++(-.6,.4) and ++(-.6,.5) .. (MDOT);
  \node at (2,0.5) {$\l$};
  \node at (DOT) {\bbullet};\node at (LDOT) {\bbullet};\node at (MDOT) {\bbullet};
\end{tikzpicture}}
\end{eqnarray}
for all $\l \geq 0$.
Now carefully applying the inductive dot slide formula, we can slide all of the $\l+1$ dots through the top crossing.  The term in which all the dots slide through the crossing is zero by the quadratic odd nilHecke relation.  What remains is a symmetric sum of terms with $\l$ dots where the crossing has been resolved.  By (the adjoint of) \eqref{eq:curldiep}, all the terms in this sum are zero except for the term in which all of the $\l$ dots are in the curl.  Hence the rightmost diagram in \eqref{eqn-one-over-betal} equals
\begin{equation}
\hackcenter{\begin{tikzpicture}[scale=0.8]
  \draw[thick] (0.5,-.25) .. controls ++(0,0.5) and ++(0,-0.5) .. (-0.5,.75)
      node[pos=0.5, shape=coordinate](X){};
  \draw[thick] (-0.5,-.25) .. controls ++(0,0.5) and ++(0,-0.5) .. (0.5,.75);
  \draw[thick] (0.5,1.5) .. controls ++(0,.4) and ++(0,.4) .. (1.5,1.5)
    node[pos=0.5, shape=coordinate](CAP){};
  \draw[thick] (0.5,-.25) .. controls ++(0,-.4) and ++(0,-.4) .. (1.5,-.25);
  \draw[thick, ->-=0.5] (1.5,1.5) -- (1.5,-.25);
  \draw[thick, ->] (-.5,1.5) -- (-.5,3);
  \draw[thick] (-.5,-.25) -- (-.5,-.5);
  \draw[thick] (-.5,.75) -- (-.5,1.5);
  \draw[thick] (.5,.75) -- (.5,1.5)
     node[pos=0.65, shape=coordinate](DOT){}
     node[pos=0, shape=coordinate](LDOT){};
  \draw[color=blue,  thick, dashed] (X) to[out=180, in=90] (-1,-.5);
  \draw[color=blue, thick, double distance=1pt, dashed]
   (DOT) .. controls ++(-1,.8) and ++(.1,.7) .. (CAP)
   node[pos=0.85,above]{$\scs \l-1$\;};
  \draw[color=blue, thick, dashed]
    (LDOT) .. controls ++(-1,.4) and ++(.2,-.5) .. (-1.25,1.5)
    .. controls ++(0,.4) and ++(0,.4) .. (-1.85,1.75)
    .. controls ++(0,-.5) and ++(0,-1.5) .. (-1.25,3);
  \node at (2,0.5) {$\l$};
  \node at (DOT) {\bbullet};\node at (LDOT) {\bbullet};
\end{tikzpicture}}
\quad = \quad - \;\;
\hackcenter{\begin{tikzpicture}
  \draw[thick, ->] (1.75,-1) -- (1.75,1);
  \draw[color=blue,  thick, dashed] (1,-1) -- (1,1);
  \node at (2,-0) {$\lambda$};
\end{tikzpicture}},
\end{equation}
by the odd nilHecke dot slide and the fact that the curl with only $\l-1$ dots is zero.  Since the degree zero bubble is equal to multiplication by 1, this shows that $\beta_{\l}=-1$ for $\l>0$ and $1/\beta_{0}=c_0^+$.

A similar calculation for $\l<0$ implies that $\beta_{\l}=-1$ for $\l<0$ and that $1/\beta_0 = c_0^-$.  Capping off \eqref{eq:EFtoFEzero} with no dots and simplifying implies that $-\beta_0 c_0^+ c_0^- =1$, completing the proof.
\end{proof}

% -------------------------------------------------------------------------------
%
\subsection{Main theorem of 2-representations}
%
% -------------------------------------------------------------------------------

We summarize the results of this section with the following theorem.

\begin{thm} \label{thm-main2rep}
A strong supercategorical action of $\sltwo$ on $\Cc$ induces a 2-representation $\Udotc \to \Cc$.
\end{thm}

%#####################################################################
%
\section{The action on cyclotomic quotients}
%
%#####################################################################

%---------------------------------------------------------------------
\subsection{Defining the action}
%---------------------------------------------------------------------

Let $R(n)$ be the odd nilHecke algebra on $n$ strands and, for a dominant integral weight $\Lambda$, let $R^\Lambda(n)$ be the corresponding cyclotomic quotient.  Let $\ONHc^\Lambda$ be the full sub-super-2-category of $\SBim$ whose objects are the algebras $R^\Lambda(n)$ for all $n\geq0$.  Set $\lambda=\Lambda-2n$; this is the weight corresponding to $R^\Lambda(n)$ when we view $R^\Lambda(n)$ as a categorified weight space.

There are three steps to the argument in this section:
\begin{enumerate}
\item Kang, Kashiwara, and Oh nearly prove in \cite{KKO,KKO2} that there is a strong supercategorical action of $\sltwo$ on $\ONHc^\Lambda$.  They omit only the brick condition \eqref{co:hom} of Definition~\ref{def_strong}.  The 1-morphisms $\Ett,\Ftt$ are the bimodule kernels of restriction and induction along the maps $\iota_{n,1}$ of \eqref{eqn-iota-onh-ab}, respectively.  The parity functor is the bimodule $\Pi R^\Lambda(n)$ as defined in Subsection \ref{subsec-conventions} (parity shift and twist the left action by the parity involution).
\item To verify condition \eqref{co:hom} of Definition~\ref{def_strong}, observe that in the action of \cite{KKO,KKO2}, a weight $\l$ is mapped to the cyclotomic quotient ring $R^\Lambda(n)$. Since $R^\Lambda(n)$ is graded Morita equivalent to an odd Grassmannian ring \cite{EKL} and odd Grassmannian rings are graded local,
\begin{equation*}
\dim_q\left(\Pi\HOM_{\ONHc^\Lambda}(R^\Lambda(n),R^\Lambda(n))\right)\in1+\N[\pi,q].
\end{equation*}
In other words, $R^\Lambda(n)$ is a brick.
\item By Theorem~\ref{thm-main2rep}, this gives a 2-representation $\Udotc\to\ONHc^\Lambda$.
\end{enumerate}

We spend the rest of this section writing down how the resulting 2-functor acts, mostly explicitly (right-oriented caps and cups are difficult to write down).  A subset of these details give an explicit description of the strong supercategorical action of \cite{KKO,KKO2}.  For short, we will sometimes abbreviate $R^\Lambda(n)$ by simply $n$ when there is no chance of confusion.  For instance, ``$(n,n+1)$-bimodule'' means ``$(R^\Lambda(n),R^\Lambda(n+1))$-bimodule.''

There is a morphism of super-2-categories mapping $\SBim$ to $\SCat$ sending a superalgebra to its supermodule category and a super bimodule to the superfunctor of tensoring with the super bimodule.
We will go between these two languages freely.
In particular, we will sometimes identify the $(A,B)$-bimodule $M$ with the superfunctor $M\otimes\mbox{--}$ from $B\text{-mod}$ to $A\text{-mod}$ (for us, $\Z$-graded supermodule categories).

\subsubsection{On objects}

The integral weight $\lambda$ is sent to the superalgebra $R^\Lambda(n)$, where $\lambda=\Lambda-2n$.  Note that $R^\Lambda(n)=0$ unless $0\leq n\leq\Lambda$.

\subsubsection{On 1-morphisms}

\begin{itemize}
\item $\Ec\onebbl\mapsto\Res^{n+1}_n$, or the $(n,n+1)$-bimodule $_nR^\Lambda(n+1)$
\item $\onebbl\Fc\mapsto\Ind^{n+1}_n$, or the $(n+1,n)$-bimodule $R^\Lambda(n+1)_n$
\item $\Pi\onebbl\mapsto\Pi_n$
\end{itemize}
In the above, $\Pi_n$ is short for $\Pi_{R^\Lambda(n)}$; the endofunctor $\Pi_A$ of $A\text{-mod}$ shifts the $\Zt$-grading by one and twists the left action by the parity involution $\iota_A$ (see \ref{eqn-parity-involution}).

\subsubsection{On 2-morphisms}

The super-2-category structure morphisms $\xi^{\pm1},\alpha_F^{\pm1}$ are sent to the corresponding structure morphisms in $\SBim$, but we will write all of them explicitly for completeness.
\begin{itemize}
\item On $\xi^{\pm1}$:
\begin{equation}\label{eqn-action-xi}\begin{split}
&\hackcenter{\begin{tikzpicture}
	\draw[thick, color=blue, dashed] (0,0) .. controls (0,.7) and (1,.7) .. (1,0);
\end{tikzpicture}}
\quad\lambda\quad\text{is sent to}\quad\xi:\Pi^2\onebl\to\onebl,\\
&\Pi^2R^\Lambda(n)\to R^\Lambda(n),\qquad x\mapsto x,\\
&\hackcenter{\begin{tikzpicture}
	\draw[thick, color=blue, dashed] (0,1) .. controls (0,.3) and (1,.3) .. (1,1);
\end{tikzpicture}}
\quad\lambda\quad\text{is sent to}\quad\xi^{-1}:\onebl\to\Pi^2\onebl,\\
&R^\Lambda(n)\to\Pi^2R^\Lambda(n),\qquad x\mapsto x,
\end{split}\end{equation}
\item On $\alpha_\Ett^{\pm1}$, $\alpha_\Ftt^{\pm1}$, and $\alpha_\Pi=\alpha_\Pi^{-1}$:
\begin{equation}\label{eqn-action-alpha}\begin{split}
\hackcenter{\begin{tikzpicture}[scale=0.5]
	\draw[thick, ->] (0,0) .. controls (0,1) and (1,1) .. (1,2)
		node[pos=.5, left](){\small$\lambda+2$\;\;}
		node[pos=.5, right](){\;\;\small$\lambda$};
	\draw[thick, color=blue, dashed] (1,0) .. controls (1,1) and (0,1) .. (0,2);
\end{tikzpicture}}
\quad\text{is sent to}\quad \alpha_\Ett:\Ett\Pi\onebl&\to\Pi\Ett\onebl,\\
_{n-1}R^\Lambda(n)\Pi&\to\Pi_{n-1}R^\Lambda(n),\qquad y\mapsto (-1)^{p(y)}y\\
\hackcenter{\begin{tikzpicture}[scale=0.5]
	\draw[thick, color=blue, dashed] (0,0) .. controls (0,1) and (1,1) .. (1,2);
	\draw[thick, ->] (1,0) .. controls (1,1) and (0,1) .. (0,2)
		node[pos=.5, left](){\small$\lambda+2$\;\;}
		node[pos=.5, right](){\;\;\small$\lambda$};
\end{tikzpicture}}
\quad\text{is sent to}\quad \alpha_\Ett^{-1}:\Pi\Ett\onebl&\to\Ett\Pi\onebl,\\
\Pi_{n-1}R^\Lambda(n)&\to_{n-1}R^\Lambda(n)\Pi,\qquad y\mapsto(-1)^{p(y)}y\\
\hackcenter{\begin{tikzpicture}[scale=0.5]
	\draw[thick, <-] (0,0) .. controls (0,1) and (1,1) .. (1,2)
		node[pos=.5, left](){\small$\lambda$\;\;}
		node[pos=.5, right](){\;\;\small$\lambda+2$};
	\draw[thick, color=blue, dashed] (1,0) .. controls (1,1) and (0,1) .. (0,2);
\end{tikzpicture}}
\quad\text{is sent to}\quad \alpha_\Ftt:\onebl\Ftt\Pi&\to\onebl\Pi\Ftt,\\
_nR^\Lambda(n-1)\Pi&\to\Pi_{n-1}R^\Lambda(n),\qquad y\mapsto (-1)^{p(y)}y\\
\hackcenter{\begin{tikzpicture}[scale=0.5]
	\draw[thick, color=blue, dashed] (0,0) .. controls (0,1) and (1,1) .. (1,2);
	\draw[thick, <-] (1,0) .. controls (1,1) and (0,1) .. (0,2)
		node[pos=.5, left](){\small$\lambda$\;\;}
		node[pos=.5, right](){\;\; \small$\lambda+2$};	
\end{tikzpicture}}
\quad\text{is sent to}\quad \alpha_\Ftt^{-1}:\onebl\Pi\Ftt&\to\onebl\Ftt\Pi,\\
\Pi R^\Lambda(n)_{n-1}&\to R^\Lambda(n)_{n-1}\Pi,\qquad y\mapsto(-1)^{p(y)}y.
\end{split}\end{equation}
\item On dots:
\begin{equation}\label{eqn-action-dots}\begin{split}
&\hackcenter{\begin{tikzpicture}
	\draw[thick, ->] (0,0) -- (0,1.5)
		node[pos=.5, shape=coordinate](DOT){}
		node[pos=.25, left](){\small$\lambda+2$\;\;}
		node[pos=.25, right](){\;\;\small$\lambda$};
	\draw[thick, color=blue, dashed] (DOT) [out=135, in=-90] to (-.5,1.5);
	\node at (DOT) {\bbullet};
\end{tikzpicture}}
\quad\text{is sent to}\quad x_\Ett:\Ett\onebl\to\Pi\Ett\onebl,\\
_{n-1}R^\Lambda(n)&\to\Pi_{n-1}R^\Lambda(n),\qquad y\mapsto x_ny,\\
&\hackcenter{\begin{tikzpicture}
	\draw[thick, <-] (0,0) -- (0,1.5)
		node[pos=.5, shape=coordinate](DOT){}
		node[pos=.25, left](){\small$\lambda$\;\;}
		node[pos=.25, right](){\;\;\small$\lambda+2$};
	\draw[thick, color=blue, dashed] (DOT) [out=135, in=-90] to (.5,1.5);
	\node at (DOT) {\bbullet};
\end{tikzpicture}}
\quad\text{is sent to}\quad x_\Ftt:\onebl\Ftt\to\onebl\Ftt\Pi,\\
R^\Lambda(n)_{n-1}&\to R^\Lambda(n)_{n-1}\Pi,\qquad y\mapsto yx_n.
\end{split}\end{equation}
\item On up- and down-crossings:
\begin{equation}\label{eqn-action-crossings}\begin{split}
&\hackcenter{\begin{tikzpicture}
	\draw[thick, ->] (0,0) .. controls (0,.5) and (.5,1) .. (.5,1.5)
		node[pos=.5, shape=coordinate](CROSSING){}
		node[pos=.25, left](){\small$\lambda+4$\;\;};
	\draw[thick, color=blue, dashed] (CROSSING) [out=135, in=-90] to (-.5,1.5);
	\draw[thick, ->] (.5,0) .. controls (.5,.5) and (0,1) .. (0,1.5)
		node[pos=.25, right](){\;\;\small$\lambda$};
\end{tikzpicture}}
\quad\text{is sent to}\quad \tau_\Ett:\Ett^2\onebl\to\Pi\Ett^2\onebl,\\
_{n-2}R^\Lambda(n)&\to\Pi_{n-2}R^\Lambda(n),\qquad y\mapsto \tau_{n-1}y\\
&\hackcenter{\begin{tikzpicture}
	\draw[thick, <-] (0,0) .. controls (0,.5) and (.5,1) .. (.5,1.5)
		node[pos=.5, shape=coordinate](CROSSING){}
		node[pos=.25, left](){\small$\lambda$\;\;};
	\draw[thick, color=blue, dashed] (CROSSING) [out=135, in=-90] to (-.5,1.5);
	\draw[thick, <-] (.5,0) .. controls (.5,.5) and (0,1) .. (0,1.5)
		node[pos=.25, right](){\;\;\small$\lambda+4$};
\end{tikzpicture}}
\quad\text{is sent to}\quad \tau_\Ftt:\onebl\Ftt^2\to\onebl\Ftt^2\Pi,\\
R^\Lambda(n)_{n-2}&\to R^\Lambda(n)_{n-2}\Pi,\qquad y\mapsto y\tau_{n-1}.
\end{split}\end{equation}
\item On the left cap and left cup:
\begin{equation}\label{eqn-action-left-cap-cup}\begin{split}
&\hackcenter{\begin{tikzpicture}
	\draw[thick, ->] (1,0) .. controls (1,.8) and (0,.8) .. (0,0)
		node[pos=.25, right](){\;\;\small$\lambda$};
\end{tikzpicture}}
\quad\text{is sent to}\quad\epsilon:\Ftt\Ett\onebl\to\onebl,\\
&R^\Lambda(n)\otimes_{n-1}R^\Lambda(n)\to R^\Lambda(n),\qquad w\otimes y\mapsto wy\\
&\hackcenter{\begin{tikzpicture}
	\draw[thick, ->] (1,1) .. controls (1,.2) and (0,.2) .. (0,1)
		node[pos=.25, right](){\;\;\small$\lambda$};
\end{tikzpicture}}
\quad\text{is sent to}\quad\eta:\onebl\to\Ett\Ftt\onebl,\\
&R^\Lambda(n)\to\hspace{.01in}_nR^\Lambda(n+1)_n,\qquad w\mapsto w.\\
\end{split}\end{equation}
\item The action of the right cap and right cup is determined, as described in Section \ref{sec-formal}, by the action of the left cap and left cup.  See Section 8 of \cite{KKO} for descriptions of these maps.
\end{itemize}

\subsubsection{Summary}

We summarize this subsection as follows.
\begin{thm} For each dominant integral weight $\Lambda$, there is a super-2-functor
\begin{equation}
\Udotc\to\ONHc^\Lambda
\end{equation}
defined by equations \eqref{eqn-action-xi}--\eqref{eqn-action-left-cap-cup}.  After taking Grothendieck groups, this action becomes the action of $\Udotpi$ on its integrable simple module $_{\Ac_\pi}V^\Lambda$.\end{thm}

%---------------------------------------------------------------------
\subsection{Consequences of the action}\label{subsec-consequences}
%---------------------------------------------------------------------

We have constructed an upper bound for the space of Homs in the 2-category $\Uc$ in Section~\ref{subsec:upper}.   The existence of the 2-representation of $\Uc$ into cyclotomic quotients provides a lower bound for the space of Homs that can be explicitly computed in certain degrees.  For example, by varying the integral weight $\Lambda$ the fact that $R^\Lambda(n)$ is a brick in the cyclotomic 2-category immediately implies the following corollary.

\begin{cor}\label{cor-ef-bricks} In the 2-category $\Udotc$ the 1-morphisms $\onebbl$ are bricks. That is, \[
\Hom_{\Udotc}(\onebbl, \onebbl \la \ell\ra) \cong
\left\{\begin{array}{ll}
  0 & \ell<0 \\
  \Bbbk & \ell=0.
\end{array}\right.
\]
\end{cor}

%#####################################################################
%
\section{(De)categorification}
%
%#####################################################################

%---------------------------------------------------------------------
\subsection{Indecomposables in $\Udotc$}
%---------------------------------------------------------------------

\begin{prop}
The divided power 1-morphisms $\Ec^{(a)}\onebbl$ and $\Fc^{(b)}\onebbl$ are bricks for all $a,b \geq 0$.
\end{prop}

\begin{proof}
This follows by induction from Corollaries~\ref{cor:homonel} and \ref{cor-ef-bricks} following arguments similar to those in \cite[Lemma 4.9]{CKL2}.
\end{proof}

\begin{prop} \label{prop_indecomp}
The 1-morphisms
\begin{enumerate}[(i)]
     \item $\Ec^{(a)}\Fc^{(b)}\onebbl\la s\ra \quad $ for $a$,$b\in \N$, $\l,s \in\Z$,
     $\l\leq b-a$,
     \item $\Fc^{(b)}\Ec^{(a)}\onebbl\la s\ra \quad$ for $a$,$b\in\N$, $\l,s \in\Z$, $\l\geq
     b-a$,
\end{enumerate}
are indecomposable. Furthermore, these indecomposables are not isomorphic unless
$\l=b-a$ in which case $ \Ec^{(a)}\Fc^{(b)}\onebb_{b-a}\la s\ra \cong
\Fc^{(b)}\Ec^{(a)}\onebb_{b-a}\la s\ra$.  In other words, the indecomposable projectives of the super-2-category $\Udotc$ categorify the canonical basis of \cite{ClarkWang}.
\end{prop}

\begin{proof}
It is straightforward to show that these 1-morphisms are bricks using adjointness and the fact that divided powers are bricks (see for example \cite[Proposition 9.9]{Lau1}).  This implies indecomposability.  Furthermore, one can show that the space of Homs between any two such elements is positively graded (in $q\N[\pi,q]$) implying that they are pairwise non-isomorphic (see for example \cite[Proposition 9.9]{Lau1}).
\end{proof}

\begin{prop} \label{prop_Uindec}\hspace{2in}
\begin{enumerate}[(i)]
  \item \label{item_indec1} Every 1-morphism $x$ in $\Hom_{\Udotc}(\l,\l')$ decomposes as a direct sum of indecomposable 1-morphisms of the form
  \begin{eqnarray}
     \onebb_{\l'}\Pi^s \Ec^{(a)}\Fc^{(b)}\onebbl\la t\ra &\quad&  \text{for $a$,$b\in \N$, $s,t \in\Z$,
     $\l\leq b-a$, } \nn\\
     \onebb_{\l'}\Pi^s \Fc^{(b)}\Ec^{(a)}\onebbl\la t\ra &\quad& \text{for $a$,$b\in\N$, $s,t \in\Z$, $\l \geq
     b-a$,} \label{eq_Bdot}
\end{eqnarray}
where $\lambda'=\lambda-2(b-a)$.
  \item \label{item_indec2} The direct sum decomposition of $x \in
 \Hom_{\Udotc}(\l,\l')$ is essentially unique, meaning that the indecomposables
 and their multiplicities are unique up to reordering the factors.
 \item \label{item_indec3}The morphisms in \eqref{item_indec1} \eqref{eq_Bdot} above are the only indecomposables in $\Udotc$ up to isomorphism.
 \item \label{item_indec4} The 1-morphisms $\Ec^{(a)}\Fc^{(b)}\onebb_{b-a}\la t \ra$ and
$\Fc^{(b)}\Ec^{(a)}\onebb_{b-a}\la t\ra$ are isomorphic in $\Udotc$.
\end{enumerate}
\end{prop}

\begin{proof}
To prove \eqref{item_indec1} it suffices to show that any element
$x=\onebb_{\l'} \Pi^{s_1}\Ec^{\alpha_1}\Fc^{\beta_1}\Pi^{s_2}\Ec^{\alpha_2} \cdots
 \Fc^{\beta_{k-1}}\Pi^{s_k}\Ec^{\alpha_k}\Fc^{\beta_k}\onebbl\la t \ra$ in $\Uc$
decomposes as a sum of elements in \eqref{eq_Bdot}.  Using the super-2-category structure the 1-morphism $x$ is isomorphic to a 1-morphism of the form
$x'=\onebb_{\l'} \Pi^{s'}\Ec^{\alpha_1}\Fc^{\beta_1}\Ec^{\alpha_2} \cdots
 \Fc^{\beta_{k-1}}\Ec^{\alpha_k}\Fc^{\beta_k}\onebbl\la t \ra$ for $s'=s_1+s_2+\dots + s_k$. Following the arguments in \cite[Proposition 9.10]{Lau1} completes the proof of (\ref{item_indec1}).

The Krull-Schmidt theorem then establishes \eqref{item_indec2},
and \eqref{item_indec3} (see Chapter I of \cite{Benson}).  The proof of \eqref{item_indec4} is identical to the proof of \cite[Corollary 9.11]{Lau1}.
\end{proof}

%---------------------------------------------------------------------
\subsection{Main theorem---categorification}
%---------------------------------------------------------------------

\begin{thm} \label{thm_Groth}
The split Grothendieck group $K_0(\Udotc)$ is isomorphic as an
$\Ac_\pi$-module to the integral covering algebra $\AUdotpi$ introduced by Clark and Wang \cite{ClarkWang}.
\end{thm}

\begin{proof}
Since $\Udotc$ has the Krull-Schmidt property (Proposition~\ref{prop_Uindec}), its
Grothendieck group is freely generated as an $\Ac_\pi$-module by the
isomorphism classes of indecomposables with no shifts.  We have shown that these
isomorphism classes of indecomposables correspond bijectively to elements in the
canonical basis of the covering algebra introduced by Clark and Wang.   Therefore, the homomorphism $\gamma \maps \AUpi \to K_0(\Udotc)$ from Proposition~\ref{prop-gamma} is an isomorphism.
\end{proof}

% ====================================================================
% REFERENCES

%\bibliographystyle{plain}
%\bibliography{ellis-bib}

%
% ====================================================================

%
\end{document}